\documentclass[psamsfonts]{amsart}
\usepackage{amssymb,amsfonts}

\usepackage{amsthm}
\usepackage{tcolorbox,longtable}
\usepackage{appendix}

\newtheorem{thm}{Theorem}[section]
\newtheorem{defi}[thm]{Definition}
\newtheorem{alg}[thm]{Algorithm}
\newtheorem{cor}[thm]{Corollary}
\newtheorem{example}[thm]{Example}
\newtheorem{rem}[thm]{Remark}

\makeatletter
\let\c@equation\c@thm
\makeatother
\numberwithin{equation}{section}

\makeatletter
\ifx\SK@label\undefined\let\SK@label\label\fi
 \let\your@thm\@thm
 \def\@thm#1#2#3{\gdef\currthmtype{#3}\your@thm{#1}{#2}{#3}}
 \def\mylabel#1{{\let\your@currentlabel\@currentlabel\def\@currentlabel
  {\currthmtype~\your@currentlabel}
 \SK@label{#1@}}\label{#1}}
 
\makeatother

\title{The algebraic Atiyah-Hirzebruch spectral sequence of real projective spectra}

\author{Guozhen Wang}
\address{Department of Mathematics, University of Copenhagen, Universitetsparken 5, 2100 Copenhagen, Denmark}
\email{guozhen@math.ku.dk}

\author{Zhouli Xu}
\address{Department of Mathematics, The University of Chicago, Chicago, IL 60637}
\email{xu@math.uchicago.edu}
\begin{document}

\maketitle

\begin{abstract}
In this note, we use Curtis's algorithm and the Lambda algebra to compute the algebraic Atiyah-Hirzebruch spectral sequence of the suspension spectrum of $\mathbb{R}P^\infty$ with the aid of a computer, which gives us its Adams $E_2$-page in the range of $t<72$. We also compute the transfer map on the Adams $E_2$-pages. These data are used in our computations of the stable homotopy groups of $\mathbb{R}P^\infty$ in \cite{WX1} and of the stable homotopy groups of spheres in \cite{WX2}.
\end{abstract}

This note gives computer-generated computations to be used in \cite{WX1} and \cite{WX2}. The data here are ``mindless" input to those papers, input that a computer can generate without human intervention. The papers \cite{WX1} and \cite{WX2} compute differentials, starting from the data presented here. We are minded to quote Frank Adams \cite[page 58-59]{Ada} from 1969:

\emph{`` $\cdots$ The history of the subject [algebraic topology] shows, in fact, that whenever a chance has arisen to show that a differential $d_r$ is non-zero, the experts have fallen on it with shouts of joy - `Here is an interesting phenomenon! Here is a chance to do some nice, clean research!' - and they have solved the problem in short order. On the other hand, the calculation of $Ext^{s,t}$ groups is necessary not only for this spectral sequence, but also for the study of cohomology operations of the $n$-th kind: each such group can be calculated by a large amount of tedious mechanical work: but the process finds few people willing to take it on. $\cdots$"}\\

\textbf{Warning}: this note contains data of interest only to experts.

\section{Notations}

We work at the prime $2$ in this paper. All cohomology groups are taken with coefficients $\mathbb{Z}/2$.

Let $\mathcal{A}$ be the Steenrod algebra.
For any $\mathcal{A}$-module $M$, we will abbreviate $Ext_{\mathcal{A}}(M,\mathbb{Z}/2)$ by $Ext(M)$.

Let $V$ be a vector space with $\{v_j\}$ an ordered basis. We say that an element $v=\sum a_iv_i$ has leading term $a_kv_k$ if $k$ is the largest number for which $a_k\neq0$.

For spectra, let $S^0$ be the sphere spectrum, and $P_1^\infty$ be the suspension spectrum of $\mathbb{R}P^\infty$. In general, we use $P_n^{n+k}$ to denote the suspension spectrum of $\mathbb{R}P^{n+k}/\mathbb{R}P^{n-1}$.

\section{The Curtis table}

We recall the notion of Curtis table in a general setting in this section.

Let $X_0\rightarrow X_1\rightarrow\dots$ be a complex of vector spaces (over $\mathbb{F}_2)$. For each $X_i$, let $\{x_{i,j}\}$ be an ordered basis.

\begin{defi} \label{cus}
A Curtis table for $X_*$ associated with the basis $\{x_{i,j}\}$ consists of a list $L_i$ for each $i$.

The items on the list $L_i$ are either an element $x_{i,j}$ for some $j$, or a tag of the form $x_{i,j}\leftarrow x_{i-1,k}$ for some $j,k$.

These lists satisfy the following:
\begin{enumerate}
\item Each element $x_{i,j}$ appears in these lists exactly once.\
\item For any $i,j$, an item of the form $x_{i,j}$ or a tag of the form $x_{i,j}\leftarrow x_{i-1,k}$  appears in the list $L_i$ if an only if there is a cycle in $X_i$ with leading term $x_{i,j}$.
\item If a tag of the form $x_{i,j}\leftarrow x_{i-1,k}$ appears in the list $L_i$, then there is an element in $X_{i-1}$ with leading term $x_{i-1,k}$ whose boundary has leading term $x_{i,j}$.
\end{enumerate}
\end{defi}

\begin{rem}
	By Theorem \ref{cut} and Corollary \ref{unq}, the Curtis table exists and is unique for a finite dimensional complex with ordered basis.
\end{rem}

The Curtis algorithm constructs a Curtis table from a basis, and can output the full cycle from the input of a leading term.

For example, the Curtis table in the usual sense is for the lambda algebra with the basis of admissible monomials in lexicographic order. In \cite{Tan} Tangora computed the Curtis table for the lambda algebra up to stem $51$.

Another example is the minimal resolution for the sphere spectrum. This case is indeed trivial in the sense that there are no tags in the Curtis table.

\section{The Curtis algorithm}

The Curtis algorithm produces a Curtis table from an ordered basis. It can be described as follows:
\begin{alg}(Curtis)
\begin{enumerate}
\item For each $i$, construct a list $L_i$ which contains every $x_{i,j}$ such that the items are ordered with $j$ ascending.
\item For $i=0,1,2,\dots$ do the following:
\begin{enumerate}
\item Construct a pointer $p$ with initial value pointing to the beginning of $L_i$.
\item If $p$ points to the end of $L_i$ (i.e. after the last element), stop and proceed to the next $i$. \label{st1}
\item If the item pointed by $p$ is tagged, move $p$ to the next item and go to Step \ref{st1}.
\item Construct a vector $c\in X_i$. Give $c$ the initial value of the item pointed by $p$.
\item Compute the boundary $b\in X_{i+1}$ of $c$. \label{st2}
\item If $b=0$, move $p$ to the next item and go to Step \ref{st1}.
\item Search the leading term $y$ of $b$ in $L_{i+1}$.
\item If $y$ is untagged, tag $y$ with the leading term of $c$. Remove the item pointed by $p$ and move $p$ to the next item. Go to Step \ref{st1}.
\item If $y$ is tagged by $z$, add $z$ to $c$. Go to Step \ref{st2}.
\end{enumerate}
\end{enumerate}
\end{alg}

\begin{example}
As an example, we compute the Curtis table for the lambda algebra for $t=3$. We start with
\begin{equation*}
\begin{split}
L_1 & = \{\lambda_2\} \\
L_2 & = \{\lambda_1\lambda_0\} \\
L_3 & = \{\lambda_0^3\}
\end{split}
\end{equation*}
We next compute the boundary of $\lambda_2$:
$$d(\lambda_2) = \lambda_1\lambda_0.$$
We therefore remove it from $L_1$ and tag $\lambda_1\lambda_0$ with $\lambda_2$. The output gives us the following:
\begin{equation*}
\begin{split}
L_1 & = \emptyset \\
L_2 & = \{\lambda_1\lambda_0\leftarrow\lambda_2\} \\
L_3 & = \{\lambda_0^3\}
\end{split}
\end{equation*}
\end{example}

\begin{thm}(Curtis) \label{cut}
The Curtis algorithm ends after finitely many steps when $X_*$ is finite dimensional. Moreover, let $Y_*$ be the graded vector space generated by those untagged items on the $L_i$'s. Denote by $C_*$ the subspace of cycles in $X_*$. There is an algorithm which constructs a map $Y_i\rightarrow C_i$ and a map $C_i\rightarrow Y_i$ which induce an isomorphism between $Y_*$ and the homology of $X_*$.
\end{thm}
\begin{proof}
See \cite{Tan}.
\end{proof}

\begin{cor} \label{unq}
The Curtis table is unique for a finite dimensional complex $X_*$ with ordered basis. In fact, it is specified in the following way:

Let $l(x)$ denote the leading term of $x$.

If there is a tag $a\leftarrow b$, then $a$ is the minimal element of the set $\{l(d(x))|l(x)=b\}$.

If an item $a$ is untagged, then it is the leading term of an element with lowest leading term in a homology class.
\end{cor}
\begin{proof}
See \cite{Tan}.
\end{proof}

\section{Curtis table and spectral sequences}

Now suppose $V$ is a filtered vector space with $\dots\subset F_iV\subset F_{i+1}V\dots\subset V$. We call an ordered set of basis $\{v_k\}$ compatible if for any $i$ there is a $k_i$ such that $F_iV$ is spanned by $\{v_k:k\leq k_i\}$.

Let $X_0\rightarrow X_1\rightarrow\dots$ be a complex of filtered vector spaces such that the differentials preserve the filtration. Then there is a spectral sequence converging to the homology of $X$ with the $E_1$-term $F_kX_i/F_{k-1}X_i$. Suppose we have compatible bases $\{x_{i,j}\}$ of $X_i$.

\begin{thm} \label{css}
The Curtis table of $X_*$ consists of the following:
\begin{enumerate}
\item The tags of the Curtis table for $(E_r,d_r)$ of the spectral sequence, for all $r\geq1$.
\item The untagged items from the $E_\infty$-term.
\end{enumerate}
Here we label the basis of $E_r$ as the following. In the $E_1$-page, we use the image of the $x_{i,j}$'s as the basis, and label them by the same name. Inductively, we use Theorem \ref{cut} to label a basis of $E_r$ by the untagged items in the Curtis table of $E_{r-1}$.
\end{thm}
\begin{proof}
We check the conditions of Definition \ref{cus}. They follow directly from the definition of the spectral sequence, the conditions for the Curtis tables of the $E_r$'s, and Theorem \ref{cut}.
\end{proof}

Consequently, we can identify the Curtis table with the table for the differentials and permanent cycles of the spectral sequence. For example, in the lambda algebra, we have a filtration by the first number of an admissible sequence. The induced spectral sequence is the algebraic EHP sequence. So the usual Curtis table can be identified with the algebraic EHP sequence. See \cite{CGM} for more details.

In practice, the Curtis table for the $E_1$ terms is often known before hand. Then we could skip those part of the Curtis algorithm dealing with the tags coming from the $E_1$ term. And we often omit this part in the output of Curtis table.

\section{The algebraic Atiyah-Hirzebruch spectral sequence} \label{s2}

Let $X$ be a spectrum. There is a filtration on $H^*(X)$ by the degrees. For any $n$ there is a short exact sequence $0\rightarrow H^{\geq n+1}(X)\rightarrow H^{\geq n}(X)\rightarrow H^{n}(X)\rightarrow0$. This induces a long exact sequence
$$\dots\rightarrow Ext(\mathbb{Z}/2)\otimes H^n(X)\rightarrow Ext(H^{\geq n}(X))\rightarrow Ext(H^{\geq n+1}(X))\rightarrow \dots$$
Combining the long exact sequences for all $n$ we get the algebraic Atiyah-Hirzebruch spectral sequence $$\oplus_n Ext(\mathbb{Z}/2)\otimes H^n(X) \Rightarrow Ext(H^*(X))$$

There is another way to look at the algebraic Atiyah-Hirzebruch spectral sequence.

Let us fix a free resolution $\dots\rightarrow F_1\rightarrow F_0\rightarrow\mathbb{F}_2$ of $\mathbb{F}_2$ as $\mathcal{A}$-modules. For example, we can take $F_*$ to be the Koszul resolution, which gives the lambda algebra constructed in \cite{B6}. We can also take $F_*$ to the the minimal resolution.

Then for $X$ a finite CW spectrum, we can identify $R\mathcal{H}om_\mathcal{A}(H^*(X),\mathbb{Z}/2)$ with the complex $C^\ast(H^\ast(X)) = Hom_{\mathcal{A}}(H^*(X)\otimes_{\mathbb{F}_2}F_*,\mathbb{F}_2)$ where we take the diagonal action of the Steenrod algebra on $H^*(X)\otimes_{\mathbb{F}_2}F_*$ using the Cartan formula.

The cell filtration on $H_*(X)$ induces a filtration on $H^*(X)\otimes_{\mathbb{F}_2}F_*$, and we can identify the algebraic Atiyah-Hirzebruch spectral sequence with the spectral sequence generated by this filtration.
In fact, the map $H^*(X)\otimes_{\mathbb{F}_2}F_*\rightarrow H^*(X)$ preserves these filtrations and induces a quasi-isomorphism on each layer. So they define equivalent sequences in the derived category, hence generate the same spectral sequence.

\section{The Curtis algorithm in computing the algebraic Atiyah-Hirzebruch spectral sequence}

Let $X$ be a finite CW spectrum.

Let $r^*_{i,j}\in F_i$ be a set of $\mathcal{A}$-basis for the free $\mathcal{A}$-module $F_i$. Let $r_{i,j}\in Hom_\mathcal{A}(F_i,\mathbb{Z}/2)$ be the dual basis.

We choose an ordered $\mathbb{F}_2$-basis $e^*_k$ of $H^*(X)$ such that elements with lower degrees come first. Let $e_k\in H_*(X)$ be the dual basis.
Then the set $\{e^*_k\otimes r^*_{i,j}\}$ is a set of $\mathcal{A}$-basis for $H^*(X)\otimes_{\mathbb{F}_2}F_*$. Let $e_k\otimes r_{i,j}\in Hom_{\mathcal{A}}(H^*(X)\otimes_{\mathbb{F}_2}F_*,\mathbb{F}_2)$ be the dual basis with the lexicographic order.

The following is a corollary of Theorem \ref{css}.
\begin{thm} The Curtis table for $C^\ast(H^\ast(X)) = Hom_{\mathcal{A}}(H^*(X)\otimes_{\mathbb{F}_2}F_*,\mathbb{F}_2)$ satisfies
\begin{enumerate}
\item If there is a tag $a\leftarrow b$ in the Curtis table of $Hom_\mathcal{A}(F_i,\mathbb{Z}/2)$, there are tags of the form $e_k\otimes a\leftarrow e_k\otimes b$. \label{tag1}
\item The table of all tags which are not contained in Case \ref{tag1} is the same as the table for the algebraic  Atiyah-Hirzebruch differentials of $X$.
\item The items not contained in the previous cases are untagged items. They correspond to the permanent cycles in the algebraic Atiyah-Hirzebruch spectral sequence.
\end{enumerate}
\end{thm}

Consequently, we can read off the $E_2$-term of the Adams spectral sequence of any truncation of $X$.
\begin{thm}
Let $X_m^n$ be the truncation of $X$ which consists of all cells of $X$ in dimensions between (and including) $m$ and $n$. Therefore in the Curtis table of $X_m^n$, all the tags are those tags in the Curtis table of $X$ lying within the corresponding range. (Note there could be more untagged items, which are just those not appearing in any tags.)
\end{thm}

\begin{proof}
This follows from the previous theorem because the Atiyah-Hirzebruch spectral sequence is truncated this way.
\end{proof}

We present two examples. The latter one is used in our computation in \cite{WX2} that the 2-primary $\pi_{61} = 0$. For notation, in the Lambda algebra, we will abbreviate an element $\lambda_{i_1}\dots\lambda_{i_n}$ by $i_1 \dots i_n$. In the Lambda complex of $P_1^\infty$, we will abbreviate an element $e_k\otimes\lambda_{i_1}\dots\lambda_{i_n}$ by $(k)i_1 \dots i_n$. The Curtis table is separated into lists labeled by $(t-s,t)$ on the top, in which those untagged items give a basis for $Ext^{s-1,t-1}(H^*(P_1^\infty))$.

\begin{example}
As a relatively easy example, we compute $Ext^{2,2+9}(H^*(P_2^8))$ using the Curtis table of $P_1^\infty$ in the Appendix.

There are only two boxes that are used in this computation: the ones labeled with $(9,3)$ and $(8,4)$. The box labeled with $(9,3)$ is the following:
\begin{equation*}
\begin{split}
 & (1)\;5\;3 \\
 & (3)\;3\;3 \\
 & (7)\;1\;1 \leftarrow (9)\;1
\end{split}
\end{equation*}

The spectrum $P_2^8$ only has cells in dimensions 2 through 8. We remove the item $(1)\;5\;3$, since it comes from the cell in dimension 1. We also remove the tag $(9)\;1$, since it comes from the cell in dimension 9. Therefore, the only items remaining in this box are $(3)\;3\;3$ and $(7)\;1\;1$.\\

The box labeled with $(8,4)$ is the following:
\begin{equation*}
\begin{split}
 & (1)\;5\;1\;1 \leftarrow (2)\;6\;1\\
 & (5)\;1\;1\;1 \leftarrow (6)\;2\;1
\end{split}
\end{equation*}

After removing the element $(1)\;5\;1\;1$, which comes from the cell in dimension 1, the element $(2)\;6\;1$ tags nothing. We move the element $(2)\;6\;1$ from the box labeled with $(8,4)$ to the one labeled with $(9,3)$. Therefore, we have the conclusion that the group $Ext^{2,2+9}(H^*(P_2^8))$ has dimension $3$, generated by
$$(3)\;3\;3, \ (7)\;1\;1, \text{~~and~~} (2)\;6\;1.$$

One can even recover the names of these generators in the algebraic Atiyah-Hirzebruch spectral sequence. See Notation 3.3 in \cite{WX2} for the notation. In $Ext(\mathbb{Z}/2)$, the elements $3\;3$, $1\;1$ and $6\;1$ all lie in the bidegrees which contain only one nontrivial element. Therefore, we can identify their Adams $E_2$-page names as $h_2^2$, $h_1^1$ and $h_0h_3$. This gives us the algebraic Atiyah-Hirzebruch $E_1$-page names of these generators:
$$h_2^2[3], \ h_1^1[7], \text{~~and~~} h_0h_3[2].$$
\end{example}

\begin{example}
We present the computation of the Adams $E_2$ page of $P_{16}^{22}$ in the 60 and 61 stem for $s\leq 7$, which is used in the proof of Lemma 8.2 in \cite{WX2}. The boxes that are used in this computation have the following labels:
$$(59, s) \text{~~for~~} s\leq 7, \text{~~and~~}(60, s'), (61, s') \text{~~for~~} s\leq 8.$$
The spectrum $P_{16}^{22}$ consists of cells in dimensions 16 through 22.\\

We start with the 60 stem.\\

We have $Ext^{1,1+60}(P_{16}^{22}) = Ext^{2,2+60}(P_{16}^{22}) = 0$, since the boxes labeled with $(60,2)$, $(59,3)$ and $(60,3)$, $(59,4)$ becomes empty.\\

We have $Ext^{3,3+60}(P_{16}^{22}) = \mathbb{Z}/2$, generated by $(19)\;11\;15\;15$ from the box labeled with $(59,5)$. The box labeled with $(60,4)$ becomes empty. Since $11\;15\;15 \in Ext^{3,3+41} = \mathbb{Z}/2$, generated by $c_2$, we identify $(19)\;11\;15\;15$ with its Atiyah-Hirzebruch name $c_2[19]$.\\

We have $Ext^{4,4+60}(P_{16}^{22}) = \mathbb{Z}/2 \oplus \mathbb{Z}/2$, generated by $(16)\;13\;13\;11\;7$ from the box labeled with $(59,6)$, and by $(20)\;19\;7\;7\;7$ from the box labeled with $(60,5)$. We find their Atiyah-Hirzebruch names $g_2[16]$ and $f_1[20]$.\\

We have $Ext^{5,5+60}(P_{16}^{22}) = \mathbb{Z}/2 \oplus \mathbb{Z}/2$, generated by $(16)\;11\;14\;5\;7\;7$ and $(21)\;7\;13\;5\;7\;7$ from the box labeled with $(59,7)$. The box labeled with $(60,6)$ becomes empty. We find their Atiyah-Hirzebruch names $h_0g_2[16]$ and $h_1e_1[21]$.\\

We have $Ext^{6,6+60}(P_{16}^{22}) = \mathbb{Z}/2 \oplus \mathbb{Z}/2 \oplus \mathbb{Z}/2$, generated by $(16)\;7\;14\;4\;5\;7\;7$ from the box labeled with $(59,8)$, and by $(20)\;5\;5\;9\;7\;7\;7$ and $(22)\;3\;5\;9\;7\;7\;7$ from the box labeled with $(60,7)$. We find their Atiyah-Hirzebruch names $h_0^2g_2[16]$, $h_0^2f_1[20]$ and $h_1x[22]$.\\

We have $Ext^{7,7+60}(P_{16}^{22}) = \mathbb{Z}/2$, generated by $(21)\;3\;5\;9\;3\;5\;7\;7$ from the box labeled with $(60,8)$. The box labeled with $(59,9)$ becomes empty. We find its Atiyah-Hirzebruch name $h_1y[21]$.\\

Similarly, one can compute the 61 stem. The computation is summarized in the following Table 1.

\pagebreak

\begin{table}[]
\caption{The Adams $E_2$ page of $P_{16}^{22}$ in the 60 and 61 stems for $s\leq 7$}
\centering
\begin{tabular}{ l l l }
$s\backslash t-s$ & 60 & 61 \\ [0.5ex] 
\hline 
7 & $h_1y[21]$ & $h_0^2h_5d_0[16]$\\
  & & $h_1y[22]$ \\ \hline
6 & $h_0^2g_2[16]$ & $h_0h_5d_0[16]$\\
  & $h_0^2f_1[20]$ & $Ph_2h_5[19]$ \\
  & $h_1x[22]$ &  \\ \hline
5 & $h_0g_2[16]$ & $h_1g_2[16]$\\
  & $h_1e_1[21]$ & $h_5d_0[16]$\\
  & & $h_1f_1[20]$\\
  & & $h_1h_5c_0[21]$\\
  & & $h_3d_1[22]$\\ \hline
4 & $g_2[16]$ & $h_0h_4^3[16]$\\
  & $f_1[20]$ & $g_2[17]$\\
  & & $f_1[21]$\\
  & & $h_1^2h_3h_5[21]$\\
  & & $h_5c_0[22]$\\ \hline
3 & $c_2[19]$ & $h_4^3[16]$\\
  & & $h_1h_3h_5[22]$ \\
\end{tabular}
\label{X}
\end{table}
\end{example}

\section{The homomorphism induced by a map}

Let $f:X\rightarrow Y$ be a map which induces the zero map on homology. Let $Z$ be the cofiber of $f$. Then the homology of $Z$ can be identified with the direct sum of $H_*(X)$ and $H_*(Y)$ as a vector space. If $x_1,\dots,x_k$ is an ordered basis of $H_*(X)$ and $y_1,\dots,y_l$ is an ordered basis of $H_*(Y)$, then $y_1,\dots,y_l,x_1,\dots,x_k$ is an ordered basis of $H_*(Z)$ with certain degree shifts. Note we do not make elements with lower degree go first here. Instead elements $y_i$ always go before elements $x_j$ regardless of degree.

Note that in this case, there is a map of Adams spectral sequence of $X$ and $Y$ which raises the Adams filtration by one, and on the $E_2$ page it is the boundary homomorphism for the Ext group for the exact sequence $0\rightarrow H^{*+1}(X)\rightarrow H^*(Z)\rightarrow H^*(Y)\rightarrow 0$. We call this the map induced by $f$.

\begin{thm} \label{t71}
The Curtis table for $C^\ast(H^\ast(Z))= Hom_{\mathcal{A}}(H^*(Z)\otimes_{\mathbb{F}_2}F_*,\mathbb{F}_2)$ from Section 5 with this ordered basis satisfies
\begin{enumerate}
\item All of the tags in the Curtis table for $C^\ast(H^\ast(X))$ and for $C^\ast(H^\ast(Y))$ also appear in the Curtis table for $C^\ast(H^\ast(Z))$.
\item The remaining tags give the table for the homomorphism on the Adams $E_2$-page induced by $f$.
\item The untagged items give basis for the kernel and cokernel of the homomorphism induced by $f$.
\end{enumerate}
\end{thm}

\begin{proof}
This follows from Theorem \ref{css} by using the filtration $Y\subset Z$, and identifying the $d_2$-differential with the attaching map $X\rightarrow Y$.
\end{proof}

So we can use the Curtis algorithm to compute the homomorphism induced by a map.

\section{The algebraic Atiyah-Hirzebruch spectral sequence of the real projective spectra}

We use the Curtis algorithm to compute the algebraic Atiyah-Hirzebruch spectral sequence for the real projective spectra. We take the lambda algebra for the resolution of $\mathbb{Z}/2$ and use the usual Curtis table for the sphere spectrum as input. We have carried out the computation through stems with $t < 72$. As a usual convention to out put the Curtis table, we abbreviate the sequence $2\;4\;1\;1$ by $*$; when there are multiple $2$'s consecutively, we replace them by the same amount of dots.

Together with the algebraic Kahn-Priddy theorem \cite{Lin} and known information of $Ext(\mathbb{Z}/2)$, this gives the Adams $E_2$-page of $P_1^\infty$ up to $t-s \leq 61$.

We also compute the transfer map. Recall that the fiber of the transfer map has one more cell than $P_1^\infty$ in dimension $-1$, and all the $Sq^i$ acts nontrivially on the class in dimension $-1$. We will use Theorem \ref{t71} to identify the table for transfer with a portion of the Curtis table for this complex.

For notation, in the Lambda algebra, we will abbreviate an element $\lambda_{i_1}\dots\lambda_{i_n}$ by $i_1 \dots i_n$. We will abbreviate an element $e_k\otimes\lambda_{i_1}\dots\lambda_{i_n}$ by $(k)i_1 \dots i_n$ in the Lambda complex of $P_1^\infty$. The symbol $o$ means zero. The Curtis table is separated into lists labeled by $(t-s,t)$ on the top, in which the untagged items give a basis for $Ext^{s-1,t-1}(H^*(P_1^\infty))$.

The table for the transfer is the output of the algorithm: (We put the table for the transfer map first since it is shorter)

In this table we only list the nontrivial items. Others either map to something with the same name, or to the only choice comparable with the algebraic Kahn-Priddy theorem. For example, $(1)$ maps to $1$, i.e. the inclusion of the bottom cell maps to $\eta$. As another example, $(5)~3$ maps to $5~3$, which can be proved independently by the Massey product
$$\langle h_2, h_1, h_2\rangle = h_1h_3.$$

We do not include such items in the transfer table.

\begin{appendices}

\pagebreak

\begin{longtable}{|l|l|}
\caption{The table for the transfer map}\\
	\hline
	(1) 7 & 5 3 \\ \hline
	(1) 5 3 & 3 3 3 \\ \hline
	(1) 6 2 3 3  & 2 4 3 3 3 \\ \hline
	(1) 15 & 13 3 \\ \hline
	(1) 13 3 & 11 3 3 \\ \hline
	(3) 15 & 11 7 \\ \hline
	(2) 13 3 & 10 5 3 \\ \hline
	(1) 11 3 3 & 9 3 3 3 \\ \hline
	(4) 6 5 3 & 5 7 3 3 \\ \hline
	(1) 8 3 3 3 & 4 5 3 3 3 \\ \hline
	(3) 8 3 3 3 & 4 7 3 3 3 \\ \hline
	(3) 11 7 & 7 7 7 \\ \hline
	(1) 6 6 5 3 & 3 5 7 3 3 \\ \hline
	(5) 11 3 3 & 3 5 7 7 \\ \hline
	(2) 7 7 7 & 4 5 7 7 \\ \hline
	(1) 6 2 3 4 4 1 1 1 & 2 4 1 1 2 4 3 3 3 \\ \hline
	(1) 4 5 7 7 & 2 3 5 7 7 \\ \hline
	(3) 13 1 2 4 1 1 1 & 4 2 2 4 5 3 3 3 \\ \hline
	(1) 8 1 1 2 4 3 3 3 & 2 2 2 2 4 5 3 3 3 \\ \hline
	(5) 13 1 2 4 1 1 1 & 6 2 2 4 5 3 3 3 \\ \hline
	(3) 8 1 1 2 4 3 3 3 & 4 2 2 2 4 5 3 3 3 \\ \hline
	(1) 6 2 2 4 5 3 3 3 & 2 2 2 2 3 5 7 3 3 \\ \hline
	(8) 3 5 7 7 & 12 9 3 3 3 \\ \hline
	(1) 15 15 & 13 11 7 \\ \hline
	(12) 5 7 7 & o \\ \hline
	(1) 6 2 3 4 4 1 1 2 4 1 1 1 & 2 4 1 1 2 4 1 1 2 4 3 3 3 \\ \hline
	(1) 31 & 29 3 \\ \hline
	(2) 12 9 3 3 3 & 9 3 6 6 5 3 \\ \hline
	(1) 5 6 2 4 5 3 3 3 & 2 2 2 3 3 6 6 5 3 \\ \hline
	(1) 29 3 & 27 3 3 \\ \hline
	(1) 13 5 7 7 & 11 3 5 7 7  \\ \hline
	(1) 9 3 6 6 5 3 & 5 5 3 6 6 5 3 \\ \hline
	(3) 12 4 5 3 3 3 & 5 5 3 6 6 5 3 \\ \hline
	(3) 31 & 27 7 \\ \hline
	(2) 29 3 & 26 5 3 \\ \hline
	(1) 27 3 3 & 25 3 3 3 \\ \hline
	(3) 9 3 5 7 7 & 5 7 3 5 7 7 \\ \hline
	(4) 12 9 3 3 3 & 5 7 3 5 7 7 \\ \hline
	(1) 8 1 1 2 4 1 1 2 4 3 3 3 & 2 2 2 2 2 2 2 2 4 5 3 3 3 \\ \hline
	(3) 13 5 7 7 & 5 9 7 7 7 \\ \hline
	(5) 12 4 5 3 3 3 & 6 5 2 3 5 7 7 \\ \hline
	(1) 5 6 2 3 5 7 3 3 & 2 4 3 3 3 6 6 5 3 \\ \hline
	(3) 8 1 1 2 4 1 1 2 4 3 3 3 & 4 2 2 2 2 2 2 2 4 5 3 3 3 \\ \hline
	(3) 27 7 & 23 7 7 \\ \hline
	(1) 5 9 3 5 7 7 & 3 5 7 3 5 7 7 \\ \hline
	(4) 5 5 3 6 6 5 3 & 4 7 3 3 6 6 5 3 \\ \hline
	(7) 31 & 23 15 \\ \hline
	(6) 29 3 & 22 13 3 \\ \hline
	(5) 27 3 3 & 21 11 3 3 \\ \hline
	(4) 25 3 3 3 & 20 9 3 3 3 \\ \hline
	(1) 14 4 5 7 7 & 3 5 9 7 7 7 \\ \hline
	(1) 23 15 & 21 11 7 \\ \hline
	(5) 27 7 & 21 11 7 \\ \hline
	(2) 23 7 7 & 20 5 7 7 \\ \hline
	(1) 17 7 7 7 & 7 13 5 7 7 \\ \hline
	(1) 8 12 9 3 3 3 & 3 5 9 3 5 7 7 \\ \hline
	(1) 21 11 7 & 19 7 7 7 \\ \hline
	(3) 23 7 7 & 19 7 7 7 \\ \hline
	(9) 13 11 7 & 11 15 7 7 \\ \hline
	(1) 20 5 7 7 & 18 3 5 7 7 \\ \hline
	(2) 17 7 7 7 & 7 14 5 7 7 \\ \hline
	(1) 7 13 5 7 7 & 5 5 9 7 7 7 \\ \hline
	(2) 20 9 3 3 3 & 17 3 6 6 5 3 \\ \hline
	(1) 11 15 7 7 & 9 11 7 7 7 \\ \hline
	(3) 17 7 7 7 & 9 11 7 7 7 \\ \hline
	(1) 17 3 6 6 5 3 & 11 12 4 5 3 3 3 \\ \hline
	(4) 20 9 3 3 3 & 12 12 9 3 3 3 \\ \hline
	(2) 17 3 6 6 5 3 & 10 9 3 6 6 5 3 \\ \hline
	(1) 11 12 4 5 3 3 3 & 9 5 5 3 6 6 5 3 \\ \hline
	(2) 5 6 5 2 3 5 7 7 & 4 5 5 5 3 6 6 5 3 \\ \hline
	(7) 23 15 & 15 15 15 \\ \hline
	(6) 21 11 7 & 14 13 11 7 \\ \hline
	(1) 13 13 11 7 & 9 15 7 7 7 \\ \hline
	(6) 20 5 7 7 & 13 13 5 7 7 +9 15 7 7 7 \\ \hline
	(7) 17 7 7 7 & 9 15 7 7 7 \\ \hline
	(5) 18 3 5 7 7 & 12 11 3 5 7 7 \\ \hline
	(3) 12 12 9 3 3 3 & 7 8 12 9 3 3 3 \\ \hline
	(1) 13 13 5 7 7 & 11 5 9 7 7 7 \\ \hline
	(8) 20 9 3 3 3 & 11 5 9 7 7 7 \\ \hline
	(2) 7 14 4 5 7 7 & 8 3 5 9 7 7 7 \\ \hline
	(6) 17 3 6 6 5 3 & 8 3 5 9 7 7 7 \\ \hline
	(3) 13 13 11 7 & 7 11 15 7 7 \\ \hline
	(2) 9 15 7 7 7 & 6 9 11 7 7 7 \\ \hline
	(1) 11 5 9 7 7 7 & 9 3 5 9 7 7 7 \\ \hline
	(1) 8 3 5 9 7 7 7 & 4 5 3 5 9 7 7 7 \\ \hline
	(2) 7 8 12 9 3 3 3 & 8 3 5 9 3 5 7 7 \\ \hline
	(11) 23 7 7 & 7 11 15 15 \\ \hline
	(8) 19 7 7 7 & 10 17 7 7 7 \\ \hline
	(3) 13 13 5 7 7 & 9 7 13 5 7 7 \\ \hline
	(2) 11 5 9 7 7 7 & 8 5 5 9 7 7 7 \\ \hline
	(4) 7 14 4 5 7 7 & 5 10 11 3 5 7 7 \\ \hline
	(2) 8 3 5 9 7 7 7 & 4 6 3 5 9 7 7 7 \\ \hline
	(1) 4 5 3 5 9 7 7 7 & 2 3 5 3 5 9 7 7 7 \\ \hline
	(1) 8 3 5 9 3 5 7 7 & 4 5 3 5 9 3 5 7 7 \\ \hline
	(4) 14 13 11 7 & 9 11 15 7 7 \\ \hline
	(1) 10 17 7 7 7 & 8 9 11 7 7 7 \\ \hline
	(5) 15 15 15 & 9 11 15 15  \\ \hline
	(6) 14 13 11 7 & 7 13 13 11 7 \\ \hline
	(3) 10 17 7 7 7 & 6 9 15 7 7 7 \\ \hline
	(2) 8 9 11 7 7 7 & 4 6 9 11 7 7 7 \\ \hline
	(4) 9 3 5 9 7 7 7 & 3 5 10 11 3 5 7 7 \\ \hline
	(1) 7 13 13 11 7 & 5 7 11 15 7 7 \\ \hline
	(3) 9 11 15 7 7 & 5 7 11 15 7 7 \\ \hline
	(3) 9 11 15 15 & 5 7 11 15 15 \\ \hline
	(1) 8 4 5 3 5 9 3 5 7 7 & 4 5 4 5 3 5 9 3 5 7 7 \\ \hline
	(3) 16 2 3 5 5 3 6 6 5 3 & 4 5 4 5 3 5 9 3 5 7 7 \\ \hline
	(1) 27 12 4 5 3 3 3 & 25 5 5 3 6 6 5 3 \\ \hline
	(20) 5 5 9 7 7 7 & o \\ \hline
	(3) 27 12 4 5 3 3 3 & 5 1 2 4 7 11 15 15 \\ \hline
	(2) 25 5 5 3 6 6 5 3 & 3 6 4 6 9 11 7 7 7 \\ \hline
	(1) 6 2 3 5 10 11 3 5 7 7 & 2 2 4 5 9 3 5 9 7 7 7 \\ \hline
	(4) 5 5 9 3 5 7 3 5 7 7 & o \\ \hline
	(7) 18 2 4 3 3 3 6 6 5 3 & 8 4 2 3 5 3 5 9 7 7 7 +2 2 4 5 9 3 5 9 7 7 7 \\ \hline
	(8) 5 7 11 15 15 & o \\ \hline
	(9) 4 7 11 15 15 & 28 11 3 5 7 7 \\ \hline
	(3) 28 12 9 3 3 3 & 23 8 12 9 3 3 3 \\ \hline
	(1) 6 2 4 7 11 15 15 & 2 4 3 4 7 11 15 15 \\ \hline
	\hline
\end{longtable}

The Curtis table for the Adams $E_2$-page of $P_1^\infty$ in the range of $t<72$ is the following:

$*=2\;4\;1\;1$

$.=2$

\addtolength{\textwidth}{2cm}
\addtolength{\hoffset}{-1cm}

\tcbset{bottom=0pt,top=0pt,left=0pt,right=0pt}
 \twocolumn

\begin{tcolorbox}[title={$(1,1)$}]
(1) 

\end{tcolorbox}
\begin{tcolorbox}[title={$(2,2)$}]
(1) 1 

\end{tcolorbox}
\begin{tcolorbox}[title={$(3,1)$}]
(3) 

\end{tcolorbox}
\begin{tcolorbox}[title={$(3,2)$}]
(2) 1 

\end{tcolorbox}
\begin{tcolorbox}[title={$(3,3)$}]
(1) 1 1 

\end{tcolorbox}
\begin{tcolorbox}[title={$(4,2)$}]
(1) 3 

(3) 1 $\leftarrow$ (5) 

\end{tcolorbox}
\begin{tcolorbox}[title={$(4,3)$}]
(1) 2 1 $\leftarrow$ (2) 3 

(2) 1 1 $\leftarrow$ (4) 1 

\end{tcolorbox}
\begin{tcolorbox}[title={$(4,4)$}]
(1) 1 1 1 $\leftarrow$ (2) 2 1 

\end{tcolorbox}
\begin{tcolorbox}[title={$(5,3)$}]
(3) 1 1 $\leftarrow$ (5) 1 

\end{tcolorbox}
\begin{tcolorbox}[title={$(5,4)$}]
(2) 1 1 1 $\leftarrow$ (4) 1 1 

\end{tcolorbox}
\begin{tcolorbox}[title={$(6,2)$}]
(3) 3 

\end{tcolorbox}
\begin{tcolorbox}[title={$(6,3)$}]
(3) 2 1 $\leftarrow$ (4) 3 

\end{tcolorbox}
\begin{tcolorbox}[title={$(6,4)$}]
(3) 1 1 1 $\leftarrow$ (4) 2 1 

\end{tcolorbox}
\begin{tcolorbox}[title={$(7,1)$}]
(7) 

\end{tcolorbox}
\begin{tcolorbox}[title={$(7,2)$}]
(6) 1 

\end{tcolorbox}
\begin{tcolorbox}[title={$(7,3)$}]
(1) 3 3 

(5) 1 1 

\end{tcolorbox}
\begin{tcolorbox}[title={$(7,4)$}]
(4) 1 1 1 

\end{tcolorbox}
\begin{tcolorbox}[title={$(8,2)$}]
(1) 7 

(5) 3 

(7) 1 $\leftarrow$ (9) 

\end{tcolorbox}
\begin{tcolorbox}[title={$(8,3)$}]
(1) 6 1 $\leftarrow$ (2) 7 

(2) 3 3 

(5) 2 1 $\leftarrow$ (6) 3 

(6) 1 1 $\leftarrow$ (8) 1 

\end{tcolorbox}
\begin{tcolorbox}[title={$(8,4)$}]
(1) 5 1 1 $\leftarrow$ (2) 6 1 

(5) 1 1 1 $\leftarrow$ (6) 2 1 

\end{tcolorbox}
\begin{tcolorbox}[title={$(8,5)$}]
(1) 4 1 1 1 $\leftarrow$ (2) 5 1 1 

\end{tcolorbox}
\begin{tcolorbox}[title={$(9,3)$}]
(1) 5 3 

(3) 3 3 

(7) 1 1 $\leftarrow$ (9) 1 

\end{tcolorbox}
\begin{tcolorbox}[title={$(9,4)$}]
(1) 2 3 3 

(6) 1 1 1 $\leftarrow$ (8) 1 1 

\end{tcolorbox}
\begin{tcolorbox}[title={$(9,5)$}]
(2) 4 1 1 1 

\end{tcolorbox}
\begin{tcolorbox}[title={$(10,2)$}]
(3) 7 

(7) 3 $\leftarrow$ (11) 

\end{tcolorbox}
\begin{tcolorbox}[title={$(10,3)$}]
(2) 5 3 

(3) 6 1 $\leftarrow$ (4) 7 

(4) 3 3 $\leftarrow$ (10) 1 

(7) 2 1 $\leftarrow$ (8) 3 

\end{tcolorbox}
\begin{tcolorbox}[title={$(10,4)$}]
(1) 3 3 3 

(2) 2 3 3 $\leftarrow$ (9) 1 1 

(3) 5 1 1 $\leftarrow$ (4) 6 1 

(7) 1 1 1 $\leftarrow$ (8) 2 1 

\end{tcolorbox}
\begin{tcolorbox}[title={$(10,5)$}]
(1) 1 2 3 3 $\leftarrow$ (8) 1 1 1 

(3) 4 1 1 1 $\leftarrow$ (4) 5 1 1 

\end{tcolorbox}
\begin{tcolorbox}[title={$(10,6)$}]
(1) * 1 

\end{tcolorbox}
\begin{tcolorbox}[title={$(11,3)$}]
(3) 5 3 $\leftarrow$ (5) 7 

(5) 3 3 $\leftarrow$ (9) 3 

\end{tcolorbox}
\begin{tcolorbox}[title={$(11,4)$}]
(2) 3 3 3 $\leftarrow$ (4) 5 3 

(3) 2 3 3 $\leftarrow$ (6) 3 3 

\end{tcolorbox}
\begin{tcolorbox}[title={$(11,5)$}]
(2) 1 2 3 3 $\leftarrow$ (4) 2 3 3 

(4) 4 1 1 1 

\end{tcolorbox}
\begin{tcolorbox}[title={$(11,6)$}]
(2) * 1 

\end{tcolorbox}
\begin{tcolorbox}[title={$(11,7)$}]
(1) 1 * 1 

\end{tcolorbox}
\begin{tcolorbox}[title={$(12,2)$}]
(11) 1 $\leftarrow$ (13) 

\end{tcolorbox}
\begin{tcolorbox}[title={$(12,3)$}]
(5) 6 1 $\leftarrow$ (6) 7 

(9) 2 1 $\leftarrow$ (10) 3 

(10) 1 1 $\leftarrow$ (12) 1 

\end{tcolorbox}
\begin{tcolorbox}[title={$(12,4)$}]
(3) 3 3 3 $\leftarrow$ (5) 5 3 

(5) 5 1 1 $\leftarrow$ (6) 6 1 

(9) 1 1 1 $\leftarrow$ (10) 2 1 

\end{tcolorbox}
\begin{tcolorbox}[title={$(12,5)$}]
(3) 1 2 3 3 $\leftarrow$ (5) 2 3 3 

(5) 4 1 1 1 $\leftarrow$ (6) 5 1 1 

\end{tcolorbox}
\begin{tcolorbox}[title={$(12,6)$}]
(1) 4 4 1 1 1 $\leftarrow$ (4) 1 2 3 3 

(3) * 1 $\leftarrow$ (6) 4 1 1 1 

\end{tcolorbox}
\begin{tcolorbox}[title={$(12,7)$}]
(1) 2 * 1 $\leftarrow$ (2) 4 4 1 1 1 

(2) 1 * 1 $\leftarrow$ (4) * 1 

\end{tcolorbox}
\begin{tcolorbox}[title={$(12,8)$}]
(1) 1 1 * 1 $\leftarrow$ (2) 2 * 1 

\end{tcolorbox}
\begin{tcolorbox}[title={$(13,3)$}]
(7) 3 3 $\leftarrow$ (11) 3 

(11) 1 1 $\leftarrow$ (13) 1 

\end{tcolorbox}
\begin{tcolorbox}[title={$(13,4)$}]
(4) 3 3 3 $\leftarrow$ (8) 3 3 

(10) 1 1 1 $\leftarrow$ (12) 1 1 

\end{tcolorbox}
\begin{tcolorbox}[title={$(13,7)$}]
(3) 1 * 1 $\leftarrow$ (5) * 1 

\end{tcolorbox}
\begin{tcolorbox}[title={$(13,8)$}]
(2) 1 1 * 1 $\leftarrow$ (4) 1 * 1 

\end{tcolorbox}
\begin{tcolorbox}[title={$(14,2)$}]
(7) 7 

\end{tcolorbox}
\begin{tcolorbox}[title={$(14,3)$}]
(6) 5 3 

(7) 6 1 $\leftarrow$ (8) 7 

(11) 2 1 $\leftarrow$ (12) 3 

\end{tcolorbox}
\begin{tcolorbox}[title={$(14,4)$}]
(5) 3 3 3 $\leftarrow$ (9) 3 3 

(6) 2 3 3 

(7) 5 1 1 $\leftarrow$ (8) 6 1 

(11) 1 1 1 $\leftarrow$ (12) 2 1 

\end{tcolorbox}
\begin{tcolorbox}[title={$(14,5)$}]
(5) 1 2 3 3 

(7) 4 1 1 1 $\leftarrow$ (8) 5 1 1 

\end{tcolorbox}
\begin{tcolorbox}[title={$(14,6)$}]
(3) 4 4 1 1 1 

\end{tcolorbox}
\begin{tcolorbox}[title={$(14,7)$}]
(3) 2 * 1 $\leftarrow$ (4) 4 4 1 1 1 

\end{tcolorbox}
\begin{tcolorbox}[title={$(14,8)$}]
(3) 1 1 * 1 $\leftarrow$ (4) 2 * 1 

\end{tcolorbox}
\begin{tcolorbox}[title={$(15,1)$}]
(15) 

\end{tcolorbox}
\begin{tcolorbox}[title={$(15,2)$}]
(14) 1 

\end{tcolorbox}
\begin{tcolorbox}[title={$(15,3)$}]
(1) 7 7 

(7) 5 3 $\leftarrow$ (9) 7 

(13) 1 1 

\end{tcolorbox}
\begin{tcolorbox}[title={$(15,4)$}]
(1) 6 5 3 $\leftarrow$ (2) 7 7 

(6) 3 3 3 $\leftarrow$ (8) 5 3 

(7) 2 3 3 $\leftarrow$ (10) 3 3 

(12) 1 1 1 

\end{tcolorbox}
\begin{tcolorbox}[title={$(15,5)$}]
(1) 6 2 3 3 

(6) 1 2 3 3 $\leftarrow$ (8) 2 3 3 

(8) 4 1 1 1 

\end{tcolorbox}
\begin{tcolorbox}[title={$(15,6)$}]
(1) 5 1 2 3 3 $\leftarrow$ (2) 6 2 3 3 

(6) * 1 

\end{tcolorbox}
\begin{tcolorbox}[title={$(15,7)$}]
(1) 3 4 4 1 1 1 $\leftarrow$ (2) 5 1 2 3 3 

(5) 1 * 1 

\end{tcolorbox}
\begin{tcolorbox}[title={$(15,8)$}]
(4) 1 1 * 1 

\end{tcolorbox}
\begin{tcolorbox}[title={$(16,2)$}]
(1) 15 

(13) 3 

(15) 1 $\leftarrow$ (17) 

\end{tcolorbox}
\begin{tcolorbox}[title={$(16,3)$}]
(1) 14 1 $\leftarrow$ (2) 15 

(9) 6 1 $\leftarrow$ (10) 7 

(13) 2 1 $\leftarrow$ (14) 3 

(14) 1 1 $\leftarrow$ (16) 1 

\end{tcolorbox}
\begin{tcolorbox}[title={$(16,4)$}]
(1) 13 1 1 $\leftarrow$ (2) 14 1 

(2) 6 5 3 

(7) 3 3 3 $\leftarrow$ (9) 5 3 

(9) 5 1 1 $\leftarrow$ (10) 6 1 

(13) 1 1 1 $\leftarrow$ (14) 2 1 

\end{tcolorbox}
\begin{tcolorbox}[title={$(16,5)$}]
(1) 12 1 1 1 $\leftarrow$ (2) 13 1 1 

(7) 1 2 3 3 $\leftarrow$ (9) 2 3 3 

(9) 4 1 1 1 $\leftarrow$ (10) 5 1 1 

\end{tcolorbox}
\begin{tcolorbox}[title={$(16,6)$}]
(1) 2 4 3 3 3 

(1) 8 4 1 1 1 $\leftarrow$ (2) 12 1 1 1 

(5) 4 4 1 1 1 $\leftarrow$ (8) 1 2 3 3 

(7) * 1 $\leftarrow$ (10) 4 1 1 1 

\end{tcolorbox}
\begin{tcolorbox}[title={$(16,7)$}]
(1) 6 * 1 $\leftarrow$ (2) 8 4 1 1 1 

(2) 3 4 4 1 1 1 

(5) 2 * 1 $\leftarrow$ (6) 4 4 1 1 1 

(6) 1 * 1 $\leftarrow$ (8) * 1 

\end{tcolorbox}
\begin{tcolorbox}[title={$(16,8)$}]
(1) 5 1 * 1 $\leftarrow$ (2) 6 * 1 

(5) 1 1 * 1 $\leftarrow$ (6) 2 * 1 

\end{tcolorbox}
\begin{tcolorbox}[title={$(16,9)$}]
(1) 4 1 1 * 1 $\leftarrow$ (2) 5 1 * 1 

\end{tcolorbox}
\begin{tcolorbox}[title={$(17,3)$}]
(1) 13 3 

(3) 7 7 

(11) 3 3 

(15) 1 1 $\leftarrow$ (17) 1 

\end{tcolorbox}
\begin{tcolorbox}[title={$(17,4)$}]
(3) 6 5 3 $\leftarrow$ (4) 7 7 

(8) 3 3 3 

(14) 1 1 1 $\leftarrow$ (16) 1 1 

\end{tcolorbox}
\begin{tcolorbox}[title={$(17,5)$}]
(3) 6 2 3 3 

\end{tcolorbox}
\begin{tcolorbox}[title={$(17,6)$}]
(2) 2 4 3 3 3 

(3) 5 1 2 3 3 $\leftarrow$ (4) 6 2 3 3 

\end{tcolorbox}
\begin{tcolorbox}[title={$(17,7)$}]
(1) 1 2 4 3 3 3 

(3) 3 4 4 1 1 1 $\leftarrow$ (4) 5 1 2 3 3 

(7) 1 * 1 $\leftarrow$ (9) * 1 

\end{tcolorbox}
\begin{tcolorbox}[title={$(17,8)$}]
(1) 2 3 4 4 1 1 1 

(6) 1 1 * 1 $\leftarrow$ (8) 1 * 1 

\end{tcolorbox}
\begin{tcolorbox}[title={$(17,9)$}]
(2) 4 1 1 * 1 

\end{tcolorbox}
\begin{tcolorbox}[title={$(18,2)$}]
(3) 15 

(11) 7 

(15) 3 $\leftarrow$ (19) 

\end{tcolorbox}
\begin{tcolorbox}[title={$(18,3)$}]
(2) 13 3 

(3) 14 1 $\leftarrow$ (4) 15 

(10) 5 3 

(11) 6 1 $\leftarrow$ (12) 7 

(12) 3 3 $\leftarrow$ (18) 1 

(15) 2 1 $\leftarrow$ (16) 3 

\end{tcolorbox}
\begin{tcolorbox}[title={$(18,4)$}]
(1) 11 3 3 

(3) 13 1 1 $\leftarrow$ (4) 14 1 

(4) 6 5 3 

(9) 3 3 3 

(10) 2 3 3 $\leftarrow$ (17) 1 1 

(11) 5 1 1 $\leftarrow$ (12) 6 1 

(15) 1 1 1 $\leftarrow$ (16) 2 1 

\end{tcolorbox}
\begin{tcolorbox}[title={$(18,5)$}]
(1) 8 3 3 3 

(3) 12 1 1 1 $\leftarrow$ (4) 13 1 1 

(9) 1 2 3 3 $\leftarrow$ (16) 1 1 1 

(11) 4 1 1 1 $\leftarrow$ (12) 5 1 1 

\end{tcolorbox}
\begin{tcolorbox}[title={$(18,6)$}]
(1) 3 6 2 3 3 $\leftarrow$ (2) 8 3 3 3 

(3) 2 4 3 3 3 $\leftarrow$ (5) 6 2 3 3 

(3) 8 4 1 1 1 $\leftarrow$ (4) 12 1 1 1 

(7) 4 4 1 1 1 $\leftarrow$ (12) 4 1 1 1 

\end{tcolorbox}

\begin{tcolorbox}[title={$(18,7)$}]
(1) ..4 3 3 3 $\leftarrow$ (2) 3 6 2 3 3 

(2) 1 2 4 3 3 3 $\leftarrow$ (4) 2 4 3 3 3 

(3) 6 * 1 $\leftarrow$ (4) 8 4 1 1 1 

(4) 3 4 4 1 1 1 $\leftarrow$ (10) * 1 

(7) 2 * 1 $\leftarrow$ (8) 4 4 1 1 1 

\end{tcolorbox}
\begin{tcolorbox}[title={$(18,8)$}]
(1) 1 1 2 4 3 3 3 $\leftarrow$ (2) ..4 3 3 3 

(2) 2 3 4 4 1 1 1 $\leftarrow$ (9) 1 * 1 

(3) 5 1 * 1 $\leftarrow$ (4) 6 * 1 

(7) 1 1 * 1 $\leftarrow$ (8) 2 * 1 

\end{tcolorbox}
\begin{tcolorbox}[title={$(18,9)$}]
(1) 1 2 3 4 4 1 1 1 $\leftarrow$ (8) 1 1 * 1 

(3) 4 1 1 * 1 $\leftarrow$ (4) 5 1 * 1 

\end{tcolorbox}
\begin{tcolorbox}[title={$(18,10)$}]
(1) * * 1 

\end{tcolorbox}

\begin{tcolorbox}[title={$(19,3)$}]
(1) 11 7 

(3) 13 3 $\leftarrow$ (5) 15 

(5) 7 7 

(11) 5 3 $\leftarrow$ (13) 7 

(13) 3 3 $\leftarrow$ (17) 3 

\end{tcolorbox}
\begin{tcolorbox}[title={$(19,4)$}]
(1) 10 5 3 $\leftarrow$ (2) 11 7 

(2) 11 3 3 $\leftarrow$ (4) 13 3 

(5) 6 5 3 $\leftarrow$ (6) 7 7 

(10) 3 3 3 $\leftarrow$ (12) 5 3 

(11) 2 3 3 $\leftarrow$ (14) 3 3 

\end{tcolorbox}
\begin{tcolorbox}[title={$(19,5)$}]
(1) 5 7 3 3 

(1) 9 3 3 3 $\leftarrow$ (2) 10 5 3 

(10) 1 2 3 3 $\leftarrow$ (12) 2 3 3 

\end{tcolorbox}
\begin{tcolorbox}[title={$(19,6)$}]
(1) 4 5 3 3 3 $\leftarrow$ (2) 5 7 3 3 

(5) 5 1 2 3 3 $\leftarrow$ (6) 6 2 3 3 

\end{tcolorbox}
\begin{tcolorbox}[title={$(19,7)$}]
(3) 1 2 4 3 3 3 $\leftarrow$ (5) 2 4 3 3 3 

(5) 3 4 4 1 1 1 $\leftarrow$ (6) 5 1 2 3 3 

\end{tcolorbox}

\begin{tcolorbox}[title={$(19,8)$}]
(2) 1 1 2 4 3 3 3 $\leftarrow$ (4) 1 2 4 3 3 3 

(3) 2 3 4 4 1 1 1 $\leftarrow$ (6) 3 4 4 1 1 1 

\end{tcolorbox}

\begin{tcolorbox}[title={$(19,9)$}]
(2) 1 2 3 4 4 1 1 1 $\leftarrow$ (4) 2 3 4 4 1 1 1 

(4) 4 1 1 * 1 

\end{tcolorbox}
\begin{tcolorbox}[title={$(19,10)$}]
(2) * * 1 

\end{tcolorbox}
\begin{tcolorbox}[title={$(19,11)$}]
(1) 1 * * 1 

\end{tcolorbox}

\small

\begin{tcolorbox}[title={$(20,2)$}]
(19) 1 $\leftarrow$ (21) 

\end{tcolorbox}
\begin{tcolorbox}[title={$(20,3)$}]
(5) 14 1 $\leftarrow$ (6) 15 

(13) 6 1 $\leftarrow$ (14) 7 

(17) 2 1 $\leftarrow$ (18) 3 

(18) 1 1 $\leftarrow$ (20) 1 

\end{tcolorbox}
\begin{tcolorbox}[title={$(20,4)$}]
(1) 5 7 7 

(3) 11 3 3 $\leftarrow$ (5) 13 3 

(5) 13 1 1 $\leftarrow$ (6) 14 1 

(6) 6 5 3 

(11) 3 3 3 $\leftarrow$ (13) 5 3 

(13) 5 1 1 $\leftarrow$ (14) 6 1 

(17) 1 1 1 $\leftarrow$ (18) 2 1 

\end{tcolorbox}
\begin{tcolorbox}[title={$(20,5)$}]
(2) 9 3 3 3 $\leftarrow$ (4) 11 3 3 

(3) 8 3 3 3 

(5) 12 1 1 1 $\leftarrow$ (6) 13 1 1 

(11) 1 2 3 3 $\leftarrow$ (13) 2 3 3 

(13) 4 1 1 1 $\leftarrow$ (14) 5 1 1 

\end{tcolorbox}
\begin{tcolorbox}[title={$(20,6)$}]
(2) 4 5 3 3 3 

(3) 3 6 2 3 3 $\leftarrow$ (4) 8 3 3 3 

(5) 8 4 1 1 1 $\leftarrow$ (6) 12 1 1 1 

(9) 4 4 1 1 1 $\leftarrow$ (12) 1 2 3 3 

(11) * 1 $\leftarrow$ (14) 4 1 1 1 

\end{tcolorbox}
\begin{tcolorbox}[title={$(20,7)$}]
(3) ..4 3 3 3 $\leftarrow$ (4) 3 6 2 3 3 

(5) 6 * 1 $\leftarrow$ (6) 8 4 1 1 1 

(9) 2 * 1 $\leftarrow$ (10) 4 4 1 1 1 

(10) 1 * 1 $\leftarrow$ (12) * 1 

\end{tcolorbox}
\begin{tcolorbox}[title={$(20,8)$}]
(3) 1 1 2 4 3 3 3 $\leftarrow$ (4) ..4 3 3 3 

(5) 5 1 * 1 $\leftarrow$ (6) 6 * 1 

(9) 1 1 * 1 $\leftarrow$ (10) 2 * 1 

\end{tcolorbox}
\begin{tcolorbox}[title={$(20,9)$}]
(3) 1 2 3 4 4 1 1 1 $\leftarrow$ (5) 2 3 4 4 1 1 1 

(5) 4 1 1 * 1 $\leftarrow$ (6) 5 1 * 1 

\end{tcolorbox}
\begin{tcolorbox}[title={$(20,10)$}]
(1) 4 4 1 1 * 1 $\leftarrow$ (4) 1 2 3 4 4 1 1 1 

(3) * * 1 $\leftarrow$ (6) 4 1 1 * 1 

\end{tcolorbox}
\begin{tcolorbox}[title={$(20,11)$}]
(1) 2 * * 1 $\leftarrow$ (2) 4 4 1 1 * 1 

(2) 1 * * 1 $\leftarrow$ (4) * * 1 

\end{tcolorbox}
\begin{tcolorbox}[title={$(20,12)$}]
(1) 1 1 * * 1 $\leftarrow$ (2) 2 * * 1 

\end{tcolorbox}
\begin{tcolorbox}[title={$(21,3)$}]
(3) 11 7 

(7) 7 7 

(15) 3 3 $\leftarrow$ (19) 3 

(19) 1 1 $\leftarrow$ (21) 1 

\end{tcolorbox}
\begin{tcolorbox}[title={$(21,4)$}]
(2) 5 7 7 

(3) 10 5 3 $\leftarrow$ (4) 11 7 

(7) 6 5 3 $\leftarrow$ (8) 7 7 

(12) 3 3 3 $\leftarrow$ (16) 3 3 

(18) 1 1 1 $\leftarrow$ (20) 1 1 

\end{tcolorbox}
\begin{tcolorbox}[title={$(21,5)$}]
(1) 6 6 5 3 

(3) 5 7 3 3 

(3) 9 3 3 3 $\leftarrow$ (4) 10 5 3 

(7) 6 2 3 3 $\leftarrow$ (14) 2 3 3 

\end{tcolorbox}
\begin{tcolorbox}[title={$(21,6)$}]
(1) 4 7 3 3 3 $\leftarrow$ (2) 6 6 5 3 

(3) 4 5 3 3 3 $\leftarrow$ (4) 5 7 3 3 

(6) 2 4 3 3 3 $\leftarrow$ (13) 1 2 3 3 

(7) 5 1 2 3 3 $\leftarrow$ (8) 6 2 3 3 

\end{tcolorbox}
\begin{tcolorbox}[title={$(21,7)$}]
(1) 2 4 5 3 3 3 $\leftarrow$ (2) 4 7 3 3 3 

(5) 1 2 4 3 3 3 $\leftarrow$ (11) 4 4 1 1 1 

(7) 3 4 4 1 1 1 $\leftarrow$ (8) 5 1 2 3 3 

(11) 1 * 1 $\leftarrow$ (13) * 1 

\end{tcolorbox}
\begin{tcolorbox}[title={$(21,8)$}]
(4) 1 1 2 4 3 3 3 $\leftarrow$ (8) 3 4 4 1 1 1 

(10) 1 1 * 1 $\leftarrow$ (12) 1 * 1 

\end{tcolorbox}
\begin{tcolorbox}[title={$(21,11)$}]
(3) 1 * * 1 $\leftarrow$ (5) * * 1 

\end{tcolorbox}
\begin{tcolorbox}[title={$(21,12)$}]
(2) 1 1 * * 1 $\leftarrow$ (4) 1 * * 1 

\end{tcolorbox}
\begin{tcolorbox}[title={$(22,2)$}]
(7) 15 

(15) 7 $\leftarrow$ (23) 

\end{tcolorbox}
\begin{tcolorbox}[title={$(22,3)$}]
(6) 13 3 

(7) 14 1 $\leftarrow$ (8) 15 

(14) 5 3 $\leftarrow$ (22) 1 

(15) 6 1 $\leftarrow$ (16) 7 

(19) 2 1 $\leftarrow$ (20) 3 

\end{tcolorbox}
\begin{tcolorbox}[title={$(22,4)$}]
(1) 7 7 7 

(3) 5 7 7 

(5) 11 3 3 

(7) 13 1 1 $\leftarrow$ (8) 14 1 

(8) 6 5 3 $\leftarrow$ (21) 1 1 

(13) 3 3 3 $\leftarrow$ (17) 3 3 

(15) 5 1 1 $\leftarrow$ (16) 6 1 

(19) 1 1 1 $\leftarrow$ (20) 2 1 

\end{tcolorbox}
\begin{tcolorbox}[title={$(22,5)$}]
(4) 9 3 3 3 

(5) 8 3 3 3 $\leftarrow$ (20) 1 1 1 

(7) 12 1 1 1 $\leftarrow$ (8) 13 1 1 

(15) 4 1 1 1 $\leftarrow$ (16) 5 1 1 

\end{tcolorbox}
\begin{tcolorbox}[title={$(22,6)$}]
(1) 3 5 7 3 3 

(4) 4 5 3 3 3 $\leftarrow$ (16) 4 1 1 1 

(5) 3 6 2 3 3 $\leftarrow$ (6) 8 3 3 3 

(7) 2 4 3 3 3 $\leftarrow$ (9) 6 2 3 3 

(7) 8 4 1 1 1 $\leftarrow$ (8) 12 1 1 1 

\end{tcolorbox}
\begin{tcolorbox}[title={$(22,7)$}]
(2) 2 4 5 3 3 3 $\leftarrow$ (14) * 1 

(5) ..4 3 3 3 $\leftarrow$ (6) 3 6 2 3 3 

(6) 1 2 4 3 3 3 $\leftarrow$ (8) 2 4 3 3 3 

(7) 6 * 1 $\leftarrow$ (8) 8 4 1 1 1 

(11) 2 * 1 $\leftarrow$ (12) 4 4 1 1 1 

\end{tcolorbox}
\begin{tcolorbox}[title={$(22,8)$}]
(5) 1 1 2 4 3 3 3 $\leftarrow$ (6) ..4 3 3 3 

(6) 2 3 4 4 1 1 1 

(7) 5 1 * 1 $\leftarrow$ (8) 6 * 1 

(11) 1 1 * 1 $\leftarrow$ (12) 2 * 1 

\end{tcolorbox}
\begin{tcolorbox}[title={$(22,9)$}]
(5) 1 2 3 4 4 1 1 1 

(7) 4 1 1 * 1 $\leftarrow$ (8) 5 1 * 1 

\end{tcolorbox}
\begin{tcolorbox}[title={$(22,10)$}]
(3) 4 4 1 1 * 1 

\end{tcolorbox}
\begin{tcolorbox}[title={$(22,11)$}]
(3) 2 * * 1 $\leftarrow$ (4) 4 4 1 1 * 1 

\end{tcolorbox}
\begin{tcolorbox}[title={$(22,12)$}]
(3) 1 1 * * 1 $\leftarrow$ (4) 2 * * 1 

\end{tcolorbox}
\begin{tcolorbox}[title={$(23,3)$}]
(5) 11 7 

(7) 13 3 $\leftarrow$ (9) 15 

(9) 7 7 $\leftarrow$ (21) 3 

(15) 5 3 $\leftarrow$ (17) 7 

\end{tcolorbox}
\begin{tcolorbox}[title={$(23,4)$}]
(2) 7 7 7 

(4) 5 7 7 

(5) 10 5 3 $\leftarrow$ (6) 11 7 

(6) 11 3 3 $\leftarrow$ (8) 13 3 

(9) 6 5 3 $\leftarrow$ (10) 7 7 

(14) 3 3 3 $\leftarrow$ (16) 5 3 

(15) 2 3 3 $\leftarrow$ (18) 3 3 

\end{tcolorbox}
\begin{tcolorbox}[title={$(23,5)$}]
(1) 3 5 7 7 

(3) 6 6 5 3 

(5) 5 7 3 3 $\leftarrow$ (10) 6 5 3 

(5) 9 3 3 3 $\leftarrow$ (6) 10 5 3 

(14) 1 2 3 3 $\leftarrow$ (16) 2 3 3 

\end{tcolorbox}
\begin{tcolorbox}[title={$(23,6)$}]
(2) 3 5 7 3 3 

(3) 4 7 3 3 3 $\leftarrow$ (4) 6 6 5 3 

(5) 4 5 3 3 3 $\leftarrow$ (6) 5 7 3 3 

(9) 5 1 2 3 3 $\leftarrow$ (10) 6 2 3 3 

\end{tcolorbox}
\begin{tcolorbox}[title={$(23,7)$}]
(3) 2 4 5 3 3 3 $\leftarrow$ (4) 4 7 3 3 3 

(7) 1 2 4 3 3 3 $\leftarrow$ (9) 2 4 3 3 3 

(9) 3 4 4 1 1 1 $\leftarrow$ (10) 5 1 2 3 3 

(13) 1 * 1 

\end{tcolorbox}
\begin{tcolorbox}[title={$(23,8)$}]
(6) 1 1 2 4 3 3 3 $\leftarrow$ (8) 1 2 4 3 3 3 

(7) 2 3 4 4 1 1 1 $\leftarrow$ (10) 3 4 4 1 1 1 

(12) 1 1 * 1 

\end{tcolorbox}
\begin{tcolorbox}[title={$(23,9)$}]
(1) 6 2 3 4 4 1 1 1 

(6) 1 2 3 4 4 1 1 1 $\leftarrow$ (8) 2 3 4 4 1 1 1 

(8) 4 1 1 * 1 

\end{tcolorbox}

\begin{tcolorbox}[title={$(23,10)$}]
(1) 5 1 2 3 4 4 1 1 1 $\leftarrow$ (2) 6 2 3 4 4 1 1 1 

(6) * * 1 

\end{tcolorbox}
\begin{tcolorbox}[title={$(23,11)$}]
(1) 3 4 4 1 1 * 1 $\leftarrow$ (2) 5 1 2 3 4 4 1 1 1 

(5) 1 * * 1 

\end{tcolorbox}
\begin{tcolorbox}[title={$(23,12)$}]
(4) 1 1 * * 1 

\end{tcolorbox}

\footnotesize

\begin{tcolorbox}[title={$(24,2)$}]
(23) 1 $\leftarrow$ (25) 

\end{tcolorbox}
\begin{tcolorbox}[title={$(24,3)$}]
(9) 14 1 $\leftarrow$ (10) 15 

(17) 6 1 $\leftarrow$ (18) 7 

(21) 2 1 $\leftarrow$ (22) 3 

(22) 1 1 $\leftarrow$ (24) 1 

\end{tcolorbox}
\begin{tcolorbox}[title={$(24,4)$}]
(3) 7 7 7 

(5) 5 7 7 $\leftarrow$ (19) 3 3 

(7) 11 3 3 $\leftarrow$ (9) 13 3 

(9) 13 1 1 $\leftarrow$ (10) 14 1 

(15) 3 3 3 $\leftarrow$ (17) 5 3 

(17) 5 1 1 $\leftarrow$ (18) 6 1 

(21) 1 1 1 $\leftarrow$ (22) 2 1 

\end{tcolorbox}
\begin{tcolorbox}[title={$(24,5)$}]
(1) 4 5 7 7 

(2) 3 5 7 7 

(6) 9 3 3 3 $\leftarrow$ (8) 11 3 3 

(7) 8 3 3 3 $\leftarrow$ (16) 3 3 3 

(9) 12 1 1 1 $\leftarrow$ (10) 13 1 1 

(15) 1 2 3 3 $\leftarrow$ (17) 2 3 3 

(17) 4 1 1 1 $\leftarrow$ (18) 5 1 1 

\end{tcolorbox}
\begin{tcolorbox}[title={$(24,6)$}]
(1) 3 6 6 5 3 

(3) 3 5 7 3 3 $\leftarrow$ (5) 6 6 5 3 

(6) 4 5 3 3 3 $\leftarrow$ (11) 6 2 3 3 

(7) 3 6 2 3 3 $\leftarrow$ (8) 8 3 3 3 

(9) 8 4 1 1 1 $\leftarrow$ (10) 12 1 1 1 

(13) 4 4 1 1 1 $\leftarrow$ (16) 1 2 3 3 

(15) * 1 $\leftarrow$ (18) 4 1 1 1 

\end{tcolorbox}
\begin{tcolorbox}[title={$(24,7)$}]
(1) 2 3 5 7 3 3 $\leftarrow$ (2) 3 6 6 5 3 

(4) 2 4 5 3 3 3 $\leftarrow$ (10) 2 4 3 3 3 

(7) ..4 3 3 3 $\leftarrow$ (8) 3 6 2 3 3 

(9) 6 * 1 $\leftarrow$ (10) 8 4 1 1 1 

(13) 2 * 1 $\leftarrow$ (14) 4 4 1 1 1 

(14) 1 * 1 $\leftarrow$ (16) * 1 

\end{tcolorbox}
\begin{tcolorbox}[title={$(24,8)$}]
(1) 13 1 * 1 $\leftarrow$ (9) 1 2 4 3 3 3 

(7) 1 1 2 4 3 3 3 $\leftarrow$ (8) ..4 3 3 3 

(9) 5 1 * 1 $\leftarrow$ (10) 6 * 1 

(13) 1 1 * 1 $\leftarrow$ (14) 2 * 1 

\end{tcolorbox}
\begin{tcolorbox}[title={$(24,9)$}]
(1) 12 1 1 * 1 $\leftarrow$ (2) 13 1 * 1 

(7) 1 2 3 4 4 1 1 1 $\leftarrow$ (9) 2 3 4 4 1 1 1 

(9) 4 1 1 * 1 $\leftarrow$ (10) 5 1 * 1 

\end{tcolorbox}
\begin{tcolorbox}[title={$(24,10)$}]
(1) * 2 4 3 3 3 

(1) 8 4 1 1 * 1 $\leftarrow$ (2) 12 1 1 * 1 

(5) 4 4 1 1 * 1 $\leftarrow$ (8) 1 2 3 4 4 1 1 1 

(7) * * 1 $\leftarrow$ (10) 4 1 1 * 1 

\end{tcolorbox}
\begin{tcolorbox}[title={$(24,11)$}]
(1) 6 * * 1 $\leftarrow$ (2) 8 4 1 1 * 1 

(2) 3 4 4 1 1 * 1 

(5) 2 * * 1 $\leftarrow$ (6) 4 4 1 1 * 1 

(6) 1 * * 1 $\leftarrow$ (8) * * 1 

\end{tcolorbox}
\begin{tcolorbox}[title={$(24,12)$}]
(1) 5 1 * * 1 $\leftarrow$ (2) 6 * * 1 

(5) 1 1 * * 1 $\leftarrow$ (6) 2 * * 1 

\end{tcolorbox}
\begin{tcolorbox}[title={$(24,13)$}]
(1) 4 1 1 * * 1 $\leftarrow$ (2) 5 1 * * 1 

\end{tcolorbox}
\begin{tcolorbox}[title={$(25,3)$}]
(7) 11 7 $\leftarrow$ (11) 15 

(11) 7 7 $\leftarrow$ (19) 7 

(23) 1 1 $\leftarrow$ (25) 1 

\end{tcolorbox}
\begin{tcolorbox}[title={$(25,4)$}]
(4) 7 7 7 $\leftarrow$ (10) 13 3 

(6) 5 7 7 $\leftarrow$ (18) 5 3 

(7) 10 5 3 $\leftarrow$ (8) 11 7 

(11) 6 5 3 $\leftarrow$ (12) 7 7 

(22) 1 1 1 $\leftarrow$ (24) 1 1 

\end{tcolorbox}
\begin{tcolorbox}[title={$(25,5)$}]
(2) 4 5 7 7 $\leftarrow$ (9) 11 3 3 

(3) 3 5 7 7 $\leftarrow$ (17) 3 3 3 

(7) 5 7 3 3 $\leftarrow$ (12) 6 5 3 

(7) 9 3 3 3 $\leftarrow$ (8) 10 5 3 

\end{tcolorbox}
\begin{tcolorbox}[title={$(25,6)$}]
(1) 2 3 5 7 7 $\leftarrow$ (8) 9 3 3 3 

(4) 3 5 7 3 3 $\leftarrow$ (9) 8 3 3 3 

(5) 4 7 3 3 3 $\leftarrow$ (6) 6 6 5 3 

(7) 4 5 3 3 3 $\leftarrow$ (8) 5 7 3 3 

(11) 5 1 2 3 3 $\leftarrow$ (12) 6 2 3 3 

\end{tcolorbox}
\begin{tcolorbox}[title={$(25,7)$}]
(2) 2 3 5 7 3 3 $\leftarrow$ (8) 4 5 3 3 3 

(5) 2 4 5 3 3 3 $\leftarrow$ (6) 4 7 3 3 3 

(11) 3 4 4 1 1 1 $\leftarrow$ (12) 5 1 2 3 3 

(15) 1 * 1 $\leftarrow$ (17) * 1 

\end{tcolorbox}
\begin{tcolorbox}[title={$(25,8)$}]
(8) 1 1 2 4 3 3 3 

(14) 1 1 * 1 $\leftarrow$ (16) 1 * 1 

\end{tcolorbox}
\begin{tcolorbox}[title={$(25,9)$}]
(3) 6 2 3 4 4 1 1 1 

\end{tcolorbox}
\begin{tcolorbox}[title={$(25,10)$}]
(2) * 2 4 3 3 3 

(3) 5 1 2 3 4 4 1 1 1 $\leftarrow$ (4) 6 2 3 4 4 1 1 1 

\end{tcolorbox}
\begin{tcolorbox}[title={$(25,11)$}]
(1) 1 * 2 4 3 3 3 

(3) 3 4 4 1 1 * 1 $\leftarrow$ (4) 5 1 2 3 4 4 1 1 1 

(7) 1 * * 1 $\leftarrow$ (9) * * 1 

\end{tcolorbox}
\begin{tcolorbox}[title={$(25,12)$}]
(1) 2 3 4 4 1 1 * 1 

(6) 1 1 * * 1 $\leftarrow$ (8) 1 * * 1 

\end{tcolorbox}
\begin{tcolorbox}[title={$(25,13)$}]
(2) 4 1 1 * * 1 

\end{tcolorbox}
\begin{tcolorbox}[title={$(26,2)$}]
(23) 3 $\leftarrow$ (27) 

\end{tcolorbox}
\begin{tcolorbox}[title={$(26,3)$}]
(11) 14 1 $\leftarrow$ (12) 15 

(19) 6 1 $\leftarrow$ (20) 7 

(20) 3 3 $\leftarrow$ (26) 1 

(23) 2 1 $\leftarrow$ (24) 3 

\end{tcolorbox}
\begin{tcolorbox}[title={$(26,4)$}]
(5) 7 7 7 $\leftarrow$ (9) 11 7 

(7) 5 7 7 $\leftarrow$ (13) 7 7 

(11) 13 1 1 $\leftarrow$ (12) 14 1 

(18) 2 3 3 $\leftarrow$ (25) 1 1 

(19) 5 1 1 $\leftarrow$ (20) 6 1 

(23) 1 1 1 $\leftarrow$ (24) 2 1 

\end{tcolorbox}
\begin{tcolorbox}[title={$(26,5)$}]
(3) 4 5 7 7 $\leftarrow$ (6) 7 7 7 

(4) 3 5 7 7 $\leftarrow$ (8) 5 7 7 

(11) 12 1 1 1 $\leftarrow$ (12) 13 1 1 

(17) 1 2 3 3 $\leftarrow$ (24) 1 1 1 

(19) 4 1 1 1 $\leftarrow$ (20) 5 1 1 

\end{tcolorbox}
\begin{tcolorbox}[title={$(26,6)$}]
(2) 2 3 5 7 7 $\leftarrow$ (4) 4 5 7 7 

(3) 3 6 6 5 3 

(5) 3 5 7 3 3 $\leftarrow$ (9) 5 7 3 3 

(9) 3 6 2 3 3 $\leftarrow$ (10) 8 3 3 3 

(11) 2 4 3 3 3 $\leftarrow$ (13) 6 2 3 3 

(11) 8 4 1 1 1 $\leftarrow$ (12) 12 1 1 1 

(15) 4 4 1 1 1 $\leftarrow$ (20) 4 1 1 1 

\end{tcolorbox}
\begin{tcolorbox}[title={$(26,7)$}]
(3) 2 3 5 7 3 3 $\leftarrow$ (4) 3 6 6 5 3 

(6) 2 4 5 3 3 3 

(9) ..4 3 3 3 $\leftarrow$ (10) 3 6 2 3 3 

(10) 1 2 4 3 3 3 $\leftarrow$ (12) 2 4 3 3 3 

(11) 6 * 1 $\leftarrow$ (12) 8 4 1 1 1 

(12) 3 4 4 1 1 1 $\leftarrow$ (18) * 1 

(15) 2 * 1 $\leftarrow$ (16) 4 4 1 1 1 

\end{tcolorbox}
\begin{tcolorbox}[title={$(26,8)$}]
(3) 13 1 * 1 

(9) 1 1 2 4 3 3 3 $\leftarrow$ (10) ..4 3 3 3 

(10) 2 3 4 4 1 1 1 $\leftarrow$ (17) 1 * 1 

(11) 5 1 * 1 $\leftarrow$ (12) 6 * 1 

(15) 1 1 * 1 $\leftarrow$ (16) 2 * 1 

\end{tcolorbox}
\begin{tcolorbox}[title={$(26,9)$}]
(1) 8 1 1 2 4 3 3 3 

(3) 12 1 1 * 1 $\leftarrow$ (4) 13 1 * 1 

(9) 1 2 3 4 4 1 1 1 $\leftarrow$ (16) 1 1 * 1 

(11) 4 1 1 * 1 $\leftarrow$ (12) 5 1 * 1 

\end{tcolorbox}
\begin{tcolorbox}[title={$(26,10)$}]
(1) 3 6 2 3 4 4 1 1 1 $\leftarrow$ (2) 8 1 1 2 4 3 3 3 

(3) * 2 4 3 3 3 $\leftarrow$ (5) 6 2 3 4 4 1 1 1 

(3) 8 4 1 1 * 1 $\leftarrow$ (4) 12 1 1 * 1 

(7) 4 4 1 1 * 1 $\leftarrow$ (12) 4 1 1 * 1 

\end{tcolorbox}
\begin{tcolorbox}[title={$(26,11)$}]
(1) 2 * 2 4 3 3 3 $\leftarrow$ (2) 3 6 2 3 4 4 1 1 1 

(2) 1 * 2 4 3 3 3 $\leftarrow$ (4) * 2 4 3 3 3 

(3) 6 * * 1 $\leftarrow$ (4) 8 4 1 1 * 1 

(4) 3 4 4 1 1 * 1 $\leftarrow$ (10) * * 1 

(7) 2 * * 1 $\leftarrow$ (8) 4 4 1 1 * 1 

\end{tcolorbox}
\begin{tcolorbox}[title={$(26,12)$}]
(1) 1 1 * 2 4 3 3 3 $\leftarrow$ (2) 2 * 2 4 3 3 3 

(2) 2 3 4 4 1 1 * 1 $\leftarrow$ (9) 1 * * 1 

(3) 5 1 * * 1 $\leftarrow$ (4) 6 * * 1 

(7) 1 1 * * 1 $\leftarrow$ (8) 2 * * 1 

\end{tcolorbox}
\begin{tcolorbox}[title={$(26,13)$}]
(1) 1 2 3 4 4 1 1 * 1 $\leftarrow$ (8) 1 1 * * 1 

(3) 4 1 1 * * 1 $\leftarrow$ (4) 5 1 * * 1 

\end{tcolorbox}
\begin{tcolorbox}[title={$(26,14)$}]
(1) * * * 1 

\end{tcolorbox}
\begin{tcolorbox}[title={$(27,3)$}]
(11) 13 3 $\leftarrow$ (13) 15 

(19) 5 3 $\leftarrow$ (21) 7 

(21) 3 3 $\leftarrow$ (25) 3 

\end{tcolorbox}
\begin{tcolorbox}[title={$(27,4)$}]
(9) 10 5 3 $\leftarrow$ (10) 11 7 

(10) 11 3 3 $\leftarrow$ (12) 13 3 

(13) 6 5 3 $\leftarrow$ (14) 7 7 

(18) 3 3 3 $\leftarrow$ (20) 5 3 

(19) 2 3 3 $\leftarrow$ (22) 3 3 

\end{tcolorbox}
\begin{tcolorbox}[title={$(27,5)$}]
(5) 3 5 7 7 $\leftarrow$ (9) 5 7 7 

(7) 6 6 5 3 $\leftarrow$ (14) 6 5 3 

(9) 9 3 3 3 $\leftarrow$ (10) 10 5 3 

(18) 1 2 3 3 $\leftarrow$ (20) 2 3 3 

\end{tcolorbox}
\begin{tcolorbox}[title={$(27,6)$}]
(3) 2 3 5 7 7 $\leftarrow$ (5) 4 5 7 7 

(6) 3 5 7 3 3 $\leftarrow$ (11) 8 3 3 3 

(7) 4 7 3 3 3 $\leftarrow$ (8) 6 6 5 3 

(9) 4 5 3 3 3 $\leftarrow$ (10) 5 7 3 3 

(13) 5 1 2 3 3 $\leftarrow$ (14) 6 2 3 3 

\end{tcolorbox}
\begin{tcolorbox}[title={$(27,7)$}]
(1) 3 3 6 6 5 3 $\leftarrow$ (4) 2 3 5 7 7 

(4) 2 3 5 7 3 3 $\leftarrow$ (10) 4 5 3 3 3 

(7) 2 4 5 3 3 3 $\leftarrow$ (8) 4 7 3 3 3 

(11) 1 2 4 3 3 3 $\leftarrow$ (13) 2 4 3 3 3 

(13) 3 4 4 1 1 1 $\leftarrow$ (14) 5 1 2 3 3 

\end{tcolorbox}
\begin{tcolorbox}[title={$(27,8)$}]
(1) 6 2 4 5 3 3 3 $\leftarrow$ (8) 2 4 5 3 3 3 

(10) 1 1 2 4 3 3 3 $\leftarrow$ (12) 1 2 4 3 3 3 

(11) 2 3 4 4 1 1 1 $\leftarrow$ (14) 3 4 4 1 1 1 

\end{tcolorbox}
\begin{tcolorbox}[title={$(27,9)$}]
(1) 4 ..4 5 3 3 3 $\leftarrow$ (2) 6 2 4 5 3 3 3 

(10) 1 2 3 4 4 1 1 1 $\leftarrow$ (12) 2 3 4 4 1 1 1 

\end{tcolorbox}
\begin{tcolorbox}[title={$(27,10)$}]
(1) ....4 5 3 3 3 $\leftarrow$ (2) 4 ..4 5 3 3 3 

(5) 5 1 2 3 4 4 1 1 1 $\leftarrow$ (6) 6 2 3 4 4 1 1 1 

\end{tcolorbox}
\begin{tcolorbox}[title={$(27,11)$}]
(3) 1 * 2 4 3 3 3 $\leftarrow$ (5) * 2 4 3 3 3 

(5) 3 4 4 1 1 * 1 $\leftarrow$ (6) 5 1 2 3 4 4 1 1 1 

\end{tcolorbox}
\begin{tcolorbox}[title={$(27,12)$}]
(2) 1 1 * 2 4 3 3 3 $\leftarrow$ (4) 1 * 2 4 3 3 3 

(3) 2 3 4 4 1 1 * 1 $\leftarrow$ (6) 3 4 4 1 1 * 1 

\end{tcolorbox}
\begin{tcolorbox}[title={$(27,13)$}]
(2) 1 2 3 4 4 1 1 * 1 $\leftarrow$ (4) 2 3 4 4 1 1 * 1 

(4) 4 1 1 * * 1 

\end{tcolorbox}
\begin{tcolorbox}[title={$(27,14)$}]
(2) * * * 1 

\end{tcolorbox}
\begin{tcolorbox}[title={$(27,15)$}]
(1) 1 * * * 1 

\end{tcolorbox}
\begin{tcolorbox}[title={$(28,2)$}]
(27) 1 $\leftarrow$ (29) 

\end{tcolorbox}
\begin{tcolorbox}[title={$(28,3)$}]
(13) 14 1 $\leftarrow$ (14) 15 

(21) 6 1 $\leftarrow$ (22) 7 

(25) 2 1 $\leftarrow$ (26) 3 

(26) 1 1 $\leftarrow$ (28) 1 

\end{tcolorbox}
\begin{tcolorbox}[title={$(28,4)$}]
(7) 7 7 7 $\leftarrow$ (11) 11 7 

(11) 11 3 3 $\leftarrow$ (13) 13 3 

(13) 13 1 1 $\leftarrow$ (14) 14 1 

(19) 3 3 3 $\leftarrow$ (21) 5 3 

(21) 5 1 1 $\leftarrow$ (22) 6 1 

(25) 1 1 1 $\leftarrow$ (26) 2 1 

\end{tcolorbox}
\begin{tcolorbox}[title={$(28,5)$}]
(6) 3 5 7 7 $\leftarrow$ (10) 5 7 7 

(10) 9 3 3 3 $\leftarrow$ (12) 11 3 3 

(13) 12 1 1 1 $\leftarrow$ (14) 13 1 1 

(19) 1 2 3 3 $\leftarrow$ (21) 2 3 3 

(21) 4 1 1 1 $\leftarrow$ (22) 5 1 1 

\end{tcolorbox}
\begin{tcolorbox}[title={$(28,6)$}]
(5) 3 6 6 5 3 $\leftarrow$ (11) 5 7 3 3 

(7) 3 5 7 3 3 $\leftarrow$ (9) 6 6 5 3 

(11) 3 6 2 3 3 $\leftarrow$ (12) 8 3 3 3 

(13) 8 4 1 1 1 $\leftarrow$ (14) 12 1 1 1 

(17) 4 4 1 1 1 $\leftarrow$ (20) 1 2 3 3 

(19) * 1 $\leftarrow$ (22) 4 1 1 1 

\end{tcolorbox}
\begin{tcolorbox}[title={$(28,7)$}]
(2) 3 3 6 6 5 3 $\leftarrow$ (8) 3 5 7 3 3 

(5) 2 3 5 7 3 3 $\leftarrow$ (6) 3 6 6 5 3 

(11) ..4 3 3 3 $\leftarrow$ (12) 3 6 2 3 3 

(13) 6 * 1 $\leftarrow$ (14) 8 4 1 1 1 

(17) 2 * 1 $\leftarrow$ (18) 4 4 1 1 1 

(18) 1 * 1 $\leftarrow$ (20) * 1 

\end{tcolorbox}
\begin{tcolorbox}[title={$(28,8)$}]
(5) 13 1 * 1 

(11) 1 1 2 4 3 3 3 $\leftarrow$ (12) ..4 3 3 3 

(13) 5 1 * 1 $\leftarrow$ (14) 6 * 1 

(17) 1 1 * 1 $\leftarrow$ (18) 2 * 1 

\end{tcolorbox}
\begin{tcolorbox}[title={$(28,9)$}]
(3) 8 1 1 2 4 3 3 3 

(5) 12 1 1 * 1 $\leftarrow$ (6) 13 1 * 1 

(11) 1 2 3 4 4 1 1 1 $\leftarrow$ (13) 2 3 4 4 1 1 1 

(13) 4 1 1 * 1 $\leftarrow$ (14) 5 1 * 1 

\end{tcolorbox}
\begin{tcolorbox}[title={$(28,10)$}]
(2) ....4 5 3 3 3 

(3) 3 6 2 3 4 4 1 1 1 $\leftarrow$ (4) 8 1 1 2 4 3 3 3 

(5) 8 4 1 1 * 1 $\leftarrow$ (6) 12 1 1 * 1 

(9) 4 4 1 1 * 1 $\leftarrow$ (12) 1 2 3 4 4 1 1 1 

(11) * * 1 $\leftarrow$ (14) 4 1 1 * 1 

\end{tcolorbox}
\begin{tcolorbox}[title={$(28,11)$}]
(3) 2 * 2 4 3 3 3 $\leftarrow$ (4) 3 6 2 3 4 4 1 1 1 

(5) 6 * * 1 $\leftarrow$ (6) 8 4 1 1 * 1 

(9) 2 * * 1 $\leftarrow$ (10) 4 4 1 1 * 1 

(10) 1 * * 1 $\leftarrow$ (12) * * 1 

\end{tcolorbox}
\begin{tcolorbox}[title={$(28,12)$}]
(3) 1 1 * 2 4 3 3 3 $\leftarrow$ (4) 2 * 2 4 3 3 3 

(5) 5 1 * * 1 $\leftarrow$ (6) 6 * * 1 

(9) 1 1 * * 1 $\leftarrow$ (10) 2 * * 1 

\end{tcolorbox}
\begin{tcolorbox}[title={$(28,13)$}]
(3) 1 2 3 4 4 1 1 * 1 $\leftarrow$ (5) 2 3 4 4 1 1 * 1 

(5) 4 1 1 * * 1 $\leftarrow$ (6) 5 1 * * 1 

\end{tcolorbox}
\begin{tcolorbox}[title={$(28,14)$}]
(1) 4 4 1 1 * * 1 $\leftarrow$ (4) 1 2 3 4 4 1 1 * 1 

(3) * * * 1 $\leftarrow$ (6) 4 1 1 * * 1 

\end{tcolorbox}
\begin{tcolorbox}[title={$(28,15)$}]
(1) 2 * * * 1 $\leftarrow$ (2) 4 4 1 1 * * 1 

(2) 1 * * * 1 $\leftarrow$ (4) * * * 1 

\end{tcolorbox}
\begin{tcolorbox}[title={$(28,16)$}]
(1) 1 1 * * * 1 $\leftarrow$ (2) 2 * * * 1 

\end{tcolorbox}
\begin{tcolorbox}[title={$(29,3)$}]
(15) 7 7 $\leftarrow$ (23) 7 

(23) 3 3 $\leftarrow$ (27) 3 

(27) 1 1 $\leftarrow$ (29) 1 

\end{tcolorbox}
\begin{tcolorbox}[title={$(29,4)$}]
(8) 7 7 7 $\leftarrow$ (22) 5 3 

(11) 10 5 3 $\leftarrow$ (12) 11 7 

(15) 6 5 3 $\leftarrow$ (16) 7 7 

(20) 3 3 3 $\leftarrow$ (24) 3 3 

(26) 1 1 1 $\leftarrow$ (28) 1 1 

\end{tcolorbox}
\begin{tcolorbox}[title={$(29,5)$}]
(6) 4 5 7 7 $\leftarrow$ (16) 6 5 3 

(7) 3 5 7 7 $\leftarrow$ (11) 5 7 7 

(11) 9 3 3 3 $\leftarrow$ (12) 10 5 3 

(15) 6 2 3 3 $\leftarrow$ (22) 2 3 3 

\end{tcolorbox}
\begin{tcolorbox}[title={$(29,6)$}]
(5) 2 3 5 7 7 $\leftarrow$ (13) 8 3 3 3 

(9) 4 7 3 3 3 $\leftarrow$ (10) 6 6 5 3 

(11) 4 5 3 3 3 $\leftarrow$ (12) 5 7 3 3 

(14) 2 4 3 3 3 $\leftarrow$ (21) 1 2 3 3 

(15) 5 1 2 3 3 $\leftarrow$ (16) 6 2 3 3 

\end{tcolorbox}
\begin{tcolorbox}[title={$(29,7)$}]
(3) 3 3 6 6 5 3 $\leftarrow$ (9) 3 5 7 3 3 

(6) 2 3 5 7 3 3 

(9) 2 4 5 3 3 3 $\leftarrow$ (10) 4 7 3 3 3 

(13) 1 2 4 3 3 3 $\leftarrow$ (19) 4 4 1 1 1 

(15) 3 4 4 1 1 1 $\leftarrow$ (16) 5 1 2 3 3 

(19) 1 * 1 $\leftarrow$ (21) * 1 

\end{tcolorbox}
\begin{tcolorbox}[title={$(29,8)$}]
(3) 6 2 4 5 3 3 3 

(12) 1 1 2 4 3 3 3 $\leftarrow$ (16) 3 4 4 1 1 1 

(18) 1 1 * 1 $\leftarrow$ (20) 1 * 1 

\end{tcolorbox}
\begin{tcolorbox}[title={$(29,9)$}]
(1) 6 ..4 5 3 3 3 

(3) 4 ..4 5 3 3 3 $\leftarrow$ (4) 6 2 4 5 3 3 3 

(7) 6 2 3 4 4 1 1 1 $\leftarrow$ (14) 2 3 4 4 1 1 1 

\end{tcolorbox}
\begin{tcolorbox}[title={$(29,10)$}]
(1) 4 ...4 5 3 3 3 $\leftarrow$ (2) 6 ..4 5 3 3 3 

(3) ....4 5 3 3 3 $\leftarrow$ (4) 4 ..4 5 3 3 3 

(6) * 2 4 3 3 3 $\leftarrow$ (13) 1 2 3 4 4 1 1 1 

(7) 5 1 2 3 4 4 1 1 1 $\leftarrow$ (8) 6 2 3 4 4 1 1 1 

\end{tcolorbox}
\begin{tcolorbox}[title={$(29,11)$}]
(1) .....4 5 3 3 3 $\leftarrow$ (2) 4 ...4 5 3 3 3 

(5) 1 * 2 4 3 3 3 $\leftarrow$ (11) 4 4 1 1 * 1 

(7) 3 4 4 1 1 * 1 $\leftarrow$ (8) 5 1 2 3 4 4 1 1 1 

(11) 1 * * 1 $\leftarrow$ (13) * * 1 

\end{tcolorbox}
\begin{tcolorbox}[title={$(29,12)$}]
(4) 1 1 * 2 4 3 3 3 $\leftarrow$ (8) 3 4 4 1 1 * 1 

(10) 1 1 * * 1 $\leftarrow$ (12) 1 * * 1 

\end{tcolorbox}
\begin{tcolorbox}[title={$(29,15)$}]
(3) 1 * * * 1 $\leftarrow$ (5) * * * 1 

\end{tcolorbox}
\begin{tcolorbox}[title={$(29,16)$}]
(2) 1 1 * * * 1 $\leftarrow$ (4) 1 * * * 1 

\end{tcolorbox}
\begin{tcolorbox}[title={$(30,2)$}]
(15) 15 

\end{tcolorbox}
\begin{tcolorbox}[title={$(30,3)$}]
(14) 13 3 

(15) 14 1 $\leftarrow$ (16) 15 

(23) 6 1 $\leftarrow$ (24) 7 

(27) 2 1 $\leftarrow$ (28) 3 

\end{tcolorbox}
\begin{tcolorbox}[title={$(30,4)$}]
(9) 7 7 7 $\leftarrow$ (17) 7 7 

(13) 11 3 3 

(15) 13 1 1 $\leftarrow$ (16) 14 1 

(21) 3 3 3 $\leftarrow$ (25) 3 3 

(23) 5 1 1 $\leftarrow$ (24) 6 1 

(27) 1 1 1 $\leftarrow$ (28) 2 1 

\end{tcolorbox}
\begin{tcolorbox}[title={$(30,5)$}]
(7) 4 5 7 7 $\leftarrow$ (10) 7 7 7 

(8) 3 5 7 7 

(12) 9 3 3 3 

(15) 12 1 1 1 $\leftarrow$ (16) 13 1 1 

(23) 4 1 1 1 $\leftarrow$ (24) 5 1 1 

\end{tcolorbox}
\begin{tcolorbox}[title={$(30,6)$}]
(6) 2 3 5 7 7 $\leftarrow$ (8) 4 5 7 7 

(7) 3 6 6 5 3 $\leftarrow$ (11) 6 6 5 3 

(12) 4 5 3 3 3 

(13) 3 6 2 3 3 $\leftarrow$ (14) 8 3 3 3 

(15) 2 4 3 3 3 $\leftarrow$ (17) 6 2 3 3 

(15) 8 4 1 1 1 $\leftarrow$ (16) 12 1 1 1 

\end{tcolorbox}
\begin{tcolorbox}[title={$(30,7)$}]
(4) 3 3 6 6 5 3 $\leftarrow$ (10) 3 5 7 3 3 

(7) 2 3 5 7 3 3 $\leftarrow$ (8) 3 6 6 5 3 

(10) 2 4 5 3 3 3 

(13) ..4 3 3 3 $\leftarrow$ (14) 3 6 2 3 3 

(14) 1 2 4 3 3 3 $\leftarrow$ (16) 2 4 3 3 3 

(15) 6 * 1 $\leftarrow$ (16) 8 4 1 1 1 

(19) 2 * 1 $\leftarrow$ (20) 4 4 1 1 1 

\end{tcolorbox}
\begin{tcolorbox}[title={$(30,8)$}]
(1) 6 2 3 5 7 3 3 $\leftarrow$ (8) 2 3 5 7 3 3 

(7) 13 1 * 1 

(13) 1 1 2 4 3 3 3 $\leftarrow$ (14) ..4 3 3 3 

(15) 5 1 * 1 $\leftarrow$ (16) 6 * 1 

(19) 1 1 * 1 $\leftarrow$ (20) 2 * 1 

\end{tcolorbox}
\begin{tcolorbox}[title={$(30,9)$}]
(1) 3 6 2 4 5 3 3 3 $\leftarrow$ (2) 6 2 3 5 7 3 3 

(5) 8 1 1 2 4 3 3 3 

(7) 12 1 1 * 1 $\leftarrow$ (8) 13 1 * 1 

(15) 4 1 1 * 1 $\leftarrow$ (16) 5 1 * 1 

\end{tcolorbox}
\begin{tcolorbox}[title={$(30,10)$}]
(1) ....3 5 7 3 3 $\leftarrow$ (2) 3 6 2 4 5 3 3 3 

(4) ....4 5 3 3 3 

(5) 3 6 2 3 4 4 1 1 1 $\leftarrow$ (6) 8 1 1 2 4 3 3 3 

(7) * 2 4 3 3 3 $\leftarrow$ (9) 6 2 3 4 4 1 1 1 

(7) 8 4 1 1 * 1 $\leftarrow$ (8) 12 1 1 * 1 

\end{tcolorbox}
\begin{tcolorbox}[title={$(30,11)$}]
(2) .....4 5 3 3 3 

(5) 2 * 2 4 3 3 3 $\leftarrow$ (6) 3 6 2 3 4 4 1 1 1 

(6) 1 * 2 4 3 3 3 $\leftarrow$ (8) * 2 4 3 3 3 

(7) 6 * * 1 $\leftarrow$ (8) 8 4 1 1 * 1 

(11) 2 * * 1 $\leftarrow$ (12) 4 4 1 1 * 1 

\end{tcolorbox}
\begin{tcolorbox}[title={$(30,12)$}]
(5) 1 1 * 2 4 3 3 3 $\leftarrow$ (6) 2 * 2 4 3 3 3 

(6) 2 3 4 4 1 1 * 1 

(7) 5 1 * * 1 $\leftarrow$ (8) 6 * * 1 

(11) 1 1 * * 1 $\leftarrow$ (12) 2 * * 1 

\end{tcolorbox}
\begin{tcolorbox}[title={$(30,13)$}]
(5) 1 2 3 4 4 1 1 * 1 

(7) 4 1 1 * * 1 $\leftarrow$ (8) 5 1 * * 1 

\end{tcolorbox}
\begin{tcolorbox}[title={$(30,14)$}]
(3) 4 4 1 1 * * 1 

\end{tcolorbox}
\begin{tcolorbox}[title={$(30,15)$}]
(3) 2 * * * 1 $\leftarrow$ (4) 4 4 1 1 * * 1 

\end{tcolorbox}
\begin{tcolorbox}[title={$(30,16)$}]
(3) 1 1 * * * 1 $\leftarrow$ (4) 2 * * * 1 

\end{tcolorbox}
\begin{tcolorbox}[title={$(31,1)$}]
(31) 

\end{tcolorbox}
\begin{tcolorbox}[title={$(31,2)$}]
(30) 1 

\end{tcolorbox}
\begin{tcolorbox}[title={$(31,3)$}]
(1) 15 15 

(13) 11 7 

(15) 13 3 $\leftarrow$ (17) 15 

(23) 5 3 $\leftarrow$ (25) 7 

(29) 1 1 

\end{tcolorbox}
\begin{tcolorbox}[title={$(31,4)$}]
(1) 14 13 3 $\leftarrow$ (2) 15 15 

(12) 5 7 7 

(13) 10 5 3 $\leftarrow$ (14) 11 7 

(14) 11 3 3 $\leftarrow$ (16) 13 3 

(17) 6 5 3 $\leftarrow$ (18) 7 7 

(22) 3 3 3 $\leftarrow$ (24) 5 3 

(23) 2 3 3 $\leftarrow$ (26) 3 3 

(28) 1 1 1 

\end{tcolorbox}
\begin{tcolorbox}[title={$(31,5)$}]
(1) 13 11 3 3 $\leftarrow$ (2) 14 13 3 

(9) 3 5 7 7 

(13) 5 7 3 3 $\leftarrow$ (18) 6 5 3 

(13) 9 3 3 3 $\leftarrow$ (14) 10 5 3 

(22) 1 2 3 3 $\leftarrow$ (24) 2 3 3 

(24) 4 1 1 1 

\end{tcolorbox}
\begin{tcolorbox}[title={$(31,6)$}]
(1) 12 9 3 3 3 $\leftarrow$ (2) 13 11 3 3 

(7) 2 3 5 7 7 $\leftarrow$ (9) 4 5 7 7 

(11) 4 7 3 3 3 $\leftarrow$ (12) 6 6 5 3 

(13) 4 5 3 3 3 $\leftarrow$ (14) 5 7 3 3 

(17) 5 1 2 3 3 $\leftarrow$ (18) 6 2 3 3 

(22) * 1 

\end{tcolorbox}
\begin{tcolorbox}[title={$(31,7)$}]
(1) 12 4 5 3 3 3 

(5) 3 3 6 6 5 3 $\leftarrow$ (8) 2 3 5 7 7 

(11) 2 4 5 3 3 3 $\leftarrow$ (12) 4 7 3 3 3 

(15) 1 2 4 3 3 3 $\leftarrow$ (17) 2 4 3 3 3 

(17) 3 4 4 1 1 1 $\leftarrow$ (18) 5 1 2 3 3 

(21) 1 * 1 

\end{tcolorbox}
\begin{tcolorbox}[title={$(31,8)$}]
(1) 10 2 4 5 3 3 3 $\leftarrow$ (2) 12 4 5 3 3 3 

(5) 6 2 4 5 3 3 3 

(14) 1 1 2 4 3 3 3 $\leftarrow$ (16) 1 2 4 3 3 3 

(15) 2 3 4 4 1 1 1 $\leftarrow$ (18) 3 4 4 1 1 1 

(20) 1 1 * 1 

\end{tcolorbox}
\begin{tcolorbox}[title={$(31,9)$}]
(1) 7 13 1 * 1 $\leftarrow$ (2) 10 2 4 5 3 3 3 

(3) 6 ..4 5 3 3 3 

(5) 4 ..4 5 3 3 3 $\leftarrow$ (6) 6 2 4 5 3 3 3 

(14) 1 2 3 4 4 1 1 1 $\leftarrow$ (16) 2 3 4 4 1 1 1 

(16) 4 1 1 * 1 

\end{tcolorbox}
\begin{tcolorbox}[title={$(31,10)$}]
(1) 5 8 1 1 2 4 3 3 3 $\leftarrow$ (2) 7 13 1 * 1 

(2) ....3 5 7 3 3 

(3) 4 ...4 5 3 3 3 $\leftarrow$ (4) 6 ..4 5 3 3 3 

(5) ....4 5 3 3 3 $\leftarrow$ (6) 4 ..4 5 3 3 3 

(9) 5 1 2 3 4 4 1 1 1 $\leftarrow$ (10) 6 2 3 4 4 1 1 1 

(14) * * 1 

\end{tcolorbox}
\begin{tcolorbox}[title={$(31,11)$}]
(1) 4 ....4 5 3 3 3 $\leftarrow$ (2) 5 8 1 1 2 4 3 3 3 

(3) .....4 5 3 3 3 $\leftarrow$ (4) 4 ...4 5 3 3 3 

(7) 1 * 2 4 3 3 3 $\leftarrow$ (9) * 2 4 3 3 3 

(9) 3 4 4 1 1 * 1 $\leftarrow$ (10) 5 1 2 3 4 4 1 1 1 

(13) 1 * * 1 

\end{tcolorbox}
\begin{tcolorbox}[title={$(31,12)$}]
(1) ......4 5 3 3 3 $\leftarrow$ (2) 4 ....4 5 3 3 3 

(6) 1 1 * 2 4 3 3 3 $\leftarrow$ (8) 1 * 2 4 3 3 3 

(7) 2 3 4 4 1 1 * 1 $\leftarrow$ (10) 3 4 4 1 1 * 1 

(12) 1 1 * * 1 

\end{tcolorbox}
\begin{tcolorbox}[title={$(31,13)$}]
(1) 6 2 3 4 4 1 1 * 1 

(6) 1 2 3 4 4 1 1 * 1 $\leftarrow$ (8) 2 3 4 4 1 1 * 1 

(8) 4 1 1 * * 1 

\end{tcolorbox}

\begin{tcolorbox}[title={$(31,14)$}]
(1) 5 1 2 3 4 4 1 1 * 1 $\leftarrow$ (2) 6 2 3 4 4 1 1 * 1 

(6) * * * 1 

\end{tcolorbox}
\begin{tcolorbox}[title={$(31,15)$}]
(1) 3 4 4 1 1 * * 1 $\leftarrow$ (2) 5 1 2 3 4 4 1 1 * 1 

(5) 1 * * * 1 

\end{tcolorbox}
\begin{tcolorbox}[title={$(31,16)$}]
(4) 1 1 * * * 1 

\end{tcolorbox}

\scriptsize

\begin{tcolorbox}[title={$(32,2)$}]
(1) 31 

(29) 3 

(31) 1 $\leftarrow$ (33) 

\end{tcolorbox}
\begin{tcolorbox}[title={$(32,3)$}]
(1) 30 1 $\leftarrow$ (2) 31 

(17) 14 1 $\leftarrow$ (18) 15 

(25) 6 1 $\leftarrow$ (26) 7 

(29) 2 1 $\leftarrow$ (30) 3 

(30) 1 1 $\leftarrow$ (32) 1 

\end{tcolorbox}
\begin{tcolorbox}[title={$(32,4)$}]
(1) 13 11 7 

(1) 29 1 1 $\leftarrow$ (2) 30 1 

(11) 7 7 7 $\leftarrow$ (19) 7 7 

(13) 5 7 7 

(15) 11 3 3 $\leftarrow$ (17) 13 3 

(17) 13 1 1 $\leftarrow$ (18) 14 1 

(23) 3 3 3 $\leftarrow$ (25) 5 3 

(25) 5 1 1 $\leftarrow$ (26) 6 1 

(29) 1 1 1 $\leftarrow$ (30) 2 1 

\end{tcolorbox}
\begin{tcolorbox}[title={$(32,5)$}]
(1) 28 1 1 1 $\leftarrow$ (2) 29 1 1 

(10) 3 5 7 7 

(14) 9 3 3 3 $\leftarrow$ (16) 11 3 3 

(15) 8 3 3 3 $\leftarrow$ (24) 3 3 3 

(17) 12 1 1 1 $\leftarrow$ (18) 13 1 1 

(23) 1 2 3 3 $\leftarrow$ (25) 2 3 3 

(25) 4 1 1 1 $\leftarrow$ (26) 5 1 1 

\end{tcolorbox}
\begin{tcolorbox}[title={$(32,6)$}]
(1) 9 3 5 7 7 

(1) 24 4 1 1 1 $\leftarrow$ (2) 28 1 1 1 

(2) 12 9 3 3 3 

(9) 3 6 6 5 3 

(11) 3 5 7 3 3 $\leftarrow$ (13) 6 6 5 3 

(14) 4 5 3 3 3 $\leftarrow$ (19) 6 2 3 3 

(15) 3 6 2 3 3 $\leftarrow$ (16) 8 3 3 3 

(17) 8 4 1 1 1 $\leftarrow$ (18) 12 1 1 1 

(21) 4 4 1 1 1 $\leftarrow$ (24) 1 2 3 3 

(23) * 1 $\leftarrow$ (26) 4 1 1 1 

\end{tcolorbox}
\begin{tcolorbox}[title={$(32,7)$}]
(1) 22 * 1 $\leftarrow$ (2) 24 4 1 1 1 

(6) 3 3 6 6 5 3 

(9) 2 3 5 7 3 3 $\leftarrow$ (10) 3 6 6 5 3 

(12) 2 4 5 3 3 3 $\leftarrow$ (18) 2 4 3 3 3 

(15) ..4 3 3 3 $\leftarrow$ (16) 3 6 2 3 3 

(17) 6 * 1 $\leftarrow$ (18) 8 4 1 1 1 

(21) 2 * 1 $\leftarrow$ (22) 4 4 1 1 1 

(22) 1 * 1 $\leftarrow$ (24) * 1 

\end{tcolorbox}
\begin{tcolorbox}[title={$(32,8)$}]
(1) 21 1 * 1 $\leftarrow$ (2) 22 * 1 

(3) 6 2 3 5 7 3 3 

(9) 13 1 * 1 $\leftarrow$ (17) 1 2 4 3 3 3 

(15) 1 1 2 4 3 3 3 $\leftarrow$ (16) ..4 3 3 3 

(17) 5 1 * 1 $\leftarrow$ (18) 6 * 1 

(21) 1 1 * 1 $\leftarrow$ (22) 2 * 1 

\end{tcolorbox}
\begin{tcolorbox}[title={$(32,9)$}]
(1) 5 6 2 4 5 3 3 3 

(1) 20 1 1 * 1 $\leftarrow$ (2) 21 1 * 1 

(3) 3 6 2 4 5 3 3 3 $\leftarrow$ (4) 6 2 3 5 7 3 3 

(7) 8 1 1 2 4 3 3 3 $\leftarrow$ (16) 1 1 2 4 3 3 3 

(9) 12 1 1 * 1 $\leftarrow$ (10) 13 1 * 1 

(15) 1 2 3 4 4 1 1 1 $\leftarrow$ (17) 2 3 4 4 1 1 1 

(17) 4 1 1 * 1 $\leftarrow$ (18) 5 1 * 1 

\end{tcolorbox}
\begin{tcolorbox}[title={$(32,10)$}]
(1) 3 6 ..4 5 3 3 3 $\leftarrow$ (2) 5 6 2 4 5 3 3 3 

(1) 16 4 1 1 * 1 $\leftarrow$ (2) 20 1 1 * 1 

(3) ....3 5 7 3 3 $\leftarrow$ (4) 3 6 2 4 5 3 3 3 

(6) ....4 5 3 3 3 $\leftarrow$ (11) 6 2 3 4 4 1 1 1 

(7) 3 6 2 3 4 4 1 1 1 $\leftarrow$ (8) 8 1 1 2 4 3 3 3 

(9) 8 4 1 1 * 1 $\leftarrow$ (10) 12 1 1 * 1 

(13) 4 4 1 1 * 1 $\leftarrow$ (16) 1 2 3 4 4 1 1 1 

(15) * * 1 $\leftarrow$ (18) 4 1 1 * 1 

\end{tcolorbox}
\begin{tcolorbox}[title={$(32,11)$}]
(1) .....3 5 7 3 3 $\leftarrow$ (2) 3 6 ..4 5 3 3 3 

(1) 14 * * 1 $\leftarrow$ (2) 16 4 1 1 * 1 

(4) .....4 5 3 3 3 $\leftarrow$ (10) * 2 4 3 3 3 

(7) 2 * 2 4 3 3 3 $\leftarrow$ (8) 3 6 2 3 4 4 1 1 1 

(9) 6 * * 1 $\leftarrow$ (10) 8 4 1 1 * 1 

(13) 2 * * 1 $\leftarrow$ (14) 4 4 1 1 * 1 

(14) 1 * * 1 $\leftarrow$ (16) * * 1 

\end{tcolorbox}
\begin{tcolorbox}[title={$(32,12)$}]
(1) 13 1 * * 1 $\leftarrow$ (2) 14 * * 1 

(2) ......4 5 3 3 3 $\leftarrow$ (9) 1 * 2 4 3 3 3 

(7) 1 1 * 2 4 3 3 3 $\leftarrow$ (8) 2 * 2 4 3 3 3 

(9) 5 1 * * 1 $\leftarrow$ (10) 6 * * 1 

(13) 1 1 * * 1 $\leftarrow$ (14) 2 * * 1 

\end{tcolorbox}
\begin{tcolorbox}[title={$(32,13)$}]
(1) 12 1 1 * * 1 $\leftarrow$ (2) 13 1 * * 1 

(7) 1 2 3 4 4 1 1 * 1 $\leftarrow$ (9) 2 3 4 4 1 1 * 1 

(9) 4 1 1 * * 1 $\leftarrow$ (10) 5 1 * * 1 

\end{tcolorbox}
\begin{tcolorbox}[title={$(32,14)$}]
(1) * * 2 4 3 3 3 

(1) 8 4 1 1 * * 1 $\leftarrow$ (2) 12 1 1 * * 1 

(5) 4 4 1 1 * * 1 $\leftarrow$ (8) 1 2 3 4 4 1 1 * 1 

(7) * * * 1 $\leftarrow$ (10) 4 1 1 * * 1 

\end{tcolorbox}
\begin{tcolorbox}[title={$(32,15)$}]
(1) 6 * * * 1 $\leftarrow$ (2) 8 4 1 1 * * 1 

(2) 3 4 4 1 1 * * 1 

(5) 2 * * * 1 $\leftarrow$ (6) 4 4 1 1 * * 1 

(6) 1 * * * 1 $\leftarrow$ (8) * * * 1 

\end{tcolorbox}
\begin{tcolorbox}[title={$(32,16)$}]
(1) 5 1 * * * 1 $\leftarrow$ (2) 6 * * * 1 

(5) 1 1 * * * 1 $\leftarrow$ (6) 2 * * * 1 

\end{tcolorbox}
\begin{tcolorbox}[title={$(32,17)$}]
(1) 4 1 1 * * * 1 $\leftarrow$ (2) 5 1 * * * 1 

\end{tcolorbox}
\begin{tcolorbox}[title={$(33,3)$}]
(1) 29 3 

(3) 15 15 

(15) 11 7 $\leftarrow$ (19) 15 

(27) 3 3 

(31) 1 1 $\leftarrow$ (33) 1 

\end{tcolorbox}
\begin{tcolorbox}[title={$(33,4)$}]
(2) 13 11 7 

(3) 14 13 3 $\leftarrow$ (4) 15 15 

(12) 7 7 7 $\leftarrow$ (18) 13 3 

(14) 5 7 7 

(15) 10 5 3 $\leftarrow$ (16) 11 7 

(19) 6 5 3 $\leftarrow$ (20) 7 7 

(30) 1 1 1 $\leftarrow$ (32) 1 1 

\end{tcolorbox}
\begin{tcolorbox}[title={$(33,5)$}]
(1) 13 5 7 7 

(3) 13 11 3 3 $\leftarrow$ (4) 14 13 3 

(10) 4 5 7 7 $\leftarrow$ (17) 11 3 3 

(11) 3 5 7 7 

(15) 5 7 3 3 $\leftarrow$ (20) 6 5 3 

(15) 9 3 3 3 $\leftarrow$ (16) 10 5 3 

\end{tcolorbox}
\begin{tcolorbox}[title={$(33,6)$}]
(2) 9 3 5 7 7 

(3) 12 9 3 3 3 $\leftarrow$ (4) 13 11 3 3 

(9) 2 3 5 7 7 $\leftarrow$ (16) 9 3 3 3 

(12) 3 5 7 3 3 $\leftarrow$ (17) 8 3 3 3 

(13) 4 7 3 3 3 $\leftarrow$ (14) 6 6 5 3 

(15) 4 5 3 3 3 $\leftarrow$ (16) 5 7 3 3 

(19) 5 1 2 3 3 $\leftarrow$ (20) 6 2 3 3 

\end{tcolorbox}
\begin{tcolorbox}[title={$(33,7)$}]
(1) 9 3 6 6 5 3 

(3) 12 4 5 3 3 3 

(7) 3 3 6 6 5 3 $\leftarrow$ (11) 3 6 6 5 3 

(10) 2 3 5 7 3 3 $\leftarrow$ (16) 4 5 3 3 3 

(13) 2 4 5 3 3 3 $\leftarrow$ (14) 4 7 3 3 3 

(19) 3 4 4 1 1 1 $\leftarrow$ (20) 5 1 2 3 3 

(23) 1 * 1 $\leftarrow$ (25) * 1 

\end{tcolorbox}
\begin{tcolorbox}[title={$(33,8)$}]
(1) 6 3 3 6 6 5 3 $\leftarrow$ (8) 3 3 6 6 5 3 

(3) 10 2 4 5 3 3 3 $\leftarrow$ (4) 12 4 5 3 3 3 

(7) 6 2 4 5 3 3 3 $\leftarrow$ (14) 2 4 5 3 3 3 

(22) 1 1 * 1 $\leftarrow$ (24) 1 * 1 

\end{tcolorbox}
\begin{tcolorbox}[title={$(33,9)$}]
(1) 3 6 2 3 5 7 3 3 $\leftarrow$ (2) 6 3 3 6 6 5 3 

(3) 7 13 1 * 1 $\leftarrow$ (4) 10 2 4 5 3 3 3 

(5) 6 ..4 5 3 3 3 $\leftarrow$ (11) 13 1 * 1 

(7) 4 ..4 5 3 3 3 $\leftarrow$ (8) 6 2 4 5 3 3 3 

\end{tcolorbox}
\begin{tcolorbox}[title={$(33,10)$}]
(1) ...3 3 6 6 5 3 $\leftarrow$ (2) 3 6 2 3 5 7 3 3 

(3) 5 8 1 1 2 4 3 3 3 $\leftarrow$ (4) 7 13 1 * 1 

(4) ....3 5 7 3 3 $\leftarrow$ (9) 8 1 1 2 4 3 3 3 

(5) 4 ...4 5 3 3 3 $\leftarrow$ (6) 6 ..4 5 3 3 3 

(7) ....4 5 3 3 3 $\leftarrow$ (8) 4 ..4 5 3 3 3 

(11) 5 1 2 3 4 4 1 1 1 $\leftarrow$ (12) 6 2 3 4 4 1 1 1 

\end{tcolorbox}
\begin{tcolorbox}[title={$(33,11)$}]
(2) .....3 5 7 3 3 $\leftarrow$ (8) ....4 5 3 3 3 

(3) 4 ....4 5 3 3 3 $\leftarrow$ (4) 5 8 1 1 2 4 3 3 3 

(5) .....4 5 3 3 3 $\leftarrow$ (6) 4 ...4 5 3 3 3 

(11) 3 4 4 1 1 * 1 $\leftarrow$ (12) 5 1 2 3 4 4 1 1 1 

(15) 1 * * 1 $\leftarrow$ (17) * * 1 

\end{tcolorbox}
\begin{tcolorbox}[title={$(33,12)$}]
(3) ......4 5 3 3 3 $\leftarrow$ (4) 4 ....4 5 3 3 3 

(8) 1 1 * 2 4 3 3 3 

(14) 1 1 * * 1 $\leftarrow$ (16) 1 * * 1 

\end{tcolorbox}
\begin{tcolorbox}[title={$(33,13)$}]
(3) 6 2 3 4 4 1 1 * 1 

\end{tcolorbox}
\begin{tcolorbox}[title={$(33,14)$}]
(2) * * 2 4 3 3 3 

(3) 5 1 2 3 4 4 1 1 * 1 $\leftarrow$ (4) 6 2 3 4 4 1 1 * 1 

\end{tcolorbox}
\begin{tcolorbox}[title={$(33,15)$}]
(1) 1 * * 2 4 3 3 3 

(3) 3 4 4 1 1 * * 1 $\leftarrow$ (4) 5 1 2 3 4 4 1 1 * 1 

(7) 1 * * * 1 $\leftarrow$ (9) * * * 1 

\end{tcolorbox}
\begin{tcolorbox}[title={$(33,16)$}]
(1) 2 3 4 4 1 1 * * 1 

(6) 1 1 * * * 1 $\leftarrow$ (8) 1 * * * 1 

\end{tcolorbox}
\begin{tcolorbox}[title={$(33,17)$}]
(2) 4 1 1 * * * 1 

\end{tcolorbox}
\begin{tcolorbox}[title={$(34,2)$}]
(3) 31 

(27) 7 

(31) 3 $\leftarrow$ (35) 

\end{tcolorbox}
\begin{tcolorbox}[title={$(34,3)$}]
(2) 29 3 

(3) 30 1 $\leftarrow$ (4) 31 

(19) 14 1 $\leftarrow$ (20) 15 

(26) 5 3 

(27) 6 1 $\leftarrow$ (28) 7 

(28) 3 3 $\leftarrow$ (34) 1 

(31) 2 1 $\leftarrow$ (32) 3 

\end{tcolorbox}
\begin{tcolorbox}[title={$(34,4)$}]
(1) 27 3 3 

(3) 13 11 7 $\leftarrow$ (5) 15 15 

(3) 29 1 1 $\leftarrow$ (4) 30 1 

(13) 7 7 7 $\leftarrow$ (17) 11 7 

(15) 5 7 7 $\leftarrow$ (21) 7 7 

(19) 13 1 1 $\leftarrow$ (20) 14 1 

(25) 3 3 3 

(26) 2 3 3 $\leftarrow$ (33) 1 1 

(27) 5 1 1 $\leftarrow$ (28) 6 1 

(31) 1 1 1 $\leftarrow$ (32) 2 1 

\end{tcolorbox}
\begin{tcolorbox}[title={$(34,5)$}]
(1) 14 5 7 7 $\leftarrow$ (4) 13 11 7 

(2) 13 5 7 7 

(3) 28 1 1 1 $\leftarrow$ (4) 29 1 1 

(11) 4 5 7 7 $\leftarrow$ (14) 7 7 7 

(12) 3 5 7 7 $\leftarrow$ (16) 5 7 7 

(19) 12 1 1 1 $\leftarrow$ (20) 13 1 1 

(25) 1 2 3 3 $\leftarrow$ (32) 1 1 1 

(27) 4 1 1 1 $\leftarrow$ (28) 5 1 1 

\end{tcolorbox}
\begin{tcolorbox}[title={$(34,6)$}]
(1) 11 3 5 7 7 $\leftarrow$ (2) 14 5 7 7 

(3) 9 3 5 7 7 

(3) 24 4 1 1 1 $\leftarrow$ (4) 28 1 1 1 

(4) 12 9 3 3 3 

(10) 2 3 5 7 7 $\leftarrow$ (12) 4 5 7 7 

(13) 3 5 7 3 3 $\leftarrow$ (17) 5 7 3 3 

(17) 3 6 2 3 3 $\leftarrow$ (18) 8 3 3 3 

(19) 2 4 3 3 3 $\leftarrow$ (21) 6 2 3 3 

(19) 8 4 1 1 1 $\leftarrow$ (20) 12 1 1 1 

(23) 4 4 1 1 1 $\leftarrow$ (28) 4 1 1 1 

\end{tcolorbox}
\begin{tcolorbox}[title={$(34,7)$}]
(2) 9 3 6 6 5 3 

(3) 22 * 1 $\leftarrow$ (4) 24 4 1 1 1 

(11) 2 3 5 7 3 3 $\leftarrow$ (12) 3 6 6 5 3 

(17) ..4 3 3 3 $\leftarrow$ (18) 3 6 2 3 3 

(18) 1 2 4 3 3 3 $\leftarrow$ (20) 2 4 3 3 3 

(19) 6 * 1 $\leftarrow$ (20) 8 4 1 1 1 

(20) 3 4 4 1 1 1 $\leftarrow$ (26) * 1 

(23) 2 * 1 $\leftarrow$ (24) 4 4 1 1 1 

\end{tcolorbox}
\begin{tcolorbox}[title={$(34,8)$}]
(1) 5 5 3 6 6 5 3 

(3) 21 1 * 1 $\leftarrow$ (4) 22 * 1 

(5) 6 2 3 5 7 3 3 

(17) 1 1 2 4 3 3 3 $\leftarrow$ (18) ..4 3 3 3 

(18) 2 3 4 4 1 1 1 $\leftarrow$ (25) 1 * 1 

(19) 5 1 * 1 $\leftarrow$ (20) 6 * 1 

(23) 1 1 * 1 $\leftarrow$ (24) 2 * 1 

\end{tcolorbox}
\begin{tcolorbox}[title={$(34,9)$}]
(3) 5 6 2 4 5 3 3 3 

(3) 20 1 1 * 1 $\leftarrow$ (4) 21 1 * 1 

(5) 3 6 2 4 5 3 3 3 $\leftarrow$ (6) 6 2 3 5 7 3 3 

(11) 12 1 1 * 1 $\leftarrow$ (12) 13 1 * 1 

(17) 1 2 3 4 4 1 1 1 $\leftarrow$ (24) 1 1 * 1 

(19) 4 1 1 * 1 $\leftarrow$ (20) 5 1 * 1 

\end{tcolorbox}
\begin{tcolorbox}[title={$(34,10)$}]
(2) ...3 3 6 6 5 3 

(3) 3 6 ..4 5 3 3 3 $\leftarrow$ (4) 5 6 2 4 5 3 3 3 

(3) 16 4 1 1 * 1 $\leftarrow$ (4) 20 1 1 * 1 

(5) ....3 5 7 3 3 $\leftarrow$ (6) 3 6 2 4 5 3 3 3 

(9) 3 6 2 3 4 4 1 1 1 $\leftarrow$ (10) 8 1 1 2 4 3 3 3 

(11) * 2 4 3 3 3 $\leftarrow$ (13) 6 2 3 4 4 1 1 1 

(11) 8 4 1 1 * 1 $\leftarrow$ (12) 12 1 1 * 1 

(15) 4 4 1 1 * 1 $\leftarrow$ (20) 4 1 1 * 1 

\end{tcolorbox}
\begin{tcolorbox}[title={$(34,11)$}]
(3) .....3 5 7 3 3 $\leftarrow$ (4) 3 6 ..4 5 3 3 3 

(3) 14 * * 1 $\leftarrow$ (4) 16 4 1 1 * 1 

(6) .....4 5 3 3 3 

(9) 2 * 2 4 3 3 3 $\leftarrow$ (10) 3 6 2 3 4 4 1 1 1 

(10) 1 * 2 4 3 3 3 $\leftarrow$ (12) * 2 4 3 3 3 

(11) 6 * * 1 $\leftarrow$ (12) 8 4 1 1 * 1 

(12) 3 4 4 1 1 * 1 $\leftarrow$ (18) * * 1 

(15) 2 * * 1 $\leftarrow$ (16) 4 4 1 1 * 1 

\end{tcolorbox}
\begin{tcolorbox}[title={$(34,12)$}]
(3) 13 1 * * 1 $\leftarrow$ (4) 14 * * 1 

(4) ......4 5 3 3 3 

(9) 1 1 * 2 4 3 3 3 $\leftarrow$ (10) 2 * 2 4 3 3 3 

(10) 2 3 4 4 1 1 * 1 $\leftarrow$ (17) 1 * * 1 

(11) 5 1 * * 1 $\leftarrow$ (12) 6 * * 1 

(15) 1 1 * * 1 $\leftarrow$ (16) 2 * * 1 

\end{tcolorbox}
\begin{tcolorbox}[title={$(34,13)$}]
(1) 8 1 1 * 2 4 3 3 3 

(3) 12 1 1 * * 1 $\leftarrow$ (4) 13 1 * * 1 

(9) 1 2 3 4 4 1 1 * 1 $\leftarrow$ (16) 1 1 * * 1 

(11) 4 1 1 * * 1 $\leftarrow$ (12) 5 1 * * 1 

\end{tcolorbox}
\begin{tcolorbox}[title={$(34,14)$}]
(1) 3 6 2 3 4 4 1 1 * 1 $\leftarrow$ (2) 8 1 1 * 2 4 3 3 3 

(3) * * 2 4 3 3 3 $\leftarrow$ (5) 6 2 3 4 4 1 1 * 1 

(3) 8 4 1 1 * * 1 $\leftarrow$ (4) 12 1 1 * * 1 

(7) 4 4 1 1 * * 1 $\leftarrow$ (12) 4 1 1 * * 1 

\end{tcolorbox}
\begin{tcolorbox}[title={$(34,15)$}]
(1) 2 * * 2 4 3 3 3 $\leftarrow$ (2) 3 6 2 3 4 4 1 1 * 1 

(2) 1 * * 2 4 3 3 3 $\leftarrow$ (4) * * 2 4 3 3 3 

(3) 6 * * * 1 $\leftarrow$ (4) 8 4 1 1 * * 1 

(4) 3 4 4 1 1 * * 1 $\leftarrow$ (10) * * * 1 

(7) 2 * * * 1 $\leftarrow$ (8) 4 4 1 1 * * 1 

\end{tcolorbox}
\begin{tcolorbox}[title={$(34,16)$}]
(1) 1 1 * * 2 4 3 3 3 $\leftarrow$ (2) 2 * * 2 4 3 3 3 

(2) 2 3 4 4 1 1 * * 1 $\leftarrow$ (9) 1 * * * 1 

(3) 5 1 * * * 1 $\leftarrow$ (4) 6 * * * 1 

(7) 1 1 * * * 1 $\leftarrow$ (8) 2 * * * 1 

\end{tcolorbox}
\begin{tcolorbox}[title={$(34,17)$}]
(1) 1 2 3 4 4 1 1 * * 1 $\leftarrow$ (8) 1 1 * * * 1 

(3) 4 1 1 * * * 1 $\leftarrow$ (4) 5 1 * * * 1 

\end{tcolorbox}
\begin{tcolorbox}[title={$(34,18)$}]
(1) * * * * 1 

\end{tcolorbox}
\begin{tcolorbox}[title={$(35,3)$}]
(1) 27 7 

(3) 29 3 $\leftarrow$ (5) 31 

(19) 13 3 $\leftarrow$ (21) 15 

(27) 5 3 $\leftarrow$ (29) 7 

(29) 3 3 $\leftarrow$ (33) 3 

\end{tcolorbox}
\begin{tcolorbox}[title={$(35,4)$}]
(1) 26 5 3 $\leftarrow$ (2) 27 7 

(2) 27 3 3 $\leftarrow$ (4) 29 3 

(5) 14 13 3 $\leftarrow$ (6) 15 15 

(17) 10 5 3 $\leftarrow$ (18) 11 7 

(18) 11 3 3 $\leftarrow$ (20) 13 3 

(21) 6 5 3 $\leftarrow$ (22) 7 7 

(26) 3 3 3 $\leftarrow$ (28) 5 3 

(27) 2 3 3 $\leftarrow$ (30) 3 3 

\end{tcolorbox}
\begin{tcolorbox}[title={$(35,5)$}]
(1) 25 3 3 3 $\leftarrow$ (2) 26 5 3 

(3) 13 5 7 7 

(5) 13 11 3 3 $\leftarrow$ (6) 14 13 3 

(13) 3 5 7 7 $\leftarrow$ (17) 5 7 7 

(15) 6 6 5 3 $\leftarrow$ (22) 6 5 3 

(17) 9 3 3 3 $\leftarrow$ (18) 10 5 3 

(26) 1 2 3 3 $\leftarrow$ (28) 2 3 3 

\end{tcolorbox}
\begin{tcolorbox}[title={$(35,6)$}]
(2) 11 3 5 7 7 $\leftarrow$ (4) 13 5 7 7 

(4) 9 3 5 7 7 

(5) 12 9 3 3 3 $\leftarrow$ (6) 13 11 3 3 

(11) 2 3 5 7 7 $\leftarrow$ (13) 4 5 7 7 

(14) 3 5 7 3 3 $\leftarrow$ (19) 8 3 3 3 

(15) 4 7 3 3 3 $\leftarrow$ (16) 6 6 5 3 

(17) 4 5 3 3 3 $\leftarrow$ (18) 5 7 3 3 

(21) 5 1 2 3 3 $\leftarrow$ (22) 6 2 3 3 

\end{tcolorbox}
\begin{tcolorbox}[title={$(35,7)$}]
(1) 5 7 3 5 7 7 

(3) 9 3 6 6 5 3 $\leftarrow$ (6) 12 9 3 3 3 

(5) 12 4 5 3 3 3 

(9) 3 3 6 6 5 3 $\leftarrow$ (12) 2 3 5 7 7 

(12) 2 3 5 7 3 3 $\leftarrow$ (18) 4 5 3 3 3 

(15) 2 4 5 3 3 3 $\leftarrow$ (16) 4 7 3 3 3 

(19) 1 2 4 3 3 3 $\leftarrow$ (21) 2 4 3 3 3 

(21) 3 4 4 1 1 1 $\leftarrow$ (22) 5 1 2 3 3 

\end{tcolorbox}
\begin{tcolorbox}[title={$(35,8)$}]
(2) 5 5 3 6 6 5 3 $\leftarrow$ (4) 9 3 6 6 5 3 

(3) 6 3 3 6 6 5 3 

(5) 10 2 4 5 3 3 3 $\leftarrow$ (6) 12 4 5 3 3 3 

(9) 6 2 4 5 3 3 3 $\leftarrow$ (16) 2 4 5 3 3 3 

(18) 1 1 2 4 3 3 3 $\leftarrow$ (20) 1 2 4 3 3 3 

(19) 2 3 4 4 1 1 1 $\leftarrow$ (22) 3 4 4 1 1 1 

\end{tcolorbox}
\begin{tcolorbox}[title={$(35,9)$}]
(1) 5 6 2 3 5 7 3 3 

(3) 3 6 2 3 5 7 3 3 $\leftarrow$ (4) 6 3 3 6 6 5 3 

(5) 7 13 1 * 1 $\leftarrow$ (6) 10 2 4 5 3 3 3 

(7) 6 ..4 5 3 3 3 $\leftarrow$ (13) 13 1 * 1 

(9) 4 ..4 5 3 3 3 $\leftarrow$ (10) 6 2 4 5 3 3 3 

(18) 1 2 3 4 4 1 1 1 $\leftarrow$ (20) 2 3 4 4 1 1 1 

\end{tcolorbox}
\begin{tcolorbox}[title={$(35,10)$}]
(1) 3 5 6 2 4 5 3 3 3 $\leftarrow$ (2) 5 6 2 3 5 7 3 3 

(3) ...3 3 6 6 5 3 $\leftarrow$ (4) 3 6 2 3 5 7 3 3 

(5) 5 8 1 1 2 4 3 3 3 $\leftarrow$ (6) 7 13 1 * 1 

(6) ....3 5 7 3 3 $\leftarrow$ (11) 8 1 1 2 4 3 3 3 

(7) 4 ...4 5 3 3 3 $\leftarrow$ (8) 6 ..4 5 3 3 3 

(9) ....4 5 3 3 3 $\leftarrow$ (10) 4 ..4 5 3 3 3 

(13) 5 1 2 3 4 4 1 1 1 $\leftarrow$ (14) 6 2 3 4 4 1 1 1 

\end{tcolorbox}
\begin{tcolorbox}[title={$(35,11)$}]
(1) ....3 3 6 6 5 3 $\leftarrow$ (2) 3 5 6 2 4 5 3 3 3 

(4) .....3 5 7 3 3 $\leftarrow$ (10) ....4 5 3 3 3 

(5) 4 ....4 5 3 3 3 $\leftarrow$ (6) 5 8 1 1 2 4 3 3 3 

(7) .....4 5 3 3 3 $\leftarrow$ (8) 4 ...4 5 3 3 3 

(11) 1 * 2 4 3 3 3 $\leftarrow$ (13) * 2 4 3 3 3 

(13) 3 4 4 1 1 * 1 $\leftarrow$ (14) 5 1 2 3 4 4 1 1 1 

\end{tcolorbox}
\begin{tcolorbox}[title={$(35,12)$}]
(1) 6 .....4 5 3 3 3 $\leftarrow$ (8) .....4 5 3 3 3 

(5) ......4 5 3 3 3 $\leftarrow$ (6) 4 ....4 5 3 3 3 

(10) 1 1 * 2 4 3 3 3 $\leftarrow$ (12) 1 * 2 4 3 3 3 

(11) 2 3 4 4 1 1 * 1 $\leftarrow$ (14) 3 4 4 1 1 * 1 

\end{tcolorbox}
\begin{tcolorbox}[title={$(35,13)$}]
(1) 4 ......4 5 3 3 3 $\leftarrow$ (2) 6 .....4 5 3 3 3 

(10) 1 2 3 4 4 1 1 * 1 $\leftarrow$ (12) 2 3 4 4 1 1 * 1 

\end{tcolorbox}
\begin{tcolorbox}[title={$(35,14)$}]
(1) ........4 5 3 3 3 $\leftarrow$ (2) 4 ......4 5 3 3 3 

(5) 5 1 2 3 4 4 1 1 * 1 $\leftarrow$ (6) 6 2 3 4 4 1 1 * 1 

\end{tcolorbox}
\begin{tcolorbox}[title={$(35,15)$}]
(3) 1 * * 2 4 3 3 3 $\leftarrow$ (5) * * 2 4 3 3 3 

(5) 3 4 4 1 1 * * 1 $\leftarrow$ (6) 5 1 2 3 4 4 1 1 * 1 

\end{tcolorbox}
\begin{tcolorbox}[title={$(35,16)$}]
(2) 1 1 * * 2 4 3 3 3 $\leftarrow$ (4) 1 * * 2 4 3 3 3 

(3) 2 3 4 4 1 1 * * 1 $\leftarrow$ (6) 3 4 4 1 1 * * 1 

\end{tcolorbox}
\begin{tcolorbox}[title={$(35,17)$}]
(2) 1 2 3 4 4 1 1 * * 1 $\leftarrow$ (4) 2 3 4 4 1 1 * * 1 

(4) 4 1 1 * * * 1 

\end{tcolorbox}
\begin{tcolorbox}[title={$(35,18)$}]
(2) * * * * 1 

\end{tcolorbox}
\begin{tcolorbox}[title={$(35,19)$}]
(1) 1 * * * * 1 

\end{tcolorbox}
\begin{tcolorbox}[title={$(36,2)$}]
(35) 1 $\leftarrow$ (37) 

\end{tcolorbox}
\begin{tcolorbox}[title={$(36,3)$}]
(5) 30 1 $\leftarrow$ (6) 31 

(21) 14 1 $\leftarrow$ (22) 15 

(29) 6 1 $\leftarrow$ (30) 7 

(33) 2 1 $\leftarrow$ (34) 3 

(34) 1 1 $\leftarrow$ (36) 1 

\end{tcolorbox}
\begin{tcolorbox}[title={$(36,4)$}]
(3) 27 3 3 $\leftarrow$ (5) 29 3 

(5) 13 11 7 

(5) 29 1 1 $\leftarrow$ (6) 30 1 

(15) 7 7 7 $\leftarrow$ (19) 11 7 

(19) 11 3 3 $\leftarrow$ (21) 13 3 

(21) 13 1 1 $\leftarrow$ (22) 14 1 

(27) 3 3 3 $\leftarrow$ (29) 5 3 

(29) 5 1 1 $\leftarrow$ (30) 6 1 

(33) 1 1 1 $\leftarrow$ (34) 2 1 

\end{tcolorbox}
\begin{tcolorbox}[title={$(36,5)$}]
(2) 25 3 3 3 $\leftarrow$ (4) 27 3 3 

(3) 14 5 7 7 $\leftarrow$ (16) 7 7 7 

(5) 28 1 1 1 $\leftarrow$ (6) 29 1 1 

(14) 3 5 7 7 $\leftarrow$ (18) 5 7 7 

(18) 9 3 3 3 $\leftarrow$ (20) 11 3 3 

(21) 12 1 1 1 $\leftarrow$ (22) 13 1 1 

(27) 1 2 3 3 $\leftarrow$ (29) 2 3 3 

(29) 4 1 1 1 $\leftarrow$ (30) 5 1 1 

\end{tcolorbox}
\begin{tcolorbox}[title={$(36,6)$}]
(1) 5 9 7 7 7 

(3) 11 3 5 7 7 $\leftarrow$ (4) 14 5 7 7 

(5) 9 3 5 7 7 

(5) 24 4 1 1 1 $\leftarrow$ (6) 28 1 1 1 

(13) 3 6 6 5 3 $\leftarrow$ (19) 5 7 3 3 

(15) 3 5 7 3 3 $\leftarrow$ (17) 6 6 5 3 

(19) 3 6 2 3 3 $\leftarrow$ (20) 8 3 3 3 

(21) 8 4 1 1 1 $\leftarrow$ (22) 12 1 1 1 

(25) 4 4 1 1 1 $\leftarrow$ (28) 1 2 3 3 

(27) * 1 $\leftarrow$ (30) 4 1 1 1 

\end{tcolorbox}
\begin{tcolorbox}[title={$(36,7)$}]
(2) 5 7 3 5 7 7 

(5) 22 * 1 $\leftarrow$ (6) 24 4 1 1 1 

(10) 3 3 6 6 5 3 $\leftarrow$ (16) 3 5 7 3 3 

(13) 2 3 5 7 3 3 $\leftarrow$ (14) 3 6 6 5 3 

(19) ..4 3 3 3 $\leftarrow$ (20) 3 6 2 3 3 

(21) 6 * 1 $\leftarrow$ (22) 8 4 1 1 1 

(25) 2 * 1 $\leftarrow$ (26) 4 4 1 1 1 

(26) 1 * 1 $\leftarrow$ (28) * 1 

\end{tcolorbox}
\begin{tcolorbox}[title={$(36,8)$}]
(1) 6 5 2 3 5 7 7 

(3) 5 5 3 6 6 5 3 $\leftarrow$ (5) 9 3 6 6 5 3 

(5) 21 1 * 1 $\leftarrow$ (6) 22 * 1 

(7) 6 2 3 5 7 3 3 $\leftarrow$ (14) 2 3 5 7 3 3 

(19) 1 1 2 4 3 3 3 $\leftarrow$ (20) ..4 3 3 3 

(21) 5 1 * 1 $\leftarrow$ (22) 6 * 1 

(25) 1 1 * 1 $\leftarrow$ (26) 2 * 1 

\end{tcolorbox}
\begin{tcolorbox}[title={$(36,9)$}]
(1) 3 6 3 3 6 6 5 3 $\leftarrow$ (2) 6 5 2 3 5 7 7 

(5) 5 6 2 4 5 3 3 3 $\leftarrow$ (11) 6 2 4 5 3 3 3 

(5) 20 1 1 * 1 $\leftarrow$ (6) 21 1 * 1 

(7) 3 6 2 4 5 3 3 3 $\leftarrow$ (8) 6 2 3 5 7 3 3 

(13) 12 1 1 * 1 $\leftarrow$ (14) 13 1 * 1 

(19) 1 2 3 4 4 1 1 1 $\leftarrow$ (21) 2 3 4 4 1 1 1 

(21) 4 1 1 * 1 $\leftarrow$ (22) 5 1 * 1 

\end{tcolorbox}
\begin{tcolorbox}[title={$(36,10)$}]
(1) 2 4 3 3 3 6 6 5 3 $\leftarrow$ (2) 3 6 3 3 6 6 5 3 

(4) ...3 3 6 6 5 3 $\leftarrow$ (9) 6 ..4 5 3 3 3 

(5) 3 6 ..4 5 3 3 3 $\leftarrow$ (6) 5 6 2 4 5 3 3 3 

(5) 16 4 1 1 * 1 $\leftarrow$ (6) 20 1 1 * 1 

(7) ....3 5 7 3 3 $\leftarrow$ (8) 3 6 2 4 5 3 3 3 

(11) 3 6 2 3 4 4 1 1 1 $\leftarrow$ (12) 8 1 1 2 4 3 3 3 

(13) 8 4 1 1 * 1 $\leftarrow$ (14) 12 1 1 * 1 

(17) 4 4 1 1 * 1 $\leftarrow$ (20) 1 2 3 4 4 1 1 1 

(19) * * 1 $\leftarrow$ (22) 4 1 1 * 1 

\end{tcolorbox}
\begin{tcolorbox}[title={$(36,11)$}]
(2) ....3 3 6 6 5 3 $\leftarrow$ (8) ....3 5 7 3 3 

(5) .....3 5 7 3 3 $\leftarrow$ (6) 3 6 ..4 5 3 3 3 

(5) 14 * * 1 $\leftarrow$ (6) 16 4 1 1 * 1 

(11) 2 * 2 4 3 3 3 $\leftarrow$ (12) 3 6 2 3 4 4 1 1 1 

(13) 6 * * 1 $\leftarrow$ (14) 8 4 1 1 * 1 

(17) 2 * * 1 $\leftarrow$ (18) 4 4 1 1 * 1 

(18) 1 * * 1 $\leftarrow$ (20) * * 1 

\end{tcolorbox}
\begin{tcolorbox}[title={$(36,12)$}]
(5) 13 1 * * 1 $\leftarrow$ (6) 14 * * 1 

(6) ......4 5 3 3 3 

(11) 1 1 * 2 4 3 3 3 $\leftarrow$ (12) 2 * 2 4 3 3 3 

(13) 5 1 * * 1 $\leftarrow$ (14) 6 * * 1 

(17) 1 1 * * 1 $\leftarrow$ (18) 2 * * 1 

\end{tcolorbox}
\begin{tcolorbox}[title={$(36,13)$}]
(3) 8 1 1 * 2 4 3 3 3 

(5) 12 1 1 * * 1 $\leftarrow$ (6) 13 1 * * 1 

(11) 1 2 3 4 4 1 1 * 1 $\leftarrow$ (13) 2 3 4 4 1 1 * 1 

(13) 4 1 1 * * 1 $\leftarrow$ (14) 5 1 * * 1 

\end{tcolorbox}
\begin{tcolorbox}[title={$(36,14)$}]
(2) ........4 5 3 3 3 

(3) 3 6 2 3 4 4 1 1 * 1 $\leftarrow$ (4) 8 1 1 * 2 4 3 3 3 

(5) 8 4 1 1 * * 1 $\leftarrow$ (6) 12 1 1 * * 1 

(9) 4 4 1 1 * * 1 $\leftarrow$ (12) 1 2 3 4 4 1 1 * 1 

(11) * * * 1 $\leftarrow$ (14) 4 1 1 * * 1 

\end{tcolorbox}
\begin{tcolorbox}[title={$(36,15)$}]
(3) 2 * * 2 4 3 3 3 $\leftarrow$ (4) 3 6 2 3 4 4 1 1 * 1 

(5) 6 * * * 1 $\leftarrow$ (6) 8 4 1 1 * * 1 

(9) 2 * * * 1 $\leftarrow$ (10) 4 4 1 1 * * 1 

(10) 1 * * * 1 $\leftarrow$ (12) * * * 1 

\end{tcolorbox}
\begin{tcolorbox}[title={$(36,16)$}]
(3) 1 1 * * 2 4 3 3 3 $\leftarrow$ (4) 2 * * 2 4 3 3 3 

(5) 5 1 * * * 1 $\leftarrow$ (6) 6 * * * 1 

(9) 1 1 * * * 1 $\leftarrow$ (10) 2 * * * 1 

\end{tcolorbox}
\begin{tcolorbox}[title={$(36,17)$}]
(3) 1 2 3 4 4 1 1 * * 1 $\leftarrow$ (5) 2 3 4 4 1 1 * * 1 

(5) 4 1 1 * * * 1 $\leftarrow$ (6) 5 1 * * * 1 

\end{tcolorbox}
\begin{tcolorbox}[title={$(36,18)$}]
(1) 4 4 1 1 * * * 1 $\leftarrow$ (4) 1 2 3 4 4 1 1 * * 1 

(3) * * * * 1 $\leftarrow$ (6) 4 1 1 * * * 1 

\end{tcolorbox}
\begin{tcolorbox}[title={$(36,19)$}]
(1) 2 * * * * 1 $\leftarrow$ (2) 4 4 1 1 * * * 1 

(2) 1 * * * * 1 $\leftarrow$ (4) * * * * 1 

\end{tcolorbox}
\begin{tcolorbox}[title={$(36,20)$}]
(1) 1 1 * * * * 1 $\leftarrow$ (2) 2 * * * * 1 

\end{tcolorbox}
\begin{tcolorbox}[title={$(37,3)$}]
(3) 27 7 

(7) 15 15 

(23) 7 7 

(31) 3 3 $\leftarrow$ (35) 3 

(35) 1 1 $\leftarrow$ (37) 1 

\end{tcolorbox}
\begin{tcolorbox}[title={$(37,4)$}]
(3) 26 5 3 $\leftarrow$ (4) 27 7 

(6) 13 11 7 

(7) 14 13 3 $\leftarrow$ (8) 15 15 

(19) 10 5 3 $\leftarrow$ (20) 11 7 

(23) 6 5 3 $\leftarrow$ (24) 7 7 

(28) 3 3 3 $\leftarrow$ (32) 3 3 

(34) 1 1 1 $\leftarrow$ (36) 1 1 

\end{tcolorbox}
\begin{tcolorbox}[title={$(37,5)$}]
(3) 25 3 3 3 $\leftarrow$ (4) 26 5 3 

(5) 13 5 7 7 

(7) 13 11 3 3 $\leftarrow$ (8) 14 13 3 

(14) 4 5 7 7 

(15) 3 5 7 7 $\leftarrow$ (19) 5 7 7 

(19) 9 3 3 3 $\leftarrow$ (20) 10 5 3 

(23) 6 2 3 3 $\leftarrow$ (30) 2 3 3 

\end{tcolorbox}
\begin{tcolorbox}[title={$(37,6)$}]
(2) 5 9 7 7 7 

(4) 11 3 5 7 7 

(6) 9 3 5 7 7 $\leftarrow$ (16) 3 5 7 7 

(7) 12 9 3 3 3 $\leftarrow$ (8) 13 11 3 3 

(13) 2 3 5 7 7 

(17) 4 7 3 3 3 $\leftarrow$ (18) 6 6 5 3 

(19) 4 5 3 3 3 $\leftarrow$ (20) 5 7 3 3 

(22) 2 4 3 3 3 $\leftarrow$ (29) 1 2 3 3 

(23) 5 1 2 3 3 $\leftarrow$ (24) 6 2 3 3 

\end{tcolorbox}
\begin{tcolorbox}[title={$(37,7)$}]
(1) 5 9 3 5 7 7 

(3) 5 7 3 5 7 7 

(7) 12 4 5 3 3 3 

(11) 3 3 6 6 5 3 $\leftarrow$ (17) 3 5 7 3 3 

(17) 2 4 5 3 3 3 $\leftarrow$ (18) 4 7 3 3 3 

(21) 1 2 4 3 3 3 $\leftarrow$ (27) 4 4 1 1 1 

(23) 3 4 4 1 1 1 $\leftarrow$ (24) 5 1 2 3 3 

(27) 1 * 1 $\leftarrow$ (29) * 1 

\end{tcolorbox}
\begin{tcolorbox}[title={$(37,8)$}]
(4) 5 5 3 6 6 5 3 

(5) 6 3 3 6 6 5 3 

(7) 10 2 4 5 3 3 3 $\leftarrow$ (8) 12 4 5 3 3 3 

(20) 1 1 2 4 3 3 3 $\leftarrow$ (24) 3 4 4 1 1 1 

(26) 1 1 * 1 $\leftarrow$ (28) 1 * 1 

\end{tcolorbox}
\begin{tcolorbox}[title={$(37,9)$}]
(3) 5 6 2 3 5 7 3 3 

(5) 3 6 2 3 5 7 3 3 $\leftarrow$ (6) 6 3 3 6 6 5 3 

(7) 7 13 1 * 1 $\leftarrow$ (8) 10 2 4 5 3 3 3 

(11) 4 ..4 5 3 3 3 $\leftarrow$ (12) 6 2 4 5 3 3 3 

(15) 6 2 3 4 4 1 1 1 $\leftarrow$ (22) 2 3 4 4 1 1 1 

\end{tcolorbox}
\begin{tcolorbox}[title={$(37,10)$}]
(2) 2 4 3 3 3 6 6 5 3 

(3) 3 5 6 2 4 5 3 3 3 $\leftarrow$ (4) 5 6 2 3 5 7 3 3 

(5) ...3 3 6 6 5 3 $\leftarrow$ (6) 3 6 2 3 5 7 3 3 

(7) 5 8 1 1 2 4 3 3 3 $\leftarrow$ (8) 7 13 1 * 1 

(9) 4 ...4 5 3 3 3 $\leftarrow$ (10) 6 ..4 5 3 3 3 

(11) ....4 5 3 3 3 $\leftarrow$ (12) 4 ..4 5 3 3 3 

(14) * 2 4 3 3 3 $\leftarrow$ (21) 1 2 3 4 4 1 1 1 

(15) 5 1 2 3 4 4 1 1 1 $\leftarrow$ (16) 6 2 3 4 4 1 1 1 

\end{tcolorbox}
\begin{tcolorbox}[title={$(37,11)$}]
(3) ....3 3 6 6 5 3 $\leftarrow$ (4) 3 5 6 2 4 5 3 3 3 

(6) .....3 5 7 3 3 

(7) 4 ....4 5 3 3 3 $\leftarrow$ (8) 5 8 1 1 2 4 3 3 3 

(9) .....4 5 3 3 3 $\leftarrow$ (10) 4 ...4 5 3 3 3 

(13) 1 * 2 4 3 3 3 $\leftarrow$ (19) 4 4 1 1 * 1 

(15) 3 4 4 1 1 * 1 $\leftarrow$ (16) 5 1 2 3 4 4 1 1 1 

(19) 1 * * 1 $\leftarrow$ (21) * * 1 

\end{tcolorbox}
\begin{tcolorbox}[title={$(37,12)$}]
(3) 6 .....4 5 3 3 3 

(7) ......4 5 3 3 3 $\leftarrow$ (8) 4 ....4 5 3 3 3 

(12) 1 1 * 2 4 3 3 3 $\leftarrow$ (16) 3 4 4 1 1 * 1 

(18) 1 1 * * 1 $\leftarrow$ (20) 1 * * 1 

\end{tcolorbox}
\begin{tcolorbox}[title={$(37,13)$}]
(1) 6 ......4 5 3 3 3 

(3) 4 ......4 5 3 3 3 $\leftarrow$ (4) 6 .....4 5 3 3 3 

(7) 6 2 3 4 4 1 1 * 1 $\leftarrow$ (14) 2 3 4 4 1 1 * 1 

\end{tcolorbox}
\begin{tcolorbox}[title={$(37,14)$}]
(1) 4 .......4 5 3 3 3 $\leftarrow$ (2) 6 ......4 5 3 3 3 

(3) ........4 5 3 3 3 $\leftarrow$ (4) 4 ......4 5 3 3 3 

(6) * * 2 4 3 3 3 $\leftarrow$ (13) 1 2 3 4 4 1 1 * 1 

(7) 5 1 2 3 4 4 1 1 * 1 $\leftarrow$ (8) 6 2 3 4 4 1 1 * 1 

\end{tcolorbox}
\begin{tcolorbox}[title={$(37,15)$}]
(1) .........4 5 3 3 3 $\leftarrow$ (2) 4 .......4 5 3 3 3 

(5) 1 * * 2 4 3 3 3 $\leftarrow$ (11) 4 4 1 1 * * 1 

(7) 3 4 4 1 1 * * 1 $\leftarrow$ (8) 5 1 2 3 4 4 1 1 * 1 

(11) 1 * * * 1 $\leftarrow$ (13) * * * 1 

\end{tcolorbox}
\begin{tcolorbox}[title={$(37,16)$}]
(4) 1 1 * * 2 4 3 3 3 $\leftarrow$ (8) 3 4 4 1 1 * * 1 

(10) 1 1 * * * 1 $\leftarrow$ (12) 1 * * * 1 

\end{tcolorbox}
\begin{tcolorbox}[title={$(37,19)$}]
(3) 1 * * * * 1 $\leftarrow$ (5) * * * * 1 

\end{tcolorbox}
\begin{tcolorbox}[title={$(37,20)$}]
(2) 1 1 * * * * 1 $\leftarrow$ (4) 1 * * * * 1 

\end{tcolorbox}
\begin{tcolorbox}[title={$(38,2)$}]
(7) 31 

(23) 15 

(31) 7 $\leftarrow$ (39) 

\end{tcolorbox}
\begin{tcolorbox}[title={$(38,3)$}]
(6) 29 3 

(7) 30 1 $\leftarrow$ (8) 31 

(22) 13 3 

(23) 14 1 $\leftarrow$ (24) 15 

(30) 5 3 $\leftarrow$ (38) 1 

(31) 6 1 $\leftarrow$ (32) 7 

(35) 2 1 $\leftarrow$ (36) 3 

\end{tcolorbox}
\begin{tcolorbox}[title={$(38,4)$}]
(1) 23 7 7 

(5) 27 3 3 

(7) 13 11 7 $\leftarrow$ (9) 15 15 

(7) 29 1 1 $\leftarrow$ (8) 30 1 

(17) 7 7 7 

(21) 11 3 3 

(23) 13 1 1 $\leftarrow$ (24) 14 1 

(24) 6 5 3 $\leftarrow$ (37) 1 1 

(29) 3 3 3 $\leftarrow$ (33) 3 3 

(31) 5 1 1 $\leftarrow$ (32) 6 1 

(35) 1 1 1 $\leftarrow$ (36) 2 1 

\end{tcolorbox}
\begin{tcolorbox}[title={$(38,5)$}]
(4) 25 3 3 3 

(5) 14 5 7 7 $\leftarrow$ (8) 13 11 7 

(6) 13 5 7 7 

(7) 28 1 1 1 $\leftarrow$ (8) 29 1 1 

(15) 4 5 7 7 $\leftarrow$ (18) 7 7 7 

(20) 9 3 3 3 

(21) 8 3 3 3 $\leftarrow$ (36) 1 1 1 

(23) 12 1 1 1 $\leftarrow$ (24) 13 1 1 

(31) 4 1 1 1 $\leftarrow$ (32) 5 1 1 

\end{tcolorbox}
\begin{tcolorbox}[title={$(38,6)$}]
(1) 14 4 5 7 7 

(3) 5 9 7 7 7 

(5) 11 3 5 7 7 $\leftarrow$ (6) 14 5 7 7 

(7) 9 3 5 7 7 $\leftarrow$ (17) 3 5 7 7 

(7) 24 4 1 1 1 $\leftarrow$ (8) 28 1 1 1 

(8) 12 9 3 3 3 

(14) 2 3 5 7 7 $\leftarrow$ (16) 4 5 7 7 

(15) 3 6 6 5 3 $\leftarrow$ (19) 6 6 5 3 

(20) 4 5 3 3 3 $\leftarrow$ (32) 4 1 1 1 

(21) 3 6 2 3 3 $\leftarrow$ (22) 8 3 3 3 

(23) 2 4 3 3 3 $\leftarrow$ (25) 6 2 3 3 

(23) 8 4 1 1 1 $\leftarrow$ (24) 12 1 1 1 

\end{tcolorbox}
\begin{tcolorbox}[title={$(38,7)$}]
(1) 13 2 3 5 7 7 $\leftarrow$ (2) 14 4 5 7 7 

(2) 5 9 3 5 7 7 

(4) 5 7 3 5 7 7 $\leftarrow$ (8) 9 3 5 7 7 

(6) 9 3 6 6 5 3 

(7) 22 * 1 $\leftarrow$ (8) 24 4 1 1 1 

(12) 3 3 6 6 5 3 $\leftarrow$ (18) 3 5 7 3 3 

(15) 2 3 5 7 3 3 $\leftarrow$ (16) 3 6 6 5 3 

(18) 2 4 5 3 3 3 $\leftarrow$ (30) * 1 

(21) ..4 3 3 3 $\leftarrow$ (22) 3 6 2 3 3 

(22) 1 2 4 3 3 3 $\leftarrow$ (24) 2 4 3 3 3 

(23) 6 * 1 $\leftarrow$ (24) 8 4 1 1 1 

(27) 2 * 1 $\leftarrow$ (28) 4 4 1 1 1 

\end{tcolorbox}
\begin{tcolorbox}[title={$(38,8)$}]
(1) 3 5 7 3 5 7 7 

(1) 7 12 4 5 3 3 3 $\leftarrow$ (2) 13 2 3 5 7 7 

(3) 6 5 2 3 5 7 7 

(5) 5 5 3 6 6 5 3 $\leftarrow$ (9) 12 4 5 3 3 3 

(7) 21 1 * 1 $\leftarrow$ (8) 22 * 1 

(9) 6 2 3 5 7 3 3 $\leftarrow$ (16) 2 3 5 7 3 3 

(15) 13 1 * 1 $\leftarrow$ (29) 1 * 1 

(21) 1 1 2 4 3 3 3 $\leftarrow$ (22) ..4 3 3 3 

(23) 5 1 * 1 $\leftarrow$ (24) 6 * 1 

(27) 1 1 * 1 $\leftarrow$ (28) 2 * 1 

\end{tcolorbox}
\begin{tcolorbox}[title={$(38,9)$}]
(1) 4 7 3 3 6 6 5 3 

(1) 5 6 3 3 6 6 5 3 $\leftarrow$ (2) 7 12 4 5 3 3 3 

(3) 3 6 3 3 6 6 5 3 $\leftarrow$ (4) 6 5 2 3 5 7 7 

(7) 5 6 2 4 5 3 3 3 $\leftarrow$ (13) 6 2 4 5 3 3 3 

(7) 20 1 1 * 1 $\leftarrow$ (8) 21 1 * 1 

(9) 3 6 2 4 5 3 3 3 $\leftarrow$ (10) 6 2 3 5 7 3 3 

(13) 8 1 1 2 4 3 3 3 $\leftarrow$ (28) 1 1 * 1 

(15) 12 1 1 * 1 $\leftarrow$ (16) 13 1 * 1 

(23) 4 1 1 * 1 $\leftarrow$ (24) 5 1 * 1 

\end{tcolorbox}
\begin{tcolorbox}[title={$(38,10)$}]
(1) 3 5 6 2 3 5 7 3 3 $\leftarrow$ (2) 5 6 3 3 6 6 5 3 

(3) 2 4 3 3 3 6 6 5 3 $\leftarrow$ (4) 3 6 3 3 6 6 5 3 

(6) ...3 3 6 6 5 3 $\leftarrow$ (11) 6 ..4 5 3 3 3 

(7) 3 6 ..4 5 3 3 3 $\leftarrow$ (8) 5 6 2 4 5 3 3 3 

(7) 16 4 1 1 * 1 $\leftarrow$ (8) 20 1 1 * 1 

(9) ....3 5 7 3 3 $\leftarrow$ (10) 3 6 2 4 5 3 3 3 

(12) ....4 5 3 3 3 $\leftarrow$ (24) 4 1 1 * 1 

(13) 3 6 2 3 4 4 1 1 1 $\leftarrow$ (14) 8 1 1 2 4 3 3 3 

(15) * 2 4 3 3 3 $\leftarrow$ (17) 6 2 3 4 4 1 1 1 

(15) 8 4 1 1 * 1 $\leftarrow$ (16) 12 1 1 * 1 

\end{tcolorbox}
\begin{tcolorbox}[title={$(38,11)$}]
(1) ..4 3 3 3 6 6 5 3 $\leftarrow$ (2) 3 5 6 2 3 5 7 3 3 

(4) ....3 3 6 6 5 3 $\leftarrow$ (10) ....3 5 7 3 3 

(7) .....3 5 7 3 3 $\leftarrow$ (8) 3 6 ..4 5 3 3 3 

(7) 14 * * 1 $\leftarrow$ (8) 16 4 1 1 * 1 

(10) .....4 5 3 3 3 $\leftarrow$ (22) * * 1 

(13) 2 * 2 4 3 3 3 $\leftarrow$ (14) 3 6 2 3 4 4 1 1 1 

(14) 1 * 2 4 3 3 3 $\leftarrow$ (16) * 2 4 3 3 3 

(15) 6 * * 1 $\leftarrow$ (16) 8 4 1 1 * 1 

(19) 2 * * 1 $\leftarrow$ (20) 4 4 1 1 * 1 

\end{tcolorbox}
\begin{tcolorbox}[title={$(38,12)$}]
(1) 6 .....3 5 7 3 3 $\leftarrow$ (8) .....3 5 7 3 3 

(7) 13 1 * * 1 $\leftarrow$ (8) 14 * * 1 

(8) ......4 5 3 3 3 $\leftarrow$ (21) 1 * * 1 

(13) 1 1 * 2 4 3 3 3 $\leftarrow$ (14) 2 * 2 4 3 3 3 

(15) 5 1 * * 1 $\leftarrow$ (16) 6 * * 1 

(19) 1 1 * * 1 $\leftarrow$ (20) 2 * * 1 

\end{tcolorbox}
\begin{tcolorbox}[title={$(38,13)$}]
(1) 3 6 .....4 5 3 3 3 $\leftarrow$ (2) 6 .....3 5 7 3 3 

(5) 8 1 1 * 2 4 3 3 3 $\leftarrow$ (20) 1 1 * * 1 

(7) 12 1 1 * * 1 $\leftarrow$ (8) 13 1 * * 1 

(15) 4 1 1 * * 1 $\leftarrow$ (16) 5 1 * * 1 

\end{tcolorbox}
\begin{tcolorbox}[title={$(38,14)$}]
(1) ........3 5 7 3 3 $\leftarrow$ (2) 3 6 .....4 5 3 3 3 

(4) ........4 5 3 3 3 $\leftarrow$ (16) 4 1 1 * * 1 

(5) 3 6 2 3 4 4 1 1 * 1 $\leftarrow$ (6) 8 1 1 * 2 4 3 3 3 

(7) * * 2 4 3 3 3 $\leftarrow$ (9) 6 2 3 4 4 1 1 * 1 

(7) 8 4 1 1 * * 1 $\leftarrow$ (8) 12 1 1 * * 1 

\end{tcolorbox}
\begin{tcolorbox}[title={$(38,15)$}]
(2) .........4 5 3 3 3 $\leftarrow$ (14) * * * 1 

(5) 2 * * 2 4 3 3 3 $\leftarrow$ (6) 3 6 2 3 4 4 1 1 * 1 

(6) 1 * * 2 4 3 3 3 $\leftarrow$ (8) * * 2 4 3 3 3 

(7) 6 * * * 1 $\leftarrow$ (8) 8 4 1 1 * * 1 

(11) 2 * * * 1 $\leftarrow$ (12) 4 4 1 1 * * 1 

\end{tcolorbox}
\begin{tcolorbox}[title={$(38,16)$}]
(5) 1 1 * * 2 4 3 3 3 $\leftarrow$ (6) 2 * * 2 4 3 3 3 

(6) 2 3 4 4 1 1 * * 1 

(7) 5 1 * * * 1 $\leftarrow$ (8) 6 * * * 1 

(11) 1 1 * * * 1 $\leftarrow$ (12) 2 * * * 1 

\end{tcolorbox}
\begin{tcolorbox}[title={$(38,17)$}]
(5) 1 2 3 4 4 1 1 * * 1 

(7) 4 1 1 * * * 1 $\leftarrow$ (8) 5 1 * * * 1 

\end{tcolorbox}
\begin{tcolorbox}[title={$(38,18)$}]
(3) 4 4 1 1 * * * 1 

\end{tcolorbox}
\begin{tcolorbox}[title={$(38,19)$}]
(3) 2 * * * * 1 $\leftarrow$ (4) 4 4 1 1 * * * 1 

\end{tcolorbox}
\begin{tcolorbox}[title={$(38,20)$}]
(3) 1 1 * * * * 1 $\leftarrow$ (4) 2 * * * * 1 

\end{tcolorbox}
\begin{tcolorbox}[title={$(39,3)$}]
(1) 23 15 

(5) 27 7 

(7) 29 3 $\leftarrow$ (9) 31 

(21) 11 7 

(23) 13 3 $\leftarrow$ (25) 15 

(25) 7 7 $\leftarrow$ (37) 3 

(31) 5 3 $\leftarrow$ (33) 7 

\end{tcolorbox}
\begin{tcolorbox}[title={$(39,4)$}]
(1) 22 13 3 $\leftarrow$ (2) 23 15 

(2) 23 7 7 

(5) 26 5 3 $\leftarrow$ (6) 27 7 

(6) 27 3 3 $\leftarrow$ (8) 29 3 

(9) 14 13 3 $\leftarrow$ (10) 15 15 

(20) 5 7 7 

(21) 10 5 3 $\leftarrow$ (22) 11 7 

(22) 11 3 3 $\leftarrow$ (24) 13 3 

(25) 6 5 3 $\leftarrow$ (26) 7 7 

(30) 3 3 3 $\leftarrow$ (32) 5 3 

(31) 2 3 3 $\leftarrow$ (34) 3 3 

\end{tcolorbox}
\begin{tcolorbox}[title={$(39,5)$}]
(1) 17 7 7 7 

(1) 21 11 3 3 $\leftarrow$ (2) 22 13 3 

(5) 25 3 3 3 $\leftarrow$ (6) 26 5 3 

(7) 13 5 7 7 

(9) 13 11 3 3 $\leftarrow$ (10) 14 13 3 

(21) 5 7 3 3 $\leftarrow$ (26) 6 5 3 

(21) 9 3 3 3 $\leftarrow$ (22) 10 5 3 

(30) 1 2 3 3 $\leftarrow$ (32) 2 3 3 

\end{tcolorbox}
\begin{tcolorbox}[title={$(39,6)$}]
(1) 20 9 3 3 3 $\leftarrow$ (2) 21 11 3 3 

(4) 5 9 7 7 7 

(6) 11 3 5 7 7 $\leftarrow$ (8) 13 5 7 7 

(9) 12 9 3 3 3 $\leftarrow$ (10) 13 11 3 3 

(15) 2 3 5 7 7 $\leftarrow$ (17) 4 5 7 7 

(19) 4 7 3 3 3 $\leftarrow$ (20) 6 6 5 3 

(21) 4 5 3 3 3 $\leftarrow$ (22) 5 7 3 3 

(25) 5 1 2 3 3 $\leftarrow$ (26) 6 2 3 3 

\end{tcolorbox}
\begin{tcolorbox}[title={$(39,7)$}]
(1) 3 5 9 7 7 7 

(1) 8 12 9 3 3 3 

(3) 5 9 3 5 7 7 

(5) 5 7 3 5 7 7 $\leftarrow$ (9) 9 3 5 7 7 

(7) 9 3 6 6 5 3 $\leftarrow$ (10) 12 9 3 3 3 

(13) 3 3 6 6 5 3 $\leftarrow$ (16) 2 3 5 7 7 

(19) 2 4 5 3 3 3 $\leftarrow$ (20) 4 7 3 3 3 

(23) 1 2 4 3 3 3 $\leftarrow$ (25) 2 4 3 3 3 

(25) 3 4 4 1 1 1 $\leftarrow$ (26) 5 1 2 3 3 

\end{tcolorbox}
\begin{tcolorbox}[title={$(39,8)$}]
(1) 6 9 3 6 6 5 3 $\leftarrow$ (2) 8 12 9 3 3 3 

(2) 3 5 7 3 5 7 7 $\leftarrow$ (4) 5 9 3 5 7 7 

(6) 5 5 3 6 6 5 3 $\leftarrow$ (8) 9 3 6 6 5 3 

(7) 6 3 3 6 6 5 3 $\leftarrow$ (14) 3 3 6 6 5 3 

(9) 10 2 4 5 3 3 3 $\leftarrow$ (10) 12 4 5 3 3 3 

(22) 1 1 2 4 3 3 3 $\leftarrow$ (24) 1 2 4 3 3 3 

(23) 2 3 4 4 1 1 1 $\leftarrow$ (26) 3 4 4 1 1 1 

\end{tcolorbox}
\begin{tcolorbox}[title={$(39,9)$}]
(1) 3 6 5 2 3 5 7 7 $\leftarrow$ (2) 6 9 3 6 6 5 3 

(2) 4 7 3 3 6 6 5 3 

(5) 5 6 2 3 5 7 3 3 $\leftarrow$ (11) 6 2 3 5 7 3 3 

(7) 3 6 2 3 5 7 3 3 $\leftarrow$ (8) 6 3 3 6 6 5 3 

(9) 7 13 1 * 1 $\leftarrow$ (10) 10 2 4 5 3 3 3 

(13) 4 ..4 5 3 3 3 $\leftarrow$ (14) 6 2 4 5 3 3 3 

(22) 1 2 3 4 4 1 1 1 $\leftarrow$ (24) 2 3 4 4 1 1 1 

\end{tcolorbox}
\begin{tcolorbox}[title={$(39,10)$}]
(1) 2 3 5 5 3 6 6 5 3 $\leftarrow$ (2) 3 6 5 2 3 5 7 7 

(4) 2 4 3 3 3 6 6 5 3 $\leftarrow$ (9) 5 6 2 4 5 3 3 3 

(5) 3 5 6 2 4 5 3 3 3 $\leftarrow$ (6) 5 6 2 3 5 7 3 3 

(7) ...3 3 6 6 5 3 $\leftarrow$ (8) 3 6 2 3 5 7 3 3 

(9) 5 8 1 1 2 4 3 3 3 $\leftarrow$ (10) 7 13 1 * 1 

(11) 4 ...4 5 3 3 3 $\leftarrow$ (12) 6 ..4 5 3 3 3 

(13) ....4 5 3 3 3 $\leftarrow$ (14) 4 ..4 5 3 3 3 

(17) 5 1 2 3 4 4 1 1 1 $\leftarrow$ (18) 6 2 3 4 4 1 1 1 

\end{tcolorbox}
\begin{tcolorbox}[title={$(39,11)$}]
(2) ..4 3 3 3 6 6 5 3 $\leftarrow$ (8) ...3 3 6 6 5 3 

(5) ....3 3 6 6 5 3 $\leftarrow$ (6) 3 5 6 2 4 5 3 3 3 

(9) 4 ....4 5 3 3 3 $\leftarrow$ (10) 5 8 1 1 2 4 3 3 3 

(11) .....4 5 3 3 3 $\leftarrow$ (12) 4 ...4 5 3 3 3 

(15) 1 * 2 4 3 3 3 $\leftarrow$ (17) * 2 4 3 3 3 

(17) 3 4 4 1 1 * 1 $\leftarrow$ (18) 5 1 2 3 4 4 1 1 1 

\end{tcolorbox}
\begin{tcolorbox}[title={$(39,12)$}]
(5) 6 .....4 5 3 3 3 

(9) ......4 5 3 3 3 $\leftarrow$ (10) 4 ....4 5 3 3 3 

(14) 1 1 * 2 4 3 3 3 $\leftarrow$ (16) 1 * 2 4 3 3 3 

(15) 2 3 4 4 1 1 * 1 $\leftarrow$ (18) 3 4 4 1 1 * 1 

\end{tcolorbox}
\begin{tcolorbox}[title={$(39,13)$}]
(3) 6 ......4 5 3 3 3 

(5) 4 ......4 5 3 3 3 $\leftarrow$ (6) 6 .....4 5 3 3 3 

(14) 1 2 3 4 4 1 1 * 1 $\leftarrow$ (16) 2 3 4 4 1 1 * 1 

\end{tcolorbox}
\begin{tcolorbox}[title={$(39,14)$}]
(2) ........3 5 7 3 3 

(3) 4 .......4 5 3 3 3 $\leftarrow$ (4) 6 ......4 5 3 3 3 

(5) ........4 5 3 3 3 $\leftarrow$ (6) 4 ......4 5 3 3 3 

(9) 5 1 2 3 4 4 1 1 * 1 $\leftarrow$ (10) 6 2 3 4 4 1 1 * 1 

\end{tcolorbox}
\begin{tcolorbox}[title={$(39,15)$}]
(3) .........4 5 3 3 3 $\leftarrow$ (4) 4 .......4 5 3 3 3 

(7) 1 * * 2 4 3 3 3 $\leftarrow$ (9) * * 2 4 3 3 3 

(9) 3 4 4 1 1 * * 1 $\leftarrow$ (10) 5 1 2 3 4 4 1 1 * 1 

(13) 1 * * * 1 

\end{tcolorbox}
\begin{tcolorbox}[title={$(39,16)$}]
(6) 1 1 * * 2 4 3 3 3 $\leftarrow$ (8) 1 * * 2 4 3 3 3 

(7) 2 3 4 4 1 1 * * 1 $\leftarrow$ (10) 3 4 4 1 1 * * 1 

(12) 1 1 * * * 1 

\end{tcolorbox}
\begin{tcolorbox}[title={$(39,17)$}]
(1) 6 2 3 4 4 1 1 * * 1 

(6) 1 2 3 4 4 1 1 * * 1 $\leftarrow$ (8) 2 3 4 4 1 1 * * 1 

(8) 4 1 1 * * * 1 

\end{tcolorbox}

\begin{tcolorbox}[title={$(39,18)$}]
(1) 5 1 2 3 4 4 1 1 * * 1 $\leftarrow$ (2) 6 2 3 4 4 1 1 * * 1 

(6) * * * * 1 

\end{tcolorbox}

\begin{tcolorbox}[title={$(39,19)$}]
(1) 3 4 4 1 1 * * * 1 $\leftarrow$ (2) 5 1 2 3 4 4 1 1 * * 1 

(5) 1 * * * * 1 

\end{tcolorbox}
\begin{tcolorbox}[title={$(39,20)$}]
(4) 1 1 * * * * 1 

\end{tcolorbox}

\tiny

\begin{tcolorbox}[title={$(40,2)$}]
(39) 1 $\leftarrow$ (41) 

\end{tcolorbox}
\begin{tcolorbox}[title={$(40,3)$}]
(9) 30 1 $\leftarrow$ (10) 31 

(25) 14 1 $\leftarrow$ (26) 15 

(33) 6 1 $\leftarrow$ (34) 7 

(37) 2 1 $\leftarrow$ (38) 3 

(38) 1 1 $\leftarrow$ (40) 1 

\end{tcolorbox}
\begin{tcolorbox}[title={$(40,4)$}]
(1) 21 11 7 

(3) 23 7 7 

(7) 27 3 3 $\leftarrow$ (9) 29 3 

(9) 13 11 7 

(9) 29 1 1 $\leftarrow$ (10) 30 1 

(19) 7 7 7 

(21) 5 7 7 $\leftarrow$ (35) 3 3 

(23) 11 3 3 $\leftarrow$ (25) 13 3 

(25) 13 1 1 $\leftarrow$ (26) 14 1 

(31) 3 3 3 $\leftarrow$ (33) 5 3 

(33) 5 1 1 $\leftarrow$ (34) 6 1 

(37) 1 1 1 $\leftarrow$ (38) 2 1 

\end{tcolorbox}
\begin{tcolorbox}[title={$(40,5)$}]
(1) 20 5 7 7 

(2) 17 7 7 7 

(6) 25 3 3 3 $\leftarrow$ (8) 27 3 3 

(7) 14 5 7 7 

(9) 28 1 1 1 $\leftarrow$ (10) 29 1 1 

(18) 3 5 7 7 

(22) 9 3 3 3 $\leftarrow$ (24) 11 3 3 

(23) 8 3 3 3 $\leftarrow$ (32) 3 3 3 

(25) 12 1 1 1 $\leftarrow$ (26) 13 1 1 

(31) 1 2 3 3 $\leftarrow$ (33) 2 3 3 

(33) 4 1 1 1 $\leftarrow$ (34) 5 1 1 

\end{tcolorbox}
\begin{tcolorbox}[title={$(40,6)$}]
(1) 7 13 5 7 7 

(2) 20 9 3 3 3 

(3) 14 4 5 7 7 

(5) 5 9 7 7 7 $\leftarrow$ (9) 13 5 7 7 

(7) 11 3 5 7 7 $\leftarrow$ (8) 14 5 7 7 

(9) 24 4 1 1 1 $\leftarrow$ (10) 28 1 1 1 

(17) 3 6 6 5 3 

(19) 3 5 7 3 3 $\leftarrow$ (21) 6 6 5 3 

(22) 4 5 3 3 3 $\leftarrow$ (27) 6 2 3 3 

(23) 3 6 2 3 3 $\leftarrow$ (24) 8 3 3 3 

(25) 8 4 1 1 1 $\leftarrow$ (26) 12 1 1 1 

(29) 4 4 1 1 1 $\leftarrow$ (32) 1 2 3 3 

(31) * 1 $\leftarrow$ (34) 4 1 1 1 

\end{tcolorbox}
\begin{tcolorbox}[title={$(40,7)$}]
(2) 3 5 9 7 7 7 $\leftarrow$ (8) 11 3 5 7 7 

(3) 13 2 3 5 7 7 $\leftarrow$ (4) 14 4 5 7 7 

(6) 5 7 3 5 7 7 $\leftarrow$ (10) 9 3 5 7 7 

(9) 22 * 1 $\leftarrow$ (10) 24 4 1 1 1 

(17) 2 3 5 7 3 3 $\leftarrow$ (18) 3 6 6 5 3 

(20) 2 4 5 3 3 3 $\leftarrow$ (26) 2 4 3 3 3 

(23) ..4 3 3 3 $\leftarrow$ (24) 3 6 2 3 3 

(25) 6 * 1 $\leftarrow$ (26) 8 4 1 1 1 

(29) 2 * 1 $\leftarrow$ (30) 4 4 1 1 1 

(30) 1 * 1 $\leftarrow$ (32) * 1 

\end{tcolorbox}
\begin{tcolorbox}[title={$(40,8)$}]
(1) 3 5 9 3 5 7 7 

(3) 3 5 7 3 5 7 7 $\leftarrow$ (5) 5 9 3 5 7 7 

(3) 7 12 4 5 3 3 3 $\leftarrow$ (4) 13 2 3 5 7 7 

(5) 6 5 2 3 5 7 7 

(7) 5 5 3 6 6 5 3 $\leftarrow$ (9) 9 3 6 6 5 3 

(9) 21 1 * 1 $\leftarrow$ (10) 22 * 1 

(17) 13 1 * 1 $\leftarrow$ (25) 1 2 4 3 3 3 

(23) 1 1 2 4 3 3 3 $\leftarrow$ (24) ..4 3 3 3 

(25) 5 1 * 1 $\leftarrow$ (26) 6 * 1 

(29) 1 1 * 1 $\leftarrow$ (30) 2 * 1 

\end{tcolorbox}
\begin{tcolorbox}[title={$(40,9)$}]
(3) 4 7 3 3 6 6 5 3 $\leftarrow$ (8) 5 5 3 6 6 5 3 

(3) 5 6 3 3 6 6 5 3 $\leftarrow$ (4) 7 12 4 5 3 3 3 

(5) 3 6 3 3 6 6 5 3 $\leftarrow$ (6) 6 5 2 3 5 7 7 

(9) 20 1 1 * 1 $\leftarrow$ (10) 21 1 * 1 

(11) 3 6 2 4 5 3 3 3 $\leftarrow$ (12) 6 2 3 5 7 3 3 

(15) 8 1 1 2 4 3 3 3 $\leftarrow$ (24) 1 1 2 4 3 3 3 

(17) 12 1 1 * 1 $\leftarrow$ (18) 13 1 * 1 

(23) 1 2 3 4 4 1 1 1 $\leftarrow$ (25) 2 3 4 4 1 1 1 

(25) 4 1 1 * 1 $\leftarrow$ (26) 5 1 * 1 

\end{tcolorbox}
\begin{tcolorbox}[title={$(40,10)$}]
(1) 2 4 7 3 3 6 6 5 3 

(2) 2 3 5 5 3 6 6 5 3 $\leftarrow$ (4) 4 7 3 3 6 6 5 3 

(3) 3 5 6 2 3 5 7 3 3 $\leftarrow$ (4) 5 6 3 3 6 6 5 3 

(5) 2 4 3 3 3 6 6 5 3 $\leftarrow$ (6) 3 6 3 3 6 6 5 3 

(9) 3 6 ..4 5 3 3 3 $\leftarrow$ (10) 5 6 2 4 5 3 3 3 

(9) 16 4 1 1 * 1 $\leftarrow$ (10) 20 1 1 * 1 

(11) ....3 5 7 3 3 $\leftarrow$ (12) 3 6 2 4 5 3 3 3 

(14) ....4 5 3 3 3 $\leftarrow$ (19) 6 2 3 4 4 1 1 1 

(15) 3 6 2 3 4 4 1 1 1 $\leftarrow$ (16) 8 1 1 2 4 3 3 3 

(17) 8 4 1 1 * 1 $\leftarrow$ (18) 12 1 1 * 1 

(21) 4 4 1 1 * 1 $\leftarrow$ (24) 1 2 3 4 4 1 1 1 

(23) * * 1 $\leftarrow$ (26) 4 1 1 * 1 

\end{tcolorbox}
\begin{tcolorbox}[title={$(40,11)$}]
(3) ..4 3 3 3 6 6 5 3 $\leftarrow$ (4) 3 5 6 2 3 5 7 3 3 

(6) ....3 3 6 6 5 3 

(9) .....3 5 7 3 3 $\leftarrow$ (10) 3 6 ..4 5 3 3 3 

(9) 14 * * 1 $\leftarrow$ (10) 16 4 1 1 * 1 

(12) .....4 5 3 3 3 $\leftarrow$ (18) * 2 4 3 3 3 

(15) 2 * 2 4 3 3 3 $\leftarrow$ (16) 3 6 2 3 4 4 1 1 1 

(17) 6 * * 1 $\leftarrow$ (18) 8 4 1 1 * 1 

(21) 2 * * 1 $\leftarrow$ (22) 4 4 1 1 * 1 

(22) 1 * * 1 $\leftarrow$ (24) * * 1 

\end{tcolorbox}
\begin{tcolorbox}[title={$(40,12)$}]
(3) 6 .....3 5 7 3 3 

(9) 13 1 * * 1 $\leftarrow$ (10) 14 * * 1 

(10) ......4 5 3 3 3 $\leftarrow$ (17) 1 * 2 4 3 3 3 

(15) 1 1 * 2 4 3 3 3 $\leftarrow$ (16) 2 * 2 4 3 3 3 

(17) 5 1 * * 1 $\leftarrow$ (18) 6 * * 1 

(21) 1 1 * * 1 $\leftarrow$ (22) 2 * * 1 

\end{tcolorbox}
\begin{tcolorbox}[title={$(40,13)$}]
(1) 5 6 .....4 5 3 3 3 

(3) 3 6 .....4 5 3 3 3 $\leftarrow$ (4) 6 .....3 5 7 3 3 

(7) 8 1 1 * 2 4 3 3 3 $\leftarrow$ (16) 1 1 * 2 4 3 3 3 

(9) 12 1 1 * * 1 $\leftarrow$ (10) 13 1 * * 1 

(15) 1 2 3 4 4 1 1 * 1 $\leftarrow$ (17) 2 3 4 4 1 1 * 1 

(17) 4 1 1 * * 1 $\leftarrow$ (18) 5 1 * * 1 

\end{tcolorbox}
\begin{tcolorbox}[title={$(40,14)$}]
(1) 3 6 ......4 5 3 3 3 $\leftarrow$ (2) 5 6 .....4 5 3 3 3 

(3) ........3 5 7 3 3 $\leftarrow$ (4) 3 6 .....4 5 3 3 3 

(6) ........4 5 3 3 3 $\leftarrow$ (11) 6 2 3 4 4 1 1 * 1 

(7) 3 6 2 3 4 4 1 1 * 1 $\leftarrow$ (8) 8 1 1 * 2 4 3 3 3 

(9) 8 4 1 1 * * 1 $\leftarrow$ (10) 12 1 1 * * 1 

(13) 4 4 1 1 * * 1 $\leftarrow$ (16) 1 2 3 4 4 1 1 * 1 

(15) * * * 1 $\leftarrow$ (18) 4 1 1 * * 1 

\end{tcolorbox}
\begin{tcolorbox}[title={$(40,15)$}]
(1) .........3 5 7 3 3 $\leftarrow$ (2) 3 6 ......4 5 3 3 3 

(4) .........4 5 3 3 3 $\leftarrow$ (10) * * 2 4 3 3 3 

(7) 2 * * 2 4 3 3 3 $\leftarrow$ (8) 3 6 2 3 4 4 1 1 * 1 

(9) 6 * * * 1 $\leftarrow$ (10) 8 4 1 1 * * 1 

(13) 2 * * * 1 $\leftarrow$ (14) 4 4 1 1 * * 1 

(14) 1 * * * 1 $\leftarrow$ (16) * * * 1 

\end{tcolorbox}
\begin{tcolorbox}[title={$(40,16)$}]
(1) 13 1 * * * 1 $\leftarrow$ (9) 1 * * 2 4 3 3 3 

(7) 1 1 * * 2 4 3 3 3 $\leftarrow$ (8) 2 * * 2 4 3 3 3 

(9) 5 1 * * * 1 $\leftarrow$ (10) 6 * * * 1 

(13) 1 1 * * * 1 $\leftarrow$ (14) 2 * * * 1 

\end{tcolorbox}
\begin{tcolorbox}[title={$(40,17)$}]
(1) 12 1 1 * * * 1 $\leftarrow$ (2) 13 1 * * * 1 

(7) 1 2 3 4 4 1 1 * * 1 $\leftarrow$ (9) 2 3 4 4 1 1 * * 1 

(9) 4 1 1 * * * 1 $\leftarrow$ (10) 5 1 * * * 1 

\end{tcolorbox}
\begin{tcolorbox}[title={$(40,18)$}]
(1) * * * 2 4 3 3 3 

(1) 8 4 1 1 * * * 1 $\leftarrow$ (2) 12 1 1 * * * 1 

(5) 4 4 1 1 * * * 1 $\leftarrow$ (8) 1 2 3 4 4 1 1 * * 1 

(7) * * * * 1 $\leftarrow$ (10) 4 1 1 * * * 1 

\end{tcolorbox}
\begin{tcolorbox}[title={$(40,19)$}]
(1) 6 * * * * 1 $\leftarrow$ (2) 8 4 1 1 * * * 1 

(2) 3 4 4 1 1 * * * 1 

(5) 2 * * * * 1 $\leftarrow$ (6) 4 4 1 1 * * * 1 

(6) 1 * * * * 1 $\leftarrow$ (8) * * * * 1 

\end{tcolorbox}
\begin{tcolorbox}[title={$(40,20)$}]
(1) 5 1 * * * * 1 $\leftarrow$ (2) 6 * * * * 1 

(5) 1 1 * * * * 1 $\leftarrow$ (6) 2 * * * * 1 

\end{tcolorbox}
\begin{tcolorbox}[title={$(40,21)$}]
(1) 4 1 1 * * * * 1 $\leftarrow$ (2) 5 1 * * * * 1 

\end{tcolorbox}
\begin{tcolorbox}[title={$(41,3)$}]
(3) 23 15 

(7) 27 7 $\leftarrow$ (11) 31 

(11) 15 15 

(23) 11 7 $\leftarrow$ (27) 15 

(27) 7 7 $\leftarrow$ (35) 7 

(39) 1 1 $\leftarrow$ (41) 1 

\end{tcolorbox}
\begin{tcolorbox}[title={$(41,4)$}]
(2) 21 11 7 

(3) 22 13 3 $\leftarrow$ (4) 23 15 

(4) 23 7 7 $\leftarrow$ (10) 29 3 

(7) 26 5 3 $\leftarrow$ (8) 27 7 

(10) 13 11 7 

(11) 14 13 3 $\leftarrow$ (12) 15 15 

(20) 7 7 7 $\leftarrow$ (26) 13 3 

(22) 5 7 7 $\leftarrow$ (34) 5 3 

(23) 10 5 3 $\leftarrow$ (24) 11 7 

(27) 6 5 3 $\leftarrow$ (28) 7 7 

(38) 1 1 1 $\leftarrow$ (40) 1 1 

\end{tcolorbox}
\begin{tcolorbox}[title={$(41,5)$}]
(1) 11 15 7 7 

(1) 19 7 7 7 

(2) 20 5 7 7 $\leftarrow$ (9) 27 3 3 

(3) 17 7 7 7 

(3) 21 11 3 3 $\leftarrow$ (4) 22 13 3 

(7) 25 3 3 3 $\leftarrow$ (8) 26 5 3 

(11) 13 11 3 3 $\leftarrow$ (12) 14 13 3 

(18) 4 5 7 7 $\leftarrow$ (25) 11 3 3 

(19) 3 5 7 7 $\leftarrow$ (33) 3 3 3 

(23) 5 7 3 3 $\leftarrow$ (28) 6 5 3 

(23) 9 3 3 3 $\leftarrow$ (24) 10 5 3 

\end{tcolorbox}
\begin{tcolorbox}[title={$(41,6)$}]
(1) 7 14 5 7 7 $\leftarrow$ (2) 11 15 7 7 

(1) 18 3 5 7 7 $\leftarrow$ (8) 25 3 3 3 

(2) 7 13 5 7 7 $\leftarrow$ (4) 17 7 7 7 

(3) 20 9 3 3 3 $\leftarrow$ (4) 21 11 3 3 

(6) 5 9 7 7 7 $\leftarrow$ (10) 13 5 7 7 

(11) 12 9 3 3 3 $\leftarrow$ (12) 13 11 3 3 

(17) 2 3 5 7 7 $\leftarrow$ (24) 9 3 3 3 

(20) 3 5 7 3 3 $\leftarrow$ (25) 8 3 3 3 

(21) 4 7 3 3 3 $\leftarrow$ (22) 6 6 5 3 

(23) 4 5 3 3 3 $\leftarrow$ (24) 5 7 3 3 

(27) 5 1 2 3 3 $\leftarrow$ (28) 6 2 3 3 

\end{tcolorbox}
\begin{tcolorbox}[title={$(41,7)$}]
(1) 5 5 9 7 7 7 $\leftarrow$ (2) 7 14 5 7 7 

(1) 17 3 6 6 5 3 

(3) 3 5 9 7 7 7 $\leftarrow$ (5) 14 4 5 7 7 

(3) 8 12 9 3 3 3 

(7) 5 7 3 5 7 7 $\leftarrow$ (11) 9 3 5 7 7 

(11) 12 4 5 3 3 3 

(15) 3 3 6 6 5 3 $\leftarrow$ (19) 3 6 6 5 3 

(18) 2 3 5 7 3 3 $\leftarrow$ (24) 4 5 3 3 3 

(21) 2 4 5 3 3 3 $\leftarrow$ (22) 4 7 3 3 3 

(27) 3 4 4 1 1 1 $\leftarrow$ (28) 5 1 2 3 3 

(31) 1 * 1 $\leftarrow$ (33) * 1 

\end{tcolorbox}
\begin{tcolorbox}[title={$(41,8)$}]
(2) 3 5 9 3 5 7 7 

(3) 6 9 3 6 6 5 3 $\leftarrow$ (4) 8 12 9 3 3 3 

(4) 3 5 7 3 5 7 7 $\leftarrow$ (8) 5 7 3 5 7 7 

(9) 6 3 3 6 6 5 3 $\leftarrow$ (16) 3 3 6 6 5 3 

(11) 10 2 4 5 3 3 3 $\leftarrow$ (12) 12 4 5 3 3 3 

(15) 6 2 4 5 3 3 3 $\leftarrow$ (22) 2 4 5 3 3 3 

(30) 1 1 * 1 $\leftarrow$ (32) 1 * 1 

\end{tcolorbox}
\begin{tcolorbox}[title={$(41,9)$}]
(1) 5 6 5 2 3 5 7 7 

(3) 3 6 5 2 3 5 7 7 $\leftarrow$ (4) 6 9 3 6 6 5 3 

(7) 5 6 2 3 5 7 3 3 $\leftarrow$ (13) 6 2 3 5 7 3 3 

(9) 3 6 2 3 5 7 3 3 $\leftarrow$ (10) 6 3 3 6 6 5 3 

(11) 7 13 1 * 1 $\leftarrow$ (12) 10 2 4 5 3 3 3 

(13) 6 ..4 5 3 3 3 $\leftarrow$ (19) 13 1 * 1 

(15) 4 ..4 5 3 3 3 $\leftarrow$ (16) 6 2 4 5 3 3 3 

\end{tcolorbox}
\begin{tcolorbox}[title={$(41,10)$}]
(2) 2 4 7 3 3 6 6 5 3 

(3) 2 3 5 5 3 6 6 5 3 $\leftarrow$ (4) 3 6 5 2 3 5 7 7 

(6) 2 4 3 3 3 6 6 5 3 $\leftarrow$ (11) 5 6 2 4 5 3 3 3 

(7) 3 5 6 2 4 5 3 3 3 $\leftarrow$ (8) 5 6 2 3 5 7 3 3 

(9) ...3 3 6 6 5 3 $\leftarrow$ (10) 3 6 2 3 5 7 3 3 

(11) 5 8 1 1 2 4 3 3 3 $\leftarrow$ (12) 7 13 1 * 1 

(12) ....3 5 7 3 3 $\leftarrow$ (17) 8 1 1 2 4 3 3 3 

(13) 4 ...4 5 3 3 3 $\leftarrow$ (14) 6 ..4 5 3 3 3 

(15) ....4 5 3 3 3 $\leftarrow$ (16) 4 ..4 5 3 3 3 

(19) 5 1 2 3 4 4 1 1 1 $\leftarrow$ (20) 6 2 3 4 4 1 1 1 

\end{tcolorbox}
\begin{tcolorbox}[title={$(41,11)$}]
(1) 1 2 4 7 3 3 6 6 5 3 

(4) ..4 3 3 3 6 6 5 3 $\leftarrow$ (10) ...3 3 6 6 5 3 

(7) ....3 3 6 6 5 3 $\leftarrow$ (8) 3 5 6 2 4 5 3 3 3 

(10) .....3 5 7 3 3 $\leftarrow$ (16) ....4 5 3 3 3 

(11) 4 ....4 5 3 3 3 $\leftarrow$ (12) 5 8 1 1 2 4 3 3 3 

(13) .....4 5 3 3 3 $\leftarrow$ (14) 4 ...4 5 3 3 3 

(19) 3 4 4 1 1 * 1 $\leftarrow$ (20) 5 1 2 3 4 4 1 1 1 

(23) 1 * * 1 $\leftarrow$ (25) * * 1 

\end{tcolorbox}
\begin{tcolorbox}[title={$(41,12)$}]
(1) 6 ....3 3 6 6 5 3 $\leftarrow$ (8) ....3 3 6 6 5 3 

(7) 6 .....4 5 3 3 3 $\leftarrow$ (14) .....4 5 3 3 3 

(11) ......4 5 3 3 3 $\leftarrow$ (12) 4 ....4 5 3 3 3 

(22) 1 1 * * 1 $\leftarrow$ (24) 1 * * 1 

\end{tcolorbox}
\begin{tcolorbox}[title={$(41,13)$}]
(1) 3 6 .....3 5 7 3 3 $\leftarrow$ (2) 6 ....3 3 6 6 5 3 

(5) 6 ......4 5 3 3 3 $\leftarrow$ (12) ......4 5 3 3 3 

(7) 4 ......4 5 3 3 3 $\leftarrow$ (8) 6 .....4 5 3 3 3 

\end{tcolorbox}
\begin{tcolorbox}[title={$(41,14)$}]
(1) .......3 3 6 6 5 3 $\leftarrow$ (2) 3 6 .....3 5 7 3 3 

(4) ........3 5 7 3 3 $\leftarrow$ (9) 8 1 1 * 2 4 3 3 3 

(5) 4 .......4 5 3 3 3 $\leftarrow$ (6) 6 ......4 5 3 3 3 

(7) ........4 5 3 3 3 $\leftarrow$ (8) 4 ......4 5 3 3 3 

(11) 5 1 2 3 4 4 1 1 * 1 $\leftarrow$ (12) 6 2 3 4 4 1 1 * 1 

\end{tcolorbox}
\begin{tcolorbox}[title={$(41,15)$}]
(2) .........3 5 7 3 3 $\leftarrow$ (8) ........4 5 3 3 3 

(5) .........4 5 3 3 3 $\leftarrow$ (6) 4 .......4 5 3 3 3 

(11) 3 4 4 1 1 * * 1 $\leftarrow$ (12) 5 1 2 3 4 4 1 1 * 1 

(15) 1 * * * 1 $\leftarrow$ (17) * * * 1 

\end{tcolorbox}
\begin{tcolorbox}[title={$(41,16)$}]
(8) 1 1 * * 2 4 3 3 3 

(14) 1 1 * * * 1 $\leftarrow$ (16) 1 * * * 1 

\end{tcolorbox}
\begin{tcolorbox}[title={$(41,17)$}]
(3) 6 2 3 4 4 1 1 * * 1 

\end{tcolorbox}
\begin{tcolorbox}[title={$(41,18)$}]
(2) * * * 2 4 3 3 3 

(3) 5 1 2 3 4 4 1 1 * * 1 $\leftarrow$ (4) 6 2 3 4 4 1 1 * * 1 

\end{tcolorbox}
\begin{tcolorbox}[title={$(41,19)$}]
(1) 1 * * * 2 4 3 3 3 

(3) 3 4 4 1 1 * * * 1 $\leftarrow$ (4) 5 1 2 3 4 4 1 1 * * 1 

(7) 1 * * * * 1 $\leftarrow$ (9) * * * * 1 

\end{tcolorbox}
\begin{tcolorbox}[title={$(41,20)$}]
(1) 2 3 4 4 1 1 * * * 1 

(6) 1 1 * * * * 1 $\leftarrow$ (8) 1 * * * * 1 

\end{tcolorbox}
\begin{tcolorbox}[title={$(41,21)$}]
(2) 4 1 1 * * * * 1 

\end{tcolorbox}
\begin{tcolorbox}[title={$(42,2)$}]
(39) 3 $\leftarrow$ (43) 

\end{tcolorbox}
\begin{tcolorbox}[title={$(42,3)$}]
(11) 30 1 $\leftarrow$ (12) 31 

(27) 14 1 $\leftarrow$ (28) 15 

(35) 6 1 $\leftarrow$ (36) 7 

(36) 3 3 $\leftarrow$ (42) 1 

(39) 2 1 $\leftarrow$ (40) 3 

\end{tcolorbox}
\begin{tcolorbox}[title={$(42,4)$}]
(1) 11 15 15 

(3) 21 11 7 $\leftarrow$ (5) 23 15 

(5) 23 7 7 $\leftarrow$ (9) 27 7 

(11) 13 11 7 $\leftarrow$ (13) 15 15 

(11) 29 1 1 $\leftarrow$ (12) 30 1 

(21) 7 7 7 $\leftarrow$ (25) 11 7 

(23) 5 7 7 $\leftarrow$ (29) 7 7 

(27) 13 1 1 $\leftarrow$ (28) 14 1 

(34) 2 3 3 $\leftarrow$ (41) 1 1 

(35) 5 1 1 $\leftarrow$ (36) 6 1 

(39) 1 1 1 $\leftarrow$ (40) 2 1 

\end{tcolorbox}
\begin{tcolorbox}[title={$(42,5)$}]
(1) 10 13 11 7 $\leftarrow$ (2) 11 15 15 

(2) 19 7 7 7 $\leftarrow$ (4) 21 11 7 

(3) 20 5 7 7 $\leftarrow$ (6) 23 7 7 

(9) 14 5 7 7 $\leftarrow$ (12) 13 11 7 

(11) 28 1 1 1 $\leftarrow$ (12) 29 1 1 

(19) 4 5 7 7 $\leftarrow$ (22) 7 7 7 

(20) 3 5 7 7 $\leftarrow$ (24) 5 7 7 

(27) 12 1 1 1 $\leftarrow$ (28) 13 1 1 

(33) 1 2 3 3 $\leftarrow$ (40) 1 1 1 

(35) 4 1 1 1 $\leftarrow$ (36) 5 1 1 

\end{tcolorbox}
\begin{tcolorbox}[title={$(42,6)$}]
(1) 9 11 7 7 7 $\leftarrow$ (2) 10 13 11 7 

(2) 18 3 5 7 7 $\leftarrow$ (4) 20 5 7 7 

(3) 7 13 5 7 7 $\leftarrow$ (5) 17 7 7 7 

(4) 20 9 3 3 3 

(7) 5 9 7 7 7 $\leftarrow$ (11) 13 5 7 7 

(9) 11 3 5 7 7 $\leftarrow$ (10) 14 5 7 7 

(11) 24 4 1 1 1 $\leftarrow$ (12) 28 1 1 1 

(12) 12 9 3 3 3 

(18) 2 3 5 7 7 $\leftarrow$ (20) 4 5 7 7 

(21) 3 5 7 3 3 $\leftarrow$ (25) 5 7 3 3 

(25) 3 6 2 3 3 $\leftarrow$ (26) 8 3 3 3 

(27) 2 4 3 3 3 $\leftarrow$ (29) 6 2 3 3 

(27) 8 4 1 1 1 $\leftarrow$ (28) 12 1 1 1 

(31) 4 4 1 1 1 $\leftarrow$ (36) 4 1 1 1 

\end{tcolorbox}
\begin{tcolorbox}[title={$(42,7)$}]
(2) 5 5 9 7 7 7 $\leftarrow$ (4) 7 13 5 7 7 

(2) 17 3 6 6 5 3 

(4) 3 5 9 7 7 7 $\leftarrow$ (8) 5 9 7 7 7 

(5) 13 2 3 5 7 7 $\leftarrow$ (6) 14 4 5 7 7 

(6) 5 9 3 5 7 7 $\leftarrow$ (12) 9 3 5 7 7 

(10) 9 3 6 6 5 3 

(11) 22 * 1 $\leftarrow$ (12) 24 4 1 1 1 

(19) 2 3 5 7 3 3 $\leftarrow$ (20) 3 6 6 5 3 

(25) ..4 3 3 3 $\leftarrow$ (26) 3 6 2 3 3 

(26) 1 2 4 3 3 3 $\leftarrow$ (28) 2 4 3 3 3 

(27) 6 * 1 $\leftarrow$ (28) 8 4 1 1 1 

(28) 3 4 4 1 1 1 $\leftarrow$ (34) * 1 

(31) 2 * 1 $\leftarrow$ (32) 4 4 1 1 1 

\end{tcolorbox}
\begin{tcolorbox}[title={$(42,8)$}]
(1) 11 12 4 5 3 3 3 

(3) 3 5 9 3 5 7 7 $\leftarrow$ (5) 8 12 9 3 3 3 

(5) 3 5 7 3 5 7 7 $\leftarrow$ (9) 5 7 3 5 7 7 

(5) 7 12 4 5 3 3 3 $\leftarrow$ (6) 13 2 3 5 7 7 

(7) 6 5 2 3 5 7 7 $\leftarrow$ (13) 12 4 5 3 3 3 

(9) 5 5 3 6 6 5 3 

(11) 21 1 * 1 $\leftarrow$ (12) 22 * 1 

(25) 1 1 2 4 3 3 3 $\leftarrow$ (26) ..4 3 3 3 

(26) 2 3 4 4 1 1 1 $\leftarrow$ (33) 1 * 1 

(27) 5 1 * 1 $\leftarrow$ (28) 6 * 1 

(31) 1 1 * 1 $\leftarrow$ (32) 2 * 1 

\end{tcolorbox}
\begin{tcolorbox}[title={$(42,9)$}]
(2) 5 6 5 2 3 5 7 7 

(5) 4 7 3 3 6 6 5 3 $\leftarrow$ (11) 6 3 3 6 6 5 3 

(5) 5 6 3 3 6 6 5 3 $\leftarrow$ (6) 7 12 4 5 3 3 3 

(7) 3 6 3 3 6 6 5 3 $\leftarrow$ (8) 6 5 2 3 5 7 7 

(11) 20 1 1 * 1 $\leftarrow$ (12) 21 1 * 1 

(13) 3 6 2 4 5 3 3 3 $\leftarrow$ (14) 6 2 3 5 7 3 3 

(19) 12 1 1 * 1 $\leftarrow$ (20) 13 1 * 1 

(25) 1 2 3 4 4 1 1 1 $\leftarrow$ (32) 1 1 * 1 

(27) 4 1 1 * 1 $\leftarrow$ (28) 5 1 * 1 

\end{tcolorbox}
\begin{tcolorbox}[title={$(42,10)$}]
(3) 2 4 7 3 3 6 6 5 3 $\leftarrow$ (6) 4 7 3 3 6 6 5 3 

(4) 2 3 5 5 3 6 6 5 3 $\leftarrow$ (9) 5 6 2 3 5 7 3 3 

(5) 3 5 6 2 3 5 7 3 3 $\leftarrow$ (6) 5 6 3 3 6 6 5 3 

(7) 2 4 3 3 3 6 6 5 3 $\leftarrow$ (8) 3 6 3 3 6 6 5 3 

(11) 3 6 ..4 5 3 3 3 $\leftarrow$ (12) 5 6 2 4 5 3 3 3 

(11) 16 4 1 1 * 1 $\leftarrow$ (12) 20 1 1 * 1 

(13) ....3 5 7 3 3 $\leftarrow$ (14) 3 6 2 4 5 3 3 3 

(17) 3 6 2 3 4 4 1 1 1 $\leftarrow$ (18) 8 1 1 2 4 3 3 3 

(19) * 2 4 3 3 3 $\leftarrow$ (21) 6 2 3 4 4 1 1 1 

(19) 8 4 1 1 * 1 $\leftarrow$ (20) 12 1 1 * 1 

(23) 4 4 1 1 * 1 $\leftarrow$ (28) 4 1 1 * 1 

\end{tcolorbox}
\begin{tcolorbox}[title={$(42,11)$}]
(1) ..4 7 3 3 6 6 5 3 $\leftarrow$ (8) 2 4 3 3 3 6 6 5 3 

(2) 1 2 4 7 3 3 6 6 5 3 $\leftarrow$ (4) 2 4 7 3 3 6 6 5 3 

(5) ..4 3 3 3 6 6 5 3 $\leftarrow$ (6) 3 5 6 2 3 5 7 3 3 

(11) .....3 5 7 3 3 $\leftarrow$ (12) 3 6 ..4 5 3 3 3 

(11) 14 * * 1 $\leftarrow$ (12) 16 4 1 1 * 1 

(17) 2 * 2 4 3 3 3 $\leftarrow$ (18) 3 6 2 3 4 4 1 1 1 

(18) 1 * 2 4 3 3 3 $\leftarrow$ (20) * 2 4 3 3 3 

(19) 6 * * 1 $\leftarrow$ (20) 8 4 1 1 * 1 

(20) 3 4 4 1 1 * 1 $\leftarrow$ (26) * * 1 

(23) 2 * * 1 $\leftarrow$ (24) 4 4 1 1 * 1 

\end{tcolorbox}
\begin{tcolorbox}[title={$(42,12)$}]
(1) 1 1 2 4 7 3 3 6 6 5 3 $\leftarrow$ (2) ..4 7 3 3 6 6 5 3 

(5) 6 .....3 5 7 3 3 

(11) 13 1 * * 1 $\leftarrow$ (12) 14 * * 1 

(17) 1 1 * 2 4 3 3 3 $\leftarrow$ (18) 2 * 2 4 3 3 3 

(18) 2 3 4 4 1 1 * 1 $\leftarrow$ (25) 1 * * 1 

(19) 5 1 * * 1 $\leftarrow$ (20) 6 * * 1 

(23) 1 1 * * 1 $\leftarrow$ (24) 2 * * 1 

\end{tcolorbox}
\begin{tcolorbox}[title={$(42,13)$}]
(3) 5 6 .....4 5 3 3 3 

(5) 3 6 .....4 5 3 3 3 $\leftarrow$ (6) 6 .....3 5 7 3 3 

(11) 12 1 1 * * 1 $\leftarrow$ (12) 13 1 * * 1 

(17) 1 2 3 4 4 1 1 * 1 $\leftarrow$ (24) 1 1 * * 1 

(19) 4 1 1 * * 1 $\leftarrow$ (20) 5 1 * * 1 

\end{tcolorbox}
\begin{tcolorbox}[title={$(42,14)$}]
(2) .......3 3 6 6 5 3 

(3) 3 6 ......4 5 3 3 3 $\leftarrow$ (4) 5 6 .....4 5 3 3 3 

(5) ........3 5 7 3 3 $\leftarrow$ (6) 3 6 .....4 5 3 3 3 

(9) 3 6 2 3 4 4 1 1 * 1 $\leftarrow$ (10) 8 1 1 * 2 4 3 3 3 

(11) * * 2 4 3 3 3 $\leftarrow$ (13) 6 2 3 4 4 1 1 * 1 

(11) 8 4 1 1 * * 1 $\leftarrow$ (12) 12 1 1 * * 1 

(15) 4 4 1 1 * * 1 $\leftarrow$ (20) 4 1 1 * * 1 

\end{tcolorbox}
\begin{tcolorbox}[title={$(42,15)$}]
(3) .........3 5 7 3 3 $\leftarrow$ (4) 3 6 ......4 5 3 3 3 

(6) .........4 5 3 3 3 

(9) 2 * * 2 4 3 3 3 $\leftarrow$ (10) 3 6 2 3 4 4 1 1 * 1 

(10) 1 * * 2 4 3 3 3 $\leftarrow$ (12) * * 2 4 3 3 3 

(11) 6 * * * 1 $\leftarrow$ (12) 8 4 1 1 * * 1 

(12) 3 4 4 1 1 * * 1 $\leftarrow$ (18) * * * 1 

(15) 2 * * * 1 $\leftarrow$ (16) 4 4 1 1 * * 1 

\end{tcolorbox}
\begin{tcolorbox}[title={$(42,16)$}]
(3) 13 1 * * * 1 

(9) 1 1 * * 2 4 3 3 3 $\leftarrow$ (10) 2 * * 2 4 3 3 3 

(10) 2 3 4 4 1 1 * * 1 $\leftarrow$ (17) 1 * * * 1 

(11) 5 1 * * * 1 $\leftarrow$ (12) 6 * * * 1 

(15) 1 1 * * * 1 $\leftarrow$ (16) 2 * * * 1 

\end{tcolorbox}
\begin{tcolorbox}[title={$(42,17)$}]
(1) 8 1 1 * * 2 4 3 3 3 

(3) 12 1 1 * * * 1 $\leftarrow$ (4) 13 1 * * * 1 

(9) 1 2 3 4 4 1 1 * * 1 $\leftarrow$ (16) 1 1 * * * 1 

(11) 4 1 1 * * * 1 $\leftarrow$ (12) 5 1 * * * 1 

\end{tcolorbox}
\begin{tcolorbox}[title={$(42,18)$}]
(1) 3 6 2 3 4 4 1 1 * * 1 $\leftarrow$ (2) 8 1 1 * * 2 4 3 3 3 

(3) * * * 2 4 3 3 3 $\leftarrow$ (5) 6 2 3 4 4 1 1 * * 1 

(3) 8 4 1 1 * * * 1 $\leftarrow$ (4) 12 1 1 * * * 1 

(7) 4 4 1 1 * * * 1 $\leftarrow$ (12) 4 1 1 * * * 1 

\end{tcolorbox}
\begin{tcolorbox}[title={$(42,19)$}]
(1) 2 * * * 2 4 3 3 3 $\leftarrow$ (2) 3 6 2 3 4 4 1 1 * * 1 

(2) 1 * * * 2 4 3 3 3 $\leftarrow$ (4) * * * 2 4 3 3 3 

(3) 6 * * * * 1 $\leftarrow$ (4) 8 4 1 1 * * * 1 

(4) 3 4 4 1 1 * * * 1 $\leftarrow$ (10) * * * * 1 

(7) 2 * * * * 1 $\leftarrow$ (8) 4 4 1 1 * * * 1 

\end{tcolorbox}
\begin{tcolorbox}[title={$(42,20)$}]
(1) 1 1 * * * 2 4 3 3 3 $\leftarrow$ (2) 2 * * * 2 4 3 3 3 

(2) 2 3 4 4 1 1 * * * 1 $\leftarrow$ (9) 1 * * * * 1 

(3) 5 1 * * * * 1 $\leftarrow$ (4) 6 * * * * 1 

(7) 1 1 * * * * 1 $\leftarrow$ (8) 2 * * * * 1 

\end{tcolorbox}
\begin{tcolorbox}[title={$(42,21)$}]
(1) 1 2 3 4 4 1 1 * * * 1 $\leftarrow$ (8) 1 1 * * * * 1 

(3) 4 1 1 * * * * 1 $\leftarrow$ (4) 5 1 * * * * 1 

\end{tcolorbox}
\begin{tcolorbox}[title={$(42,22)$}]
(1) * * * * * 1 

\end{tcolorbox}
\begin{tcolorbox}[title={$(43,3)$}]
(11) 29 3 $\leftarrow$ (13) 31 

(27) 13 3 $\leftarrow$ (29) 15 

(35) 5 3 $\leftarrow$ (37) 7 

(37) 3 3 $\leftarrow$ (41) 3 

\end{tcolorbox}
\begin{tcolorbox}[title={$(43,4)$}]
(5) 22 13 3 $\leftarrow$ (6) 23 15 

(9) 26 5 3 $\leftarrow$ (10) 27 7 

(10) 27 3 3 $\leftarrow$ (12) 29 3 

(13) 14 13 3 $\leftarrow$ (14) 15 15 

(25) 10 5 3 $\leftarrow$ (26) 11 7 

(26) 11 3 3 $\leftarrow$ (28) 13 3 

(29) 6 5 3 $\leftarrow$ (30) 7 7 

(34) 3 3 3 $\leftarrow$ (36) 5 3 

(35) 2 3 3 $\leftarrow$ (38) 3 3 

\end{tcolorbox}
\begin{tcolorbox}[title={$(43,5)$}]
(3) 11 15 7 7 

(3) 19 7 7 7 $\leftarrow$ (5) 21 11 7 

(5) 21 11 3 3 $\leftarrow$ (6) 22 13 3 

(9) 25 3 3 3 $\leftarrow$ (10) 26 5 3 

(13) 13 11 3 3 $\leftarrow$ (14) 14 13 3 

(21) 3 5 7 7 $\leftarrow$ (25) 5 7 7 

(23) 6 6 5 3 $\leftarrow$ (30) 6 5 3 

(25) 9 3 3 3 $\leftarrow$ (26) 10 5 3 

(34) 1 2 3 3 $\leftarrow$ (36) 2 3 3 

\end{tcolorbox}
\begin{tcolorbox}[title={$(43,6)$}]
(2) 9 11 7 7 7 $\leftarrow$ (6) 17 7 7 7 

(3) 7 14 5 7 7 $\leftarrow$ (4) 11 15 7 7 

(3) 18 3 5 7 7 $\leftarrow$ (5) 20 5 7 7 

(5) 20 9 3 3 3 $\leftarrow$ (6) 21 11 3 3 

(10) 11 3 5 7 7 $\leftarrow$ (12) 13 5 7 7 

(13) 12 9 3 3 3 $\leftarrow$ (14) 13 11 3 3 

(19) 2 3 5 7 7 $\leftarrow$ (21) 4 5 7 7 

(22) 3 5 7 3 3 $\leftarrow$ (27) 8 3 3 3 

(23) 4 7 3 3 3 $\leftarrow$ (24) 6 6 5 3 

(25) 4 5 3 3 3 $\leftarrow$ (26) 5 7 3 3 

(29) 5 1 2 3 3 $\leftarrow$ (30) 6 2 3 3 

\end{tcolorbox}
\begin{tcolorbox}[title={$(43,7)$}]
(1) 12 12 9 3 3 3 $\leftarrow$ (4) 18 3 5 7 7 

(3) 5 5 9 7 7 7 $\leftarrow$ (4) 7 14 5 7 7 

(3) 17 3 6 6 5 3 $\leftarrow$ (6) 20 9 3 3 3 

(5) 3 5 9 7 7 7 $\leftarrow$ (9) 5 9 7 7 7 

(7) 5 9 3 5 7 7 $\leftarrow$ (13) 9 3 5 7 7 

(11) 9 3 6 6 5 3 $\leftarrow$ (14) 12 9 3 3 3 

(17) 3 3 6 6 5 3 $\leftarrow$ (20) 2 3 5 7 7 

(20) 2 3 5 7 3 3 $\leftarrow$ (26) 4 5 3 3 3 

(23) 2 4 5 3 3 3 $\leftarrow$ (24) 4 7 3 3 3 

(27) 1 2 4 3 3 3 $\leftarrow$ (29) 2 4 3 3 3 

(29) 3 4 4 1 1 1 $\leftarrow$ (30) 5 1 2 3 3 

\end{tcolorbox}
\begin{tcolorbox}[title={$(43,8)$}]
(1) 10 9 3 6 6 5 3 $\leftarrow$ (2) 12 12 9 3 3 3 

(2) 11 12 4 5 3 3 3 $\leftarrow$ (4) 17 3 6 6 5 3 

(4) 3 5 9 3 5 7 7 $\leftarrow$ (10) 5 7 3 5 7 7 

(5) 6 9 3 6 6 5 3 $\leftarrow$ (6) 8 12 9 3 3 3 

(6) 3 5 7 3 5 7 7 $\leftarrow$ (8) 5 9 3 5 7 7 

(10) 5 5 3 6 6 5 3 $\leftarrow$ (12) 9 3 6 6 5 3 

(13) 10 2 4 5 3 3 3 $\leftarrow$ (14) 12 4 5 3 3 3 

(17) 6 2 4 5 3 3 3 $\leftarrow$ (24) 2 4 5 3 3 3 

(26) 1 1 2 4 3 3 3 $\leftarrow$ (28) 1 2 4 3 3 3 

(27) 2 3 4 4 1 1 1 $\leftarrow$ (30) 3 4 4 1 1 1 

\end{tcolorbox}
\begin{tcolorbox}[title={$(43,9)$}]
(1) 9 5 5 3 6 6 5 3 $\leftarrow$ (2) 10 9 3 6 6 5 3 

(3) 5 6 5 2 3 5 7 7 $\leftarrow$ (9) 6 5 2 3 5 7 7 

(5) 3 6 5 2 3 5 7 7 $\leftarrow$ (6) 6 9 3 6 6 5 3 

(11) 3 6 2 3 5 7 3 3 $\leftarrow$ (12) 6 3 3 6 6 5 3 

(13) 7 13 1 * 1 $\leftarrow$ (14) 10 2 4 5 3 3 3 

(15) 6 ..4 5 3 3 3 $\leftarrow$ (21) 13 1 * 1 

(17) 4 ..4 5 3 3 3 $\leftarrow$ (18) 6 2 4 5 3 3 3 

(26) 1 2 3 4 4 1 1 1 $\leftarrow$ (28) 2 3 4 4 1 1 1 

\end{tcolorbox}
\begin{tcolorbox}[title={$(43,10)$}]
(1) 4 5 5 5 3 6 6 5 3 $\leftarrow$ (4) 5 6 5 2 3 5 7 7 

(5) 2 3 5 5 3 6 6 5 3 $\leftarrow$ (6) 3 6 5 2 3 5 7 7 

(9) 3 5 6 2 4 5 3 3 3 $\leftarrow$ (10) 5 6 2 3 5 7 3 3 

(11) ...3 3 6 6 5 3 $\leftarrow$ (12) 3 6 2 3 5 7 3 3 

(13) 5 8 1 1 2 4 3 3 3 $\leftarrow$ (14) 7 13 1 * 1 

(14) ....3 5 7 3 3 $\leftarrow$ (19) 8 1 1 2 4 3 3 3 

(15) 4 ...4 5 3 3 3 $\leftarrow$ (16) 6 ..4 5 3 3 3 

(17) ....4 5 3 3 3 $\leftarrow$ (18) 4 ..4 5 3 3 3 

(21) 5 1 2 3 4 4 1 1 1 $\leftarrow$ (22) 6 2 3 4 4 1 1 1 

\end{tcolorbox}
\begin{tcolorbox}[title={$(43,11)$}]
(3) 1 2 4 7 3 3 6 6 5 3 $\leftarrow$ (5) 2 4 7 3 3 6 6 5 3 

(6) ..4 3 3 3 6 6 5 3 

(9) ....3 3 6 6 5 3 $\leftarrow$ (10) 3 5 6 2 4 5 3 3 3 

(12) .....3 5 7 3 3 $\leftarrow$ (18) ....4 5 3 3 3 

(13) 4 ....4 5 3 3 3 $\leftarrow$ (14) 5 8 1 1 2 4 3 3 3 

(15) .....4 5 3 3 3 $\leftarrow$ (16) 4 ...4 5 3 3 3 

(19) 1 * 2 4 3 3 3 $\leftarrow$ (21) * 2 4 3 3 3 

(21) 3 4 4 1 1 * 1 $\leftarrow$ (22) 5 1 2 3 4 4 1 1 1 

\end{tcolorbox}
\begin{tcolorbox}[title={$(43,12)$}]
(2) 1 1 2 4 7 3 3 6 6 5 3 $\leftarrow$ (4) 1 2 4 7 3 3 6 6 5 3 

(3) 6 ....3 3 6 6 5 3 

(9) 6 .....4 5 3 3 3 $\leftarrow$ (16) .....4 5 3 3 3 

(13) ......4 5 3 3 3 $\leftarrow$ (14) 4 ....4 5 3 3 3 

(18) 1 1 * 2 4 3 3 3 $\leftarrow$ (20) 1 * 2 4 3 3 3 

(19) 2 3 4 4 1 1 * 1 $\leftarrow$ (22) 3 4 4 1 1 * 1 

\end{tcolorbox}
\begin{tcolorbox}[title={$(43,13)$}]
(1) 5 6 .....3 5 7 3 3 

(3) 3 6 .....3 5 7 3 3 $\leftarrow$ (4) 6 ....3 3 6 6 5 3 

(7) 6 ......4 5 3 3 3 $\leftarrow$ (14) ......4 5 3 3 3 

(9) 4 ......4 5 3 3 3 $\leftarrow$ (10) 6 .....4 5 3 3 3 

(18) 1 2 3 4 4 1 1 * 1 $\leftarrow$ (20) 2 3 4 4 1 1 * 1 

\end{tcolorbox}
\begin{tcolorbox}[title={$(43,14)$}]
(1) 3 5 6 .....4 5 3 3 3 $\leftarrow$ (2) 5 6 .....3 5 7 3 3 

(3) .......3 3 6 6 5 3 $\leftarrow$ (4) 3 6 .....3 5 7 3 3 

(6) ........3 5 7 3 3 $\leftarrow$ (11) 8 1 1 * 2 4 3 3 3 

(7) 4 .......4 5 3 3 3 $\leftarrow$ (8) 6 ......4 5 3 3 3 

(9) ........4 5 3 3 3 $\leftarrow$ (10) 4 ......4 5 3 3 3 

(13) 5 1 2 3 4 4 1 1 * 1 $\leftarrow$ (14) 6 2 3 4 4 1 1 * 1 

\end{tcolorbox}
\begin{tcolorbox}[title={$(43,15)$}]
(1) ........3 3 6 6 5 3 $\leftarrow$ (2) 3 5 6 .....4 5 3 3 3 

(4) .........3 5 7 3 3 $\leftarrow$ (10) ........4 5 3 3 3 

(7) .........4 5 3 3 3 $\leftarrow$ (8) 4 .......4 5 3 3 3 

(11) 1 * * 2 4 3 3 3 $\leftarrow$ (13) * * 2 4 3 3 3 

(13) 3 4 4 1 1 * * 1 $\leftarrow$ (14) 5 1 2 3 4 4 1 1 * 1 

\end{tcolorbox}
\begin{tcolorbox}[title={$(43,16)$}]
(1) 6 .........4 5 3 3 3 $\leftarrow$ (8) .........4 5 3 3 3 

(10) 1 1 * * 2 4 3 3 3 $\leftarrow$ (12) 1 * * 2 4 3 3 3 

(11) 2 3 4 4 1 1 * * 1 $\leftarrow$ (14) 3 4 4 1 1 * * 1 

\end{tcolorbox}
\begin{tcolorbox}[title={$(43,17)$}]
(1) 4 ..........4 5 3 3 3 $\leftarrow$ (2) 6 .........4 5 3 3 3 

(10) 1 2 3 4 4 1 1 * * 1 $\leftarrow$ (12) 2 3 4 4 1 1 * * 1 

\end{tcolorbox}
\begin{tcolorbox}[title={$(43,18)$}]
(1) ............4 5 3 3 3 $\leftarrow$ (2) 4 ..........4 5 3 3 3 

(5) 5 1 2 3 4 4 1 1 * * 1 $\leftarrow$ (6) 6 2 3 4 4 1 1 * * 1 

\end{tcolorbox}
\begin{tcolorbox}[title={$(43,19)$}]
(3) 1 * * * 2 4 3 3 3 $\leftarrow$ (5) * * * 2 4 3 3 3 

(5) 3 4 4 1 1 * * * 1 $\leftarrow$ (6) 5 1 2 3 4 4 1 1 * * 1 

\end{tcolorbox}
\begin{tcolorbox}[title={$(43,20)$}]
(2) 1 1 * * * 2 4 3 3 3 $\leftarrow$ (4) 1 * * * 2 4 3 3 3 

(3) 2 3 4 4 1 1 * * * 1 $\leftarrow$ (6) 3 4 4 1 1 * * * 1 

\end{tcolorbox}
\begin{tcolorbox}[title={$(43,21)$}]
(2) 1 2 3 4 4 1 1 * * * 1 $\leftarrow$ (4) 2 3 4 4 1 1 * * * 1 

(4) 4 1 1 * * * * 1 

\end{tcolorbox}
\begin{tcolorbox}[title={$(43,22)$}]
(2) * * * * * 1 

\end{tcolorbox}
\begin{tcolorbox}[title={$(43,23)$}]
(1) 1 * * * * * 1 

\end{tcolorbox}
\begin{tcolorbox}[title={$(44,2)$}]
(43) 1 $\leftarrow$ (45) 

\end{tcolorbox}
\begin{tcolorbox}[title={$(44,3)$}]
(13) 30 1 $\leftarrow$ (14) 31 

(29) 14 1 $\leftarrow$ (30) 15 

(37) 6 1 $\leftarrow$ (38) 7 

(41) 2 1 $\leftarrow$ (42) 3 

(42) 1 1 $\leftarrow$ (44) 1 

\end{tcolorbox}
\begin{tcolorbox}[title={$(44,4)$}]
(3) 11 15 15 

(7) 23 7 7 $\leftarrow$ (11) 27 7 

(11) 27 3 3 $\leftarrow$ (13) 29 3 

(13) 13 11 7 

(13) 29 1 1 $\leftarrow$ (14) 30 1 

(23) 7 7 7 $\leftarrow$ (27) 11 7 

(27) 11 3 3 $\leftarrow$ (29) 13 3 

(29) 13 1 1 $\leftarrow$ (30) 14 1 

(35) 3 3 3 $\leftarrow$ (37) 5 3 

(37) 5 1 1 $\leftarrow$ (38) 6 1 

(41) 1 1 1 $\leftarrow$ (42) 2 1 

\end{tcolorbox}
\begin{tcolorbox}[title={$(44,5)$}]
(3) 10 13 11 7 $\leftarrow$ (4) 11 15 15 

(4) 19 7 7 7 $\leftarrow$ (8) 23 7 7 

(10) 25 3 3 3 $\leftarrow$ (12) 27 3 3 

(11) 14 5 7 7 

(13) 28 1 1 1 $\leftarrow$ (14) 29 1 1 

(22) 3 5 7 7 $\leftarrow$ (26) 5 7 7 

(26) 9 3 3 3 $\leftarrow$ (28) 11 3 3 

(29) 12 1 1 1 $\leftarrow$ (30) 13 1 1 

(35) 1 2 3 3 $\leftarrow$ (37) 2 3 3 

(37) 4 1 1 1 $\leftarrow$ (38) 5 1 1 

\end{tcolorbox}
\begin{tcolorbox}[title={$(44,6)$}]
(3) 9 11 7 7 7 $\leftarrow$ (4) 10 13 11 7 

(5) 7 13 5 7 7 

(7) 14 4 5 7 7 

(11) 11 3 5 7 7 $\leftarrow$ (12) 14 5 7 7 

(13) 24 4 1 1 1 $\leftarrow$ (14) 28 1 1 1 

(21) 3 6 6 5 3 $\leftarrow$ (27) 5 7 3 3 

(23) 3 5 7 3 3 $\leftarrow$ (25) 6 6 5 3 

(27) 3 6 2 3 3 $\leftarrow$ (28) 8 3 3 3 

(29) 8 4 1 1 1 $\leftarrow$ (30) 12 1 1 1 

(33) 4 4 1 1 1 $\leftarrow$ (36) 1 2 3 3 

(35) * 1 $\leftarrow$ (38) 4 1 1 1 

\end{tcolorbox}
\begin{tcolorbox}[title={$(44,7)$}]
(4) 5 5 9 7 7 7 

(6) 3 5 9 7 7 7 $\leftarrow$ (10) 5 9 7 7 7 

(7) 13 2 3 5 7 7 $\leftarrow$ (8) 14 4 5 7 7 

(13) 22 * 1 $\leftarrow$ (14) 24 4 1 1 1 

(18) 3 3 6 6 5 3 $\leftarrow$ (24) 3 5 7 3 3 

(21) 2 3 5 7 3 3 $\leftarrow$ (22) 3 6 6 5 3 

(27) ..4 3 3 3 $\leftarrow$ (28) 3 6 2 3 3 

(29) 6 * 1 $\leftarrow$ (30) 8 4 1 1 1 

(33) 2 * 1 $\leftarrow$ (34) 4 4 1 1 1 

(34) 1 * 1 $\leftarrow$ (36) * 1 

\end{tcolorbox}
\begin{tcolorbox}[title={$(44,8)$}]
(3) 11 12 4 5 3 3 3 $\leftarrow$ (5) 17 3 6 6 5 3 

(5) 3 5 9 3 5 7 7 $\leftarrow$ (11) 5 7 3 5 7 7 

(7) 3 5 7 3 5 7 7 $\leftarrow$ (9) 5 9 3 5 7 7 

(7) 7 12 4 5 3 3 3 $\leftarrow$ (8) 13 2 3 5 7 7 

(11) 5 5 3 6 6 5 3 $\leftarrow$ (13) 9 3 6 6 5 3 

(13) 21 1 * 1 $\leftarrow$ (14) 22 * 1 

(15) 6 2 3 5 7 3 3 $\leftarrow$ (22) 2 3 5 7 3 3 

(27) 1 1 2 4 3 3 3 $\leftarrow$ (28) ..4 3 3 3 

(29) 5 1 * 1 $\leftarrow$ (30) 6 * 1 

(33) 1 1 * 1 $\leftarrow$ (34) 2 * 1 

\end{tcolorbox}
\begin{tcolorbox}[title={$(44,9)$}]
(2) 9 5 5 3 6 6 5 3 $\leftarrow$ (4) 11 12 4 5 3 3 3 

(7) 4 7 3 3 6 6 5 3 $\leftarrow$ (12) 5 5 3 6 6 5 3 

(7) 5 6 3 3 6 6 5 3 $\leftarrow$ (8) 7 12 4 5 3 3 3 

(9) 3 6 3 3 6 6 5 3 $\leftarrow$ (10) 6 5 2 3 5 7 7 

(13) 5 6 2 4 5 3 3 3 $\leftarrow$ (19) 6 2 4 5 3 3 3 

(13) 20 1 1 * 1 $\leftarrow$ (14) 21 1 * 1 

(15) 3 6 2 4 5 3 3 3 $\leftarrow$ (16) 6 2 3 5 7 3 3 

(21) 12 1 1 * 1 $\leftarrow$ (22) 13 1 * 1 

(27) 1 2 3 4 4 1 1 1 $\leftarrow$ (29) 2 3 4 4 1 1 1 

(29) 4 1 1 * 1 $\leftarrow$ (30) 5 1 * 1 

\end{tcolorbox}
\begin{tcolorbox}[title={$(44,10)$}]
(2) 4 5 5 5 3 6 6 5 3 

(6) 2 3 5 5 3 6 6 5 3 $\leftarrow$ (8) 4 7 3 3 6 6 5 3 

(7) 3 5 6 2 3 5 7 3 3 $\leftarrow$ (8) 5 6 3 3 6 6 5 3 

(9) 2 4 3 3 3 6 6 5 3 $\leftarrow$ (10) 3 6 3 3 6 6 5 3 

(12) ...3 3 6 6 5 3 $\leftarrow$ (17) 6 ..4 5 3 3 3 

(13) 3 6 ..4 5 3 3 3 $\leftarrow$ (14) 5 6 2 4 5 3 3 3 

(13) 16 4 1 1 * 1 $\leftarrow$ (14) 20 1 1 * 1 

(15) ....3 5 7 3 3 $\leftarrow$ (16) 3 6 2 4 5 3 3 3 

(19) 3 6 2 3 4 4 1 1 1 $\leftarrow$ (20) 8 1 1 2 4 3 3 3 

(21) 8 4 1 1 * 1 $\leftarrow$ (22) 12 1 1 * 1 

(25) 4 4 1 1 * 1 $\leftarrow$ (28) 1 2 3 4 4 1 1 1 

(27) * * 1 $\leftarrow$ (30) 4 1 1 * 1 

\end{tcolorbox}
\begin{tcolorbox}[title={$(44,11)$}]
(3) ..4 7 3 3 6 6 5 3 $\leftarrow$ (6) 2 4 7 3 3 6 6 5 3 

(7) ..4 3 3 3 6 6 5 3 $\leftarrow$ (8) 3 5 6 2 3 5 7 3 3 

(10) ....3 3 6 6 5 3 $\leftarrow$ (16) ....3 5 7 3 3 

(13) .....3 5 7 3 3 $\leftarrow$ (14) 3 6 ..4 5 3 3 3 

(13) 14 * * 1 $\leftarrow$ (14) 16 4 1 1 * 1 

(19) 2 * 2 4 3 3 3 $\leftarrow$ (20) 3 6 2 3 4 4 1 1 1 

(21) 6 * * 1 $\leftarrow$ (22) 8 4 1 1 * 1 

(25) 2 * * 1 $\leftarrow$ (26) 4 4 1 1 * 1 

(26) 1 * * 1 $\leftarrow$ (28) * * 1 

\end{tcolorbox}
\begin{tcolorbox}[title={$(44,12)$}]
(1) 6 ..4 3 3 3 6 6 5 3 $\leftarrow$ (5) 1 2 4 7 3 3 6 6 5 3 

(3) 1 1 2 4 7 3 3 6 6 5 3 $\leftarrow$ (4) ..4 7 3 3 6 6 5 3 

(7) 6 .....3 5 7 3 3 $\leftarrow$ (14) .....3 5 7 3 3 

(13) 13 1 * * 1 $\leftarrow$ (14) 14 * * 1 

(19) 1 1 * 2 4 3 3 3 $\leftarrow$ (20) 2 * 2 4 3 3 3 

(21) 5 1 * * 1 $\leftarrow$ (22) 6 * * 1 

(25) 1 1 * * 1 $\leftarrow$ (26) 2 * * 1 

\end{tcolorbox}
\begin{tcolorbox}[title={$(44,13)$}]
(1) 3 6 ....3 3 6 6 5 3 $\leftarrow$ (2) 6 ..4 3 3 3 6 6 5 3 

(5) 5 6 .....4 5 3 3 3 $\leftarrow$ (11) 6 .....4 5 3 3 3 

(7) 3 6 .....4 5 3 3 3 $\leftarrow$ (8) 6 .....3 5 7 3 3 

(13) 12 1 1 * * 1 $\leftarrow$ (14) 13 1 * * 1 

(19) 1 2 3 4 4 1 1 * 1 $\leftarrow$ (21) 2 3 4 4 1 1 * 1 

(21) 4 1 1 * * 1 $\leftarrow$ (22) 5 1 * * 1 

\end{tcolorbox}
\begin{tcolorbox}[title={$(44,14)$}]
(1) .....4 3 3 3 6 6 5 3 $\leftarrow$ (2) 3 6 ....3 3 6 6 5 3 

(4) .......3 3 6 6 5 3 $\leftarrow$ (9) 6 ......4 5 3 3 3 

(5) 3 6 ......4 5 3 3 3 $\leftarrow$ (6) 5 6 .....4 5 3 3 3 

(7) ........3 5 7 3 3 $\leftarrow$ (8) 3 6 .....4 5 3 3 3 

(11) 3 6 2 3 4 4 1 1 * 1 $\leftarrow$ (12) 8 1 1 * 2 4 3 3 3 

(13) 8 4 1 1 * * 1 $\leftarrow$ (14) 12 1 1 * * 1 

(17) 4 4 1 1 * * 1 $\leftarrow$ (20) 1 2 3 4 4 1 1 * 1 

(19) * * * 1 $\leftarrow$ (22) 4 1 1 * * 1 

\end{tcolorbox}
\begin{tcolorbox}[title={$(44,15)$}]
(2) ........3 3 6 6 5 3 $\leftarrow$ (8) ........3 5 7 3 3 

(5) .........3 5 7 3 3 $\leftarrow$ (6) 3 6 ......4 5 3 3 3 

(11) 2 * * 2 4 3 3 3 $\leftarrow$ (12) 3 6 2 3 4 4 1 1 * 1 

(13) 6 * * * 1 $\leftarrow$ (14) 8 4 1 1 * * 1 

(17) 2 * * * 1 $\leftarrow$ (18) 4 4 1 1 * * 1 

(18) 1 * * * 1 $\leftarrow$ (20) * * * 1 

\end{tcolorbox}
\begin{tcolorbox}[title={$(44,16)$}]
(5) 13 1 * * * 1 

(11) 1 1 * * 2 4 3 3 3 $\leftarrow$ (12) 2 * * 2 4 3 3 3 

(13) 5 1 * * * 1 $\leftarrow$ (14) 6 * * * 1 

(17) 1 1 * * * 1 $\leftarrow$ (18) 2 * * * 1 

\end{tcolorbox}
\begin{tcolorbox}[title={$(44,17)$}]
(3) 8 1 1 * * 2 4 3 3 3 

(5) 12 1 1 * * * 1 $\leftarrow$ (6) 13 1 * * * 1 

(11) 1 2 3 4 4 1 1 * * 1 $\leftarrow$ (13) 2 3 4 4 1 1 * * 1 

(13) 4 1 1 * * * 1 $\leftarrow$ (14) 5 1 * * * 1 

\end{tcolorbox}
\begin{tcolorbox}[title={$(44,18)$}]
(2) ............4 5 3 3 3 

(3) 3 6 2 3 4 4 1 1 * * 1 $\leftarrow$ (4) 8 1 1 * * 2 4 3 3 3 

(5) 8 4 1 1 * * * 1 $\leftarrow$ (6) 12 1 1 * * * 1 

(9) 4 4 1 1 * * * 1 $\leftarrow$ (12) 1 2 3 4 4 1 1 * * 1 

(11) * * * * 1 $\leftarrow$ (14) 4 1 1 * * * 1 

\end{tcolorbox}
\begin{tcolorbox}[title={$(44,19)$}]
(3) 2 * * * 2 4 3 3 3 $\leftarrow$ (4) 3 6 2 3 4 4 1 1 * * 1 

(5) 6 * * * * 1 $\leftarrow$ (6) 8 4 1 1 * * * 1 

(9) 2 * * * * 1 $\leftarrow$ (10) 4 4 1 1 * * * 1 

(10) 1 * * * * 1 $\leftarrow$ (12) * * * * 1 

\end{tcolorbox}
\begin{tcolorbox}[title={$(44,20)$}]
(3) 1 1 * * * 2 4 3 3 3 $\leftarrow$ (4) 2 * * * 2 4 3 3 3 

(5) 5 1 * * * * 1 $\leftarrow$ (6) 6 * * * * 1 

(9) 1 1 * * * * 1 $\leftarrow$ (10) 2 * * * * 1 

\end{tcolorbox}
\begin{tcolorbox}[title={$(44,21)$}]
(3) 1 2 3 4 4 1 1 * * * 1 $\leftarrow$ (5) 2 3 4 4 1 1 * * * 1 

(5) 4 1 1 * * * * 1 $\leftarrow$ (6) 5 1 * * * * 1 

\end{tcolorbox}
\begin{tcolorbox}[title={$(44,22)$}]
(1) 4 4 1 1 * * * * 1 $\leftarrow$ (4) 1 2 3 4 4 1 1 * * * 1 

(3) * * * * * 1 $\leftarrow$ (6) 4 1 1 * * * * 1 

\end{tcolorbox}
\begin{tcolorbox}[title={$(44,23)$}]
(1) 2 * * * * * 1 $\leftarrow$ (2) 4 4 1 1 * * * * 1 

(2) 1 * * * * * 1 $\leftarrow$ (4) * * * * * 1 

\end{tcolorbox}
\begin{tcolorbox}[title={$(44,24)$}]
(1) 1 1 * * * * * 1 $\leftarrow$ (2) 2 * * * * * 1 

\end{tcolorbox}
\begin{tcolorbox}[title={$(45,3)$}]
(7) 23 15 

(15) 15 15 

(31) 7 7 $\leftarrow$ (39) 7 

(39) 3 3 $\leftarrow$ (43) 3 

(43) 1 1 $\leftarrow$ (45) 1 

\end{tcolorbox}
\begin{tcolorbox}[title={$(45,4)$}]
(6) 21 11 7 

(7) 22 13 3 $\leftarrow$ (8) 23 15 

(11) 26 5 3 $\leftarrow$ (12) 27 7 

(14) 13 11 7 

(15) 14 13 3 $\leftarrow$ (16) 15 15 

(24) 7 7 7 $\leftarrow$ (38) 5 3 

(27) 10 5 3 $\leftarrow$ (28) 11 7 

(31) 6 5 3 $\leftarrow$ (32) 7 7 

(36) 3 3 3 $\leftarrow$ (40) 3 3 

(42) 1 1 1 $\leftarrow$ (44) 1 1 

\end{tcolorbox}
\begin{tcolorbox}[title={$(45,5)$}]
(1) 13 13 11 7 

(5) 11 15 7 7 

(5) 19 7 7 7 $\leftarrow$ (9) 23 7 7 

(6) 20 5 7 7 

(7) 17 7 7 7 

(7) 21 11 3 3 $\leftarrow$ (8) 22 13 3 

(11) 25 3 3 3 $\leftarrow$ (12) 26 5 3 

(13) 13 5 7 7 

(15) 13 11 3 3 $\leftarrow$ (16) 14 13 3 

(22) 4 5 7 7 $\leftarrow$ (32) 6 5 3 

(23) 3 5 7 7 $\leftarrow$ (27) 5 7 7 

(27) 9 3 3 3 $\leftarrow$ (28) 10 5 3 

(31) 6 2 3 3 $\leftarrow$ (38) 2 3 3 

\end{tcolorbox}
\begin{tcolorbox}[title={$(45,6)$}]
(1) 11 14 5 7 7 $\leftarrow$ (2) 13 13 11 7 

(4) 9 11 7 7 7 

(5) 7 14 5 7 7 $\leftarrow$ (6) 11 15 7 7 

(5) 18 3 5 7 7 

(6) 7 13 5 7 7 $\leftarrow$ (8) 17 7 7 7 

(7) 20 9 3 3 3 $\leftarrow$ (8) 21 11 3 3 

(12) 11 3 5 7 7 

(14) 9 3 5 7 7 $\leftarrow$ (24) 3 5 7 7 

(15) 12 9 3 3 3 $\leftarrow$ (16) 13 11 3 3 

(21) 2 3 5 7 7 $\leftarrow$ (29) 8 3 3 3 

(25) 4 7 3 3 3 $\leftarrow$ (26) 6 6 5 3 

(27) 4 5 3 3 3 $\leftarrow$ (28) 5 7 3 3 

(30) 2 4 3 3 3 $\leftarrow$ (37) 1 2 3 3 

(31) 5 1 2 3 3 $\leftarrow$ (32) 6 2 3 3 

\end{tcolorbox}
\begin{tcolorbox}[title={$(45,7)$}]
(1) 7 14 4 5 7 7 $\leftarrow$ (2) 11 14 5 7 7 

(3) 12 12 9 3 3 3 

(5) 5 5 9 7 7 7 $\leftarrow$ (6) 7 14 5 7 7 

(7) 3 5 9 7 7 7 $\leftarrow$ (9) 14 4 5 7 7 

(7) 8 12 9 3 3 3 

(15) 12 4 5 3 3 3 $\leftarrow$ (28) 4 5 3 3 3 

(19) 3 3 6 6 5 3 $\leftarrow$ (25) 3 5 7 3 3 

(25) 2 4 5 3 3 3 $\leftarrow$ (26) 4 7 3 3 3 

(29) 1 2 4 3 3 3 $\leftarrow$ (35) 4 4 1 1 1 

(31) 3 4 4 1 1 1 $\leftarrow$ (32) 5 1 2 3 3 

(35) 1 * 1 $\leftarrow$ (37) * 1 

\end{tcolorbox}
\begin{tcolorbox}[title={$(45,8)$}]
(3) 10 9 3 6 6 5 3 $\leftarrow$ (4) 12 12 9 3 3 3 

(6) 3 5 9 3 5 7 7 $\leftarrow$ (10) 5 9 3 5 7 7 

(7) 6 9 3 6 6 5 3 $\leftarrow$ (8) 8 12 9 3 3 3 

(8) 3 5 7 3 5 7 7 

(13) 6 3 3 6 6 5 3 $\leftarrow$ (26) 2 4 5 3 3 3 

(15) 10 2 4 5 3 3 3 $\leftarrow$ (16) 12 4 5 3 3 3 

(28) 1 1 2 4 3 3 3 $\leftarrow$ (32) 3 4 4 1 1 1 

(34) 1 1 * 1 $\leftarrow$ (36) 1 * 1 

\end{tcolorbox}
\begin{tcolorbox}[title={$(45,9)$}]
(3) 9 5 5 3 6 6 5 3 $\leftarrow$ (4) 10 9 3 6 6 5 3 

(5) 5 6 5 2 3 5 7 7 

(7) 3 6 5 2 3 5 7 7 $\leftarrow$ (8) 6 9 3 6 6 5 3 

(11) 5 6 2 3 5 7 3 3 $\leftarrow$ (23) 13 1 * 1 

(13) 3 6 2 3 5 7 3 3 $\leftarrow$ (14) 6 3 3 6 6 5 3 

(15) 7 13 1 * 1 $\leftarrow$ (16) 10 2 4 5 3 3 3 

(19) 4 ..4 5 3 3 3 $\leftarrow$ (20) 6 2 4 5 3 3 3 

(23) 6 2 3 4 4 1 1 1 $\leftarrow$ (30) 2 3 4 4 1 1 1 

\end{tcolorbox}
\begin{tcolorbox}[title={$(45,10)$}]
(3) 4 5 5 5 3 6 6 5 3 $\leftarrow$ (6) 5 6 5 2 3 5 7 7 

(7) 2 3 5 5 3 6 6 5 3 $\leftarrow$ (8) 3 6 5 2 3 5 7 7 

(10) 2 4 3 3 3 6 6 5 3 $\leftarrow$ (21) 8 1 1 2 4 3 3 3 

(11) 3 5 6 2 4 5 3 3 3 $\leftarrow$ (12) 5 6 2 3 5 7 3 3 

(13) ...3 3 6 6 5 3 $\leftarrow$ (14) 3 6 2 3 5 7 3 3 

(15) 5 8 1 1 2 4 3 3 3 $\leftarrow$ (16) 7 13 1 * 1 

(17) 4 ...4 5 3 3 3 $\leftarrow$ (18) 6 ..4 5 3 3 3 

(19) ....4 5 3 3 3 $\leftarrow$ (20) 4 ..4 5 3 3 3 

(22) * 2 4 3 3 3 $\leftarrow$ (29) 1 2 3 4 4 1 1 1 

(23) 5 1 2 3 4 4 1 1 1 $\leftarrow$ (24) 6 2 3 4 4 1 1 1 

\end{tcolorbox}
\begin{tcolorbox}[title={$(45,11)$}]
(1) 2 4 5 5 5 3 6 6 5 3 $\leftarrow$ (4) 4 5 5 5 3 6 6 5 3 

(8) ..4 3 3 3 6 6 5 3 $\leftarrow$ (20) ....4 5 3 3 3 

(11) ....3 3 6 6 5 3 $\leftarrow$ (12) 3 5 6 2 4 5 3 3 3 

(15) 4 ....4 5 3 3 3 $\leftarrow$ (16) 5 8 1 1 2 4 3 3 3 

(17) .....4 5 3 3 3 $\leftarrow$ (18) 4 ...4 5 3 3 3 

(21) 1 * 2 4 3 3 3 $\leftarrow$ (27) 4 4 1 1 * 1 

(23) 3 4 4 1 1 * 1 $\leftarrow$ (24) 5 1 2 3 4 4 1 1 1 

(27) 1 * * 1 $\leftarrow$ (29) * * 1 

\end{tcolorbox}
\begin{tcolorbox}[title={$(45,12)$}]
(4) 1 1 2 4 7 3 3 6 6 5 3 

(5) 6 ....3 3 6 6 5 3 $\leftarrow$ (18) .....4 5 3 3 3 

(15) ......4 5 3 3 3 $\leftarrow$ (16) 4 ....4 5 3 3 3 

(20) 1 1 * 2 4 3 3 3 $\leftarrow$ (24) 3 4 4 1 1 * 1 

(26) 1 1 * * 1 $\leftarrow$ (28) 1 * * 1 

\end{tcolorbox}
\begin{tcolorbox}[title={$(45,13)$}]
(3) 5 6 .....3 5 7 3 3 $\leftarrow$ (16) ......4 5 3 3 3 

(5) 3 6 .....3 5 7 3 3 $\leftarrow$ (6) 6 ....3 3 6 6 5 3 

(11) 4 ......4 5 3 3 3 $\leftarrow$ (12) 6 .....4 5 3 3 3 

(15) 6 2 3 4 4 1 1 * 1 $\leftarrow$ (22) 2 3 4 4 1 1 * 1 

\end{tcolorbox}
\begin{tcolorbox}[title={$(45,14)$}]
(2) .....4 3 3 3 6 6 5 3 $\leftarrow$ (13) 8 1 1 * 2 4 3 3 3 

(3) 3 5 6 .....4 5 3 3 3 $\leftarrow$ (4) 5 6 .....3 5 7 3 3 

(5) .......3 3 6 6 5 3 $\leftarrow$ (6) 3 6 .....3 5 7 3 3 

(9) 4 .......4 5 3 3 3 $\leftarrow$ (10) 6 ......4 5 3 3 3 

(11) ........4 5 3 3 3 $\leftarrow$ (12) 4 ......4 5 3 3 3 

(14) * * 2 4 3 3 3 $\leftarrow$ (21) 1 2 3 4 4 1 1 * 1 

(15) 5 1 2 3 4 4 1 1 * 1 $\leftarrow$ (16) 6 2 3 4 4 1 1 * 1 

\end{tcolorbox}
\begin{tcolorbox}[title={$(45,15)$}]
(3) ........3 3 6 6 5 3 $\leftarrow$ (4) 3 5 6 .....4 5 3 3 3 

(6) .........3 5 7 3 3 

(9) .........4 5 3 3 3 $\leftarrow$ (10) 4 .......4 5 3 3 3 

(13) 1 * * 2 4 3 3 3 $\leftarrow$ (19) 4 4 1 1 * * 1 

(15) 3 4 4 1 1 * * 1 $\leftarrow$ (16) 5 1 2 3 4 4 1 1 * 1 

(19) 1 * * * 1 $\leftarrow$ (21) * * * 1 

\end{tcolorbox}
\begin{tcolorbox}[title={$(45,16)$}]
(3) 6 .........4 5 3 3 3 

(12) 1 1 * * 2 4 3 3 3 $\leftarrow$ (16) 3 4 4 1 1 * * 1 

(18) 1 1 * * * 1 $\leftarrow$ (20) 1 * * * 1 

\end{tcolorbox}
\begin{tcolorbox}[title={$(45,17)$}]
(1) 6 ..........4 5 3 3 3 

(3) 4 ..........4 5 3 3 3 $\leftarrow$ (4) 6 .........4 5 3 3 3 

(7) 6 2 3 4 4 1 1 * * 1 $\leftarrow$ (14) 2 3 4 4 1 1 * * 1 

\end{tcolorbox}
\begin{tcolorbox}[title={$(45,18)$}]
(1) 4 ...........4 5 3 3 3 $\leftarrow$ (2) 6 ..........4 5 3 3 3 

(3) ............4 5 3 3 3 $\leftarrow$ (4) 4 ..........4 5 3 3 3 

(6) * * * 2 4 3 3 3 $\leftarrow$ (13) 1 2 3 4 4 1 1 * * 1 

(7) 5 1 2 3 4 4 1 1 * * 1 $\leftarrow$ (8) 6 2 3 4 4 1 1 * * 1 

\end{tcolorbox}
\begin{tcolorbox}[title={$(45,19)$}]
(1) .............4 5 3 3 3 $\leftarrow$ (2) 4 ...........4 5 3 3 3 

(5) 1 * * * 2 4 3 3 3 $\leftarrow$ (11) 4 4 1 1 * * * 1 

(7) 3 4 4 1 1 * * * 1 $\leftarrow$ (8) 5 1 2 3 4 4 1 1 * * 1 

(11) 1 * * * * 1 $\leftarrow$ (13) * * * * 1 

\end{tcolorbox}
\begin{tcolorbox}[title={$(45,20)$}]
(4) 1 1 * * * 2 4 3 3 3 $\leftarrow$ (8) 3 4 4 1 1 * * * 1 

(10) 1 1 * * * * 1 $\leftarrow$ (12) 1 * * * * 1 

\end{tcolorbox}
\begin{tcolorbox}[title={$(45,23)$}]
(3) 1 * * * * * 1 $\leftarrow$ (5) * * * * * 1 

\end{tcolorbox}
\begin{tcolorbox}[title={$(45,24)$}]
(2) 1 1 * * * * * 1 $\leftarrow$ (4) 1 * * * * * 1 

\end{tcolorbox}
\begin{tcolorbox}[title={$(46,2)$}]
(15) 31 

(31) 15 $\leftarrow$ (47) 

\end{tcolorbox}
\begin{tcolorbox}[title={$(46,3)$}]
(14) 29 3 

(15) 30 1 $\leftarrow$ (16) 31 

(30) 13 3 $\leftarrow$ (46) 1 

(31) 14 1 $\leftarrow$ (32) 15 

(39) 6 1 $\leftarrow$ (40) 7 

(43) 2 1 $\leftarrow$ (44) 3 

\end{tcolorbox}
\begin{tcolorbox}[title={$(46,4)$}]
(1) 15 15 15 

(5) 11 15 15 

(7) 21 11 7 $\leftarrow$ (9) 23 15 

(13) 27 3 3 

(15) 13 11 7 $\leftarrow$ (17) 15 15 

(15) 29 1 1 $\leftarrow$ (16) 30 1 

(25) 7 7 7 $\leftarrow$ (33) 7 7 

(29) 11 3 3 $\leftarrow$ (45) 1 1 

(31) 13 1 1 $\leftarrow$ (32) 14 1 

(37) 3 3 3 $\leftarrow$ (41) 3 3 

(39) 5 1 1 $\leftarrow$ (40) 6 1 

(43) 1 1 1 $\leftarrow$ (44) 2 1 

\end{tcolorbox}
\begin{tcolorbox}[title={$(46,5)$}]
(1) 14 13 11 7 $\leftarrow$ (2) 15 15 15 

(5) 10 13 11 7 $\leftarrow$ (6) 11 15 15 

(6) 19 7 7 7 $\leftarrow$ (8) 21 11 7 

(7) 20 5 7 7 $\leftarrow$ (10) 23 7 7 

(12) 25 3 3 3 

(13) 14 5 7 7 $\leftarrow$ (16) 13 11 7 

(14) 13 5 7 7 $\leftarrow$ (28) 5 7 7 

(15) 28 1 1 1 $\leftarrow$ (16) 29 1 1 

(23) 4 5 7 7 $\leftarrow$ (26) 7 7 7 

(28) 9 3 3 3 $\leftarrow$ (44) 1 1 1 

(31) 12 1 1 1 $\leftarrow$ (32) 13 1 1 

(39) 4 1 1 1 $\leftarrow$ (40) 5 1 1 

\end{tcolorbox}
\begin{tcolorbox}[title={$(46,6)$}]
(1) 9 15 7 7 7 

(1) 13 13 5 7 7 

(5) 9 11 7 7 7 $\leftarrow$ (6) 10 13 11 7 

(6) 18 3 5 7 7 $\leftarrow$ (8) 20 5 7 7 

(7) 7 13 5 7 7 $\leftarrow$ (9) 17 7 7 7 

(8) 20 9 3 3 3 

(11) 5 9 7 7 7 

(13) 11 3 5 7 7 $\leftarrow$ (14) 14 5 7 7 

(15) 9 3 5 7 7 $\leftarrow$ (25) 3 5 7 7 

(15) 24 4 1 1 1 $\leftarrow$ (16) 28 1 1 1 

(16) 12 9 3 3 3 $\leftarrow$ (40) 4 1 1 1 

(22) 2 3 5 7 7 $\leftarrow$ (24) 4 5 7 7 

(23) 3 6 6 5 3 $\leftarrow$ (27) 6 6 5 3 

(29) 3 6 2 3 3 $\leftarrow$ (30) 8 3 3 3 

(31) 2 4 3 3 3 $\leftarrow$ (33) 6 2 3 3 

(31) 8 4 1 1 1 $\leftarrow$ (32) 12 1 1 1 

\end{tcolorbox}
\begin{tcolorbox}[title={$(46,7)$}]
(1) 12 11 3 5 7 7 $\leftarrow$ (2) 13 13 5 7 7 

(2) 7 14 4 5 7 7 

(6) 5 5 9 7 7 7 $\leftarrow$ (8) 7 13 5 7 7 

(6) 17 3 6 6 5 3 

(8) 3 5 9 7 7 7 

(9) 13 2 3 5 7 7 $\leftarrow$ (10) 14 4 5 7 7 

(12) 5 7 3 5 7 7 $\leftarrow$ (16) 9 3 5 7 7 

(14) 9 3 6 6 5 3 $\leftarrow$ (38) * 1 

(15) 22 * 1 $\leftarrow$ (16) 24 4 1 1 1 

(20) 3 3 6 6 5 3 $\leftarrow$ (26) 3 5 7 3 3 

(23) 2 3 5 7 3 3 $\leftarrow$ (24) 3 6 6 5 3 

(29) ..4 3 3 3 $\leftarrow$ (30) 3 6 2 3 3 

(30) 1 2 4 3 3 3 $\leftarrow$ (32) 2 4 3 3 3 

(31) 6 * 1 $\leftarrow$ (32) 8 4 1 1 1 

(35) 2 * 1 $\leftarrow$ (36) 4 4 1 1 1 

\end{tcolorbox}
\begin{tcolorbox}[title={$(46,8)$}]
(1) 7 8 12 9 3 3 3 $\leftarrow$ (2) 12 11 3 5 7 7 

(5) 11 12 4 5 3 3 3 

(7) 3 5 9 3 5 7 7 $\leftarrow$ (9) 8 12 9 3 3 3 

(9) 3 5 7 3 5 7 7 

(9) 7 12 4 5 3 3 3 $\leftarrow$ (10) 13 2 3 5 7 7 

(11) 6 5 2 3 5 7 7 $\leftarrow$ (37) 1 * 1 

(13) 5 5 3 6 6 5 3 $\leftarrow$ (17) 12 4 5 3 3 3 

(15) 21 1 * 1 $\leftarrow$ (16) 22 * 1 

(17) 6 2 3 5 7 3 3 $\leftarrow$ (24) 2 3 5 7 3 3 

(29) 1 1 2 4 3 3 3 $\leftarrow$ (30) ..4 3 3 3 

(31) 5 1 * 1 $\leftarrow$ (32) 6 * 1 

(35) 1 1 * 1 $\leftarrow$ (36) 2 * 1 

\end{tcolorbox}
\begin{tcolorbox}[title={$(46,9)$}]
(4) 9 5 5 3 6 6 5 3 

(9) 4 7 3 3 6 6 5 3 $\leftarrow$ (36) 1 1 * 1 

(9) 5 6 3 3 6 6 5 3 $\leftarrow$ (10) 7 12 4 5 3 3 3 

(11) 3 6 3 3 6 6 5 3 $\leftarrow$ (12) 6 5 2 3 5 7 7 

(15) 5 6 2 4 5 3 3 3 $\leftarrow$ (21) 6 2 4 5 3 3 3 

(15) 20 1 1 * 1 $\leftarrow$ (16) 21 1 * 1 

(17) 3 6 2 4 5 3 3 3 $\leftarrow$ (18) 6 2 3 5 7 3 3 

(23) 12 1 1 * 1 $\leftarrow$ (24) 13 1 * 1 

(31) 4 1 1 * 1 $\leftarrow$ (32) 5 1 * 1 

\end{tcolorbox}
\begin{tcolorbox}[title={$(46,10)$}]
(1) 5 5 6 5 2 3 5 7 7 

(7) 2 4 7 3 3 6 6 5 3 $\leftarrow$ (10) 4 7 3 3 6 6 5 3 

(8) 2 3 5 5 3 6 6 5 3 $\leftarrow$ (32) 4 1 1 * 1 

(9) 3 5 6 2 3 5 7 3 3 $\leftarrow$ (10) 5 6 3 3 6 6 5 3 

(11) 2 4 3 3 3 6 6 5 3 $\leftarrow$ (12) 3 6 3 3 6 6 5 3 

(14) ...3 3 6 6 5 3 $\leftarrow$ (19) 6 ..4 5 3 3 3 

(15) 3 6 ..4 5 3 3 3 $\leftarrow$ (16) 5 6 2 4 5 3 3 3 

(15) 16 4 1 1 * 1 $\leftarrow$ (16) 20 1 1 * 1 

(17) ....3 5 7 3 3 $\leftarrow$ (18) 3 6 2 4 5 3 3 3 

(21) 3 6 2 3 4 4 1 1 1 $\leftarrow$ (22) 8 1 1 2 4 3 3 3 

(23) * 2 4 3 3 3 $\leftarrow$ (25) 6 2 3 4 4 1 1 1 

(23) 8 4 1 1 * 1 $\leftarrow$ (24) 12 1 1 * 1 

\end{tcolorbox}
\begin{tcolorbox}[title={$(46,11)$}]
(2) 2 4 5 5 5 3 6 6 5 3 

(5) ..4 7 3 3 6 6 5 3 $\leftarrow$ (30) * * 1 

(6) 1 2 4 7 3 3 6 6 5 3 $\leftarrow$ (8) 2 4 7 3 3 6 6 5 3 

(9) ..4 3 3 3 6 6 5 3 $\leftarrow$ (10) 3 5 6 2 3 5 7 3 3 

(12) ....3 3 6 6 5 3 $\leftarrow$ (18) ....3 5 7 3 3 

(15) .....3 5 7 3 3 $\leftarrow$ (16) 3 6 ..4 5 3 3 3 

(15) 14 * * 1 $\leftarrow$ (16) 16 4 1 1 * 1 

(21) 2 * 2 4 3 3 3 $\leftarrow$ (22) 3 6 2 3 4 4 1 1 1 

(22) 1 * 2 4 3 3 3 $\leftarrow$ (24) * 2 4 3 3 3 

(23) 6 * * 1 $\leftarrow$ (24) 8 4 1 1 * 1 

(27) 2 * * 1 $\leftarrow$ (28) 4 4 1 1 * 1 

\end{tcolorbox}
\begin{tcolorbox}[title={$(46,12)$}]
(3) 6 ..4 3 3 3 6 6 5 3 $\leftarrow$ (29) 1 * * 1 

(5) 1 1 2 4 7 3 3 6 6 5 3 $\leftarrow$ (6) ..4 7 3 3 6 6 5 3 

(9) 6 .....3 5 7 3 3 $\leftarrow$ (16) .....3 5 7 3 3 

(15) 13 1 * * 1 $\leftarrow$ (16) 14 * * 1 

(21) 1 1 * 2 4 3 3 3 $\leftarrow$ (22) 2 * 2 4 3 3 3 

(23) 5 1 * * 1 $\leftarrow$ (24) 6 * * 1 

(27) 1 1 * * 1 $\leftarrow$ (28) 2 * * 1 

\end{tcolorbox}
\begin{tcolorbox}[title={$(46,13)$}]
(1) 4 1 1 2 4 7 3 3 6 6 5 3 $\leftarrow$ (28) 1 1 * * 1 

(3) 3 6 ....3 3 6 6 5 3 $\leftarrow$ (4) 6 ..4 3 3 3 6 6 5 3 

(7) 5 6 .....4 5 3 3 3 $\leftarrow$ (13) 6 .....4 5 3 3 3 

(9) 3 6 .....4 5 3 3 3 $\leftarrow$ (10) 6 .....3 5 7 3 3 

(15) 12 1 1 * * 1 $\leftarrow$ (16) 13 1 * * 1 

(23) 4 1 1 * * 1 $\leftarrow$ (24) 5 1 * * 1 

\end{tcolorbox}
\begin{tcolorbox}[title={$(46,14)$}]
(3) .....4 3 3 3 6 6 5 3 $\leftarrow$ (4) 3 6 ....3 3 6 6 5 3 

(6) .......3 3 6 6 5 3 $\leftarrow$ (11) 6 ......4 5 3 3 3 

(7) 3 6 ......4 5 3 3 3 $\leftarrow$ (8) 5 6 .....4 5 3 3 3 

(9) ........3 5 7 3 3 $\leftarrow$ (10) 3 6 .....4 5 3 3 3 

(12) ........4 5 3 3 3 

(13) 3 6 2 3 4 4 1 1 * 1 $\leftarrow$ (14) 8 1 1 * 2 4 3 3 3 

(15) * * 2 4 3 3 3 $\leftarrow$ (17) 6 2 3 4 4 1 1 * 1 

(15) 8 4 1 1 * * 1 $\leftarrow$ (16) 12 1 1 * * 1 

\end{tcolorbox}
\begin{tcolorbox}[title={$(46,15)$}]
(4) ........3 3 6 6 5 3 $\leftarrow$ (10) ........3 5 7 3 3 

(7) .........3 5 7 3 3 $\leftarrow$ (8) 3 6 ......4 5 3 3 3 

(10) .........4 5 3 3 3 

(13) 2 * * 2 4 3 3 3 $\leftarrow$ (14) 3 6 2 3 4 4 1 1 * 1 

(14) 1 * * 2 4 3 3 3 $\leftarrow$ (16) * * 2 4 3 3 3 

(15) 6 * * * 1 $\leftarrow$ (16) 8 4 1 1 * * 1 

(19) 2 * * * 1 $\leftarrow$ (20) 4 4 1 1 * * 1 

\end{tcolorbox}
\begin{tcolorbox}[title={$(46,16)$}]
(1) 6 .........3 5 7 3 3 $\leftarrow$ (8) .........3 5 7 3 3 

(7) 13 1 * * * 1 

(13) 1 1 * * 2 4 3 3 3 $\leftarrow$ (14) 2 * * 2 4 3 3 3 

(15) 5 1 * * * 1 $\leftarrow$ (16) 6 * * * 1 

(19) 1 1 * * * 1 $\leftarrow$ (20) 2 * * * 1 

\end{tcolorbox}
\begin{tcolorbox}[title={$(46,17)$}]
(1) 3 6 .........4 5 3 3 3 $\leftarrow$ (2) 6 .........3 5 7 3 3 

(5) 8 1 1 * * 2 4 3 3 3 

(7) 12 1 1 * * * 1 $\leftarrow$ (8) 13 1 * * * 1 

(15) 4 1 1 * * * 1 $\leftarrow$ (16) 5 1 * * * 1 

\end{tcolorbox}
\begin{tcolorbox}[title={$(46,18)$}]
(1) ............3 5 7 3 3 $\leftarrow$ (2) 3 6 .........4 5 3 3 3 

(4) ............4 5 3 3 3 

(5) 3 6 2 3 4 4 1 1 * * 1 $\leftarrow$ (6) 8 1 1 * * 2 4 3 3 3 

(7) * * * 2 4 3 3 3 $\leftarrow$ (9) 6 2 3 4 4 1 1 * * 1 

(7) 8 4 1 1 * * * 1 $\leftarrow$ (8) 12 1 1 * * * 1 

\end{tcolorbox}
\begin{tcolorbox}[title={$(46,19)$}]
(2) .............4 5 3 3 3 

(5) 2 * * * 2 4 3 3 3 $\leftarrow$ (6) 3 6 2 3 4 4 1 1 * * 1 

(6) 1 * * * 2 4 3 3 3 $\leftarrow$ (8) * * * 2 4 3 3 3 

(7) 6 * * * * 1 $\leftarrow$ (8) 8 4 1 1 * * * 1 

(11) 2 * * * * 1 $\leftarrow$ (12) 4 4 1 1 * * * 1 

\end{tcolorbox}
\begin{tcolorbox}[title={$(46,20)$}]
(5) 1 1 * * * 2 4 3 3 3 $\leftarrow$ (6) 2 * * * 2 4 3 3 3 

(6) 2 3 4 4 1 1 * * * 1 

(7) 5 1 * * * * 1 $\leftarrow$ (8) 6 * * * * 1 

(11) 1 1 * * * * 1 $\leftarrow$ (12) 2 * * * * 1 

\end{tcolorbox}
\begin{tcolorbox}[title={$(46,21)$}]
(5) 1 2 3 4 4 1 1 * * * 1 

(7) 4 1 1 * * * * 1 $\leftarrow$ (8) 5 1 * * * * 1 

\end{tcolorbox}
\begin{tcolorbox}[title={$(46,22)$}]
(3) 4 4 1 1 * * * * 1 

\end{tcolorbox}
\begin{tcolorbox}[title={$(46,23)$}]
(3) 2 * * * * * 1 $\leftarrow$ (4) 4 4 1 1 * * * * 1 

\end{tcolorbox}
\begin{tcolorbox}[title={$(46,24)$}]
(3) 1 1 * * * * * 1 $\leftarrow$ (4) 2 * * * * * 1 

\end{tcolorbox}
\begin{tcolorbox}[title={$(47,3)$}]
(13) 27 7 

(15) 29 3 $\leftarrow$ (17) 31 

(29) 11 7 $\leftarrow$ (45) 3 

(31) 13 3 $\leftarrow$ (33) 15 

(39) 5 3 $\leftarrow$ (41) 7 

\end{tcolorbox}
\begin{tcolorbox}[title={$(47,4)$}]
(9) 22 13 3 $\leftarrow$ (10) 23 15 

(13) 26 5 3 $\leftarrow$ (14) 27 7 

(14) 27 3 3 $\leftarrow$ (16) 29 3 

(17) 14 13 3 $\leftarrow$ (18) 15 15 

(29) 10 5 3 $\leftarrow$ (30) 11 7 

(30) 11 3 3 $\leftarrow$ (32) 13 3 

(33) 6 5 3 $\leftarrow$ (34) 7 7 

(38) 3 3 3 $\leftarrow$ (40) 5 3 

(39) 2 3 3 $\leftarrow$ (42) 3 3 

\end{tcolorbox}
\begin{tcolorbox}[title={$(47,5)$}]
(2) 14 13 11 7 

(3) 13 13 11 7 

(7) 11 15 7 7 

(7) 19 7 7 7 $\leftarrow$ (9) 21 11 7 

(9) 21 11 3 3 $\leftarrow$ (10) 22 13 3 

(13) 25 3 3 3 $\leftarrow$ (14) 26 5 3 

(15) 13 5 7 7 $\leftarrow$ (29) 5 7 7 

(17) 13 11 3 3 $\leftarrow$ (18) 14 13 3 

(29) 5 7 3 3 $\leftarrow$ (34) 6 5 3 

(29) 9 3 3 3 $\leftarrow$ (30) 10 5 3 

(38) 1 2 3 3 $\leftarrow$ (40) 2 3 3 

\end{tcolorbox}
\begin{tcolorbox}[title={$(47,6)$}]
(2) 9 15 7 7 7 

(3) 11 14 5 7 7 $\leftarrow$ (4) 13 13 11 7 

(6) 9 11 7 7 7 

(7) 7 14 5 7 7 $\leftarrow$ (8) 11 15 7 7 

(7) 18 3 5 7 7 $\leftarrow$ (9) 20 5 7 7 

(9) 20 9 3 3 3 $\leftarrow$ (10) 21 11 3 3 

(12) 5 9 7 7 7 $\leftarrow$ (26) 3 5 7 7 

(14) 11 3 5 7 7 $\leftarrow$ (16) 13 5 7 7 

(17) 12 9 3 3 3 $\leftarrow$ (18) 13 11 3 3 

(23) 2 3 5 7 7 $\leftarrow$ (25) 4 5 7 7 

(27) 4 7 3 3 3 $\leftarrow$ (28) 6 6 5 3 

(29) 4 5 3 3 3 $\leftarrow$ (30) 5 7 3 3 

(33) 5 1 2 3 3 $\leftarrow$ (34) 6 2 3 3 

\end{tcolorbox}
\begin{tcolorbox}[title={$(47,7)$}]
(1) 11 5 9 7 7 7 

(3) 7 14 4 5 7 7 $\leftarrow$ (4) 11 14 5 7 7 

(5) 12 12 9 3 3 3 $\leftarrow$ (8) 18 3 5 7 7 

(7) 5 5 9 7 7 7 $\leftarrow$ (8) 7 14 5 7 7 

(7) 17 3 6 6 5 3 $\leftarrow$ (10) 20 9 3 3 3 

(9) 3 5 9 7 7 7 

(11) 5 9 3 5 7 7 $\leftarrow$ (25) 3 6 6 5 3 

(13) 5 7 3 5 7 7 $\leftarrow$ (17) 9 3 5 7 7 

(15) 9 3 6 6 5 3 $\leftarrow$ (18) 12 9 3 3 3 

(21) 3 3 6 6 5 3 $\leftarrow$ (24) 2 3 5 7 7 

(27) 2 4 5 3 3 3 $\leftarrow$ (28) 4 7 3 3 3 

(31) 1 2 4 3 3 3 $\leftarrow$ (33) 2 4 3 3 3 

(33) 3 4 4 1 1 1 $\leftarrow$ (34) 5 1 2 3 3 

\end{tcolorbox}
\begin{tcolorbox}[title={$(47,8)$}]
(1) 8 3 5 9 7 7 7 

(2) 7 8 12 9 3 3 3 

(5) 10 9 3 6 6 5 3 $\leftarrow$ (6) 12 12 9 3 3 3 

(6) 11 12 4 5 3 3 3 $\leftarrow$ (8) 17 3 6 6 5 3 

(8) 3 5 9 3 5 7 7 

(9) 6 9 3 6 6 5 3 $\leftarrow$ (10) 8 12 9 3 3 3 

(10) 3 5 7 3 5 7 7 $\leftarrow$ (12) 5 9 3 5 7 7 

(14) 5 5 3 6 6 5 3 $\leftarrow$ (16) 9 3 6 6 5 3 

(15) 6 3 3 6 6 5 3 $\leftarrow$ (22) 3 3 6 6 5 3 

(17) 10 2 4 5 3 3 3 $\leftarrow$ (18) 12 4 5 3 3 3 

(30) 1 1 2 4 3 3 3 $\leftarrow$ (32) 1 2 4 3 3 3 

(31) 2 3 4 4 1 1 1 $\leftarrow$ (34) 3 4 4 1 1 1 

\end{tcolorbox}
\begin{tcolorbox}[title={$(47,9)$}]
(1) 9 3 5 7 3 5 7 7 

(5) 9 5 5 3 6 6 5 3 $\leftarrow$ (6) 10 9 3 6 6 5 3 

(7) 5 6 5 2 3 5 7 7 $\leftarrow$ (13) 6 5 2 3 5 7 7 

(9) 3 6 5 2 3 5 7 7 $\leftarrow$ (10) 6 9 3 6 6 5 3 

(13) 5 6 2 3 5 7 3 3 $\leftarrow$ (19) 6 2 3 5 7 3 3 

(15) 3 6 2 3 5 7 3 3 $\leftarrow$ (16) 6 3 3 6 6 5 3 

(17) 7 13 1 * 1 $\leftarrow$ (18) 10 2 4 5 3 3 3 

(21) 4 ..4 5 3 3 3 $\leftarrow$ (22) 6 2 4 5 3 3 3 

(30) 1 2 3 4 4 1 1 1 $\leftarrow$ (32) 2 3 4 4 1 1 1 

\end{tcolorbox}
\begin{tcolorbox}[title={$(47,10)$}]
(2) 5 5 6 5 2 3 5 7 7 

(5) 4 5 5 5 3 6 6 5 3 $\leftarrow$ (8) 5 6 5 2 3 5 7 7 

(9) 2 3 5 5 3 6 6 5 3 $\leftarrow$ (10) 3 6 5 2 3 5 7 7 

(12) 2 4 3 3 3 6 6 5 3 $\leftarrow$ (17) 5 6 2 4 5 3 3 3 

(13) 3 5 6 2 4 5 3 3 3 $\leftarrow$ (14) 5 6 2 3 5 7 3 3 

(15) ...3 3 6 6 5 3 $\leftarrow$ (16) 3 6 2 3 5 7 3 3 

(17) 5 8 1 1 2 4 3 3 3 $\leftarrow$ (18) 7 13 1 * 1 

(19) 4 ...4 5 3 3 3 $\leftarrow$ (20) 6 ..4 5 3 3 3 

(21) ....4 5 3 3 3 $\leftarrow$ (22) 4 ..4 5 3 3 3 

(25) 5 1 2 3 4 4 1 1 1 $\leftarrow$ (26) 6 2 3 4 4 1 1 1 

\end{tcolorbox}
\begin{tcolorbox}[title={$(47,11)$}]
(3) 2 4 5 5 5 3 6 6 5 3 $\leftarrow$ (6) 4 5 5 5 3 6 6 5 3 

(7) 1 2 4 7 3 3 6 6 5 3 $\leftarrow$ (9) 2 4 7 3 3 6 6 5 3 

(10) ..4 3 3 3 6 6 5 3 $\leftarrow$ (16) ...3 3 6 6 5 3 

(13) ....3 3 6 6 5 3 $\leftarrow$ (14) 3 5 6 2 4 5 3 3 3 

(17) 4 ....4 5 3 3 3 $\leftarrow$ (18) 5 8 1 1 2 4 3 3 3 

(19) .....4 5 3 3 3 $\leftarrow$ (20) 4 ...4 5 3 3 3 

(23) 1 * 2 4 3 3 3 $\leftarrow$ (25) * 2 4 3 3 3 

(25) 3 4 4 1 1 * 1 $\leftarrow$ (26) 5 1 2 3 4 4 1 1 1 

\end{tcolorbox}
\begin{tcolorbox}[title={$(47,12)$}]
(1) ..4 5 5 5 3 6 6 5 3 $\leftarrow$ (4) 2 4 5 5 5 3 6 6 5 3 

(6) 1 1 2 4 7 3 3 6 6 5 3 $\leftarrow$ (8) 1 2 4 7 3 3 6 6 5 3 

(7) 6 ....3 3 6 6 5 3 $\leftarrow$ (14) ....3 3 6 6 5 3 

(17) ......4 5 3 3 3 $\leftarrow$ (18) 4 ....4 5 3 3 3 

(22) 1 1 * 2 4 3 3 3 $\leftarrow$ (24) 1 * 2 4 3 3 3 

(23) 2 3 4 4 1 1 * 1 $\leftarrow$ (26) 3 4 4 1 1 * 1 

\end{tcolorbox}
\begin{tcolorbox}[title={$(47,13)$}]
(2) 4 1 1 2 4 7 3 3 6 6 5 3 

(5) 5 6 .....3 5 7 3 3 $\leftarrow$ (11) 6 .....3 5 7 3 3 

(7) 3 6 .....3 5 7 3 3 $\leftarrow$ (8) 6 ....3 3 6 6 5 3 

(13) 4 ......4 5 3 3 3 $\leftarrow$ (14) 6 .....4 5 3 3 3 

(22) 1 2 3 4 4 1 1 * 1 $\leftarrow$ (24) 2 3 4 4 1 1 * 1 

(24) 4 1 1 * * 1 

\end{tcolorbox}
\begin{tcolorbox}[title={$(47,14)$}]
(4) .....4 3 3 3 6 6 5 3 $\leftarrow$ (9) 5 6 .....4 5 3 3 3 

(5) 3 5 6 .....4 5 3 3 3 $\leftarrow$ (6) 5 6 .....3 5 7 3 3 

(7) .......3 3 6 6 5 3 $\leftarrow$ (8) 3 6 .....3 5 7 3 3 

(11) 4 .......4 5 3 3 3 $\leftarrow$ (12) 6 ......4 5 3 3 3 

(13) ........4 5 3 3 3 $\leftarrow$ (14) 4 ......4 5 3 3 3 

(17) 5 1 2 3 4 4 1 1 * 1 $\leftarrow$ (18) 6 2 3 4 4 1 1 * 1 

(22) * * * 1 

\end{tcolorbox}
\begin{tcolorbox}[title={$(47,15)$}]
(1) 1.........4 5 3 3 3 $\leftarrow$ (8) .......3 3 6 6 5 3 

(5) ........3 3 6 6 5 3 $\leftarrow$ (6) 3 5 6 .....4 5 3 3 3 

(11) .........4 5 3 3 3 $\leftarrow$ (12) 4 .......4 5 3 3 3 

(15) 1 * * 2 4 3 3 3 $\leftarrow$ (17) * * 2 4 3 3 3 

(17) 3 4 4 1 1 * * 1 $\leftarrow$ (18) 5 1 2 3 4 4 1 1 * 1 

(21) 1 * * * 1 

\end{tcolorbox}
\begin{tcolorbox}[title={$(47,16)$}]
(1) 10 .........4 5 3 3 3 $\leftarrow$ (2) 1.........4 5 3 3 3 

(5) 6 .........4 5 3 3 3 

(14) 1 1 * * 2 4 3 3 3 $\leftarrow$ (16) 1 * * 2 4 3 3 3 

(15) 2 3 4 4 1 1 * * 1 $\leftarrow$ (18) 3 4 4 1 1 * * 1 

(20) 1 1 * * * 1 

\end{tcolorbox}
\begin{tcolorbox}[title={$(47,17)$}]
(1) 7 13 1 * * * 1 $\leftarrow$ (2) 10 .........4 5 3 3 3 

(3) 6 ..........4 5 3 3 3 

(5) 4 ..........4 5 3 3 3 $\leftarrow$ (6) 6 .........4 5 3 3 3 

(14) 1 2 3 4 4 1 1 * * 1 $\leftarrow$ (16) 2 3 4 4 1 1 * * 1 

(16) 4 1 1 * * * 1 

\end{tcolorbox}
\begin{tcolorbox}[title={$(47,18)$}]
(1) 5 8 1 1 * * 2 4 3 3 3 $\leftarrow$ (2) 7 13 1 * * * 1 

(2) ............3 5 7 3 3 

(3) 4 ...........4 5 3 3 3 $\leftarrow$ (4) 6 ..........4 5 3 3 3 

(5) ............4 5 3 3 3 $\leftarrow$ (6) 4 ..........4 5 3 3 3 

(9) 5 1 2 3 4 4 1 1 * * 1 $\leftarrow$ (10) 6 2 3 4 4 1 1 * * 1 

(14) * * * * 1 

\end{tcolorbox}
\begin{tcolorbox}[title={$(47,19)$}]
(1) 4 ............4 5 3 3 3 $\leftarrow$ (2) 5 8 1 1 * * 2 4 3 3 3 

(3) .............4 5 3 3 3 $\leftarrow$ (4) 4 ...........4 5 3 3 3 

(7) 1 * * * 2 4 3 3 3 $\leftarrow$ (9) * * * 2 4 3 3 3 

(9) 3 4 4 1 1 * * * 1 $\leftarrow$ (10) 5 1 2 3 4 4 1 1 * * 1 

(13) 1 * * * * 1 

\end{tcolorbox}
\begin{tcolorbox}[title={$(47,20)$}]
(1) ..............4 5 3 3 3 $\leftarrow$ (2) 4 ............4 5 3 3 3 

(6) 1 1 * * * 2 4 3 3 3 $\leftarrow$ (8) 1 * * * 2 4 3 3 3 

(7) 2 3 4 4 1 1 * * * 1 $\leftarrow$ (10) 3 4 4 1 1 * * * 1 

(12) 1 1 * * * * 1 

\end{tcolorbox}
\begin{tcolorbox}[title={$(47,21)$}]
(1) 6 2 3 4 4 1 1 * * * 1 

(6) 1 2 3 4 4 1 1 * * * 1 $\leftarrow$ (8) 2 3 4 4 1 1 * * * 1 

(8) 4 1 1 * * * * 1 

\end{tcolorbox}
\begin{tcolorbox}[title={$(47,22)$}]
(1) 5 1 2 3 4 4 1 1 * * * 1 $\leftarrow$ (2) 6 2 3 4 4 1 1 * * * 1 

(6) * * * * * 1 

\end{tcolorbox}
\begin{tcolorbox}[title={$(47,23)$}]
(1) 3 4 4 1 1 * * * * 1 $\leftarrow$ (2) 5 1 2 3 4 4 1 1 * * * 1 

(5) 1 * * * * * 1 

\end{tcolorbox}
\begin{tcolorbox}[title={$(47,24)$}]
(4) 1 1 * * * * * 1 

\end{tcolorbox}
\begin{tcolorbox}[title={$(48,2)$}]
(47) 1 $\leftarrow$ (49) 

\end{tcolorbox}
\begin{tcolorbox}[title={$(48,3)$}]
(17) 30 1 $\leftarrow$ (18) 31 

(33) 14 1 $\leftarrow$ (34) 15 

(41) 6 1 $\leftarrow$ (42) 7 

(45) 2 1 $\leftarrow$ (46) 3 

(46) 1 1 $\leftarrow$ (48) 1 

\end{tcolorbox}
\begin{tcolorbox}[title={$(48,4)$}]
(3) 15 15 15 

(7) 11 15 15 

(11) 23 7 7 

(15) 27 3 3 $\leftarrow$ (17) 29 3 

(17) 13 11 7 $\leftarrow$ (43) 3 3 

(17) 29 1 1 $\leftarrow$ (18) 30 1 

(27) 7 7 7 $\leftarrow$ (35) 7 7 

(31) 11 3 3 $\leftarrow$ (33) 13 3 

(33) 13 1 1 $\leftarrow$ (34) 14 1 

(39) 3 3 3 $\leftarrow$ (41) 5 3 

(41) 5 1 1 $\leftarrow$ (42) 6 1 

(45) 1 1 1 $\leftarrow$ (46) 2 1 

\end{tcolorbox}
\begin{tcolorbox}[title={$(48,5)$}]
(3) 14 13 11 7 $\leftarrow$ (4) 15 15 15 

(7) 10 13 11 7 $\leftarrow$ (8) 11 15 15 

(8) 19 7 7 7 

(10) 17 7 7 7 

(14) 25 3 3 3 $\leftarrow$ (16) 27 3 3 

(15) 14 5 7 7 $\leftarrow$ (30) 5 7 7 

(17) 28 1 1 1 $\leftarrow$ (18) 29 1 1 

(30) 9 3 3 3 $\leftarrow$ (32) 11 3 3 

(31) 8 3 3 3 $\leftarrow$ (40) 3 3 3 

(33) 12 1 1 1 $\leftarrow$ (34) 13 1 1 

(39) 1 2 3 3 $\leftarrow$ (41) 2 3 3 

(41) 4 1 1 1 $\leftarrow$ (42) 5 1 1 

\end{tcolorbox}
\begin{tcolorbox}[title={$(48,6)$}]
(1) 7 11 15 7 7 

(3) 9 15 7 7 7 $\leftarrow$ (5) 13 13 11 7 

(3) 13 13 5 7 7 

(7) 9 11 7 7 7 $\leftarrow$ (8) 10 13 11 7 

(9) 7 13 5 7 7 

(11) 14 4 5 7 7 $\leftarrow$ (27) 3 5 7 7 

(13) 5 9 7 7 7 $\leftarrow$ (17) 13 5 7 7 

(15) 11 3 5 7 7 $\leftarrow$ (16) 14 5 7 7 

(17) 24 4 1 1 1 $\leftarrow$ (18) 28 1 1 1 

(27) 3 5 7 3 3 $\leftarrow$ (29) 6 6 5 3 

(30) 4 5 3 3 3 $\leftarrow$ (35) 6 2 3 3 

(31) 3 6 2 3 3 $\leftarrow$ (32) 8 3 3 3 

(33) 8 4 1 1 1 $\leftarrow$ (34) 12 1 1 1 

(37) 4 4 1 1 1 $\leftarrow$ (40) 1 2 3 3 

(39) * 1 $\leftarrow$ (42) 4 1 1 1 

\end{tcolorbox}
\begin{tcolorbox}[title={$(48,7)$}]
(1) 6 9 11 7 7 7 $\leftarrow$ (2) 7 11 15 7 7 

(2) 11 5 9 7 7 7 

(3) 12 11 3 5 7 7 $\leftarrow$ (4) 13 13 5 7 7 

(4) 7 14 4 5 7 7 

(8) 5 5 9 7 7 7 

(10) 3 5 9 7 7 7 $\leftarrow$ (16) 11 3 5 7 7 

(11) 13 2 3 5 7 7 $\leftarrow$ (12) 14 4 5 7 7 

(14) 5 7 3 5 7 7 $\leftarrow$ (18) 9 3 5 7 7 

(17) 22 * 1 $\leftarrow$ (18) 24 4 1 1 1 

(25) 2 3 5 7 3 3 $\leftarrow$ (26) 3 6 6 5 3 

(28) 2 4 5 3 3 3 $\leftarrow$ (34) 2 4 3 3 3 

(31) ..4 3 3 3 $\leftarrow$ (32) 3 6 2 3 3 

(33) 6 * 1 $\leftarrow$ (34) 8 4 1 1 1 

(37) 2 * 1 $\leftarrow$ (38) 4 4 1 1 1 

(38) 1 * 1 $\leftarrow$ (40) * 1 

\end{tcolorbox}
\begin{tcolorbox}[title={$(48,8)$}]
(1) 9 3 5 9 7 7 7 

(2) 8 3 5 9 7 7 7 

(3) 7 8 12 9 3 3 3 $\leftarrow$ (4) 12 11 3 5 7 7 

(7) 11 12 4 5 3 3 3 $\leftarrow$ (9) 17 3 6 6 5 3 

(9) 3 5 9 3 5 7 7 $\leftarrow$ (19) 12 4 5 3 3 3 

(11) 3 5 7 3 5 7 7 $\leftarrow$ (13) 5 9 3 5 7 7 

(11) 7 12 4 5 3 3 3 $\leftarrow$ (12) 13 2 3 5 7 7 

(15) 5 5 3 6 6 5 3 $\leftarrow$ (17) 9 3 6 6 5 3 

(17) 21 1 * 1 $\leftarrow$ (18) 22 * 1 

(25) 13 1 * 1 $\leftarrow$ (33) 1 2 4 3 3 3 

(31) 1 1 2 4 3 3 3 $\leftarrow$ (32) ..4 3 3 3 

(33) 5 1 * 1 $\leftarrow$ (34) 6 * 1 

(37) 1 1 * 1 $\leftarrow$ (38) 2 * 1 

\end{tcolorbox}
\begin{tcolorbox}[title={$(48,9)$}]
(1) 4 5 3 5 9 7 7 7 

(1) 8 3 5 9 3 5 7 7 

(2) 9 3 5 7 3 5 7 7 

(6) 9 5 5 3 6 6 5 3 $\leftarrow$ (8) 11 12 4 5 3 3 3 

(11) 4 7 3 3 6 6 5 3 $\leftarrow$ (16) 5 5 3 6 6 5 3 

(11) 5 6 3 3 6 6 5 3 $\leftarrow$ (12) 7 12 4 5 3 3 3 

(13) 3 6 3 3 6 6 5 3 $\leftarrow$ (14) 6 5 2 3 5 7 7 

(17) 20 1 1 * 1 $\leftarrow$ (18) 21 1 * 1 

(19) 3 6 2 4 5 3 3 3 $\leftarrow$ (20) 6 2 3 5 7 3 3 

(23) 8 1 1 2 4 3 3 3 $\leftarrow$ (32) 1 1 2 4 3 3 3 

(25) 12 1 1 * 1 $\leftarrow$ (26) 13 1 * 1 

(31) 1 2 3 4 4 1 1 1 $\leftarrow$ (33) 2 3 4 4 1 1 1 

(33) 4 1 1 * 1 $\leftarrow$ (34) 5 1 * 1 

\end{tcolorbox}
\begin{tcolorbox}[title={$(48,10)$}]
(3) 5 5 6 5 2 3 5 7 7 $\leftarrow$ (9) 5 6 5 2 3 5 7 7 

(10) 2 3 5 5 3 6 6 5 3 $\leftarrow$ (12) 4 7 3 3 6 6 5 3 

(11) 3 5 6 2 3 5 7 3 3 $\leftarrow$ (12) 5 6 3 3 6 6 5 3 

(13) 2 4 3 3 3 6 6 5 3 $\leftarrow$ (14) 3 6 3 3 6 6 5 3 

(17) 3 6 ..4 5 3 3 3 $\leftarrow$ (18) 5 6 2 4 5 3 3 3 

(17) 16 4 1 1 * 1 $\leftarrow$ (18) 20 1 1 * 1 

(19) ....3 5 7 3 3 $\leftarrow$ (20) 3 6 2 4 5 3 3 3 

(22) ....4 5 3 3 3 $\leftarrow$ (27) 6 2 3 4 4 1 1 1 

(23) 3 6 2 3 4 4 1 1 1 $\leftarrow$ (24) 8 1 1 2 4 3 3 3 

(25) 8 4 1 1 * 1 $\leftarrow$ (26) 12 1 1 * 1 

(29) 4 4 1 1 * 1 $\leftarrow$ (32) 1 2 3 4 4 1 1 1 

(31) * * 1 $\leftarrow$ (34) 4 1 1 * 1 

\end{tcolorbox}
\begin{tcolorbox}[title={$(48,11)$}]
(1) 4 3 5 6 5 2 3 5 7 7 $\leftarrow$ (4) 5 5 6 5 2 3 5 7 7 

(7) ..4 7 3 3 6 6 5 3 $\leftarrow$ (10) 2 4 7 3 3 6 6 5 3 

(11) ..4 3 3 3 6 6 5 3 $\leftarrow$ (12) 3 5 6 2 3 5 7 3 3 

(17) .....3 5 7 3 3 $\leftarrow$ (18) 3 6 ..4 5 3 3 3 

(17) 14 * * 1 $\leftarrow$ (18) 16 4 1 1 * 1 

(20) .....4 5 3 3 3 $\leftarrow$ (26) * 2 4 3 3 3 

(23) 2 * 2 4 3 3 3 $\leftarrow$ (24) 3 6 2 3 4 4 1 1 1 

(25) 6 * * 1 $\leftarrow$ (26) 8 4 1 1 * 1 

(29) 2 * * 1 $\leftarrow$ (30) 4 4 1 1 * 1 

(30) 1 * * 1 $\leftarrow$ (32) * * 1 

\end{tcolorbox}
\begin{tcolorbox}[title={$(48,12)$}]
(2) ..4 5 5 5 3 6 6 5 3 

(5) 6 ..4 3 3 3 6 6 5 3 $\leftarrow$ (9) 1 2 4 7 3 3 6 6 5 3 

(7) 1 1 2 4 7 3 3 6 6 5 3 $\leftarrow$ (8) ..4 7 3 3 6 6 5 3 

(17) 13 1 * * 1 $\leftarrow$ (18) 14 * * 1 

(18) ......4 5 3 3 3 $\leftarrow$ (25) 1 * 2 4 3 3 3 

(23) 1 1 * 2 4 3 3 3 $\leftarrow$ (24) 2 * 2 4 3 3 3 

(25) 5 1 * * 1 $\leftarrow$ (26) 6 * * 1 

(29) 1 1 * * 1 $\leftarrow$ (30) 2 * * 1 

\end{tcolorbox}
\begin{tcolorbox}[title={$(48,13)$}]
(3) 4 1 1 2 4 7 3 3 6 6 5 3 $\leftarrow$ (8) 1 1 2 4 7 3 3 6 6 5 3 

(5) 3 6 ....3 3 6 6 5 3 $\leftarrow$ (6) 6 ..4 3 3 3 6 6 5 3 

(11) 3 6 .....4 5 3 3 3 $\leftarrow$ (12) 6 .....3 5 7 3 3 

(15) 8 1 1 * 2 4 3 3 3 $\leftarrow$ (24) 1 1 * 2 4 3 3 3 

(17) 12 1 1 * * 1 $\leftarrow$ (18) 13 1 * * 1 

(23) 1 2 3 4 4 1 1 * 1 $\leftarrow$ (25) 2 3 4 4 1 1 * 1 

(25) 4 1 1 * * 1 $\leftarrow$ (26) 5 1 * * 1 

\end{tcolorbox}
\begin{tcolorbox}[title={$(48,14)$}]
(1) * 2 4 7 3 3 6 6 5 3 

(1) 24 4 1 1 * * 1 $\leftarrow$ (4) 4 1 1 2 4 7 3 3 6 6 5 3 

(5) .....4 3 3 3 6 6 5 3 $\leftarrow$ (6) 3 6 ....3 3 6 6 5 3 

(9) 3 6 ......4 5 3 3 3 $\leftarrow$ (10) 5 6 .....4 5 3 3 3 

(11) ........3 5 7 3 3 $\leftarrow$ (12) 3 6 .....4 5 3 3 3 

(14) ........4 5 3 3 3 $\leftarrow$ (19) 6 2 3 4 4 1 1 * 1 

(15) 3 6 2 3 4 4 1 1 * 1 $\leftarrow$ (16) 8 1 1 * 2 4 3 3 3 

(17) 8 4 1 1 * * 1 $\leftarrow$ (18) 12 1 1 * * 1 

(21) 4 4 1 1 * * 1 $\leftarrow$ (24) 1 2 3 4 4 1 1 * 1 

(23) * * * 1 $\leftarrow$ (26) 4 1 1 * * 1 

\end{tcolorbox}
\begin{tcolorbox}[title={$(48,15)$}]
(1) 22 * * * 1 $\leftarrow$ (2) 24 4 1 1 * * 1 

(6) ........3 3 6 6 5 3 

(9) .........3 5 7 3 3 $\leftarrow$ (10) 3 6 ......4 5 3 3 3 

(12) .........4 5 3 3 3 $\leftarrow$ (18) * * 2 4 3 3 3 

(15) 2 * * 2 4 3 3 3 $\leftarrow$ (16) 3 6 2 3 4 4 1 1 * 1 

(17) 6 * * * 1 $\leftarrow$ (18) 8 4 1 1 * * 1 

(21) 2 * * * 1 $\leftarrow$ (22) 4 4 1 1 * * 1 

(22) 1 * * * 1 $\leftarrow$ (24) * * * 1 

\end{tcolorbox}
\begin{tcolorbox}[title={$(48,16)$}]
(1) 21 1 * * * 1 $\leftarrow$ (2) 22 * * * 1 

(3) 6 .........3 5 7 3 3 

(9) 13 1 * * * 1 $\leftarrow$ (17) 1 * * 2 4 3 3 3 

(15) 1 1 * * 2 4 3 3 3 $\leftarrow$ (16) 2 * * 2 4 3 3 3 

(17) 5 1 * * * 1 $\leftarrow$ (18) 6 * * * 1 

(21) 1 1 * * * 1 $\leftarrow$ (22) 2 * * * 1 

\end{tcolorbox}
\begin{tcolorbox}[title={$(48,17)$}]
(1) 5 6 .........4 5 3 3 3 

(1) 20 1 1 * * * 1 $\leftarrow$ (2) 21 1 * * * 1 

(3) 3 6 .........4 5 3 3 3 $\leftarrow$ (4) 6 .........3 5 7 3 3 

(7) 8 1 1 * * 2 4 3 3 3 $\leftarrow$ (16) 1 1 * * 2 4 3 3 3 

(9) 12 1 1 * * * 1 $\leftarrow$ (10) 13 1 * * * 1 

(15) 1 2 3 4 4 1 1 * * 1 $\leftarrow$ (17) 2 3 4 4 1 1 * * 1 

(17) 4 1 1 * * * 1 $\leftarrow$ (18) 5 1 * * * 1 

\end{tcolorbox}
\begin{tcolorbox}[title={$(48,18)$}]
(1) 3 6 ..........4 5 3 3 3 $\leftarrow$ (2) 5 6 .........4 5 3 3 3 

(1) 16 4 1 1 * * * 1 $\leftarrow$ (2) 20 1 1 * * * 1 

(3) ............3 5 7 3 3 $\leftarrow$ (4) 3 6 .........4 5 3 3 3 

(6) ............4 5 3 3 3 $\leftarrow$ (11) 6 2 3 4 4 1 1 * * 1 

(7) 3 6 2 3 4 4 1 1 * * 1 $\leftarrow$ (8) 8 1 1 * * 2 4 3 3 3 

(9) 8 4 1 1 * * * 1 $\leftarrow$ (10) 12 1 1 * * * 1 

(13) 4 4 1 1 * * * 1 $\leftarrow$ (16) 1 2 3 4 4 1 1 * * 1 

(15) * * * * 1 $\leftarrow$ (18) 4 1 1 * * * 1 

\end{tcolorbox}
\begin{tcolorbox}[title={$(48,19)$}]
(1) .............3 5 7 3 3 $\leftarrow$ (2) 3 6 ..........4 5 3 3 3 

(1) 14 * * * * 1 $\leftarrow$ (2) 16 4 1 1 * * * 1 

(4) .............4 5 3 3 3 $\leftarrow$ (10) * * * 2 4 3 3 3 

(7) 2 * * * 2 4 3 3 3 $\leftarrow$ (8) 3 6 2 3 4 4 1 1 * * 1 

(9) 6 * * * * 1 $\leftarrow$ (10) 8 4 1 1 * * * 1 

(13) 2 * * * * 1 $\leftarrow$ (14) 4 4 1 1 * * * 1 

(14) 1 * * * * 1 $\leftarrow$ (16) * * * * 1 

\end{tcolorbox}
\begin{tcolorbox}[title={$(48,20)$}]
(1) 13 1 * * * * 1 $\leftarrow$ (2) 14 * * * * 1 

(2) ..............4 5 3 3 3 $\leftarrow$ (9) 1 * * * 2 4 3 3 3 

(7) 1 1 * * * 2 4 3 3 3 $\leftarrow$ (8) 2 * * * 2 4 3 3 3 

(9) 5 1 * * * * 1 $\leftarrow$ (10) 6 * * * * 1 

(13) 1 1 * * * * 1 $\leftarrow$ (14) 2 * * * * 1 

\end{tcolorbox}
\begin{tcolorbox}[title={$(48,21)$}]
(1) 12 1 1 * * * * 1 $\leftarrow$ (2) 13 1 * * * * 1 

(7) 1 2 3 4 4 1 1 * * * 1 $\leftarrow$ (9) 2 3 4 4 1 1 * * * 1 

(9) 4 1 1 * * * * 1 $\leftarrow$ (10) 5 1 * * * * 1 

\end{tcolorbox}
\begin{tcolorbox}[title={$(48,22)$}]
(1) * * * * 2 4 3 3 3 

(1) 8 4 1 1 * * * * 1 $\leftarrow$ (2) 12 1 1 * * * * 1 

(5) 4 4 1 1 * * * * 1 $\leftarrow$ (8) 1 2 3 4 4 1 1 * * * 1 

(7) * * * * * 1 $\leftarrow$ (10) 4 1 1 * * * * 1 

\end{tcolorbox}
\begin{tcolorbox}[title={$(48,23)$}]
(1) 6 * * * * * 1 $\leftarrow$ (2) 8 4 1 1 * * * * 1 

(2) 3 4 4 1 1 * * * * 1 

(5) 2 * * * * * 1 $\leftarrow$ (6) 4 4 1 1 * * * * 1 

(6) 1 * * * * * 1 $\leftarrow$ (8) * * * * * 1 

\end{tcolorbox}
\begin{tcolorbox}[title={$(48,24)$}]
(1) 5 1 * * * * * 1 $\leftarrow$ (2) 6 * * * * * 1 

(5) 1 1 * * * * * 1 $\leftarrow$ (6) 2 * * * * * 1 

\end{tcolorbox}
\begin{tcolorbox}[title={$(48,25)$}]
(1) 4 1 1 * * * * * 1 $\leftarrow$ (2) 5 1 * * * * * 1 

\end{tcolorbox}
\begin{tcolorbox}[title={$(49,3)$}]
(11) 23 15 

(15) 27 7 $\leftarrow$ (19) 31 

(19) 15 15 $\leftarrow$ (43) 7 

(31) 11 7 $\leftarrow$ (35) 15 

(47) 1 1 $\leftarrow$ (49) 1 

\end{tcolorbox}
\begin{tcolorbox}[title={$(49,4)$}]
(10) 21 11 7 

(11) 22 13 3 $\leftarrow$ (12) 23 15 

(12) 23 7 7 $\leftarrow$ (18) 29 3 

(15) 26 5 3 $\leftarrow$ (16) 27 7 

(18) 13 11 7 $\leftarrow$ (42) 5 3 

(19) 14 13 3 $\leftarrow$ (20) 15 15 

(28) 7 7 7 $\leftarrow$ (34) 13 3 

(31) 10 5 3 $\leftarrow$ (32) 11 7 

(35) 6 5 3 $\leftarrow$ (36) 7 7 

(46) 1 1 1 $\leftarrow$ (48) 1 1 

\end{tcolorbox}
\begin{tcolorbox}[title={$(49,5)$}]
(1) 7 11 15 15 

(4) 14 13 11 7 

(9) 11 15 7 7 

(9) 19 7 7 7 

(10) 20 5 7 7 $\leftarrow$ (17) 27 3 3 

(11) 17 7 7 7 $\leftarrow$ (41) 3 3 3 

(11) 21 11 3 3 $\leftarrow$ (12) 22 13 3 

(15) 25 3 3 3 $\leftarrow$ (16) 26 5 3 

(19) 13 11 3 3 $\leftarrow$ (20) 14 13 3 

(26) 4 5 7 7 $\leftarrow$ (33) 11 3 3 

(31) 5 7 3 3 $\leftarrow$ (36) 6 5 3 

(31) 9 3 3 3 $\leftarrow$ (32) 10 5 3 

\end{tcolorbox}
\begin{tcolorbox}[title={$(49,6)$}]
(1) 10 17 7 7 7 

(4) 9 15 7 7 7 

(5) 11 14 5 7 7 $\leftarrow$ (6) 13 13 11 7 

(8) 9 11 7 7 7 

(9) 7 14 5 7 7 $\leftarrow$ (10) 11 15 7 7 

(9) 18 3 5 7 7 $\leftarrow$ (16) 25 3 3 3 

(10) 7 13 5 7 7 $\leftarrow$ (12) 17 7 7 7 

(11) 20 9 3 3 3 $\leftarrow$ (12) 21 11 3 3 

(14) 5 9 7 7 7 $\leftarrow$ (18) 13 5 7 7 

(19) 12 9 3 3 3 $\leftarrow$ (20) 13 11 3 3 

(25) 2 3 5 7 7 $\leftarrow$ (32) 9 3 3 3 

(28) 3 5 7 3 3 $\leftarrow$ (33) 8 3 3 3 

(29) 4 7 3 3 3 $\leftarrow$ (30) 6 6 5 3 

(31) 4 5 3 3 3 $\leftarrow$ (32) 5 7 3 3 

(35) 5 1 2 3 3 $\leftarrow$ (36) 6 2 3 3 

\end{tcolorbox}
\begin{tcolorbox}[title={$(49,7)$}]
(1) 9 7 13 5 7 7 $\leftarrow$ (2) 10 17 7 7 7 

(2) 6 9 11 7 7 7 

(3) 11 5 9 7 7 7 $\leftarrow$ (5) 13 13 5 7 7 

(5) 7 14 4 5 7 7 $\leftarrow$ (6) 11 14 5 7 7 

(7) 12 12 9 3 3 3 $\leftarrow$ (12) 20 9 3 3 3 

(9) 5 5 9 7 7 7 $\leftarrow$ (10) 7 14 5 7 7 

(11) 3 5 9 7 7 7 $\leftarrow$ (13) 14 4 5 7 7 

(11) 8 12 9 3 3 3 $\leftarrow$ (20) 12 9 3 3 3 

(15) 5 7 3 5 7 7 $\leftarrow$ (19) 9 3 5 7 7 

(23) 3 3 6 6 5 3 $\leftarrow$ (27) 3 6 6 5 3 

(26) 2 3 5 7 3 3 $\leftarrow$ (32) 4 5 3 3 3 

(29) 2 4 5 3 3 3 $\leftarrow$ (30) 4 7 3 3 3 

(35) 3 4 4 1 1 1 $\leftarrow$ (36) 5 1 2 3 3 

(39) 1 * 1 $\leftarrow$ (41) * 1 

\end{tcolorbox}
\begin{tcolorbox}[title={$(49,8)$}]
(1) 5 10 11 3 5 7 7 

(1) 8 5 5 9 7 7 7 $\leftarrow$ (2) 9 7 13 5 7 7 

(2) 9 3 5 9 7 7 7 $\leftarrow$ (4) 11 5 9 7 7 7 

(3) 8 3 5 9 7 7 7 $\leftarrow$ (6) 7 14 4 5 7 7 

(4) 7 8 12 9 3 3 3 $\leftarrow$ (10) 17 3 6 6 5 3 

(7) 10 9 3 6 6 5 3 $\leftarrow$ (8) 12 12 9 3 3 3 

(10) 3 5 9 3 5 7 7 $\leftarrow$ (18) 9 3 6 6 5 3 

(11) 6 9 3 6 6 5 3 $\leftarrow$ (12) 8 12 9 3 3 3 

(12) 3 5 7 3 5 7 7 $\leftarrow$ (16) 5 7 3 5 7 7 

(17) 6 3 3 6 6 5 3 $\leftarrow$ (24) 3 3 6 6 5 3 

(19) 10 2 4 5 3 3 3 $\leftarrow$ (20) 12 4 5 3 3 3 

(23) 6 2 4 5 3 3 3 $\leftarrow$ (30) 2 4 5 3 3 3 

(38) 1 1 * 1 $\leftarrow$ (40) 1 * 1 

\end{tcolorbox}
\begin{tcolorbox}[title={$(49,9)$}]
(1) 4 6 3 5 9 7 7 7 $\leftarrow$ (2) 5 10 11 3 5 7 7 

(2) 4 5 3 5 9 7 7 7 $\leftarrow$ (4) 8 3 5 9 7 7 7 

(2) 8 3 5 9 3 5 7 7 $\leftarrow$ (9) 11 12 4 5 3 3 3 

(3) 9 3 5 7 3 5 7 7 $\leftarrow$ (17) 5 5 3 6 6 5 3 

(7) 9 5 5 3 6 6 5 3 $\leftarrow$ (8) 10 9 3 6 6 5 3 

(11) 3 6 5 2 3 5 7 7 $\leftarrow$ (12) 6 9 3 6 6 5 3 

(15) 5 6 2 3 5 7 3 3 $\leftarrow$ (21) 6 2 3 5 7 3 3 

(17) 3 6 2 3 5 7 3 3 $\leftarrow$ (18) 6 3 3 6 6 5 3 

(19) 7 13 1 * 1 $\leftarrow$ (20) 10 2 4 5 3 3 3 

(21) 6 ..4 5 3 3 3 $\leftarrow$ (27) 13 1 * 1 

(23) 4 ..4 5 3 3 3 $\leftarrow$ (24) 6 2 4 5 3 3 3 

\end{tcolorbox}
\begin{tcolorbox}[title={$(49,10)$}]
(1) 2 3 5 3 5 9 7 7 7 $\leftarrow$ (2) 4 6 3 5 9 7 7 7 

(1) 4 5 3 5 9 3 5 7 7 $\leftarrow$ (8) 9 5 5 3 6 6 5 3 

(7) 4 5 5 5 3 6 6 5 3 $\leftarrow$ (10) 5 6 5 2 3 5 7 7 

(11) 2 3 5 5 3 6 6 5 3 $\leftarrow$ (12) 3 6 5 2 3 5 7 7 

(14) 2 4 3 3 3 6 6 5 3 $\leftarrow$ (19) 5 6 2 4 5 3 3 3 

(15) 3 5 6 2 4 5 3 3 3 $\leftarrow$ (16) 5 6 2 3 5 7 3 3 

(17) ...3 3 6 6 5 3 $\leftarrow$ (18) 3 6 2 3 5 7 3 3 

(19) 5 8 1 1 2 4 3 3 3 $\leftarrow$ (20) 7 13 1 * 1 

(20) ....3 5 7 3 3 $\leftarrow$ (25) 8 1 1 2 4 3 3 3 

(21) 4 ...4 5 3 3 3 $\leftarrow$ (22) 6 ..4 5 3 3 3 

(23) ....4 5 3 3 3 $\leftarrow$ (24) 4 ..4 5 3 3 3 

(27) 5 1 2 3 4 4 1 1 1 $\leftarrow$ (28) 6 2 3 4 4 1 1 1 

\end{tcolorbox}
\begin{tcolorbox}[title={$(49,11)$}]
(2) 4 3 5 6 5 2 3 5 7 7 

(5) 2 4 5 5 5 3 6 6 5 3 $\leftarrow$ (8) 4 5 5 5 3 6 6 5 3 

(12) ..4 3 3 3 6 6 5 3 $\leftarrow$ (18) ...3 3 6 6 5 3 

(15) ....3 3 6 6 5 3 $\leftarrow$ (16) 3 5 6 2 4 5 3 3 3 

(18) .....3 5 7 3 3 $\leftarrow$ (24) ....4 5 3 3 3 

(19) 4 ....4 5 3 3 3 $\leftarrow$ (20) 5 8 1 1 2 4 3 3 3 

(21) .....4 5 3 3 3 $\leftarrow$ (22) 4 ...4 5 3 3 3 

(27) 3 4 4 1 1 * 1 $\leftarrow$ (28) 5 1 2 3 4 4 1 1 1 

(31) 1 * * 1 $\leftarrow$ (33) * * 1 

\end{tcolorbox}
\begin{tcolorbox}[title={$(49,12)$}]
(3) ..4 5 5 5 3 6 6 5 3 $\leftarrow$ (6) 2 4 5 5 5 3 6 6 5 3 

(9) 6 ....3 3 6 6 5 3 $\leftarrow$ (16) ....3 3 6 6 5 3 

(15) 6 .....4 5 3 3 3 $\leftarrow$ (22) .....4 5 3 3 3 

(19) ......4 5 3 3 3 $\leftarrow$ (20) 4 ....4 5 3 3 3 

(30) 1 1 * * 1 $\leftarrow$ (32) 1 * * 1 

\end{tcolorbox}
\begin{tcolorbox}[title={$(49,13)$}]
(1) ...4 5 5 5 3 6 6 5 3 $\leftarrow$ (4) ..4 5 5 5 3 6 6 5 3 

(7) 5 6 .....3 5 7 3 3 $\leftarrow$ (13) 6 .....3 5 7 3 3 

(9) 3 6 .....3 5 7 3 3 $\leftarrow$ (10) 6 ....3 3 6 6 5 3 

(13) 6 ......4 5 3 3 3 $\leftarrow$ (20) ......4 5 3 3 3 

(15) 4 ......4 5 3 3 3 $\leftarrow$ (16) 6 .....4 5 3 3 3 

\end{tcolorbox}
\begin{tcolorbox}[title={$(49,14)$}]
(2) * 2 4 7 3 3 6 6 5 3 

(6) .....4 3 3 3 6 6 5 3 $\leftarrow$ (11) 5 6 .....4 5 3 3 3 

(7) 3 5 6 .....4 5 3 3 3 $\leftarrow$ (8) 5 6 .....3 5 7 3 3 

(9) .......3 3 6 6 5 3 $\leftarrow$ (10) 3 6 .....3 5 7 3 3 

(12) ........3 5 7 3 3 $\leftarrow$ (17) 8 1 1 * 2 4 3 3 3 

(13) 4 .......4 5 3 3 3 $\leftarrow$ (14) 6 ......4 5 3 3 3 

(15) ........4 5 3 3 3 $\leftarrow$ (16) 4 ......4 5 3 3 3 

(19) 5 1 2 3 4 4 1 1 * 1 $\leftarrow$ (20) 6 2 3 4 4 1 1 * 1 

\end{tcolorbox}
\begin{tcolorbox}[title={$(49,15)$}]
(1) 1 * 2 4 7 3 3 6 6 5 3 

(3) 1.........4 5 3 3 3 $\leftarrow$ (10) .......3 3 6 6 5 3 

(7) ........3 3 6 6 5 3 $\leftarrow$ (8) 3 5 6 .....4 5 3 3 3 

(10) .........3 5 7 3 3 $\leftarrow$ (16) ........4 5 3 3 3 

(13) .........4 5 3 3 3 $\leftarrow$ (14) 4 .......4 5 3 3 3 

(19) 3 4 4 1 1 * * 1 $\leftarrow$ (20) 5 1 2 3 4 4 1 1 * 1 

(23) 1 * * * 1 $\leftarrow$ (25) * * * 1 

\end{tcolorbox}
\begin{tcolorbox}[title={$(49,16)$}]
(1) 6 ........3 3 6 6 5 3 $\leftarrow$ (8) ........3 3 6 6 5 3 

(3) 10 .........4 5 3 3 3 $\leftarrow$ (4) 1.........4 5 3 3 3 

(7) 6 .........4 5 3 3 3 $\leftarrow$ (14) .........4 5 3 3 3 

(22) 1 1 * * * 1 $\leftarrow$ (24) 1 * * * 1 

\end{tcolorbox}
\begin{tcolorbox}[title={$(49,17)$}]
(1) 3 6 .........3 5 7 3 3 $\leftarrow$ (2) 6 ........3 3 6 6 5 3 

(3) 7 13 1 * * * 1 $\leftarrow$ (4) 10 .........4 5 3 3 3 

(5) 6 ..........4 5 3 3 3 $\leftarrow$ (11) 13 1 * * * 1 

(7) 4 ..........4 5 3 3 3 $\leftarrow$ (8) 6 .........4 5 3 3 3 

\end{tcolorbox}
\begin{tcolorbox}[title={$(49,18)$}]
(1) ...........3 3 6 6 5 3 $\leftarrow$ (2) 3 6 .........3 5 7 3 3 

(3) 5 8 1 1 * * 2 4 3 3 3 $\leftarrow$ (4) 7 13 1 * * * 1 

(4) ............3 5 7 3 3 $\leftarrow$ (9) 8 1 1 * * 2 4 3 3 3 

(5) 4 ...........4 5 3 3 3 $\leftarrow$ (6) 6 ..........4 5 3 3 3 

(7) ............4 5 3 3 3 $\leftarrow$ (8) 4 ..........4 5 3 3 3 

(11) 5 1 2 3 4 4 1 1 * * 1 $\leftarrow$ (12) 6 2 3 4 4 1 1 * * 1 

\end{tcolorbox}
\begin{tcolorbox}[title={$(49,19)$}]
(2) .............3 5 7 3 3 $\leftarrow$ (8) ............4 5 3 3 3 

(3) 4 ............4 5 3 3 3 $\leftarrow$ (4) 5 8 1 1 * * 2 4 3 3 3 

(5) .............4 5 3 3 3 $\leftarrow$ (6) 4 ...........4 5 3 3 3 

(11) 3 4 4 1 1 * * * 1 $\leftarrow$ (12) 5 1 2 3 4 4 1 1 * * 1 

(15) 1 * * * * 1 $\leftarrow$ (17) * * * * 1 

\end{tcolorbox}
\begin{tcolorbox}[title={$(49,20)$}]
(3) ..............4 5 3 3 3 $\leftarrow$ (4) 4 ............4 5 3 3 3 

(8) 1 1 * * * 2 4 3 3 3 

(14) 1 1 * * * * 1 $\leftarrow$ (16) 1 * * * * 1 

\end{tcolorbox}
\begin{tcolorbox}[title={$(49,21)$}]
(3) 6 2 3 4 4 1 1 * * * 1 

\end{tcolorbox}
\begin{tcolorbox}[title={$(49,22)$}]
(2) * * * * 2 4 3 3 3 

(3) 5 1 2 3 4 4 1 1 * * * 1 $\leftarrow$ (4) 6 2 3 4 4 1 1 * * * 1 

\end{tcolorbox}
\begin{tcolorbox}[title={$(49,23)$}]
(1) 1 * * * * 2 4 3 3 3 

(3) 3 4 4 1 1 * * * * 1 $\leftarrow$ (4) 5 1 2 3 4 4 1 1 * * * 1 

(7) 1 * * * * * 1 $\leftarrow$ (9) * * * * * 1 

\end{tcolorbox}
\begin{tcolorbox}[title={$(49,24)$}]
(1) 2 3 4 4 1 1 * * * * 1 

(6) 1 1 * * * * * 1 $\leftarrow$ (8) 1 * * * * * 1 

\end{tcolorbox}
\begin{tcolorbox}[title={$(49,25)$}]
(2) 4 1 1 * * * * * 1 

\end{tcolorbox}
\begin{tcolorbox}[title={$(50,2)$}]
(47) 3 $\leftarrow$ (51) 

\end{tcolorbox}
\begin{tcolorbox}[title={$(50,3)$}]
(19) 30 1 $\leftarrow$ (20) 31 

(35) 14 1 $\leftarrow$ (36) 15 

(43) 6 1 $\leftarrow$ (44) 7 

(44) 3 3 $\leftarrow$ (50) 1 

(47) 2 1 $\leftarrow$ (48) 3 

\end{tcolorbox}
\begin{tcolorbox}[title={$(50,4)$}]
(5) 15 15 15 

(9) 11 15 15 

(11) 21 11 7 $\leftarrow$ (13) 23 15 

(13) 23 7 7 $\leftarrow$ (17) 27 7 

(19) 13 11 7 $\leftarrow$ (21) 15 15 

(19) 29 1 1 $\leftarrow$ (20) 30 1 

(29) 7 7 7 $\leftarrow$ (33) 11 7 

(31) 5 7 7 $\leftarrow$ (37) 7 7 

(35) 13 1 1 $\leftarrow$ (36) 14 1 

(42) 2 3 3 $\leftarrow$ (49) 1 1 

(43) 5 1 1 $\leftarrow$ (44) 6 1 

(47) 1 1 1 $\leftarrow$ (48) 2 1 

\end{tcolorbox}
\begin{tcolorbox}[title={$(50,5)$}]
(2) 7 11 15 15 

(5) 14 13 11 7 $\leftarrow$ (6) 15 15 15 

(9) 10 13 11 7 $\leftarrow$ (10) 11 15 15 

(10) 19 7 7 7 $\leftarrow$ (12) 21 11 7 

(11) 20 5 7 7 $\leftarrow$ (14) 23 7 7 

(17) 14 5 7 7 $\leftarrow$ (20) 13 11 7 

(19) 28 1 1 1 $\leftarrow$ (20) 29 1 1 

(27) 4 5 7 7 $\leftarrow$ (30) 7 7 7 

(28) 3 5 7 7 $\leftarrow$ (32) 5 7 7 

(35) 12 1 1 1 $\leftarrow$ (36) 13 1 1 

(41) 1 2 3 3 $\leftarrow$ (48) 1 1 1 

(43) 4 1 1 1 $\leftarrow$ (44) 5 1 1 

\end{tcolorbox}
\begin{tcolorbox}[title={$(50,6)$}]
(1) 9 11 15 7 7 

(3) 7 11 15 7 7 

(5) 9 15 7 7 7 

(9) 9 11 7 7 7 $\leftarrow$ (10) 10 13 11 7 

(10) 18 3 5 7 7 $\leftarrow$ (12) 20 5 7 7 

(11) 7 13 5 7 7 $\leftarrow$ (13) 17 7 7 7 

(15) 5 9 7 7 7 $\leftarrow$ (19) 13 5 7 7 

(17) 11 3 5 7 7 $\leftarrow$ (18) 14 5 7 7 

(19) 24 4 1 1 1 $\leftarrow$ (20) 28 1 1 1 

(26) 2 3 5 7 7 $\leftarrow$ (28) 4 5 7 7 

(29) 3 5 7 3 3 $\leftarrow$ (33) 5 7 3 3 

(33) 3 6 2 3 3 $\leftarrow$ (34) 8 3 3 3 

(35) 2 4 3 3 3 $\leftarrow$ (37) 6 2 3 3 

(35) 8 4 1 1 1 $\leftarrow$ (36) 12 1 1 1 

(39) 4 4 1 1 1 $\leftarrow$ (44) 4 1 1 1 

\end{tcolorbox}
\begin{tcolorbox}[title={$(50,7)$}]
(1) 8 9 11 7 7 7 $\leftarrow$ (2) 9 11 15 7 7 

(3) 6 9 11 7 7 7 $\leftarrow$ (4) 7 11 15 7 7 

(5) 12 11 3 5 7 7 $\leftarrow$ (6) 13 13 5 7 7 

(10) 5 5 9 7 7 7 $\leftarrow$ (12) 7 13 5 7 7 

(12) 3 5 9 7 7 7 $\leftarrow$ (16) 5 9 7 7 7 

(13) 13 2 3 5 7 7 $\leftarrow$ (14) 14 4 5 7 7 

(14) 5 9 3 5 7 7 $\leftarrow$ (20) 9 3 5 7 7 

(19) 22 * 1 $\leftarrow$ (20) 24 4 1 1 1 

(27) 2 3 5 7 3 3 $\leftarrow$ (28) 3 6 6 5 3 

(33) ..4 3 3 3 $\leftarrow$ (34) 3 6 2 3 3 

(34) 1 2 4 3 3 3 $\leftarrow$ (36) 2 4 3 3 3 

(35) 6 * 1 $\leftarrow$ (36) 8 4 1 1 1 

(36) 3 4 4 1 1 1 $\leftarrow$ (42) * 1 

(39) 2 * 1 $\leftarrow$ (40) 4 4 1 1 1 

\end{tcolorbox}
\begin{tcolorbox}[title={$(50,8)$}]
(2) 8 5 5 9 7 7 7 

(3) 9 3 5 9 7 7 7 $\leftarrow$ (5) 11 5 9 7 7 7 

(5) 7 8 12 9 3 3 3 $\leftarrow$ (6) 12 11 3 5 7 7 

(11) 3 5 9 3 5 7 7 $\leftarrow$ (13) 8 12 9 3 3 3 

(13) 3 5 7 3 5 7 7 $\leftarrow$ (17) 5 7 3 5 7 7 

(13) 7 12 4 5 3 3 3 $\leftarrow$ (14) 13 2 3 5 7 7 

(15) 6 5 2 3 5 7 7 $\leftarrow$ (21) 12 4 5 3 3 3 

(19) 21 1 * 1 $\leftarrow$ (20) 22 * 1 

(33) 1 1 2 4 3 3 3 $\leftarrow$ (34) ..4 3 3 3 

(34) 2 3 4 4 1 1 1 $\leftarrow$ (41) 1 * 1 

(35) 5 1 * 1 $\leftarrow$ (36) 6 * 1 

(39) 1 1 * 1 $\leftarrow$ (40) 2 * 1 

\end{tcolorbox}
\begin{tcolorbox}[title={$(50,9)$}]
(3) 4 5 3 5 9 7 7 7 $\leftarrow$ (5) 8 3 5 9 7 7 7 

(3) 8 3 5 9 3 5 7 7 $\leftarrow$ (6) 7 8 12 9 3 3 3 

(4) 9 3 5 7 3 5 7 7 

(13) 4 7 3 3 6 6 5 3 $\leftarrow$ (19) 6 3 3 6 6 5 3 

(13) 5 6 3 3 6 6 5 3 $\leftarrow$ (14) 7 12 4 5 3 3 3 

(15) 3 6 3 3 6 6 5 3 $\leftarrow$ (16) 6 5 2 3 5 7 7 

(19) 20 1 1 * 1 $\leftarrow$ (20) 21 1 * 1 

(21) 3 6 2 4 5 3 3 3 $\leftarrow$ (22) 6 2 3 5 7 3 3 

(27) 12 1 1 * 1 $\leftarrow$ (28) 13 1 * 1 

(33) 1 2 3 4 4 1 1 1 $\leftarrow$ (40) 1 1 * 1 

(35) 4 1 1 * 1 $\leftarrow$ (36) 5 1 * 1 

\end{tcolorbox}
\begin{tcolorbox}[title={$(50,10)$}]
(2) 2 3 5 3 5 9 7 7 7 $\leftarrow$ (4) 4 5 3 5 9 7 7 7 

(2) 4 5 3 5 9 3 5 7 7 $\leftarrow$ (4) 8 3 5 9 3 5 7 7 

(5) 5 5 6 5 2 3 5 7 7 

(11) 2 4 7 3 3 6 6 5 3 $\leftarrow$ (14) 4 7 3 3 6 6 5 3 

(12) 2 3 5 5 3 6 6 5 3 $\leftarrow$ (17) 5 6 2 3 5 7 3 3 

(13) 3 5 6 2 3 5 7 3 3 $\leftarrow$ (14) 5 6 3 3 6 6 5 3 

(15) 2 4 3 3 3 6 6 5 3 $\leftarrow$ (16) 3 6 3 3 6 6 5 3 

(19) 3 6 ..4 5 3 3 3 $\leftarrow$ (20) 5 6 2 4 5 3 3 3 

(19) 16 4 1 1 * 1 $\leftarrow$ (20) 20 1 1 * 1 

(21) ....3 5 7 3 3 $\leftarrow$ (22) 3 6 2 4 5 3 3 3 

(25) 3 6 2 3 4 4 1 1 1 $\leftarrow$ (26) 8 1 1 2 4 3 3 3 

(27) * 2 4 3 3 3 $\leftarrow$ (29) 6 2 3 4 4 1 1 1 

(27) 8 4 1 1 * 1 $\leftarrow$ (28) 12 1 1 * 1 

(31) 4 4 1 1 * 1 $\leftarrow$ (36) 4 1 1 * 1 

\end{tcolorbox}
\begin{tcolorbox}[title={$(50,11)$}]
(3) 4 3 5 6 5 2 3 5 7 7 $\leftarrow$ (6) 5 5 6 5 2 3 5 7 7 

(9) ..4 7 3 3 6 6 5 3 $\leftarrow$ (16) 2 4 3 3 3 6 6 5 3 

(10) 1 2 4 7 3 3 6 6 5 3 $\leftarrow$ (12) 2 4 7 3 3 6 6 5 3 

(13) ..4 3 3 3 6 6 5 3 $\leftarrow$ (14) 3 5 6 2 3 5 7 3 3 

(19) .....3 5 7 3 3 $\leftarrow$ (20) 3 6 ..4 5 3 3 3 

(19) 14 * * 1 $\leftarrow$ (20) 16 4 1 1 * 1 

(25) 2 * 2 4 3 3 3 $\leftarrow$ (26) 3 6 2 3 4 4 1 1 1 

(26) 1 * 2 4 3 3 3 $\leftarrow$ (28) * 2 4 3 3 3 

(27) 6 * * 1 $\leftarrow$ (28) 8 4 1 1 * 1 

(28) 3 4 4 1 1 * 1 $\leftarrow$ (34) * * 1 

(31) 2 * * 1 $\leftarrow$ (32) 4 4 1 1 * 1 

\end{tcolorbox}
\begin{tcolorbox}[title={$(50,12)$}]
(1) 2 4 3 5 6 5 2 3 5 7 7 $\leftarrow$ (4) 4 3 5 6 5 2 3 5 7 7 

(7) 6 ..4 3 3 3 6 6 5 3 $\leftarrow$ (14) ..4 3 3 3 6 6 5 3 

(9) 1 1 2 4 7 3 3 6 6 5 3 $\leftarrow$ (10) ..4 7 3 3 6 6 5 3 

(19) 13 1 * * 1 $\leftarrow$ (20) 14 * * 1 

(25) 1 1 * 2 4 3 3 3 $\leftarrow$ (26) 2 * 2 4 3 3 3 

(26) 2 3 4 4 1 1 * 1 $\leftarrow$ (33) 1 * * 1 

(27) 5 1 * * 1 $\leftarrow$ (28) 6 * * 1 

(31) 1 1 * * 1 $\leftarrow$ (32) 2 * * 1 

\end{tcolorbox}
\begin{tcolorbox}[title={$(50,13)$}]
(2) ...4 5 5 5 3 6 6 5 3 

(5) 4 1 1 2 4 7 3 3 6 6 5 3 $\leftarrow$ (11) 6 ....3 3 6 6 5 3 

(7) 3 6 ....3 3 6 6 5 3 $\leftarrow$ (8) 6 ..4 3 3 3 6 6 5 3 

(13) 3 6 .....4 5 3 3 3 $\leftarrow$ (14) 6 .....3 5 7 3 3 

(19) 12 1 1 * * 1 $\leftarrow$ (20) 13 1 * * 1 

(25) 1 2 3 4 4 1 1 * 1 $\leftarrow$ (32) 1 1 * * 1 

(27) 4 1 1 * * 1 $\leftarrow$ (28) 5 1 * * 1 

\end{tcolorbox}
\begin{tcolorbox}[title={$(50,14)$}]
(3) * 2 4 7 3 3 6 6 5 3 $\leftarrow$ (6) 4 1 1 2 4 7 3 3 6 6 5 3 

(3) 24 4 1 1 * * 1 $\leftarrow$ (9) 5 6 .....3 5 7 3 3 

(7) .....4 3 3 3 6 6 5 3 $\leftarrow$ (8) 3 6 ....3 3 6 6 5 3 

(11) 3 6 ......4 5 3 3 3 $\leftarrow$ (12) 5 6 .....4 5 3 3 3 

(13) ........3 5 7 3 3 $\leftarrow$ (14) 3 6 .....4 5 3 3 3 

(17) 3 6 2 3 4 4 1 1 * 1 $\leftarrow$ (18) 8 1 1 * 2 4 3 3 3 

(19) * * 2 4 3 3 3 $\leftarrow$ (21) 6 2 3 4 4 1 1 * 1 

(19) 8 4 1 1 * * 1 $\leftarrow$ (20) 12 1 1 * * 1 

(23) 4 4 1 1 * * 1 $\leftarrow$ (28) 4 1 1 * * 1 

\end{tcolorbox}
\begin{tcolorbox}[title={$(50,15)$}]
(1) 2 * 2 4 7 3 3 6 6 5 3 $\leftarrow$ (8) .....4 3 3 3 6 6 5 3 

(2) 1 * 2 4 7 3 3 6 6 5 3 $\leftarrow$ (4) * 2 4 7 3 3 6 6 5 3 

(3) 22 * * * 1 $\leftarrow$ (4) 24 4 1 1 * * 1 

(11) .........3 5 7 3 3 $\leftarrow$ (12) 3 6 ......4 5 3 3 3 

(17) 2 * * 2 4 3 3 3 $\leftarrow$ (18) 3 6 2 3 4 4 1 1 * 1 

(18) 1 * * 2 4 3 3 3 $\leftarrow$ (20) * * 2 4 3 3 3 

(19) 6 * * * 1 $\leftarrow$ (20) 8 4 1 1 * * 1 

(20) 3 4 4 1 1 * * 1 $\leftarrow$ (26) * * * 1 

(23) 2 * * * 1 $\leftarrow$ (24) 4 4 1 1 * * 1 

\end{tcolorbox}
\begin{tcolorbox}[title={$(50,16)$}]
(1) 1 1 * 2 4 7 3 3 6 6 5 3 $\leftarrow$ (2) 2 * 2 4 7 3 3 6 6 5 3 

(3) 21 1 * * * 1 $\leftarrow$ (4) 22 * * * 1 

(5) 6 .........3 5 7 3 3 

(17) 1 1 * * 2 4 3 3 3 $\leftarrow$ (18) 2 * * 2 4 3 3 3 

(18) 2 3 4 4 1 1 * * 1 $\leftarrow$ (25) 1 * * * 1 

(19) 5 1 * * * 1 $\leftarrow$ (20) 6 * * * 1 

(23) 1 1 * * * 1 $\leftarrow$ (24) 2 * * * 1 

\end{tcolorbox}
\begin{tcolorbox}[title={$(50,17)$}]
(3) 5 6 .........4 5 3 3 3 

(3) 20 1 1 * * * 1 $\leftarrow$ (4) 21 1 * * * 1 

(5) 3 6 .........4 5 3 3 3 $\leftarrow$ (6) 6 .........3 5 7 3 3 

(11) 12 1 1 * * * 1 $\leftarrow$ (12) 13 1 * * * 1 

(17) 1 2 3 4 4 1 1 * * 1 $\leftarrow$ (24) 1 1 * * * 1 

(19) 4 1 1 * * * 1 $\leftarrow$ (20) 5 1 * * * 1 

\end{tcolorbox}
\begin{tcolorbox}[title={$(50,18)$}]
(2) ...........3 3 6 6 5 3 

(3) 3 6 ..........4 5 3 3 3 $\leftarrow$ (4) 5 6 .........4 5 3 3 3 

(3) 16 4 1 1 * * * 1 $\leftarrow$ (4) 20 1 1 * * * 1 

(5) ............3 5 7 3 3 $\leftarrow$ (6) 3 6 .........4 5 3 3 3 

(9) 3 6 2 3 4 4 1 1 * * 1 $\leftarrow$ (10) 8 1 1 * * 2 4 3 3 3 

(11) * * * 2 4 3 3 3 $\leftarrow$ (13) 6 2 3 4 4 1 1 * * 1 

(11) 8 4 1 1 * * * 1 $\leftarrow$ (12) 12 1 1 * * * 1 

(15) 4 4 1 1 * * * 1 $\leftarrow$ (20) 4 1 1 * * * 1 

\end{tcolorbox}
\begin{tcolorbox}[title={$(50,19)$}]
(3) .............3 5 7 3 3 $\leftarrow$ (4) 3 6 ..........4 5 3 3 3 

(3) 14 * * * * 1 $\leftarrow$ (4) 16 4 1 1 * * * 1 

(6) .............4 5 3 3 3 

(9) 2 * * * 2 4 3 3 3 $\leftarrow$ (10) 3 6 2 3 4 4 1 1 * * 1 

(10) 1 * * * 2 4 3 3 3 $\leftarrow$ (12) * * * 2 4 3 3 3 

(11) 6 * * * * 1 $\leftarrow$ (12) 8 4 1 1 * * * 1 

(12) 3 4 4 1 1 * * * 1 $\leftarrow$ (18) * * * * 1 

(15) 2 * * * * 1 $\leftarrow$ (16) 4 4 1 1 * * * 1 

\end{tcolorbox}
\begin{tcolorbox}[title={$(50,20)$}]
(3) 13 1 * * * * 1 $\leftarrow$ (4) 14 * * * * 1 

(4) ..............4 5 3 3 3 

(9) 1 1 * * * 2 4 3 3 3 $\leftarrow$ (10) 2 * * * 2 4 3 3 3 

(10) 2 3 4 4 1 1 * * * 1 $\leftarrow$ (17) 1 * * * * 1 

(11) 5 1 * * * * 1 $\leftarrow$ (12) 6 * * * * 1 

(15) 1 1 * * * * 1 $\leftarrow$ (16) 2 * * * * 1 

\end{tcolorbox}
\begin{tcolorbox}[title={$(50,21)$}]
(1) 8 1 1 * * * 2 4 3 3 3 

(3) 12 1 1 * * * * 1 $\leftarrow$ (4) 13 1 * * * * 1 

(9) 1 2 3 4 4 1 1 * * * 1 $\leftarrow$ (16) 1 1 * * * * 1 

(11) 4 1 1 * * * * 1 $\leftarrow$ (12) 5 1 * * * * 1 

\end{tcolorbox}
\begin{tcolorbox}[title={$(50,22)$}]
(1) 3 6 2 3 4 4 1 1 * * * 1 $\leftarrow$ (2) 8 1 1 * * * 2 4 3 3 3 

(3) * * * * 2 4 3 3 3 $\leftarrow$ (5) 6 2 3 4 4 1 1 * * * 1 

(3) 8 4 1 1 * * * * 1 $\leftarrow$ (4) 12 1 1 * * * * 1 

(7) 4 4 1 1 * * * * 1 $\leftarrow$ (12) 4 1 1 * * * * 1 

\end{tcolorbox}
\begin{tcolorbox}[title={$(50,23)$}]
(1) 2 * * * * 2 4 3 3 3 $\leftarrow$ (2) 3 6 2 3 4 4 1 1 * * * 1 

(2) 1 * * * * 2 4 3 3 3 $\leftarrow$ (4) * * * * 2 4 3 3 3 

(3) 6 * * * * * 1 $\leftarrow$ (4) 8 4 1 1 * * * * 1 

(4) 3 4 4 1 1 * * * * 1 $\leftarrow$ (10) * * * * * 1 

(7) 2 * * * * * 1 $\leftarrow$ (8) 4 4 1 1 * * * * 1 

\end{tcolorbox}
\begin{tcolorbox}[title={$(50,24)$}]
(1) 1 1 * * * * 2 4 3 3 3 $\leftarrow$ (2) 2 * * * * 2 4 3 3 3 

(2) 2 3 4 4 1 1 * * * * 1 $\leftarrow$ (9) 1 * * * * * 1 

(3) 5 1 * * * * * 1 $\leftarrow$ (4) 6 * * * * * 1 

(7) 1 1 * * * * * 1 $\leftarrow$ (8) 2 * * * * * 1 

\end{tcolorbox}
\begin{tcolorbox}[title={$(51,3)$}]
(19) 29 3 $\leftarrow$ (21) 31 

(35) 13 3 $\leftarrow$ (37) 15 

(43) 5 3 $\leftarrow$ (45) 7 

(45) 3 3 $\leftarrow$ (49) 3 

\end{tcolorbox}
\begin{tcolorbox}[title={$(51,4)$}]
(13) 22 13 3 $\leftarrow$ (14) 23 15 

(17) 26 5 3 $\leftarrow$ (18) 27 7 

(18) 27 3 3 $\leftarrow$ (20) 29 3 

(21) 14 13 3 $\leftarrow$ (22) 15 15 

(33) 10 5 3 $\leftarrow$ (34) 11 7 

(34) 11 3 3 $\leftarrow$ (36) 13 3 

(37) 6 5 3 $\leftarrow$ (38) 7 7 

(42) 3 3 3 $\leftarrow$ (44) 5 3 

(43) 2 3 3 $\leftarrow$ (46) 3 3 

\end{tcolorbox}
\begin{tcolorbox}[title={$(51,5)$}]
(1) 9 11 15 15 

(3) 7 11 15 15 

(6) 14 13 11 7 

(7) 13 13 11 7 

(11) 11 15 7 7 $\leftarrow$ (21) 13 11 7 

(11) 19 7 7 7 $\leftarrow$ (13) 21 11 7 

(13) 21 11 3 3 $\leftarrow$ (14) 22 13 3 

(17) 25 3 3 3 $\leftarrow$ (18) 26 5 3 

(21) 13 11 3 3 $\leftarrow$ (22) 14 13 3 

(29) 3 5 7 7 $\leftarrow$ (33) 5 7 7 

(31) 6 6 5 3 $\leftarrow$ (38) 6 5 3 

(33) 9 3 3 3 $\leftarrow$ (34) 10 5 3 

(42) 1 2 3 3 $\leftarrow$ (44) 2 3 3 

\end{tcolorbox}
\begin{tcolorbox}[title={$(51,6)$}]
(3) 10 17 7 7 7 

(6) 9 15 7 7 7 

(7) 11 14 5 7 7 $\leftarrow$ (8) 13 13 11 7 

(10) 9 11 7 7 7 $\leftarrow$ (14) 17 7 7 7 

(11) 7 14 5 7 7 $\leftarrow$ (12) 11 15 7 7 

(11) 18 3 5 7 7 $\leftarrow$ (13) 20 5 7 7 

(13) 20 9 3 3 3 $\leftarrow$ (14) 21 11 3 3 

(18) 11 3 5 7 7 $\leftarrow$ (20) 13 5 7 7 

(21) 12 9 3 3 3 $\leftarrow$ (22) 13 11 3 3 

(27) 2 3 5 7 7 $\leftarrow$ (29) 4 5 7 7 

(30) 3 5 7 3 3 $\leftarrow$ (35) 8 3 3 3 

(31) 4 7 3 3 3 $\leftarrow$ (32) 6 6 5 3 

(33) 4 5 3 3 3 $\leftarrow$ (34) 5 7 3 3 

(37) 5 1 2 3 3 $\leftarrow$ (38) 6 2 3 3 

\end{tcolorbox}
\begin{tcolorbox}[title={$(51,7)$}]
(1) 5 9 15 7 7 7 

(2) 8 9 11 7 7 7 

(3) 9 7 13 5 7 7 $\leftarrow$ (4) 10 17 7 7 7 

(4) 6 9 11 7 7 7 

(7) 7 14 4 5 7 7 $\leftarrow$ (8) 11 14 5 7 7 

(9) 12 12 9 3 3 3 $\leftarrow$ (12) 18 3 5 7 7 

(11) 5 5 9 7 7 7 $\leftarrow$ (12) 7 14 5 7 7 

(11) 17 3 6 6 5 3 $\leftarrow$ (14) 20 9 3 3 3 

(13) 3 5 9 7 7 7 $\leftarrow$ (17) 5 9 7 7 7 

(15) 5 9 3 5 7 7 $\leftarrow$ (21) 9 3 5 7 7 

(19) 9 3 6 6 5 3 $\leftarrow$ (22) 12 9 3 3 3 

(25) 3 3 6 6 5 3 $\leftarrow$ (28) 2 3 5 7 7 

(28) 2 3 5 7 3 3 $\leftarrow$ (34) 4 5 3 3 3 

(31) 2 4 5 3 3 3 $\leftarrow$ (32) 4 7 3 3 3 

(35) 1 2 4 3 3 3 $\leftarrow$ (37) 2 4 3 3 3 

(37) 3 4 4 1 1 1 $\leftarrow$ (38) 5 1 2 3 3 

\end{tcolorbox}
\begin{tcolorbox}[title={$(51,8)$}]
(3) 5 10 11 3 5 7 7 

(3) 8 5 5 9 7 7 7 $\leftarrow$ (4) 9 7 13 5 7 7 

(4) 9 3 5 9 7 7 7 

(9) 10 9 3 6 6 5 3 $\leftarrow$ (10) 12 12 9 3 3 3 

(10) 11 12 4 5 3 3 3 $\leftarrow$ (12) 17 3 6 6 5 3 

(12) 3 5 9 3 5 7 7 $\leftarrow$ (18) 5 7 3 5 7 7 

(13) 6 9 3 6 6 5 3 $\leftarrow$ (14) 8 12 9 3 3 3 

(14) 3 5 7 3 5 7 7 $\leftarrow$ (16) 5 9 3 5 7 7 

(18) 5 5 3 6 6 5 3 $\leftarrow$ (20) 9 3 6 6 5 3 

(21) 10 2 4 5 3 3 3 $\leftarrow$ (22) 12 4 5 3 3 3 

(25) 6 2 4 5 3 3 3 $\leftarrow$ (32) 2 4 5 3 3 3 

(34) 1 1 2 4 3 3 3 $\leftarrow$ (36) 1 2 4 3 3 3 

(35) 2 3 4 4 1 1 1 $\leftarrow$ (38) 3 4 4 1 1 1 

\end{tcolorbox}
\begin{tcolorbox}[title={$(51,9)$}]
(3) 4 6 3 5 9 7 7 7 $\leftarrow$ (4) 5 10 11 3 5 7 7 

(5) 9 3 5 7 3 5 7 7 

(9) 9 5 5 3 6 6 5 3 $\leftarrow$ (10) 10 9 3 6 6 5 3 

(11) 5 6 5 2 3 5 7 7 $\leftarrow$ (17) 6 5 2 3 5 7 7 

(13) 3 6 5 2 3 5 7 7 $\leftarrow$ (14) 6 9 3 6 6 5 3 

(19) 3 6 2 3 5 7 3 3 $\leftarrow$ (20) 6 3 3 6 6 5 3 

(21) 7 13 1 * 1 $\leftarrow$ (22) 10 2 4 5 3 3 3 

(23) 6 ..4 5 3 3 3 $\leftarrow$ (29) 13 1 * 1 

(25) 4 ..4 5 3 3 3 $\leftarrow$ (26) 6 2 4 5 3 3 3 

(34) 1 2 3 4 4 1 1 1 $\leftarrow$ (36) 2 3 4 4 1 1 1 

\end{tcolorbox}
\begin{tcolorbox}[title={$(51,10)$}]
(3) 2 3 5 3 5 9 7 7 7 $\leftarrow$ (4) 4 6 3 5 9 7 7 7 

(3) 4 5 3 5 9 3 5 7 7 $\leftarrow$ (5) 8 3 5 9 3 5 7 7 

(9) 4 5 5 5 3 6 6 5 3 $\leftarrow$ (12) 5 6 5 2 3 5 7 7 

(13) 2 3 5 5 3 6 6 5 3 $\leftarrow$ (14) 3 6 5 2 3 5 7 7 

(17) 3 5 6 2 4 5 3 3 3 $\leftarrow$ (18) 5 6 2 3 5 7 3 3 

(19) ...3 3 6 6 5 3 $\leftarrow$ (20) 3 6 2 3 5 7 3 3 

(21) 5 8 1 1 2 4 3 3 3 $\leftarrow$ (22) 7 13 1 * 1 

(22) ....3 5 7 3 3 $\leftarrow$ (27) 8 1 1 2 4 3 3 3 

(23) 4 ...4 5 3 3 3 $\leftarrow$ (24) 6 ..4 5 3 3 3 

(25) ....4 5 3 3 3 $\leftarrow$ (26) 4 ..4 5 3 3 3 

(29) 5 1 2 3 4 4 1 1 1 $\leftarrow$ (30) 6 2 3 4 4 1 1 1 

\end{tcolorbox}
\begin{tcolorbox}[title={$(51,11)$}]
(1) 5 5 5 6 5 2 3 5 7 7 

(7) 2 4 5 5 5 3 6 6 5 3 $\leftarrow$ (10) 4 5 5 5 3 6 6 5 3 

(11) 1 2 4 7 3 3 6 6 5 3 $\leftarrow$ (13) 2 4 7 3 3 6 6 5 3 

(17) ....3 3 6 6 5 3 $\leftarrow$ (18) 3 5 6 2 4 5 3 3 3 

(20) .....3 5 7 3 3 $\leftarrow$ (26) ....4 5 3 3 3 

(21) 4 ....4 5 3 3 3 $\leftarrow$ (22) 5 8 1 1 2 4 3 3 3 

(23) .....4 5 3 3 3 $\leftarrow$ (24) 4 ...4 5 3 3 3 

(27) 1 * 2 4 3 3 3 $\leftarrow$ (29) * 2 4 3 3 3 

(29) 3 4 4 1 1 * 1 $\leftarrow$ (30) 5 1 2 3 4 4 1 1 1 

\end{tcolorbox}
\begin{tcolorbox}[title={$(51,12)$}]
(2) 2 4 3 5 6 5 2 3 5 7 7 

(5) ..4 5 5 5 3 6 6 5 3 $\leftarrow$ (8) 2 4 5 5 5 3 6 6 5 3 

(10) 1 1 2 4 7 3 3 6 6 5 3 $\leftarrow$ (12) 1 2 4 7 3 3 6 6 5 3 

(17) 6 .....4 5 3 3 3 $\leftarrow$ (24) .....4 5 3 3 3 

(21) ......4 5 3 3 3 $\leftarrow$ (22) 4 ....4 5 3 3 3 

(26) 1 1 * 2 4 3 3 3 $\leftarrow$ (28) 1 * 2 4 3 3 3 

(27) 2 3 4 4 1 1 * 1 $\leftarrow$ (30) 3 4 4 1 1 * 1 

\end{tcolorbox}
\begin{tcolorbox}[title={$(51,13)$}]
(3) ...4 5 5 5 3 6 6 5 3 $\leftarrow$ (6) ..4 5 5 5 3 6 6 5 3 

(11) 3 6 .....3 5 7 3 3 $\leftarrow$ (12) 6 ....3 3 6 6 5 3 

(15) 6 ......4 5 3 3 3 $\leftarrow$ (22) ......4 5 3 3 3 

(17) 4 ......4 5 3 3 3 $\leftarrow$ (18) 6 .....4 5 3 3 3 

(26) 1 2 3 4 4 1 1 * 1 $\leftarrow$ (28) 2 3 4 4 1 1 * 1 

\end{tcolorbox}
\begin{tcolorbox}[title={$(51,14)$}]
(1) ....4 5 5 5 3 6 6 5 3 $\leftarrow$ (4) ...4 5 5 5 3 6 6 5 3 

(9) 3 5 6 .....4 5 3 3 3 $\leftarrow$ (10) 5 6 .....3 5 7 3 3 

(11) .......3 3 6 6 5 3 $\leftarrow$ (12) 3 6 .....3 5 7 3 3 

(14) ........3 5 7 3 3 $\leftarrow$ (19) 8 1 1 * 2 4 3 3 3 

(15) 4 .......4 5 3 3 3 $\leftarrow$ (16) 6 ......4 5 3 3 3 

(17) ........4 5 3 3 3 $\leftarrow$ (18) 4 ......4 5 3 3 3 

(21) 5 1 2 3 4 4 1 1 * 1 $\leftarrow$ (22) 6 2 3 4 4 1 1 * 1 

\end{tcolorbox}
\begin{tcolorbox}[title={$(51,15)$}]
(3) 1 * 2 4 7 3 3 6 6 5 3 $\leftarrow$ (5) * 2 4 7 3 3 6 6 5 3 

(5) 1.........4 5 3 3 3 

(9) ........3 3 6 6 5 3 $\leftarrow$ (10) 3 5 6 .....4 5 3 3 3 

(12) .........3 5 7 3 3 $\leftarrow$ (18) ........4 5 3 3 3 

(15) .........4 5 3 3 3 $\leftarrow$ (16) 4 .......4 5 3 3 3 

(19) 1 * * 2 4 3 3 3 $\leftarrow$ (21) * * 2 4 3 3 3 

(21) 3 4 4 1 1 * * 1 $\leftarrow$ (22) 5 1 2 3 4 4 1 1 * 1 

\end{tcolorbox}
\begin{tcolorbox}[title={$(51,16)$}]
(2) 1 1 * 2 4 7 3 3 6 6 5 3 $\leftarrow$ (4) 1 * 2 4 7 3 3 6 6 5 3 

(3) 6 ........3 3 6 6 5 3 

(5) 10 .........4 5 3 3 3 $\leftarrow$ (6) 1.........4 5 3 3 3 

(9) 6 .........4 5 3 3 3 $\leftarrow$ (16) .........4 5 3 3 3 

(18) 1 1 * * 2 4 3 3 3 $\leftarrow$ (20) 1 * * 2 4 3 3 3 

(19) 2 3 4 4 1 1 * * 1 $\leftarrow$ (22) 3 4 4 1 1 * * 1 

\end{tcolorbox}
\begin{tcolorbox}[title={$(51,17)$}]
(1) 5 6 .........3 5 7 3 3 

(3) 3 6 .........3 5 7 3 3 $\leftarrow$ (4) 6 ........3 3 6 6 5 3 

(5) 7 13 1 * * * 1 $\leftarrow$ (6) 10 .........4 5 3 3 3 

(7) 6 ..........4 5 3 3 3 $\leftarrow$ (13) 13 1 * * * 1 

(9) 4 ..........4 5 3 3 3 $\leftarrow$ (10) 6 .........4 5 3 3 3 

(18) 1 2 3 4 4 1 1 * * 1 $\leftarrow$ (20) 2 3 4 4 1 1 * * 1 

\end{tcolorbox}
\begin{tcolorbox}[title={$(51,18)$}]
(1) 3 5 6 .........4 5 3 3 3 $\leftarrow$ (2) 5 6 .........3 5 7 3 3 

(3) ...........3 3 6 6 5 3 $\leftarrow$ (4) 3 6 .........3 5 7 3 3 

(5) 5 8 1 1 * * 2 4 3 3 3 $\leftarrow$ (6) 7 13 1 * * * 1 

(6) ............3 5 7 3 3 $\leftarrow$ (11) 8 1 1 * * 2 4 3 3 3 

(7) 4 ...........4 5 3 3 3 $\leftarrow$ (8) 6 ..........4 5 3 3 3 

(9) ............4 5 3 3 3 $\leftarrow$ (10) 4 ..........4 5 3 3 3 

(13) 5 1 2 3 4 4 1 1 * * 1 $\leftarrow$ (14) 6 2 3 4 4 1 1 * * 1 

\end{tcolorbox}
\begin{tcolorbox}[title={$(51,19)$}]
(1) ............3 3 6 6 5 3 $\leftarrow$ (2) 3 5 6 .........4 5 3 3 3 

(4) .............3 5 7 3 3 $\leftarrow$ (10) ............4 5 3 3 3 

(5) 4 ............4 5 3 3 3 $\leftarrow$ (6) 5 8 1 1 * * 2 4 3 3 3 

(7) .............4 5 3 3 3 $\leftarrow$ (8) 4 ...........4 5 3 3 3 

(11) 1 * * * 2 4 3 3 3 $\leftarrow$ (13) * * * 2 4 3 3 3 

(13) 3 4 4 1 1 * * * 1 $\leftarrow$ (14) 5 1 2 3 4 4 1 1 * * 1 

\end{tcolorbox}
\begin{tcolorbox}[title={$(51,20)$}]
(1) 6 .............4 5 3 3 3 $\leftarrow$ (8) .............4 5 3 3 3 

(5) ..............4 5 3 3 3 $\leftarrow$ (6) 4 ............4 5 3 3 3 

(10) 1 1 * * * 2 4 3 3 3 $\leftarrow$ (12) 1 * * * 2 4 3 3 3 

(11) 2 3 4 4 1 1 * * * 1 $\leftarrow$ (14) 3 4 4 1 1 * * * 1 

\end{tcolorbox}
\begin{tcolorbox}[title={$(51,21)$}]
(1) 4 ..............4 5 3 3 3 $\leftarrow$ (2) 6 .............4 5 3 3 3 

(10) 1 2 3 4 4 1 1 * * * 1 $\leftarrow$ (12) 2 3 4 4 1 1 * * * 1 

\end{tcolorbox}
\begin{tcolorbox}[title={$(51,22)$}]
(1) ................4 5 3 3 3 $\leftarrow$ (2) 4 ..............4 5 3 3 3 

(5) 5 1 2 3 4 4 1 1 * * * 1 $\leftarrow$ (6) 6 2 3 4 4 1 1 * * * 1 

\end{tcolorbox}
\begin{tcolorbox}[title={$(51,23)$}]
(3) 1 * * * * 2 4 3 3 3 $\leftarrow$ (5) * * * * 2 4 3 3 3 

(5) 3 4 4 1 1 * * * * 1 $\leftarrow$ (6) 5 1 2 3 4 4 1 1 * * * 1 

\end{tcolorbox}
\begin{tcolorbox}[title={$(52,2)$}]
(51) 1 $\leftarrow$ (53) 

\end{tcolorbox}
\begin{tcolorbox}[title={$(52,3)$}]
(21) 30 1 $\leftarrow$ (22) 31 

(37) 14 1 $\leftarrow$ (38) 15 

(45) 6 1 $\leftarrow$ (46) 7 

(49) 2 1 $\leftarrow$ (50) 3 

(50) 1 1 $\leftarrow$ (52) 1 

\end{tcolorbox}
\begin{tcolorbox}[title={$(52,4)$}]
(7) 15 15 15 

(11) 11 15 15 $\leftarrow$ (39) 7 7 

(15) 23 7 7 $\leftarrow$ (19) 27 7 

(19) 27 3 3 $\leftarrow$ (21) 29 3 

(21) 29 1 1 $\leftarrow$ (22) 30 1 

(31) 7 7 7 $\leftarrow$ (35) 11 7 

(35) 11 3 3 $\leftarrow$ (37) 13 3 

(37) 13 1 1 $\leftarrow$ (38) 14 1 

(43) 3 3 3 $\leftarrow$ (45) 5 3 

(45) 5 1 1 $\leftarrow$ (46) 6 1 

(49) 1 1 1 $\leftarrow$ (50) 2 1 

\end{tcolorbox}
\begin{tcolorbox}[title={$(52,5)$}]
(2) 9 11 15 15 

(4) 7 11 15 15 

(7) 14 13 11 7 $\leftarrow$ (8) 15 15 15 

(11) 10 13 11 7 $\leftarrow$ (12) 11 15 15 

(12) 19 7 7 7 $\leftarrow$ (16) 23 7 7 

(18) 25 3 3 3 $\leftarrow$ (20) 27 3 3 

(19) 14 5 7 7 $\leftarrow$ (32) 7 7 7 

(21) 28 1 1 1 $\leftarrow$ (22) 29 1 1 

(30) 3 5 7 7 $\leftarrow$ (34) 5 7 7 

(34) 9 3 3 3 $\leftarrow$ (36) 11 3 3 

(37) 12 1 1 1 $\leftarrow$ (38) 13 1 1 

(43) 1 2 3 3 $\leftarrow$ (45) 2 3 3 

(45) 4 1 1 1 $\leftarrow$ (46) 5 1 1 

\end{tcolorbox}
\begin{tcolorbox}[title={$(52,6)$}]
(1) 7 13 13 11 7 

(3) 9 11 15 7 7 

(5) 7 11 15 7 7 

(7) 9 15 7 7 7 $\leftarrow$ (9) 13 13 11 7 

(7) 13 13 5 7 7 $\leftarrow$ (14) 20 5 7 7 

(11) 9 11 7 7 7 $\leftarrow$ (12) 10 13 11 7 

(13) 7 13 5 7 7 $\leftarrow$ (21) 13 5 7 7 

(15) 14 4 5 7 7 $\leftarrow$ (30) 4 5 7 7 

(19) 11 3 5 7 7 $\leftarrow$ (20) 14 5 7 7 

(21) 24 4 1 1 1 $\leftarrow$ (22) 28 1 1 1 

(29) 3 6 6 5 3 $\leftarrow$ (35) 5 7 3 3 

(31) 3 5 7 3 3 $\leftarrow$ (33) 6 6 5 3 

(35) 3 6 2 3 3 $\leftarrow$ (36) 8 3 3 3 

(37) 8 4 1 1 1 $\leftarrow$ (38) 12 1 1 1 

(41) 4 4 1 1 1 $\leftarrow$ (44) 1 2 3 3 

(43) * 1 $\leftarrow$ (46) 4 1 1 1 

\end{tcolorbox}
\begin{tcolorbox}[title={$(52,7)$}]
(1) 6 9 15 7 7 7 $\leftarrow$ (2) 7 13 13 11 7 

(2) 5 9 15 7 7 7 

(3) 8 9 11 7 7 7 $\leftarrow$ (4) 9 11 15 7 7 

(5) 6 9 11 7 7 7 $\leftarrow$ (6) 7 11 15 7 7 

(6) 11 5 9 7 7 7 $\leftarrow$ (13) 18 3 5 7 7 

(7) 12 11 3 5 7 7 $\leftarrow$ (8) 13 13 5 7 7 

(8) 7 14 4 5 7 7 $\leftarrow$ (29) 2 3 5 7 7 

(12) 5 5 9 7 7 7 $\leftarrow$ (20) 11 3 5 7 7 

(14) 3 5 9 7 7 7 $\leftarrow$ (18) 5 9 7 7 7 

(15) 13 2 3 5 7 7 $\leftarrow$ (16) 14 4 5 7 7 

(21) 22 * 1 $\leftarrow$ (22) 24 4 1 1 1 

(26) 3 3 6 6 5 3 $\leftarrow$ (32) 3 5 7 3 3 

(29) 2 3 5 7 3 3 $\leftarrow$ (30) 3 6 6 5 3 

(35) ..4 3 3 3 $\leftarrow$ (36) 3 6 2 3 3 

(37) 6 * 1 $\leftarrow$ (38) 8 4 1 1 1 

(41) 2 * 1 $\leftarrow$ (42) 4 4 1 1 1 

(42) 1 * 1 $\leftarrow$ (44) * 1 

\end{tcolorbox}
\begin{tcolorbox}[title={$(52,8)$}]
(1) 4 6 9 11 7 7 7 $\leftarrow$ (2) 6 9 15 7 7 7 

(4) 8 5 5 9 7 7 7 

(5) 9 3 5 9 7 7 7 $\leftarrow$ (11) 12 12 9 3 3 3 

(6) 8 3 5 9 7 7 7 $\leftarrow$ (23) 12 4 5 3 3 3 

(7) 7 8 12 9 3 3 3 $\leftarrow$ (8) 12 11 3 5 7 7 

(11) 11 12 4 5 3 3 3 $\leftarrow$ (13) 17 3 6 6 5 3 

(13) 3 5 9 3 5 7 7 $\leftarrow$ (19) 5 7 3 5 7 7 

(15) 3 5 7 3 5 7 7 $\leftarrow$ (17) 5 9 3 5 7 7 

(15) 7 12 4 5 3 3 3 $\leftarrow$ (16) 13 2 3 5 7 7 

(19) 5 5 3 6 6 5 3 $\leftarrow$ (21) 9 3 6 6 5 3 

(21) 21 1 * 1 $\leftarrow$ (22) 22 * 1 

(23) 6 2 3 5 7 3 3 $\leftarrow$ (30) 2 3 5 7 3 3 

(35) 1 1 2 4 3 3 3 $\leftarrow$ (36) ..4 3 3 3 

(37) 5 1 * 1 $\leftarrow$ (38) 6 * 1 

(41) 1 1 * 1 $\leftarrow$ (42) 2 * 1 

\end{tcolorbox}
\begin{tcolorbox}[title={$(52,9)$}]
(1) 3 5 10 11 3 5 7 7 

(5) 4 5 3 5 9 7 7 7 $\leftarrow$ (21) 6 3 3 6 6 5 3 

(6) 9 3 5 7 3 5 7 7 $\leftarrow$ (16) 3 5 7 3 5 7 7 

(10) 9 5 5 3 6 6 5 3 $\leftarrow$ (12) 11 12 4 5 3 3 3 

(15) 4 7 3 3 6 6 5 3 $\leftarrow$ (20) 5 5 3 6 6 5 3 

(15) 5 6 3 3 6 6 5 3 $\leftarrow$ (16) 7 12 4 5 3 3 3 

(17) 3 6 3 3 6 6 5 3 $\leftarrow$ (18) 6 5 2 3 5 7 7 

(21) 5 6 2 4 5 3 3 3 $\leftarrow$ (27) 6 2 4 5 3 3 3 

(21) 20 1 1 * 1 $\leftarrow$ (22) 21 1 * 1 

(23) 3 6 2 4 5 3 3 3 $\leftarrow$ (24) 6 2 3 5 7 3 3 

(29) 12 1 1 * 1 $\leftarrow$ (30) 13 1 * 1 

(35) 1 2 3 4 4 1 1 1 $\leftarrow$ (37) 2 3 4 4 1 1 1 

(37) 4 1 1 * 1 $\leftarrow$ (38) 5 1 * 1 

\end{tcolorbox}
\begin{tcolorbox}[title={$(52,10)$}]
(1) 5 9 3 5 7 3 5 7 7 

(4) 2 3 5 3 5 9 7 7 7 $\leftarrow$ (19) 5 6 2 3 5 7 3 3 

(4) 4 5 3 5 9 3 5 7 7 

(7) 5 5 6 5 2 3 5 7 7 $\leftarrow$ (13) 5 6 5 2 3 5 7 7 

(14) 2 3 5 5 3 6 6 5 3 $\leftarrow$ (16) 4 7 3 3 6 6 5 3 

(15) 3 5 6 2 3 5 7 3 3 $\leftarrow$ (16) 5 6 3 3 6 6 5 3 

(17) 2 4 3 3 3 6 6 5 3 $\leftarrow$ (18) 3 6 3 3 6 6 5 3 

(20) ...3 3 6 6 5 3 $\leftarrow$ (25) 6 ..4 5 3 3 3 

(21) 3 6 ..4 5 3 3 3 $\leftarrow$ (22) 5 6 2 4 5 3 3 3 

(21) 16 4 1 1 * 1 $\leftarrow$ (22) 20 1 1 * 1 

(23) ....3 5 7 3 3 $\leftarrow$ (24) 3 6 2 4 5 3 3 3 

(27) 3 6 2 3 4 4 1 1 1 $\leftarrow$ (28) 8 1 1 2 4 3 3 3 

(29) 8 4 1 1 * 1 $\leftarrow$ (30) 12 1 1 * 1 

(33) 4 4 1 1 * 1 $\leftarrow$ (36) 1 2 3 4 4 1 1 1 

(35) * * 1 $\leftarrow$ (38) 4 1 1 * 1 

\end{tcolorbox}
\begin{tcolorbox}[title={$(52,11)$}]
(2) 5 5 5 6 5 2 3 5 7 7 

(5) 4 3 5 6 5 2 3 5 7 7 $\leftarrow$ (8) 5 5 6 5 2 3 5 7 7 

(11) ..4 7 3 3 6 6 5 3 $\leftarrow$ (14) 2 4 7 3 3 6 6 5 3 

(15) ..4 3 3 3 6 6 5 3 $\leftarrow$ (16) 3 5 6 2 3 5 7 3 3 

(18) ....3 3 6 6 5 3 $\leftarrow$ (24) ....3 5 7 3 3 

(21) .....3 5 7 3 3 $\leftarrow$ (22) 3 6 ..4 5 3 3 3 

(21) 14 * * 1 $\leftarrow$ (22) 16 4 1 1 * 1 

(27) 2 * 2 4 3 3 3 $\leftarrow$ (28) 3 6 2 3 4 4 1 1 1 

(29) 6 * * 1 $\leftarrow$ (30) 8 4 1 1 * 1 

(33) 2 * * 1 $\leftarrow$ (34) 4 4 1 1 * 1 

(34) 1 * * 1 $\leftarrow$ (36) * * 1 

\end{tcolorbox}
\begin{tcolorbox}[title={$(52,12)$}]
(3) 2 4 3 5 6 5 2 3 5 7 7 $\leftarrow$ (6) 4 3 5 6 5 2 3 5 7 7 

(9) 6 ..4 3 3 3 6 6 5 3 $\leftarrow$ (13) 1 2 4 7 3 3 6 6 5 3 

(11) 1 1 2 4 7 3 3 6 6 5 3 $\leftarrow$ (12) ..4 7 3 3 6 6 5 3 

(15) 6 .....3 5 7 3 3 $\leftarrow$ (22) .....3 5 7 3 3 

(21) 13 1 * * 1 $\leftarrow$ (22) 14 * * 1 

(27) 1 1 * 2 4 3 3 3 $\leftarrow$ (28) 2 * 2 4 3 3 3 

(29) 5 1 * * 1 $\leftarrow$ (30) 6 * * 1 

(33) 1 1 * * 1 $\leftarrow$ (34) 2 * * 1 

\end{tcolorbox}
\begin{tcolorbox}[title={$(52,13)$}]
(1) ..4 3 5 6 5 2 3 5 7 7 $\leftarrow$ (4) 2 4 3 5 6 5 2 3 5 7 7 

(7) 4 1 1 2 4 7 3 3 6 6 5 3 $\leftarrow$ (12) 1 1 2 4 7 3 3 6 6 5 3 

(9) 3 6 ....3 3 6 6 5 3 $\leftarrow$ (10) 6 ..4 3 3 3 6 6 5 3 

(13) 5 6 .....4 5 3 3 3 $\leftarrow$ (19) 6 .....4 5 3 3 3 

(15) 3 6 .....4 5 3 3 3 $\leftarrow$ (16) 6 .....3 5 7 3 3 

(21) 12 1 1 * * 1 $\leftarrow$ (22) 13 1 * * 1 

(27) 1 2 3 4 4 1 1 * 1 $\leftarrow$ (29) 2 3 4 4 1 1 * 1 

(29) 4 1 1 * * 1 $\leftarrow$ (30) 5 1 * * 1 

\end{tcolorbox}
\begin{tcolorbox}[title={$(52,14)$}]
(2) ....4 5 5 5 3 6 6 5 3 

(5) 24 4 1 1 * * 1 $\leftarrow$ (8) 4 1 1 2 4 7 3 3 6 6 5 3 

(9) .....4 3 3 3 6 6 5 3 $\leftarrow$ (10) 3 6 ....3 3 6 6 5 3 

(12) .......3 3 6 6 5 3 $\leftarrow$ (17) 6 ......4 5 3 3 3 

(13) 3 6 ......4 5 3 3 3 $\leftarrow$ (14) 5 6 .....4 5 3 3 3 

(15) ........3 5 7 3 3 $\leftarrow$ (16) 3 6 .....4 5 3 3 3 

(19) 3 6 2 3 4 4 1 1 * 1 $\leftarrow$ (20) 8 1 1 * 2 4 3 3 3 

(21) 8 4 1 1 * * 1 $\leftarrow$ (22) 12 1 1 * * 1 

(25) 4 4 1 1 * * 1 $\leftarrow$ (28) 1 2 3 4 4 1 1 * 1 

(27) * * * 1 $\leftarrow$ (30) 4 1 1 * * 1 

\end{tcolorbox}
\begin{tcolorbox}[title={$(52,15)$}]
(3) 2 * 2 4 7 3 3 6 6 5 3 $\leftarrow$ (6) * 2 4 7 3 3 6 6 5 3 

(5) 22 * * * 1 $\leftarrow$ (6) 24 4 1 1 * * 1 

(10) ........3 3 6 6 5 3 $\leftarrow$ (16) ........3 5 7 3 3 

(13) .........3 5 7 3 3 $\leftarrow$ (14) 3 6 ......4 5 3 3 3 

(19) 2 * * 2 4 3 3 3 $\leftarrow$ (20) 3 6 2 3 4 4 1 1 * 1 

(21) 6 * * * 1 $\leftarrow$ (22) 8 4 1 1 * * 1 

(25) 2 * * * 1 $\leftarrow$ (26) 4 4 1 1 * * 1 

(26) 1 * * * 1 $\leftarrow$ (28) * * * 1 

\end{tcolorbox}
\begin{tcolorbox}[title={$(52,16)$}]
(1) 6 ......4 3 3 3 6 6 5 3 $\leftarrow$ (5) 1 * 2 4 7 3 3 6 6 5 3 

(3) 1 1 * 2 4 7 3 3 6 6 5 3 $\leftarrow$ (4) 2 * 2 4 7 3 3 6 6 5 3 

(5) 21 1 * * * 1 $\leftarrow$ (6) 22 * * * 1 

(7) 6 .........3 5 7 3 3 $\leftarrow$ (14) .........3 5 7 3 3 

(19) 1 1 * * 2 4 3 3 3 $\leftarrow$ (20) 2 * * 2 4 3 3 3 

(21) 5 1 * * * 1 $\leftarrow$ (22) 6 * * * 1 

(25) 1 1 * * * 1 $\leftarrow$ (26) 2 * * * 1 

\end{tcolorbox}
\begin{tcolorbox}[title={$(52,17)$}]
(1) 3 6 ........3 3 6 6 5 3 $\leftarrow$ (2) 6 ......4 3 3 3 6 6 5 3 

(5) 5 6 .........4 5 3 3 3 $\leftarrow$ (11) 6 .........4 5 3 3 3 

(5) 20 1 1 * * * 1 $\leftarrow$ (6) 21 1 * * * 1 

(7) 3 6 .........4 5 3 3 3 $\leftarrow$ (8) 6 .........3 5 7 3 3 

(13) 12 1 1 * * * 1 $\leftarrow$ (14) 13 1 * * * 1 

(19) 1 2 3 4 4 1 1 * * 1 $\leftarrow$ (21) 2 3 4 4 1 1 * * 1 

(21) 4 1 1 * * * 1 $\leftarrow$ (22) 5 1 * * * 1 

\end{tcolorbox}
\begin{tcolorbox}[title={$(52,18)$}]
(1) .........4 3 3 3 6 6 5 3 $\leftarrow$ (2) 3 6 ........3 3 6 6 5 3 

(4) ...........3 3 6 6 5 3 $\leftarrow$ (9) 6 ..........4 5 3 3 3 

(5) 3 6 ..........4 5 3 3 3 $\leftarrow$ (6) 5 6 .........4 5 3 3 3 

(5) 16 4 1 1 * * * 1 $\leftarrow$ (6) 20 1 1 * * * 1 

(7) ............3 5 7 3 3 $\leftarrow$ (8) 3 6 .........4 5 3 3 3 

(11) 3 6 2 3 4 4 1 1 * * 1 $\leftarrow$ (12) 8 1 1 * * 2 4 3 3 3 

(13) 8 4 1 1 * * * 1 $\leftarrow$ (14) 12 1 1 * * * 1 

(17) 4 4 1 1 * * * 1 $\leftarrow$ (20) 1 2 3 4 4 1 1 * * 1 

(19) * * * * 1 $\leftarrow$ (22) 4 1 1 * * * 1 

\end{tcolorbox}
\begin{tcolorbox}[title={$(52,19)$}]
(2) ............3 3 6 6 5 3 $\leftarrow$ (8) ............3 5 7 3 3 

(5) .............3 5 7 3 3 $\leftarrow$ (6) 3 6 ..........4 5 3 3 3 

(5) 14 * * * * 1 $\leftarrow$ (6) 16 4 1 1 * * * 1 

(11) 2 * * * 2 4 3 3 3 $\leftarrow$ (12) 3 6 2 3 4 4 1 1 * * 1 

(13) 6 * * * * 1 $\leftarrow$ (14) 8 4 1 1 * * * 1 

(17) 2 * * * * 1 $\leftarrow$ (18) 4 4 1 1 * * * 1 

(18) 1 * * * * 1 $\leftarrow$ (20) * * * * 1 

\end{tcolorbox}
\begin{tcolorbox}[title={$(52,20)$}]
(5) 13 1 * * * * 1 $\leftarrow$ (6) 14 * * * * 1 

(6) ..............4 5 3 3 3 

(11) 1 1 * * * 2 4 3 3 3 $\leftarrow$ (12) 2 * * * 2 4 3 3 3 

(13) 5 1 * * * * 1 $\leftarrow$ (14) 6 * * * * 1 

(17) 1 1 * * * * 1 $\leftarrow$ (18) 2 * * * * 1 

\end{tcolorbox}
\begin{tcolorbox}[title={$(52,21)$}]
(3) 8 1 1 * * * 2 4 3 3 3 

(5) 12 1 1 * * * * 1 $\leftarrow$ (6) 13 1 * * * * 1 

(11) 1 2 3 4 4 1 1 * * * 1 $\leftarrow$ (13) 2 3 4 4 1 1 * * * 1 

(13) 4 1 1 * * * * 1 $\leftarrow$ (14) 5 1 * * * * 1 

\end{tcolorbox}
\begin{tcolorbox}[title={$(52,22)$}]
(2) ................4 5 3 3 3 

(3) 3 6 2 3 4 4 1 1 * * * 1 $\leftarrow$ (4) 8 1 1 * * * 2 4 3 3 3 

(5) 8 4 1 1 * * * * 1 $\leftarrow$ (6) 12 1 1 * * * * 1 

(9) 4 4 1 1 * * * * 1 $\leftarrow$ (12) 1 2 3 4 4 1 1 * * * 1 

(11) * * * * * 1 $\leftarrow$ (14) 4 1 1 * * * * 1 

\end{tcolorbox}
\begin{tcolorbox}[title={$(53,3)$}]
(15) 23 15 $\leftarrow$ (23) 31 

(23) 15 15 $\leftarrow$ (39) 15 

(47) 3 3 $\leftarrow$ (51) 3 

(51) 1 1 $\leftarrow$ (53) 1 

\end{tcolorbox}
\begin{tcolorbox}[title={$(53,4)$}]
(14) 21 11 7 $\leftarrow$ (22) 29 3 

(15) 22 13 3 $\leftarrow$ (16) 23 15 

(19) 26 5 3 $\leftarrow$ (20) 27 7 

(22) 13 11 7 $\leftarrow$ (38) 13 3 

(23) 14 13 3 $\leftarrow$ (24) 15 15 

(35) 10 5 3 $\leftarrow$ (36) 11 7 

(39) 6 5 3 $\leftarrow$ (40) 7 7 

(44) 3 3 3 $\leftarrow$ (48) 3 3 

(50) 1 1 1 $\leftarrow$ (52) 1 1 

\end{tcolorbox}
\begin{tcolorbox}[title={$(53,5)$}]
(3) 9 11 15 15 

(5) 7 11 15 15 

(8) 14 13 11 7 $\leftarrow$ (21) 27 3 3 

(13) 11 15 7 7 $\leftarrow$ (37) 11 3 3 

(13) 19 7 7 7 $\leftarrow$ (17) 23 7 7 

(15) 17 7 7 7 $\leftarrow$ (33) 7 7 7 

(15) 21 11 3 3 $\leftarrow$ (16) 22 13 3 

(19) 25 3 3 3 $\leftarrow$ (20) 26 5 3 

(23) 13 11 3 3 $\leftarrow$ (24) 14 13 3 

(31) 3 5 7 7 $\leftarrow$ (35) 5 7 7 

(35) 9 3 3 3 $\leftarrow$ (36) 10 5 3 

(39) 6 2 3 3 $\leftarrow$ (46) 2 3 3 

\end{tcolorbox}
\begin{tcolorbox}[title={$(53,6)$}]
(1) 4 7 11 15 15 

(5) 10 17 7 7 7 $\leftarrow$ (20) 25 3 3 3 

(8) 9 15 7 7 7 $\leftarrow$ (36) 9 3 3 3 

(9) 11 14 5 7 7 $\leftarrow$ (10) 13 13 11 7 

(12) 9 11 7 7 7 $\leftarrow$ (22) 13 5 7 7 

(13) 7 14 5 7 7 $\leftarrow$ (14) 11 15 7 7 

(14) 7 13 5 7 7 $\leftarrow$ (16) 17 7 7 7 

(15) 20 9 3 3 3 $\leftarrow$ (16) 21 11 3 3 

(22) 9 3 5 7 7 $\leftarrow$ (32) 3 5 7 7 

(23) 12 9 3 3 3 $\leftarrow$ (24) 13 11 3 3 

(33) 4 7 3 3 3 $\leftarrow$ (34) 6 6 5 3 

(35) 4 5 3 3 3 $\leftarrow$ (36) 5 7 3 3 

(38) 2 4 3 3 3 $\leftarrow$ (45) 1 2 3 3 

(39) 5 1 2 3 3 $\leftarrow$ (40) 6 2 3 3 

\end{tcolorbox}
\begin{tcolorbox}[title={$(53,7)$}]
(1) 5 7 11 15 7 7 

(3) 5 9 15 7 7 7 

(4) 8 9 11 7 7 7 $\leftarrow$ (16) 20 9 3 3 3 

(5) 9 7 13 5 7 7 $\leftarrow$ (6) 10 17 7 7 7 

(6) 6 9 11 7 7 7 $\leftarrow$ (19) 5 9 7 7 7 

(7) 11 5 9 7 7 7 $\leftarrow$ (9) 13 13 5 7 7 

(9) 7 14 4 5 7 7 $\leftarrow$ (10) 11 14 5 7 7 

(13) 5 5 9 7 7 7 $\leftarrow$ (14) 7 14 5 7 7 

(15) 3 5 9 7 7 7 $\leftarrow$ (17) 14 4 5 7 7 

(15) 8 12 9 3 3 3 $\leftarrow$ (24) 12 9 3 3 3 

(27) 3 3 6 6 5 3 $\leftarrow$ (33) 3 5 7 3 3 

(33) 2 4 5 3 3 3 $\leftarrow$ (34) 4 7 3 3 3 

(37) 1 2 4 3 3 3 $\leftarrow$ (43) 4 4 1 1 1 

(39) 3 4 4 1 1 1 $\leftarrow$ (40) 5 1 2 3 3 

(43) 1 * 1 $\leftarrow$ (45) * 1 

\end{tcolorbox}
\begin{tcolorbox}[title={$(53,8)$}]
(2) 4 6 9 11 7 7 7 $\leftarrow$ (14) 17 3 6 6 5 3 

(5) 5 10 11 3 5 7 7 $\leftarrow$ (16) 3 5 9 7 7 7 

(5) 8 5 5 9 7 7 7 $\leftarrow$ (6) 9 7 13 5 7 7 

(6) 9 3 5 9 7 7 7 $\leftarrow$ (8) 11 5 9 7 7 7 

(7) 8 3 5 9 7 7 7 $\leftarrow$ (10) 7 14 4 5 7 7 

(8) 7 8 12 9 3 3 3 $\leftarrow$ (22) 9 3 6 6 5 3 

(11) 10 9 3 6 6 5 3 $\leftarrow$ (12) 12 12 9 3 3 3 

(14) 3 5 9 3 5 7 7 $\leftarrow$ (18) 5 9 3 5 7 7 

(15) 6 9 3 6 6 5 3 $\leftarrow$ (16) 8 12 9 3 3 3 

(23) 10 2 4 5 3 3 3 $\leftarrow$ (24) 12 4 5 3 3 3 

(36) 1 1 2 4 3 3 3 $\leftarrow$ (40) 3 4 4 1 1 1 

(42) 1 1 * 1 $\leftarrow$ (44) 1 * 1 

\end{tcolorbox}
\begin{tcolorbox}[title={$(53,9)$}]
(1) 4 8 5 5 9 7 7 7 

(2) 3 5 10 11 3 5 7 7 

(5) 4 6 3 5 9 7 7 7 $\leftarrow$ (6) 5 10 11 3 5 7 7 

(6) 4 5 3 5 9 7 7 7 $\leftarrow$ (8) 8 3 5 9 7 7 7 

(6) 8 3 5 9 3 5 7 7 $\leftarrow$ (19) 6 5 2 3 5 7 7 

(7) 9 3 5 7 3 5 7 7 $\leftarrow$ (17) 3 5 7 3 5 7 7 

(11) 9 5 5 3 6 6 5 3 $\leftarrow$ (12) 10 9 3 6 6 5 3 

(15) 3 6 5 2 3 5 7 7 $\leftarrow$ (16) 6 9 3 6 6 5 3 

(21) 3 6 2 3 5 7 3 3 $\leftarrow$ (22) 6 3 3 6 6 5 3 

(23) 7 13 1 * 1 $\leftarrow$ (24) 10 2 4 5 3 3 3 

(27) 4 ..4 5 3 3 3 $\leftarrow$ (28) 6 2 4 5 3 3 3 

(31) 6 2 3 4 4 1 1 1 $\leftarrow$ (38) 2 3 4 4 1 1 1 

\end{tcolorbox}
\begin{tcolorbox}[title={$(53,10)$}]
(2) 5 9 3 5 7 3 5 7 7 $\leftarrow$ (8) 9 3 5 7 3 5 7 7 

(5) 2 3 5 3 5 9 7 7 7 $\leftarrow$ (6) 4 6 3 5 9 7 7 7 

(5) 4 5 3 5 9 3 5 7 7 $\leftarrow$ (17) 4 7 3 3 6 6 5 3 

(11) 4 5 5 5 3 6 6 5 3 $\leftarrow$ (14) 5 6 5 2 3 5 7 7 

(15) 2 3 5 5 3 6 6 5 3 $\leftarrow$ (16) 3 6 5 2 3 5 7 7 

(18) 2 4 3 3 3 6 6 5 3 

(19) 3 5 6 2 4 5 3 3 3 $\leftarrow$ (20) 5 6 2 3 5 7 3 3 

(21) ...3 3 6 6 5 3 $\leftarrow$ (22) 3 6 2 3 5 7 3 3 

(23) 5 8 1 1 2 4 3 3 3 $\leftarrow$ (24) 7 13 1 * 1 

(25) 4 ...4 5 3 3 3 $\leftarrow$ (26) 6 ..4 5 3 3 3 

(27) ....4 5 3 3 3 $\leftarrow$ (28) 4 ..4 5 3 3 3 

(30) * 2 4 3 3 3 $\leftarrow$ (37) 1 2 3 4 4 1 1 1 

(31) 5 1 2 3 4 4 1 1 1 $\leftarrow$ (32) 6 2 3 4 4 1 1 1 

\end{tcolorbox}
\begin{tcolorbox}[title={$(53,11)$}]
(3) 5 5 5 6 5 2 3 5 7 7 $\leftarrow$ (9) 5 5 6 5 2 3 5 7 7 

(9) 2 4 5 5 5 3 6 6 5 3 $\leftarrow$ (12) 4 5 5 5 3 6 6 5 3 

(16) ..4 3 3 3 6 6 5 3 

(19) ....3 3 6 6 5 3 $\leftarrow$ (20) 3 5 6 2 4 5 3 3 3 

(23) 4 ....4 5 3 3 3 $\leftarrow$ (24) 5 8 1 1 2 4 3 3 3 

(25) .....4 5 3 3 3 $\leftarrow$ (26) 4 ...4 5 3 3 3 

(29) 1 * 2 4 3 3 3 $\leftarrow$ (35) 4 4 1 1 * 1 

(31) 3 4 4 1 1 * 1 $\leftarrow$ (32) 5 1 2 3 4 4 1 1 1 

(35) 1 * * 1 $\leftarrow$ (37) * * 1 

\end{tcolorbox}
\begin{tcolorbox}[title={$(53,12)$}]
(1) 4 3 5 5 6 5 2 3 5 7 7 $\leftarrow$ (4) 5 5 5 6 5 2 3 5 7 7 

(7) ..4 5 5 5 3 6 6 5 3 $\leftarrow$ (10) 2 4 5 5 5 3 6 6 5 3 

(13) 6 ....3 3 6 6 5 3 

(23) ......4 5 3 3 3 $\leftarrow$ (24) 4 ....4 5 3 3 3 

(28) 1 1 * 2 4 3 3 3 $\leftarrow$ (32) 3 4 4 1 1 * 1 

(34) 1 1 * * 1 $\leftarrow$ (36) 1 * * 1 

\end{tcolorbox}
\begin{tcolorbox}[title={$(53,13)$}]
(2) ..4 3 5 6 5 2 3 5 7 7 

(5) ...4 5 5 5 3 6 6 5 3 $\leftarrow$ (8) ..4 5 5 5 3 6 6 5 3 

(11) 5 6 .....3 5 7 3 3 

(13) 3 6 .....3 5 7 3 3 $\leftarrow$ (14) 6 ....3 3 6 6 5 3 

(19) 4 ......4 5 3 3 3 $\leftarrow$ (20) 6 .....4 5 3 3 3 

(23) 6 2 3 4 4 1 1 * 1 $\leftarrow$ (30) 2 3 4 4 1 1 * 1 

\end{tcolorbox}
\begin{tcolorbox}[title={$(53,14)$}]
(3) ....4 5 5 5 3 6 6 5 3 $\leftarrow$ (6) ...4 5 5 5 3 6 6 5 3 

(10) .....4 3 3 3 6 6 5 3 

(11) 3 5 6 .....4 5 3 3 3 $\leftarrow$ (12) 5 6 .....3 5 7 3 3 

(13) .......3 3 6 6 5 3 $\leftarrow$ (14) 3 6 .....3 5 7 3 3 

(17) 4 .......4 5 3 3 3 $\leftarrow$ (18) 6 ......4 5 3 3 3 

(19) ........4 5 3 3 3 $\leftarrow$ (20) 4 ......4 5 3 3 3 

(22) * * 2 4 3 3 3 $\leftarrow$ (29) 1 2 3 4 4 1 1 * 1 

(23) 5 1 2 3 4 4 1 1 * 1 $\leftarrow$ (24) 6 2 3 4 4 1 1 * 1 

\end{tcolorbox}
\begin{tcolorbox}[title={$(53,15)$}]
(1) .....4 5 5 5 3 6 6 5 3 $\leftarrow$ (4) ....4 5 5 5 3 6 6 5 3 

(7) 1.........4 5 3 3 3 

(11) ........3 3 6 6 5 3 $\leftarrow$ (12) 3 5 6 .....4 5 3 3 3 

(17) .........4 5 3 3 3 $\leftarrow$ (18) 4 .......4 5 3 3 3 

(21) 1 * * 2 4 3 3 3 $\leftarrow$ (27) 4 4 1 1 * * 1 

(23) 3 4 4 1 1 * * 1 $\leftarrow$ (24) 5 1 2 3 4 4 1 1 * 1 

(27) 1 * * * 1 $\leftarrow$ (29) * * * 1 

\end{tcolorbox}
\begin{tcolorbox}[title={$(53,16)$}]
(4) 1 1 * 2 4 7 3 3 6 6 5 3 

(5) 6 ........3 3 6 6 5 3 

(7) 10 .........4 5 3 3 3 $\leftarrow$ (8) 1.........4 5 3 3 3 

(20) 1 1 * * 2 4 3 3 3 $\leftarrow$ (24) 3 4 4 1 1 * * 1 

(26) 1 1 * * * 1 $\leftarrow$ (28) 1 * * * 1 

\end{tcolorbox}
\begin{tcolorbox}[title={$(53,17)$}]
(3) 5 6 .........3 5 7 3 3 

(5) 3 6 .........3 5 7 3 3 $\leftarrow$ (6) 6 ........3 3 6 6 5 3 

(7) 7 13 1 * * * 1 $\leftarrow$ (8) 10 .........4 5 3 3 3 

(11) 4 ..........4 5 3 3 3 $\leftarrow$ (12) 6 .........4 5 3 3 3 

(15) 6 2 3 4 4 1 1 * * 1 $\leftarrow$ (22) 2 3 4 4 1 1 * * 1 

\end{tcolorbox}
\begin{tcolorbox}[title={$(53,18)$}]
(2) .........4 3 3 3 6 6 5 3 

(3) 3 5 6 .........4 5 3 3 3 $\leftarrow$ (4) 5 6 .........3 5 7 3 3 

(5) ...........3 3 6 6 5 3 $\leftarrow$ (6) 3 6 .........3 5 7 3 3 

(7) 5 8 1 1 * * 2 4 3 3 3 $\leftarrow$ (8) 7 13 1 * * * 1 

(9) 4 ...........4 5 3 3 3 $\leftarrow$ (10) 6 ..........4 5 3 3 3 

(11) ............4 5 3 3 3 $\leftarrow$ (12) 4 ..........4 5 3 3 3 

(14) * * * 2 4 3 3 3 $\leftarrow$ (21) 1 2 3 4 4 1 1 * * 1 

(15) 5 1 2 3 4 4 1 1 * * 1 $\leftarrow$ (16) 6 2 3 4 4 1 1 * * 1 

\end{tcolorbox}
\begin{tcolorbox}[title={$(53,19)$}]
(3) ............3 3 6 6 5 3 $\leftarrow$ (4) 3 5 6 .........4 5 3 3 3 

(6) .............3 5 7 3 3 

(7) 4 ............4 5 3 3 3 $\leftarrow$ (8) 5 8 1 1 * * 2 4 3 3 3 

(9) .............4 5 3 3 3 $\leftarrow$ (10) 4 ...........4 5 3 3 3 

(13) 1 * * * 2 4 3 3 3 $\leftarrow$ (19) 4 4 1 1 * * * 1 

(15) 3 4 4 1 1 * * * 1 $\leftarrow$ (16) 5 1 2 3 4 4 1 1 * * 1 

(19) 1 * * * * 1 $\leftarrow$ (21) * * * * 1 

\end{tcolorbox}
\begin{tcolorbox}[title={$(53,20)$}]
(3) 6 .............4 5 3 3 3 

(7) ..............4 5 3 3 3 $\leftarrow$ (8) 4 ............4 5 3 3 3 

(12) 1 1 * * * 2 4 3 3 3 $\leftarrow$ (16) 3 4 4 1 1 * * * 1 

(18) 1 1 * * * * 1 $\leftarrow$ (20) 1 * * * * 1 

\end{tcolorbox}
\begin{tcolorbox}[title={$(53,21)$}]
(1) 6 ..............4 5 3 3 3 

(3) 4 ..............4 5 3 3 3 $\leftarrow$ (4) 6 .............4 5 3 3 3 

(7) 6 2 3 4 4 1 1 * * * 1 $\leftarrow$ (14) 2 3 4 4 1 1 * * * 1 

\end{tcolorbox}
\begin{tcolorbox}[title={$(54,2)$}]
(47) 7 $\leftarrow$ (55) 

\end{tcolorbox}
\begin{tcolorbox}[title={$(54,3)$}]
(23) 30 1 $\leftarrow$ (24) 31 

(39) 14 1 $\leftarrow$ (40) 15 

(46) 5 3 $\leftarrow$ (54) 1 

(47) 6 1 $\leftarrow$ (48) 7 

(51) 2 1 $\leftarrow$ (52) 3 

\end{tcolorbox}
\begin{tcolorbox}[title={$(54,4)$}]
(9) 15 15 15 $\leftarrow$ (21) 27 7 

(13) 11 15 15 $\leftarrow$ (37) 11 7 

(15) 21 11 7 $\leftarrow$ (17) 23 15 

(23) 13 11 7 $\leftarrow$ (25) 15 15 

(23) 29 1 1 $\leftarrow$ (24) 30 1 

(39) 13 1 1 $\leftarrow$ (40) 14 1 

(40) 6 5 3 $\leftarrow$ (53) 1 1 

(45) 3 3 3 $\leftarrow$ (49) 3 3 

(47) 5 1 1 $\leftarrow$ (48) 6 1 

(51) 1 1 1 $\leftarrow$ (52) 2 1 

\end{tcolorbox}
\begin{tcolorbox}[title={$(54,5)$}]
(4) 9 11 15 15 

(6) 7 11 15 15 $\leftarrow$ (36) 5 7 7 

(9) 14 13 11 7 $\leftarrow$ (10) 15 15 15 

(13) 10 13 11 7 $\leftarrow$ (14) 11 15 15 

(14) 19 7 7 7 $\leftarrow$ (16) 21 11 7 

(15) 20 5 7 7 $\leftarrow$ (18) 23 7 7 

(21) 14 5 7 7 $\leftarrow$ (24) 13 11 7 

(23) 28 1 1 1 $\leftarrow$ (24) 29 1 1 

(31) 4 5 7 7 $\leftarrow$ (34) 7 7 7 

(37) 8 3 3 3 $\leftarrow$ (52) 1 1 1 

(39) 12 1 1 1 $\leftarrow$ (40) 13 1 1 

(47) 4 1 1 1 $\leftarrow$ (48) 5 1 1 

\end{tcolorbox}
\begin{tcolorbox}[title={$(54,6)$}]
(1) 5 7 11 15 15 

(2) 4 7 11 15 15 

(3) 7 13 13 11 7 

(5) 9 11 15 7 7 $\leftarrow$ (10) 14 13 11 7 

(7) 7 11 15 7 7 $\leftarrow$ (11) 13 13 11 7 

(9) 9 15 7 7 7 $\leftarrow$ (23) 13 5 7 7 

(13) 9 11 7 7 7 $\leftarrow$ (14) 10 13 11 7 

(14) 18 3 5 7 7 $\leftarrow$ (16) 20 5 7 7 

(15) 7 13 5 7 7 $\leftarrow$ (17) 17 7 7 7 

(21) 11 3 5 7 7 $\leftarrow$ (22) 14 5 7 7 

(23) 9 3 5 7 7 $\leftarrow$ (33) 3 5 7 7 

(23) 24 4 1 1 1 $\leftarrow$ (24) 28 1 1 1 

(30) 2 3 5 7 7 $\leftarrow$ (32) 4 5 7 7 

(31) 3 6 6 5 3 $\leftarrow$ (35) 6 6 5 3 

(36) 4 5 3 3 3 $\leftarrow$ (48) 4 1 1 1 

(37) 3 6 2 3 3 $\leftarrow$ (38) 8 3 3 3 

(39) 2 4 3 3 3 $\leftarrow$ (41) 6 2 3 3 

(39) 8 4 1 1 1 $\leftarrow$ (40) 12 1 1 1 

\end{tcolorbox}
\begin{tcolorbox}[title={$(54,7)$}]
(2) 5 7 11 15 7 7 

(3) 6 9 15 7 7 7 $\leftarrow$ (4) 7 13 13 11 7 

(4) 5 9 15 7 7 7 $\leftarrow$ (20) 5 9 7 7 7 

(5) 8 9 11 7 7 7 $\leftarrow$ (6) 9 11 15 7 7 

(7) 6 9 11 7 7 7 $\leftarrow$ (8) 7 11 15 7 7 

(9) 12 11 3 5 7 7 $\leftarrow$ (10) 13 13 5 7 7 

(14) 5 5 9 7 7 7 $\leftarrow$ (16) 7 13 5 7 7 

(17) 13 2 3 5 7 7 $\leftarrow$ (18) 14 4 5 7 7 

(20) 5 7 3 5 7 7 $\leftarrow$ (24) 9 3 5 7 7 

(23) 22 * 1 $\leftarrow$ (24) 24 4 1 1 1 

(28) 3 3 6 6 5 3 $\leftarrow$ (34) 3 5 7 3 3 

(31) 2 3 5 7 3 3 $\leftarrow$ (32) 3 6 6 5 3 

(34) 2 4 5 3 3 3 $\leftarrow$ (46) * 1 

(37) ..4 3 3 3 $\leftarrow$ (38) 3 6 2 3 3 

(38) 1 2 4 3 3 3 $\leftarrow$ (40) 2 4 3 3 3 

(39) 6 * 1 $\leftarrow$ (40) 8 4 1 1 1 

(43) 2 * 1 $\leftarrow$ (44) 4 4 1 1 1 

\end{tcolorbox}
\begin{tcolorbox}[title={$(54,8)$}]
(1) 3 5 9 15 7 7 7 

(3) 4 6 9 11 7 7 7 $\leftarrow$ (4) 6 9 15 7 7 7 

(6) 8 5 5 9 7 7 7 $\leftarrow$ (19) 5 9 3 5 7 7 

(7) 9 3 5 9 7 7 7 $\leftarrow$ (9) 11 5 9 7 7 7 

(9) 7 8 12 9 3 3 3 $\leftarrow$ (10) 12 11 3 5 7 7 

(13) 11 12 4 5 3 3 3 

(15) 3 5 9 3 5 7 7 $\leftarrow$ (17) 8 12 9 3 3 3 

(17) 7 12 4 5 3 3 3 $\leftarrow$ (18) 13 2 3 5 7 7 

(21) 5 5 3 6 6 5 3 $\leftarrow$ (25) 12 4 5 3 3 3 

(23) 21 1 * 1 $\leftarrow$ (24) 22 * 1 

(25) 6 2 3 5 7 3 3 $\leftarrow$ (32) 2 3 5 7 3 3 

(31) 13 1 * 1 $\leftarrow$ (45) 1 * 1 

(37) 1 1 2 4 3 3 3 $\leftarrow$ (38) ..4 3 3 3 

(39) 5 1 * 1 $\leftarrow$ (40) 6 * 1 

(43) 1 1 * 1 $\leftarrow$ (44) 2 * 1 

\end{tcolorbox}
\begin{tcolorbox}[title={$(54,9)$}]
(2) 4 8 5 5 9 7 7 7 $\leftarrow$ (16) 3 5 9 3 5 7 7 

(3) 3 5 10 11 3 5 7 7 $\leftarrow$ (8) 9 3 5 9 7 7 7 

(7) 4 5 3 5 9 7 7 7 $\leftarrow$ (9) 8 3 5 9 7 7 7 

(7) 8 3 5 9 3 5 7 7 $\leftarrow$ (10) 7 8 12 9 3 3 3 

(12) 9 5 5 3 6 6 5 3 

(17) 5 6 3 3 6 6 5 3 $\leftarrow$ (18) 7 12 4 5 3 3 3 

(19) 3 6 3 3 6 6 5 3 $\leftarrow$ (20) 6 5 2 3 5 7 7 

(23) 5 6 2 4 5 3 3 3 $\leftarrow$ (29) 6 2 4 5 3 3 3 

(23) 20 1 1 * 1 $\leftarrow$ (24) 21 1 * 1 

(25) 3 6 2 4 5 3 3 3 $\leftarrow$ (26) 6 2 3 5 7 3 3 

(29) 8 1 1 2 4 3 3 3 $\leftarrow$ (44) 1 1 * 1 

(31) 12 1 1 * 1 $\leftarrow$ (32) 13 1 * 1 

(39) 4 1 1 * 1 $\leftarrow$ (40) 5 1 * 1 

\end{tcolorbox}
\begin{tcolorbox}[title={$(54,10)$}]
(1) 2 3 5 10 11 3 5 7 7 

(3) 5 9 3 5 7 3 5 7 7 $\leftarrow$ (9) 9 3 5 7 3 5 7 7 

(6) 2 3 5 3 5 9 7 7 7 $\leftarrow$ (8) 4 5 3 5 9 7 7 7 

(6) 4 5 3 5 9 3 5 7 7 $\leftarrow$ (8) 8 3 5 9 3 5 7 7 

(15) 2 4 7 3 3 6 6 5 3 $\leftarrow$ (18) 4 7 3 3 6 6 5 3 

(16) 2 3 5 5 3 6 6 5 3 

(17) 3 5 6 2 3 5 7 3 3 $\leftarrow$ (18) 5 6 3 3 6 6 5 3 

(19) 2 4 3 3 3 6 6 5 3 $\leftarrow$ (20) 3 6 3 3 6 6 5 3 

(22) ...3 3 6 6 5 3 $\leftarrow$ (27) 6 ..4 5 3 3 3 

(23) 3 6 ..4 5 3 3 3 $\leftarrow$ (24) 5 6 2 4 5 3 3 3 

(23) 16 4 1 1 * 1 $\leftarrow$ (24) 20 1 1 * 1 

(25) ....3 5 7 3 3 $\leftarrow$ (26) 3 6 2 4 5 3 3 3 

(28) ....4 5 3 3 3 $\leftarrow$ (40) 4 1 1 * 1 

(29) 3 6 2 3 4 4 1 1 1 $\leftarrow$ (30) 8 1 1 2 4 3 3 3 

(31) * 2 4 3 3 3 $\leftarrow$ (33) 6 2 3 4 4 1 1 1 

(31) 8 4 1 1 * 1 $\leftarrow$ (32) 12 1 1 * 1 

\end{tcolorbox}
\begin{tcolorbox}[title={$(54,11)$}]
(1) 18 2 4 3 3 3 6 6 5 3 

(7) 4 3 5 6 5 2 3 5 7 7 $\leftarrow$ (10) 5 5 6 5 2 3 5 7 7 

(13) ..4 7 3 3 6 6 5 3 

(14) 1 2 4 7 3 3 6 6 5 3 $\leftarrow$ (16) 2 4 7 3 3 6 6 5 3 

(17) ..4 3 3 3 6 6 5 3 $\leftarrow$ (18) 3 5 6 2 3 5 7 3 3 

(20) ....3 3 6 6 5 3 $\leftarrow$ (26) ....3 5 7 3 3 

(23) .....3 5 7 3 3 $\leftarrow$ (24) 3 6 ..4 5 3 3 3 

(23) 14 * * 1 $\leftarrow$ (24) 16 4 1 1 * 1 

(26) .....4 5 3 3 3 $\leftarrow$ (38) * * 1 

(29) 2 * 2 4 3 3 3 $\leftarrow$ (30) 3 6 2 3 4 4 1 1 1 

(30) 1 * 2 4 3 3 3 $\leftarrow$ (32) * 2 4 3 3 3 

(31) 6 * * 1 $\leftarrow$ (32) 8 4 1 1 * 1 

(35) 2 * * 1 $\leftarrow$ (36) 4 4 1 1 * 1 

\end{tcolorbox}
\begin{tcolorbox}[title={$(54,12)$}]
(1) 16 ..4 3 3 3 6 6 5 3 $\leftarrow$ (2) 18 2 4 3 3 3 6 6 5 3 

(2) 4 3 5 5 6 5 2 3 5 7 7 

(5) 2 4 3 5 6 5 2 3 5 7 7 $\leftarrow$ (8) 4 3 5 6 5 2 3 5 7 7 

(11) 6 ..4 3 3 3 6 6 5 3 

(13) 1 1 2 4 7 3 3 6 6 5 3 $\leftarrow$ (14) ..4 7 3 3 6 6 5 3 

(17) 6 .....3 5 7 3 3 $\leftarrow$ (24) .....3 5 7 3 3 

(23) 13 1 * * 1 $\leftarrow$ (24) 14 * * 1 

(24) ......4 5 3 3 3 $\leftarrow$ (37) 1 * * 1 

(29) 1 1 * 2 4 3 3 3 $\leftarrow$ (30) 2 * 2 4 3 3 3 

(31) 5 1 * * 1 $\leftarrow$ (32) 6 * * 1 

(35) 1 1 * * 1 $\leftarrow$ (36) 2 * * 1 

\end{tcolorbox}
\begin{tcolorbox}[title={$(54,13)$}]
(1) 13 6 ....3 3 6 6 5 3 $\leftarrow$ (2) 16 ..4 3 3 3 6 6 5 3 

(3) ..4 3 5 6 5 2 3 5 7 7 $\leftarrow$ (6) 2 4 3 5 6 5 2 3 5 7 7 

(9) 4 1 1 2 4 7 3 3 6 6 5 3 

(11) 3 6 ....3 3 6 6 5 3 $\leftarrow$ (12) 6 ..4 3 3 3 6 6 5 3 

(15) 5 6 .....4 5 3 3 3 $\leftarrow$ (21) 6 .....4 5 3 3 3 

(17) 3 6 .....4 5 3 3 3 $\leftarrow$ (18) 6 .....3 5 7 3 3 

(21) 8 1 1 * 2 4 3 3 3 $\leftarrow$ (36) 1 1 * * 1 

(23) 12 1 1 * * 1 $\leftarrow$ (24) 13 1 * * 1 

(31) 4 1 1 * * 1 $\leftarrow$ (32) 5 1 * * 1 

\end{tcolorbox}
\begin{tcolorbox}[title={$(54,14)$}]
(1) ...4 3 5 6 5 2 3 5 7 7 $\leftarrow$ (4) ..4 3 5 6 5 2 3 5 7 7 

(1) 11 5 6 .....3 5 7 3 3 $\leftarrow$ (2) 13 6 ....3 3 6 6 5 3 

(7) * 2 4 7 3 3 6 6 5 3 $\leftarrow$ (10) 4 1 1 2 4 7 3 3 6 6 5 3 

(7) 24 4 1 1 * * 1 

(11) .....4 3 3 3 6 6 5 3 $\leftarrow$ (12) 3 6 ....3 3 6 6 5 3 

(14) .......3 3 6 6 5 3 $\leftarrow$ (19) 6 ......4 5 3 3 3 

(15) 3 6 ......4 5 3 3 3 $\leftarrow$ (16) 5 6 .....4 5 3 3 3 

(17) ........3 5 7 3 3 $\leftarrow$ (18) 3 6 .....4 5 3 3 3 

(20) ........4 5 3 3 3 $\leftarrow$ (32) 4 1 1 * * 1 

(21) 3 6 2 3 4 4 1 1 * 1 $\leftarrow$ (22) 8 1 1 * 2 4 3 3 3 

(23) * * 2 4 3 3 3 $\leftarrow$ (25) 6 2 3 4 4 1 1 * 1 

(23) 8 4 1 1 * * 1 $\leftarrow$ (24) 12 1 1 * * 1 

\end{tcolorbox}
\begin{tcolorbox}[title={$(54,15)$}]
(1) 10 .....4 3 3 3 6 6 5 3 $\leftarrow$ (2) 11 5 6 .....3 5 7 3 3 

(2) .....4 5 5 5 3 6 6 5 3 

(5) 2 * 2 4 7 3 3 6 6 5 3 

(6) 1 * 2 4 7 3 3 6 6 5 3 $\leftarrow$ (8) * 2 4 7 3 3 6 6 5 3 

(7) 22 * * * 1 $\leftarrow$ (8) 24 4 1 1 * * 1 

(12) ........3 3 6 6 5 3 $\leftarrow$ (18) ........3 5 7 3 3 

(15) .........3 5 7 3 3 $\leftarrow$ (16) 3 6 ......4 5 3 3 3 

(18) .........4 5 3 3 3 $\leftarrow$ (30) * * * 1 

(21) 2 * * 2 4 3 3 3 $\leftarrow$ (22) 3 6 2 3 4 4 1 1 * 1 

(22) 1 * * 2 4 3 3 3 $\leftarrow$ (24) * * 2 4 3 3 3 

(23) 6 * * * 1 $\leftarrow$ (24) 8 4 1 1 * * 1 

(27) 2 * * * 1 $\leftarrow$ (28) 4 4 1 1 * * 1 

\end{tcolorbox}
\begin{tcolorbox}[title={$(54,16)$}]
(1) 7 1.........4 5 3 3 3 $\leftarrow$ (2) 10 .....4 3 3 3 6 6 5 3 

(3) 6 ......4 3 3 3 6 6 5 3 

(5) 1 1 * 2 4 7 3 3 6 6 5 3 $\leftarrow$ (6) 2 * 2 4 7 3 3 6 6 5 3 

(7) 21 1 * * * 1 $\leftarrow$ (8) 22 * * * 1 

(9) 6 .........3 5 7 3 3 $\leftarrow$ (16) .........3 5 7 3 3 

(15) 13 1 * * * 1 $\leftarrow$ (29) 1 * * * 1 

(21) 1 1 * * 2 4 3 3 3 $\leftarrow$ (22) 2 * * 2 4 3 3 3 

(23) 5 1 * * * 1 $\leftarrow$ (24) 6 * * * 1 

(27) 1 1 * * * 1 $\leftarrow$ (28) 2 * * * 1 

\end{tcolorbox}
\begin{tcolorbox}[title={$(54,17)$}]
(1) 4 1 1 * 2 4 7 3 3 6 6 5 3 

(1) 5 6 ........3 3 6 6 5 3 $\leftarrow$ (2) 7 1.........4 5 3 3 3 

(3) 3 6 ........3 3 6 6 5 3 $\leftarrow$ (4) 6 ......4 3 3 3 6 6 5 3 

(7) 5 6 .........4 5 3 3 3 $\leftarrow$ (13) 6 .........4 5 3 3 3 

(7) 20 1 1 * * * 1 $\leftarrow$ (8) 21 1 * * * 1 

(9) 3 6 .........4 5 3 3 3 $\leftarrow$ (10) 6 .........3 5 7 3 3 

(13) 8 1 1 * * 2 4 3 3 3 $\leftarrow$ (28) 1 1 * * * 1 

(15) 12 1 1 * * * 1 $\leftarrow$ (16) 13 1 * * * 1 

(23) 4 1 1 * * * 1 $\leftarrow$ (24) 5 1 * * * 1 

\end{tcolorbox}
\begin{tcolorbox}[title={$(54,18)$}]
(1) 3 5 6 .........3 5 7 3 3 $\leftarrow$ (2) 5 6 ........3 3 6 6 5 3 

(3) .........4 3 3 3 6 6 5 3 $\leftarrow$ (4) 3 6 ........3 3 6 6 5 3 

(6) ...........3 3 6 6 5 3 $\leftarrow$ (11) 6 ..........4 5 3 3 3 

(7) 3 6 ..........4 5 3 3 3 $\leftarrow$ (8) 5 6 .........4 5 3 3 3 

(7) 16 4 1 1 * * * 1 $\leftarrow$ (8) 20 1 1 * * * 1 

(9) ............3 5 7 3 3 $\leftarrow$ (10) 3 6 .........4 5 3 3 3 

(12) ............4 5 3 3 3 $\leftarrow$ (24) 4 1 1 * * * 1 

(13) 3 6 2 3 4 4 1 1 * * 1 $\leftarrow$ (14) 8 1 1 * * 2 4 3 3 3 

(15) * * * 2 4 3 3 3 $\leftarrow$ (17) 6 2 3 4 4 1 1 * * 1 

(15) 8 4 1 1 * * * 1 $\leftarrow$ (16) 12 1 1 * * * 1 

\end{tcolorbox}
\begin{tcolorbox}[title={$(54,19)$}]
(1) ..........4 3 3 3 6 6 5 3 $\leftarrow$ (2) 3 5 6 .........3 5 7 3 3 

(4) ............3 3 6 6 5 3 $\leftarrow$ (10) ............3 5 7 3 3 

(7) .............3 5 7 3 3 $\leftarrow$ (8) 3 6 ..........4 5 3 3 3 

(7) 14 * * * * 1 $\leftarrow$ (8) 16 4 1 1 * * * 1 

(10) .............4 5 3 3 3 $\leftarrow$ (22) * * * * 1 

(13) 2 * * * 2 4 3 3 3 $\leftarrow$ (14) 3 6 2 3 4 4 1 1 * * 1 

(14) 1 * * * 2 4 3 3 3 $\leftarrow$ (16) * * * 2 4 3 3 3 

(15) 6 * * * * 1 $\leftarrow$ (16) 8 4 1 1 * * * 1 

(19) 2 * * * * 1 $\leftarrow$ (20) 4 4 1 1 * * * 1 

\end{tcolorbox}
\begin{tcolorbox}[title={$(54,20)$}]
(1) 6 .............3 5 7 3 3 $\leftarrow$ (8) .............3 5 7 3 3 

(7) 13 1 * * * * 1 $\leftarrow$ (8) 14 * * * * 1 

(8) ..............4 5 3 3 3 $\leftarrow$ (21) 1 * * * * 1 

(13) 1 1 * * * 2 4 3 3 3 $\leftarrow$ (14) 2 * * * 2 4 3 3 3 

(15) 5 1 * * * * 1 $\leftarrow$ (16) 6 * * * * 1 

(19) 1 1 * * * * 1 $\leftarrow$ (20) 2 * * * * 1 

\end{tcolorbox}
\begin{tcolorbox}[title={$(55,3)$}]
(23) 29 3 $\leftarrow$ (25) 31 

(39) 13 3 $\leftarrow$ (41) 15 

(41) 7 7 $\leftarrow$ (53) 3 

(47) 5 3 $\leftarrow$ (49) 7 

\end{tcolorbox}
\begin{tcolorbox}[title={$(55,4)$}]
(17) 22 13 3 $\leftarrow$ (18) 23 15 

(21) 26 5 3 $\leftarrow$ (22) 27 7 

(22) 27 3 3 $\leftarrow$ (24) 29 3 

(25) 14 13 3 $\leftarrow$ (26) 15 15 

(37) 10 5 3 $\leftarrow$ (38) 11 7 

(38) 11 3 3 $\leftarrow$ (40) 13 3 

(41) 6 5 3 $\leftarrow$ (42) 7 7 

(46) 3 3 3 $\leftarrow$ (48) 5 3 

(47) 2 3 3 $\leftarrow$ (50) 3 3 

\end{tcolorbox}
\begin{tcolorbox}[title={$(55,5)$}]
(5) 9 11 15 15 $\leftarrow$ (19) 23 7 7 

(7) 7 11 15 15 $\leftarrow$ (35) 7 7 7 

(15) 11 15 7 7 $\leftarrow$ (25) 13 11 7 

(15) 19 7 7 7 $\leftarrow$ (17) 21 11 7 

(17) 21 11 3 3 $\leftarrow$ (18) 22 13 3 

(21) 25 3 3 3 $\leftarrow$ (22) 26 5 3 

(25) 13 11 3 3 $\leftarrow$ (26) 14 13 3 

(37) 5 7 3 3 $\leftarrow$ (42) 6 5 3 

(37) 9 3 3 3 $\leftarrow$ (38) 10 5 3 

(46) 1 2 3 3 $\leftarrow$ (48) 2 3 3 

\end{tcolorbox}
\begin{tcolorbox}[title={$(55,6)$}]
(2) 5 7 11 15 15 

(3) 4 7 11 15 15 $\leftarrow$ (34) 3 5 7 7 

(7) 10 17 7 7 7 $\leftarrow$ (16) 19 7 7 7 

(10) 9 15 7 7 7 $\leftarrow$ (18) 17 7 7 7 

(11) 11 14 5 7 7 $\leftarrow$ (12) 13 13 11 7 

(14) 9 11 7 7 7 $\leftarrow$ (23) 14 5 7 7 

(15) 7 14 5 7 7 $\leftarrow$ (16) 11 15 7 7 

(15) 18 3 5 7 7 $\leftarrow$ (17) 20 5 7 7 

(17) 20 9 3 3 3 $\leftarrow$ (18) 21 11 3 3 

(22) 11 3 5 7 7 $\leftarrow$ (24) 13 5 7 7 

(25) 12 9 3 3 3 $\leftarrow$ (26) 13 11 3 3 

(31) 2 3 5 7 7 $\leftarrow$ (33) 4 5 7 7 

(35) 4 7 3 3 3 $\leftarrow$ (36) 6 6 5 3 

(37) 4 5 3 3 3 $\leftarrow$ (38) 5 7 3 3 

(41) 5 1 2 3 3 $\leftarrow$ (42) 6 2 3 3 

\end{tcolorbox}
\begin{tcolorbox}[title={$(55,7)$}]
(1) 2 4 7 11 15 15 

(3) 5 7 11 15 7 7 $\leftarrow$ (5) 7 13 13 11 7 

(5) 5 9 15 7 7 7 $\leftarrow$ (19) 14 4 5 7 7 

(6) 8 9 11 7 7 7 $\leftarrow$ (11) 13 13 5 7 7 

(7) 9 7 13 5 7 7 $\leftarrow$ (8) 10 17 7 7 7 

(8) 6 9 11 7 7 7 $\leftarrow$ (17) 7 13 5 7 7 

(11) 7 14 4 5 7 7 $\leftarrow$ (12) 11 14 5 7 7 

(13) 12 12 9 3 3 3 $\leftarrow$ (16) 18 3 5 7 7 

(15) 5 5 9 7 7 7 $\leftarrow$ (16) 7 14 5 7 7 

(15) 17 3 6 6 5 3 $\leftarrow$ (18) 20 9 3 3 3 

(17) 3 5 9 7 7 7 $\leftarrow$ (33) 3 6 6 5 3 

(21) 5 7 3 5 7 7 $\leftarrow$ (25) 9 3 5 7 7 

(23) 9 3 6 6 5 3 $\leftarrow$ (26) 12 9 3 3 3 

(29) 3 3 6 6 5 3 $\leftarrow$ (32) 2 3 5 7 7 

(35) 2 4 5 3 3 3 $\leftarrow$ (36) 4 7 3 3 3 

(39) 1 2 4 3 3 3 $\leftarrow$ (41) 2 4 3 3 3 

(41) 3 4 4 1 1 1 $\leftarrow$ (42) 5 1 2 3 3 

\end{tcolorbox}
\begin{tcolorbox}[title={$(55,8)$}]
(2) 3 5 9 15 7 7 7 $\leftarrow$ (16) 5 5 9 7 7 7 

(4) 4 6 9 11 7 7 7 $\leftarrow$ (10) 11 5 9 7 7 7 

(7) 5 10 11 3 5 7 7 $\leftarrow$ (12) 7 14 4 5 7 7 

(7) 8 5 5 9 7 7 7 $\leftarrow$ (8) 9 7 13 5 7 7 

(13) 10 9 3 6 6 5 3 $\leftarrow$ (14) 12 12 9 3 3 3 

(14) 11 12 4 5 3 3 3 $\leftarrow$ (16) 17 3 6 6 5 3 

(17) 6 9 3 6 6 5 3 $\leftarrow$ (18) 8 12 9 3 3 3 

(18) 3 5 7 3 5 7 7 $\leftarrow$ (20) 5 9 3 5 7 7 

(22) 5 5 3 6 6 5 3 $\leftarrow$ (24) 9 3 6 6 5 3 

(23) 6 3 3 6 6 5 3 $\leftarrow$ (30) 3 3 6 6 5 3 

(25) 10 2 4 5 3 3 3 $\leftarrow$ (26) 12 4 5 3 3 3 

(38) 1 1 2 4 3 3 3 $\leftarrow$ (40) 1 2 4 3 3 3 

(39) 2 3 4 4 1 1 1 $\leftarrow$ (42) 3 4 4 1 1 1 

\end{tcolorbox}
\begin{tcolorbox}[title={$(55,9)$}]
(1) 13 11 12 4 5 3 3 3 $\leftarrow$ (8) 8 5 5 9 7 7 7 

(3) 4 8 5 5 9 7 7 7 $\leftarrow$ (17) 3 5 9 3 5 7 7 

(4) 3 5 10 11 3 5 7 7 $\leftarrow$ (10) 8 3 5 9 7 7 7 

(7) 4 6 3 5 9 7 7 7 $\leftarrow$ (8) 5 10 11 3 5 7 7 

(13) 9 5 5 3 6 6 5 3 $\leftarrow$ (14) 10 9 3 6 6 5 3 

(15) 5 6 5 2 3 5 7 7 $\leftarrow$ (21) 6 5 2 3 5 7 7 

(17) 3 6 5 2 3 5 7 7 $\leftarrow$ (18) 6 9 3 6 6 5 3 

(21) 5 6 2 3 5 7 3 3 $\leftarrow$ (27) 6 2 3 5 7 3 3 

(23) 3 6 2 3 5 7 3 3 $\leftarrow$ (24) 6 3 3 6 6 5 3 

(25) 7 13 1 * 1 $\leftarrow$ (26) 10 2 4 5 3 3 3 

(29) 4 ..4 5 3 3 3 $\leftarrow$ (30) 6 2 4 5 3 3 3 

(38) 1 2 3 4 4 1 1 1 $\leftarrow$ (40) 2 3 4 4 1 1 1 

\end{tcolorbox}
\begin{tcolorbox}[title={$(55,10)$}]
(1) 12 9 5 5 3 6 6 5 3 $\leftarrow$ (2) 13 11 12 4 5 3 3 3 

(2) 2 3 5 10 11 3 5 7 7 $\leftarrow$ (9) 4 5 3 5 9 7 7 7 

(4) 5 9 3 5 7 3 5 7 7 

(7) 2 3 5 3 5 9 7 7 7 $\leftarrow$ (8) 4 6 3 5 9 7 7 7 

(7) 4 5 3 5 9 3 5 7 7 $\leftarrow$ (9) 8 3 5 9 3 5 7 7 

(13) 4 5 5 5 3 6 6 5 3 $\leftarrow$ (16) 5 6 5 2 3 5 7 7 

(17) 2 3 5 5 3 6 6 5 3 $\leftarrow$ (18) 3 6 5 2 3 5 7 7 

(20) 2 4 3 3 3 6 6 5 3 $\leftarrow$ (25) 5 6 2 4 5 3 3 3 

(21) 3 5 6 2 4 5 3 3 3 $\leftarrow$ (22) 5 6 2 3 5 7 3 3 

(23) ...3 3 6 6 5 3 $\leftarrow$ (24) 3 6 2 3 5 7 3 3 

(25) 5 8 1 1 2 4 3 3 3 $\leftarrow$ (26) 7 13 1 * 1 

(27) 4 ...4 5 3 3 3 $\leftarrow$ (28) 6 ..4 5 3 3 3 

(29) ....4 5 3 3 3 $\leftarrow$ (30) 4 ..4 5 3 3 3 

(33) 5 1 2 3 4 4 1 1 1 $\leftarrow$ (34) 6 2 3 4 4 1 1 1 

\end{tcolorbox}
\begin{tcolorbox}[title={$(55,11)$}]
(1) 1 2 3 5 10 11 3 5 7 7 $\leftarrow$ (2) 12 9 5 5 3 6 6 5 3 

(1) 16 2 3 5 5 3 6 6 5 3 

(5) 5 5 5 6 5 2 3 5 7 7 

(11) 2 4 5 5 5 3 6 6 5 3 $\leftarrow$ (14) 4 5 5 5 3 6 6 5 3 

(15) 1 2 4 7 3 3 6 6 5 3 $\leftarrow$ (17) 2 4 7 3 3 6 6 5 3 

(18) ..4 3 3 3 6 6 5 3 $\leftarrow$ (24) ...3 3 6 6 5 3 

(21) ....3 3 6 6 5 3 $\leftarrow$ (22) 3 5 6 2 4 5 3 3 3 

(25) 4 ....4 5 3 3 3 $\leftarrow$ (26) 5 8 1 1 2 4 3 3 3 

(27) .....4 5 3 3 3 $\leftarrow$ (28) 4 ...4 5 3 3 3 

(31) 1 * 2 4 3 3 3 $\leftarrow$ (33) * 2 4 3 3 3 

(33) 3 4 4 1 1 * 1 $\leftarrow$ (34) 5 1 2 3 4 4 1 1 1 

\end{tcolorbox}
\begin{tcolorbox}[title={$(55,12)$}]
(1) 2 4 2 3 5 3 5 9 7 7 7 

(1) 13 ..4 7 3 3 6 6 5 3 $\leftarrow$ (2) 16 2 3 5 5 3 6 6 5 3 

(3) 4 3 5 5 6 5 2 3 5 7 7 $\leftarrow$ (6) 5 5 5 6 5 2 3 5 7 7 

(9) ..4 5 5 5 3 6 6 5 3 $\leftarrow$ (12) 2 4 5 5 5 3 6 6 5 3 

(14) 1 1 2 4 7 3 3 6 6 5 3 $\leftarrow$ (16) 1 2 4 7 3 3 6 6 5 3 

(15) 6 ....3 3 6 6 5 3 $\leftarrow$ (22) ....3 3 6 6 5 3 

(25) ......4 5 3 3 3 $\leftarrow$ (26) 4 ....4 5 3 3 3 

(30) 1 1 * 2 4 3 3 3 $\leftarrow$ (32) 1 * 2 4 3 3 3 

(31) 2 3 4 4 1 1 * 1 $\leftarrow$ (34) 3 4 4 1 1 * 1 

\end{tcolorbox}
\begin{tcolorbox}[title={$(55,13)$}]
(1) 2 4 3 5 5 6 5 2 3 5 7 7 $\leftarrow$ (4) 4 3 5 5 6 5 2 3 5 7 7 

(1) 11 6 ..4 3 3 3 6 6 5 3 $\leftarrow$ (2) 13 ..4 7 3 3 6 6 5 3 

(7) ...4 5 5 5 3 6 6 5 3 $\leftarrow$ (10) ..4 5 5 5 3 6 6 5 3 

(13) 5 6 .....3 5 7 3 3 $\leftarrow$ (19) 6 .....3 5 7 3 3 

(15) 3 6 .....3 5 7 3 3 $\leftarrow$ (16) 6 ....3 3 6 6 5 3 

(21) 4 ......4 5 3 3 3 $\leftarrow$ (22) 6 .....4 5 3 3 3 

(30) 1 2 3 4 4 1 1 * 1 $\leftarrow$ (32) 2 3 4 4 1 1 * 1 

\end{tcolorbox}
\begin{tcolorbox}[title={$(55,14)$}]
(1) 9 4 1 1 2 4 7 3 3 6 6 5 3 $\leftarrow$ (2) 11 6 ..4 3 3 3 6 6 5 3 

(2) ...4 3 5 6 5 2 3 5 7 7 

(5) ....4 5 5 5 3 6 6 5 3 $\leftarrow$ (8) ...4 5 5 5 3 6 6 5 3 

(12) .....4 3 3 3 6 6 5 3 $\leftarrow$ (17) 5 6 .....4 5 3 3 3 

(13) 3 5 6 .....4 5 3 3 3 $\leftarrow$ (14) 5 6 .....3 5 7 3 3 

(15) .......3 3 6 6 5 3 $\leftarrow$ (16) 3 6 .....3 5 7 3 3 

(19) 4 .......4 5 3 3 3 $\leftarrow$ (20) 6 ......4 5 3 3 3 

(21) ........4 5 3 3 3 $\leftarrow$ (22) 4 ......4 5 3 3 3 

(25) 5 1 2 3 4 4 1 1 * 1 $\leftarrow$ (26) 6 2 3 4 4 1 1 * 1 

\end{tcolorbox}
\begin{tcolorbox}[title={$(55,15)$}]
(1) 8 ..1 ..4 7 3 3 6 6 5 3 $\leftarrow$ (2) 9 4 1 1 2 4 7 3 3 6 6 5 3 

(3) .....4 5 5 5 3 6 6 5 3 $\leftarrow$ (6) ....4 5 5 5 3 6 6 5 3 

(7) 1 * 2 4 7 3 3 6 6 5 3 $\leftarrow$ (9) * 2 4 7 3 3 6 6 5 3 

(9) 1.........4 5 3 3 3 $\leftarrow$ (16) .......3 3 6 6 5 3 

(13) ........3 3 6 6 5 3 $\leftarrow$ (14) 3 5 6 .....4 5 3 3 3 

(19) .........4 5 3 3 3 $\leftarrow$ (20) 4 .......4 5 3 3 3 

(23) 1 * * 2 4 3 3 3 $\leftarrow$ (25) * * 2 4 3 3 3 

(25) 3 4 4 1 1 * * 1 $\leftarrow$ (26) 5 1 2 3 4 4 1 1 * 1 

\end{tcolorbox}
\begin{tcolorbox}[title={$(55,16)$}]
(1) ......4 5 5 5 3 6 6 5 3 $\leftarrow$ (4) .....4 5 5 5 3 6 6 5 3 

(1) 5 2 * 2 4 7 3 3 6 6 5 3 $\leftarrow$ (2) 8 ..1 ..4 7 3 3 6 6 5 3 

(6) 1 1 * 2 4 7 3 3 6 6 5 3 $\leftarrow$ (8) 1 * 2 4 7 3 3 6 6 5 3 

(7) 6 ........3 3 6 6 5 3 $\leftarrow$ (14) ........3 3 6 6 5 3 

(9) 10 .........4 5 3 3 3 $\leftarrow$ (10) 1.........4 5 3 3 3 

(22) 1 1 * * 2 4 3 3 3 $\leftarrow$ (24) 1 * * 2 4 3 3 3 

(23) 2 3 4 4 1 1 * * 1 $\leftarrow$ (26) 3 4 4 1 1 * * 1 

\end{tcolorbox}
\begin{tcolorbox}[title={$(55,17)$}]
(1) 3 6 ......4 3 3 3 6 6 5 3 $\leftarrow$ (2) 5 2 * 2 4 7 3 3 6 6 5 3 

(2) 4 1 1 * 2 4 7 3 3 6 6 5 3 

(5) 5 6 .........3 5 7 3 3 $\leftarrow$ (11) 6 .........3 5 7 3 3 

(7) 3 6 .........3 5 7 3 3 $\leftarrow$ (8) 6 ........3 3 6 6 5 3 

(9) 7 13 1 * * * 1 $\leftarrow$ (10) 10 .........4 5 3 3 3 

(13) 4 ..........4 5 3 3 3 $\leftarrow$ (14) 6 .........4 5 3 3 3 

(22) 1 2 3 4 4 1 1 * * 1 $\leftarrow$ (24) 2 3 4 4 1 1 * * 1 

\end{tcolorbox}
\begin{tcolorbox}[title={$(55,18)$}]
(1) ..1 2 * 2 4 7 3 3 6 6 5 3 $\leftarrow$ (2) 3 6 ......4 3 3 3 6 6 5 3 

(4) .........4 3 3 3 6 6 5 3 $\leftarrow$ (9) 5 6 .........4 5 3 3 3 

(5) 3 5 6 .........4 5 3 3 3 $\leftarrow$ (6) 5 6 .........3 5 7 3 3 

(7) ...........3 3 6 6 5 3 $\leftarrow$ (8) 3 6 .........3 5 7 3 3 

(9) 5 8 1 1 * * 2 4 3 3 3 $\leftarrow$ (10) 7 13 1 * * * 1 

(11) 4 ...........4 5 3 3 3 $\leftarrow$ (12) 6 ..........4 5 3 3 3 

(13) ............4 5 3 3 3 $\leftarrow$ (14) 4 ..........4 5 3 3 3 

(17) 5 1 2 3 4 4 1 1 * * 1 $\leftarrow$ (18) 6 2 3 4 4 1 1 * * 1 

\end{tcolorbox}
\begin{tcolorbox}[title={$(55,19)$}]
(2) ..........4 3 3 3 6 6 5 3 $\leftarrow$ (8) ...........3 3 6 6 5 3 

(5) ............3 3 6 6 5 3 $\leftarrow$ (6) 3 5 6 .........4 5 3 3 3 

(9) 4 ............4 5 3 3 3 $\leftarrow$ (10) 5 8 1 1 * * 2 4 3 3 3 

(11) .............4 5 3 3 3 $\leftarrow$ (12) 4 ...........4 5 3 3 3 

(15) 1 * * * 2 4 3 3 3 $\leftarrow$ (17) * * * 2 4 3 3 3 

(17) 3 4 4 1 1 * * * 1 $\leftarrow$ (18) 5 1 2 3 4 4 1 1 * * 1 

\end{tcolorbox}
\begin{tcolorbox}[title={$(56,2)$}]
(55) 1 $\leftarrow$ (57) 

\end{tcolorbox}
\begin{tcolorbox}[title={$(56,3)$}]
(25) 30 1 $\leftarrow$ (26) 31 

(41) 14 1 $\leftarrow$ (42) 15 

(49) 6 1 $\leftarrow$ (50) 7 

(53) 2 1 $\leftarrow$ (54) 3 

(54) 1 1 $\leftarrow$ (56) 1 

\end{tcolorbox}
\begin{tcolorbox}[title={$(56,4)$}]
(11) 15 15 15 $\leftarrow$ (19) 23 15 

(15) 11 15 15 $\leftarrow$ (27) 15 15 

(23) 27 3 3 $\leftarrow$ (25) 29 3 

(25) 29 1 1 $\leftarrow$ (26) 30 1 

(37) 5 7 7 $\leftarrow$ (51) 3 3 

(39) 11 3 3 $\leftarrow$ (41) 13 3 

(41) 13 1 1 $\leftarrow$ (42) 14 1 

(47) 3 3 3 $\leftarrow$ (49) 5 3 

(49) 5 1 1 $\leftarrow$ (50) 6 1 

(53) 1 1 1 $\leftarrow$ (54) 2 1 

\end{tcolorbox}
\begin{tcolorbox}[title={$(56,5)$}]
(6) 9 11 15 15 $\leftarrow$ (18) 21 11 7 

(8) 7 11 15 15 $\leftarrow$ (26) 13 11 7 

(11) 14 13 11 7 $\leftarrow$ (12) 15 15 15 

(15) 10 13 11 7 $\leftarrow$ (16) 11 15 15 

(22) 25 3 3 3 $\leftarrow$ (24) 27 3 3 

(25) 28 1 1 1 $\leftarrow$ (26) 29 1 1 

(38) 9 3 3 3 $\leftarrow$ (40) 11 3 3 

(39) 8 3 3 3 $\leftarrow$ (48) 3 3 3 

(41) 12 1 1 1 $\leftarrow$ (42) 13 1 1 

(47) 1 2 3 3 $\leftarrow$ (49) 2 3 3 

(49) 4 1 1 1 $\leftarrow$ (50) 5 1 1 

\end{tcolorbox}
\begin{tcolorbox}[title={$(56,6)$}]
(3) 5 7 11 15 15 $\leftarrow$ (17) 19 7 7 7 

(4) 4 7 11 15 15 $\leftarrow$ (19) 17 7 7 7 

(7) 9 11 15 7 7 $\leftarrow$ (12) 14 13 11 7 

(9) 7 11 15 7 7 $\leftarrow$ (17) 11 15 7 7 

(11) 9 15 7 7 7 $\leftarrow$ (13) 13 13 11 7 

(15) 9 11 7 7 7 $\leftarrow$ (16) 10 13 11 7 

(21) 5 9 7 7 7 $\leftarrow$ (25) 13 5 7 7 

(23) 11 3 5 7 7 $\leftarrow$ (24) 14 5 7 7 

(25) 24 4 1 1 1 $\leftarrow$ (26) 28 1 1 1 

(35) 3 5 7 3 3 $\leftarrow$ (37) 6 6 5 3 

(38) 4 5 3 3 3 $\leftarrow$ (43) 6 2 3 3 

(39) 3 6 2 3 3 $\leftarrow$ (40) 8 3 3 3 

(41) 8 4 1 1 1 $\leftarrow$ (42) 12 1 1 1 

(45) 4 4 1 1 1 $\leftarrow$ (48) 1 2 3 3 

(47) * 1 $\leftarrow$ (50) 4 1 1 1 

\end{tcolorbox}
\begin{tcolorbox}[title={$(56,7)$}]
(2) 2 4 7 11 15 15 $\leftarrow$ (16) 9 11 7 7 7 

(4) 5 7 11 15 7 7 $\leftarrow$ (9) 10 17 7 7 7 

(5) 6 9 15 7 7 7 $\leftarrow$ (6) 7 13 13 11 7 

(6) 5 9 15 7 7 7 $\leftarrow$ (12) 9 15 7 7 7 

(7) 8 9 11 7 7 7 $\leftarrow$ (8) 9 11 15 7 7 

(9) 6 9 11 7 7 7 $\leftarrow$ (10) 7 11 15 7 7 

(11) 12 11 3 5 7 7 $\leftarrow$ (12) 13 13 5 7 7 

(18) 3 5 9 7 7 7 $\leftarrow$ (24) 11 3 5 7 7 

(19) 13 2 3 5 7 7 $\leftarrow$ (20) 14 4 5 7 7 

(22) 5 7 3 5 7 7 $\leftarrow$ (26) 9 3 5 7 7 

(25) 22 * 1 $\leftarrow$ (26) 24 4 1 1 1 

(33) 2 3 5 7 3 3 $\leftarrow$ (34) 3 6 6 5 3 

(36) 2 4 5 3 3 3 $\leftarrow$ (42) 2 4 3 3 3 

(39) ..4 3 3 3 $\leftarrow$ (40) 3 6 2 3 3 

(41) 6 * 1 $\leftarrow$ (42) 8 4 1 1 1 

(45) 2 * 1 $\leftarrow$ (46) 4 4 1 1 1 

(46) 1 * 1 $\leftarrow$ (48) * 1 

\end{tcolorbox}
\begin{tcolorbox}[title={$(56,8)$}]
(1) 1 2 4 7 11 15 15 $\leftarrow$ (8) 8 9 11 7 7 7 

(3) 3 5 9 15 7 7 7 $\leftarrow$ (10) 6 9 11 7 7 7 

(5) 4 6 9 11 7 7 7 $\leftarrow$ (6) 6 9 15 7 7 7 

(9) 9 3 5 9 7 7 7 $\leftarrow$ (19) 8 12 9 3 3 3 

(11) 7 8 12 9 3 3 3 $\leftarrow$ (12) 12 11 3 5 7 7 

(15) 11 12 4 5 3 3 3 $\leftarrow$ (17) 17 3 6 6 5 3 

(19) 3 5 7 3 5 7 7 $\leftarrow$ (21) 5 9 3 5 7 7 

(19) 7 12 4 5 3 3 3 $\leftarrow$ (20) 13 2 3 5 7 7 

(23) 5 5 3 6 6 5 3 $\leftarrow$ (25) 9 3 6 6 5 3 

(25) 21 1 * 1 $\leftarrow$ (26) 22 * 1 

(33) 13 1 * 1 $\leftarrow$ (41) 1 2 4 3 3 3 

(39) 1 1 2 4 3 3 3 $\leftarrow$ (40) ..4 3 3 3 

(41) 5 1 * 1 $\leftarrow$ (42) 6 * 1 

(45) 1 1 * 1 $\leftarrow$ (46) 2 * 1 

\end{tcolorbox}
\begin{tcolorbox}[title={$(56,9)$}]
(4) 4 8 5 5 9 7 7 7 

(5) 3 5 10 11 3 5 7 7 $\leftarrow$ (9) 5 10 11 3 5 7 7 

(10) 9 3 5 7 3 5 7 7 

(14) 9 5 5 3 6 6 5 3 $\leftarrow$ (16) 11 12 4 5 3 3 3 

(19) 4 7 3 3 6 6 5 3 $\leftarrow$ (24) 5 5 3 6 6 5 3 

(19) 5 6 3 3 6 6 5 3 $\leftarrow$ (20) 7 12 4 5 3 3 3 

(21) 3 6 3 3 6 6 5 3 $\leftarrow$ (22) 6 5 2 3 5 7 7 

(25) 20 1 1 * 1 $\leftarrow$ (26) 21 1 * 1 

(27) 3 6 2 4 5 3 3 3 $\leftarrow$ (28) 6 2 3 5 7 3 3 

(31) 8 1 1 2 4 3 3 3 $\leftarrow$ (40) 1 1 2 4 3 3 3 

(33) 12 1 1 * 1 $\leftarrow$ (34) 13 1 * 1 

(39) 1 2 3 4 4 1 1 1 $\leftarrow$ (41) 2 3 4 4 1 1 1 

(41) 4 1 1 * 1 $\leftarrow$ (42) 5 1 * 1 

\end{tcolorbox}
\begin{tcolorbox}[title={$(56,10)$}]
(3) 2 3 5 10 11 3 5 7 7 $\leftarrow$ (6) 3 5 10 11 3 5 7 7 

(5) 5 9 3 5 7 3 5 7 7 

(8) 2 3 5 3 5 9 7 7 7 

(8) 4 5 3 5 9 3 5 7 7 

(11) 5 5 6 5 2 3 5 7 7 $\leftarrow$ (17) 5 6 5 2 3 5 7 7 

(18) 2 3 5 5 3 6 6 5 3 $\leftarrow$ (20) 4 7 3 3 6 6 5 3 

(19) 3 5 6 2 3 5 7 3 3 $\leftarrow$ (20) 5 6 3 3 6 6 5 3 

(21) 2 4 3 3 3 6 6 5 3 $\leftarrow$ (22) 3 6 3 3 6 6 5 3 

(25) 3 6 ..4 5 3 3 3 $\leftarrow$ (26) 5 6 2 4 5 3 3 3 

(25) 16 4 1 1 * 1 $\leftarrow$ (26) 20 1 1 * 1 

(27) ....3 5 7 3 3 $\leftarrow$ (28) 3 6 2 4 5 3 3 3 

(30) ....4 5 3 3 3 $\leftarrow$ (35) 6 2 3 4 4 1 1 1 

(31) 3 6 2 3 4 4 1 1 1 $\leftarrow$ (32) 8 1 1 2 4 3 3 3 

(33) 8 4 1 1 * 1 $\leftarrow$ (34) 12 1 1 * 1 

(37) 4 4 1 1 * 1 $\leftarrow$ (40) 1 2 3 4 4 1 1 1 

(39) * * 1 $\leftarrow$ (42) 4 1 1 * 1 

\end{tcolorbox}
\begin{tcolorbox}[title={$(56,11)$}]
(2) 1 2 3 5 10 11 3 5 7 7 $\leftarrow$ (4) 2 3 5 10 11 3 5 7 7 

(3) 18 2 4 3 3 3 6 6 5 3 

(9) 4 3 5 6 5 2 3 5 7 7 $\leftarrow$ (12) 5 5 6 5 2 3 5 7 7 

(15) ..4 7 3 3 6 6 5 3 $\leftarrow$ (18) 2 4 7 3 3 6 6 5 3 

(19) ..4 3 3 3 6 6 5 3 $\leftarrow$ (20) 3 5 6 2 3 5 7 3 3 

(25) .....3 5 7 3 3 $\leftarrow$ (26) 3 6 ..4 5 3 3 3 

(25) 14 * * 1 $\leftarrow$ (26) 16 4 1 1 * 1 

(28) .....4 5 3 3 3 $\leftarrow$ (34) * 2 4 3 3 3 

(31) 2 * 2 4 3 3 3 $\leftarrow$ (32) 3 6 2 3 4 4 1 1 1 

(33) 6 * * 1 $\leftarrow$ (34) 8 4 1 1 * 1 

(37) 2 * * 1 $\leftarrow$ (38) 4 4 1 1 * 1 

(38) 1 * * 1 $\leftarrow$ (40) * * 1 

\end{tcolorbox}
\begin{tcolorbox}[title={$(56,12)$}]
(1) 5 5 5 5 6 5 2 3 5 7 7 

(2) 2 4 2 3 5 3 5 9 7 7 7 

(3) 16 ..4 3 3 3 6 6 5 3 $\leftarrow$ (4) 18 2 4 3 3 3 6 6 5 3 

(7) 2 4 3 5 6 5 2 3 5 7 7 $\leftarrow$ (10) 4 3 5 6 5 2 3 5 7 7 

(13) 6 ..4 3 3 3 6 6 5 3 $\leftarrow$ (17) 1 2 4 7 3 3 6 6 5 3 

(15) 1 1 2 4 7 3 3 6 6 5 3 $\leftarrow$ (16) ..4 7 3 3 6 6 5 3 

(25) 13 1 * * 1 $\leftarrow$ (26) 14 * * 1 

(26) ......4 5 3 3 3 $\leftarrow$ (33) 1 * 2 4 3 3 3 

(31) 1 1 * 2 4 3 3 3 $\leftarrow$ (32) 2 * 2 4 3 3 3 

(33) 5 1 * * 1 $\leftarrow$ (34) 6 * * 1 

(37) 1 1 * * 1 $\leftarrow$ (38) 2 * * 1 

\end{tcolorbox}
\begin{tcolorbox}[title={$(56,13)$}]
(1) 1 2 4 2 3 5 3 5 9 7 7 7 

(2) 2 4 3 5 5 6 5 2 3 5 7 7 

(3) 13 6 ....3 3 6 6 5 3 $\leftarrow$ (4) 16 ..4 3 3 3 6 6 5 3 

(5) ..4 3 5 6 5 2 3 5 7 7 $\leftarrow$ (8) 2 4 3 5 6 5 2 3 5 7 7 

(11) 4 1 1 2 4 7 3 3 6 6 5 3 $\leftarrow$ (16) 1 1 2 4 7 3 3 6 6 5 3 

(13) 3 6 ....3 3 6 6 5 3 $\leftarrow$ (14) 6 ..4 3 3 3 6 6 5 3 

(19) 3 6 .....4 5 3 3 3 $\leftarrow$ (20) 6 .....3 5 7 3 3 

(23) 8 1 1 * 2 4 3 3 3 $\leftarrow$ (32) 1 1 * 2 4 3 3 3 

(25) 12 1 1 * * 1 $\leftarrow$ (26) 13 1 * * 1 

(31) 1 2 3 4 4 1 1 * 1 $\leftarrow$ (33) 2 3 4 4 1 1 * 1 

(33) 4 1 1 * * 1 $\leftarrow$ (34) 5 1 * * 1 

\end{tcolorbox}
\begin{tcolorbox}[title={$(56,14)$}]
(3) ...4 3 5 6 5 2 3 5 7 7 $\leftarrow$ (6) ..4 3 5 6 5 2 3 5 7 7 

(3) 11 5 6 .....3 5 7 3 3 $\leftarrow$ (4) 13 6 ....3 3 6 6 5 3 

(9) 24 4 1 1 * * 1 $\leftarrow$ (12) 4 1 1 2 4 7 3 3 6 6 5 3 

(13) .....4 3 3 3 6 6 5 3 $\leftarrow$ (14) 3 6 ....3 3 6 6 5 3 

(17) 3 6 ......4 5 3 3 3 $\leftarrow$ (18) 5 6 .....4 5 3 3 3 

(19) ........3 5 7 3 3 $\leftarrow$ (20) 3 6 .....4 5 3 3 3 

(22) ........4 5 3 3 3 $\leftarrow$ (27) 6 2 3 4 4 1 1 * 1 

(23) 3 6 2 3 4 4 1 1 * 1 $\leftarrow$ (24) 8 1 1 * 2 4 3 3 3 

(25) 8 4 1 1 * * 1 $\leftarrow$ (26) 12 1 1 * * 1 

(29) 4 4 1 1 * * 1 $\leftarrow$ (32) 1 2 3 4 4 1 1 * 1 

(31) * * * 1 $\leftarrow$ (34) 4 1 1 * * 1 

\end{tcolorbox}
\begin{tcolorbox}[title={$(56,15)$}]
(1) ....4 3 5 6 5 2 3 5 7 7 $\leftarrow$ (4) ...4 3 5 6 5 2 3 5 7 7 

(3) 10 .....4 3 3 3 6 6 5 3 $\leftarrow$ (4) 11 5 6 .....3 5 7 3 3 

(7) 2 * 2 4 7 3 3 6 6 5 3 $\leftarrow$ (10) * 2 4 7 3 3 6 6 5 3 

(9) 22 * * * 1 $\leftarrow$ (10) 24 4 1 1 * * 1 

(17) .........3 5 7 3 3 $\leftarrow$ (18) 3 6 ......4 5 3 3 3 

(20) .........4 5 3 3 3 $\leftarrow$ (26) * * 2 4 3 3 3 

(23) 2 * * 2 4 3 3 3 $\leftarrow$ (24) 3 6 2 3 4 4 1 1 * 1 

(25) 6 * * * 1 $\leftarrow$ (26) 8 4 1 1 * * 1 

(29) 2 * * * 1 $\leftarrow$ (30) 4 4 1 1 * * 1 

(30) 1 * * * 1 $\leftarrow$ (32) * * * 1 

\end{tcolorbox}
\begin{tcolorbox}[title={$(56,16)$}]
(2) ......4 5 5 5 3 6 6 5 3 

(3) 7 1.........4 5 3 3 3 $\leftarrow$ (4) 10 .....4 3 3 3 6 6 5 3 

(5) 6 ......4 3 3 3 6 6 5 3 $\leftarrow$ (9) 1 * 2 4 7 3 3 6 6 5 3 

(7) 1 1 * 2 4 7 3 3 6 6 5 3 $\leftarrow$ (8) 2 * 2 4 7 3 3 6 6 5 3 

(9) 21 1 * * * 1 $\leftarrow$ (10) 22 * * * 1 

(17) 13 1 * * * 1 $\leftarrow$ (25) 1 * * 2 4 3 3 3 

(23) 1 1 * * 2 4 3 3 3 $\leftarrow$ (24) 2 * * 2 4 3 3 3 

(25) 5 1 * * * 1 $\leftarrow$ (26) 6 * * * 1 

(29) 1 1 * * * 1 $\leftarrow$ (30) 2 * * * 1 

\end{tcolorbox}
\begin{tcolorbox}[title={$(56,17)$}]
(3) 4 1 1 * 2 4 7 3 3 6 6 5 3 $\leftarrow$ (8) 1 1 * 2 4 7 3 3 6 6 5 3 

(3) 5 6 ........3 3 6 6 5 3 $\leftarrow$ (4) 7 1.........4 5 3 3 3 

(5) 3 6 ........3 3 6 6 5 3 $\leftarrow$ (6) 6 ......4 3 3 3 6 6 5 3 

(9) 20 1 1 * * * 1 $\leftarrow$ (10) 21 1 * * * 1 

(11) 3 6 .........4 5 3 3 3 $\leftarrow$ (12) 6 .........3 5 7 3 3 

(15) 8 1 1 * * 2 4 3 3 3 $\leftarrow$ (24) 1 1 * * 2 4 3 3 3 

(17) 12 1 1 * * * 1 $\leftarrow$ (18) 13 1 * * * 1 

(23) 1 2 3 4 4 1 1 * * 1 $\leftarrow$ (25) 2 3 4 4 1 1 * * 1 

(25) 4 1 1 * * * 1 $\leftarrow$ (26) 5 1 * * * 1 

\end{tcolorbox}
\begin{tcolorbox}[title={$(56,18)$}]
(1) * * 2 4 7 3 3 6 6 5 3 

(2) ..1 2 * 2 4 7 3 3 6 6 5 3 $\leftarrow$ (4) 4 1 1 * 2 4 7 3 3 6 6 5 3 

(3) 3 5 6 .........3 5 7 3 3 $\leftarrow$ (4) 5 6 ........3 3 6 6 5 3 

(5) .........4 3 3 3 6 6 5 3 $\leftarrow$ (6) 3 6 ........3 3 6 6 5 3 

(9) 3 6 ..........4 5 3 3 3 $\leftarrow$ (10) 5 6 .........4 5 3 3 3 

(9) 16 4 1 1 * * * 1 $\leftarrow$ (10) 20 1 1 * * * 1 

(11) ............3 5 7 3 3 $\leftarrow$ (12) 3 6 .........4 5 3 3 3 

(14) ............4 5 3 3 3 $\leftarrow$ (19) 6 2 3 4 4 1 1 * * 1 

(15) 3 6 2 3 4 4 1 1 * * 1 $\leftarrow$ (16) 8 1 1 * * 2 4 3 3 3 

(17) 8 4 1 1 * * * 1 $\leftarrow$ (18) 12 1 1 * * * 1 

(21) 4 4 1 1 * * * 1 $\leftarrow$ (24) 1 2 3 4 4 1 1 * * 1 

(23) * * * * 1 $\leftarrow$ (26) 4 1 1 * * * 1 

\end{tcolorbox}
\begin{tcolorbox}[title={$(57,3)$}]
(23) 27 7 $\leftarrow$ (27) 31 

(39) 11 7 $\leftarrow$ (43) 15 

(43) 7 7 $\leftarrow$ (51) 7 

(55) 1 1 $\leftarrow$ (57) 1 

\end{tcolorbox}
\begin{tcolorbox}[title={$(57,4)$}]
(19) 22 13 3 $\leftarrow$ (20) 23 15 

(20) 23 7 7 $\leftarrow$ (26) 29 3 

(23) 26 5 3 $\leftarrow$ (24) 27 7 

(27) 14 13 3 $\leftarrow$ (28) 15 15 

(36) 7 7 7 $\leftarrow$ (42) 13 3 

(38) 5 7 7 $\leftarrow$ (50) 5 3 

(39) 10 5 3 $\leftarrow$ (40) 11 7 

(43) 6 5 3 $\leftarrow$ (44) 7 7 

(54) 1 1 1 $\leftarrow$ (56) 1 1 

\end{tcolorbox}
\begin{tcolorbox}[title={$(57,5)$}]
(7) 9 11 15 15 $\leftarrow$ (13) 15 15 15 

(9) 7 11 15 15 $\leftarrow$ (17) 11 15 15 

(18) 20 5 7 7 $\leftarrow$ (25) 27 3 3 

(19) 21 11 3 3 $\leftarrow$ (20) 22 13 3 

(23) 25 3 3 3 $\leftarrow$ (24) 26 5 3 

(27) 13 11 3 3 $\leftarrow$ (28) 14 13 3 

(34) 4 5 7 7 $\leftarrow$ (41) 11 3 3 

(35) 3 5 7 7 $\leftarrow$ (49) 3 3 3 

(39) 5 7 3 3 $\leftarrow$ (44) 6 5 3 

(39) 9 3 3 3 $\leftarrow$ (40) 10 5 3 

\end{tcolorbox}
\begin{tcolorbox}[title={$(57,6)$}]
(4) 5 7 11 15 15 $\leftarrow$ (8) 9 11 15 15 

(5) 4 7 11 15 15 $\leftarrow$ (10) 7 11 15 15 

(13) 11 14 5 7 7 $\leftarrow$ (14) 13 13 11 7 

(17) 7 14 5 7 7 $\leftarrow$ (18) 11 15 7 7 

(17) 18 3 5 7 7 $\leftarrow$ (24) 25 3 3 3 

(18) 7 13 5 7 7 $\leftarrow$ (20) 17 7 7 7 

(19) 20 9 3 3 3 $\leftarrow$ (20) 21 11 3 3 

(22) 5 9 7 7 7 $\leftarrow$ (26) 13 5 7 7 

(27) 12 9 3 3 3 $\leftarrow$ (28) 13 11 3 3 

(33) 2 3 5 7 7 $\leftarrow$ (40) 9 3 3 3 

(36) 3 5 7 3 3 $\leftarrow$ (41) 8 3 3 3 

(37) 4 7 3 3 3 $\leftarrow$ (38) 6 6 5 3 

(39) 4 5 3 3 3 $\leftarrow$ (40) 5 7 3 3 

(43) 5 1 2 3 3 $\leftarrow$ (44) 6 2 3 3 

\end{tcolorbox}
\begin{tcolorbox}[title={$(57,7)$}]
(3) 2 4 7 11 15 15 $\leftarrow$ (6) 4 7 11 15 15 

(5) 5 7 11 15 7 7 $\leftarrow$ (9) 9 11 15 7 7 

(7) 5 9 15 7 7 7 $\leftarrow$ (13) 9 15 7 7 7 

(9) 9 7 13 5 7 7 $\leftarrow$ (10) 10 17 7 7 7 

(11) 11 5 9 7 7 7 $\leftarrow$ (13) 13 13 5 7 7 

(13) 7 14 4 5 7 7 $\leftarrow$ (14) 11 14 5 7 7 

(15) 12 12 9 3 3 3 $\leftarrow$ (20) 20 9 3 3 3 

(17) 5 5 9 7 7 7 $\leftarrow$ (18) 7 14 5 7 7 

(19) 3 5 9 7 7 7 $\leftarrow$ (21) 14 4 5 7 7 

(23) 5 7 3 5 7 7 $\leftarrow$ (27) 9 3 5 7 7 

(27) 12 4 5 3 3 3 

(31) 3 3 6 6 5 3 $\leftarrow$ (35) 3 6 6 5 3 

(34) 2 3 5 7 3 3 $\leftarrow$ (40) 4 5 3 3 3 

(37) 2 4 5 3 3 3 $\leftarrow$ (38) 4 7 3 3 3 

(43) 3 4 4 1 1 1 $\leftarrow$ (44) 5 1 2 3 3 

(47) 1 * 1 $\leftarrow$ (49) * 1 

\end{tcolorbox}
\begin{tcolorbox}[title={$(57,8)$}]
(2) 1 2 4 7 11 15 15 $\leftarrow$ (4) 2 4 7 11 15 15 

(4) 3 5 9 15 7 7 7 $\leftarrow$ (8) 5 9 15 7 7 7 

(6) 4 6 9 11 7 7 7 

(9) 8 5 5 9 7 7 7 $\leftarrow$ (10) 9 7 13 5 7 7 

(10) 9 3 5 9 7 7 7 $\leftarrow$ (12) 11 5 9 7 7 7 

(11) 8 3 5 9 7 7 7 $\leftarrow$ (14) 7 14 4 5 7 7 

(12) 7 8 12 9 3 3 3 $\leftarrow$ (18) 17 3 6 6 5 3 

(15) 10 9 3 6 6 5 3 $\leftarrow$ (16) 12 12 9 3 3 3 

(18) 3 5 9 3 5 7 7 

(19) 6 9 3 6 6 5 3 $\leftarrow$ (20) 8 12 9 3 3 3 

(20) 3 5 7 3 5 7 7 $\leftarrow$ (24) 5 7 3 5 7 7 

(25) 6 3 3 6 6 5 3 $\leftarrow$ (32) 3 3 6 6 5 3 

(27) 10 2 4 5 3 3 3 $\leftarrow$ (28) 12 4 5 3 3 3 

(31) 6 2 4 5 3 3 3 $\leftarrow$ (38) 2 4 5 3 3 3 

(46) 1 1 * 1 $\leftarrow$ (48) 1 * 1 

\end{tcolorbox}
\begin{tcolorbox}[title={$(57,9)$}]
(3) 13 11 12 4 5 3 3 3 

(5) 4 8 5 5 9 7 7 7 $\leftarrow$ (10) 8 5 5 9 7 7 7 

(9) 4 6 3 5 9 7 7 7 $\leftarrow$ (10) 5 10 11 3 5 7 7 

(10) 4 5 3 5 9 7 7 7 $\leftarrow$ (12) 8 3 5 9 7 7 7 

(10) 8 3 5 9 3 5 7 7 $\leftarrow$ (17) 11 12 4 5 3 3 3 

(11) 9 3 5 7 3 5 7 7 

(15) 9 5 5 3 6 6 5 3 $\leftarrow$ (16) 10 9 3 6 6 5 3 

(19) 3 6 5 2 3 5 7 7 $\leftarrow$ (20) 6 9 3 6 6 5 3 

(23) 5 6 2 3 5 7 3 3 $\leftarrow$ (29) 6 2 3 5 7 3 3 

(25) 3 6 2 3 5 7 3 3 $\leftarrow$ (26) 6 3 3 6 6 5 3 

(27) 7 13 1 * 1 $\leftarrow$ (28) 10 2 4 5 3 3 3 

(29) 6 ..4 5 3 3 3 $\leftarrow$ (35) 13 1 * 1 

(31) 4 ..4 5 3 3 3 $\leftarrow$ (32) 6 2 4 5 3 3 3 

\end{tcolorbox}
\begin{tcolorbox}[title={$(57,10)$}]
(1) 4 4 8 5 5 9 7 7 7 

(3) 12 9 5 5 3 6 6 5 3 $\leftarrow$ (4) 13 11 12 4 5 3 3 3 

(6) 5 9 3 5 7 3 5 7 7 $\leftarrow$ (12) 9 3 5 7 3 5 7 7 

(9) 2 3 5 3 5 9 7 7 7 $\leftarrow$ (10) 4 6 3 5 9 7 7 7 

(9) 4 5 3 5 9 3 5 7 7 $\leftarrow$ (16) 9 5 5 3 6 6 5 3 

(15) 4 5 5 5 3 6 6 5 3 $\leftarrow$ (18) 5 6 5 2 3 5 7 7 

(19) 2 3 5 5 3 6 6 5 3 $\leftarrow$ (20) 3 6 5 2 3 5 7 7 

(22) 2 4 3 3 3 6 6 5 3 $\leftarrow$ (27) 5 6 2 4 5 3 3 3 

(23) 3 5 6 2 4 5 3 3 3 $\leftarrow$ (24) 5 6 2 3 5 7 3 3 

(25) ...3 3 6 6 5 3 $\leftarrow$ (26) 3 6 2 3 5 7 3 3 

(27) 5 8 1 1 2 4 3 3 3 $\leftarrow$ (28) 7 13 1 * 1 

(28) ....3 5 7 3 3 $\leftarrow$ (33) 8 1 1 2 4 3 3 3 

(29) 4 ...4 5 3 3 3 $\leftarrow$ (30) 6 ..4 5 3 3 3 

(31) ....4 5 3 3 3 $\leftarrow$ (32) 4 ..4 5 3 3 3 

(35) 5 1 2 3 4 4 1 1 1 $\leftarrow$ (36) 6 2 3 4 4 1 1 1 

\end{tcolorbox}
\begin{tcolorbox}[title={$(57,11)$}]
(1) 5 5 9 3 5 7 3 5 7 7 

(1) 8 4 5 3 5 9 3 5 7 7 

(3) 1 2 3 5 10 11 3 5 7 7 $\leftarrow$ (4) 12 9 5 5 3 6 6 5 3 

(3) 16 2 3 5 5 3 6 6 5 3 

(7) 5 5 5 6 5 2 3 5 7 7 $\leftarrow$ (13) 5 5 6 5 2 3 5 7 7 

(13) 2 4 5 5 5 3 6 6 5 3 $\leftarrow$ (16) 4 5 5 5 3 6 6 5 3 

(20) ..4 3 3 3 6 6 5 3 $\leftarrow$ (26) ...3 3 6 6 5 3 

(23) ....3 3 6 6 5 3 $\leftarrow$ (24) 3 5 6 2 4 5 3 3 3 

(26) .....3 5 7 3 3 $\leftarrow$ (32) ....4 5 3 3 3 

(27) 4 ....4 5 3 3 3 $\leftarrow$ (28) 5 8 1 1 2 4 3 3 3 

(29) .....4 5 3 3 3 $\leftarrow$ (30) 4 ...4 5 3 3 3 

(35) 3 4 4 1 1 * 1 $\leftarrow$ (36) 5 1 2 3 4 4 1 1 1 

(39) 1 * * 1 $\leftarrow$ (41) * * 1 

\end{tcolorbox}
\begin{tcolorbox}[title={$(57,12)$}]
(1) 4 4 2 3 5 3 5 9 7 7 7 $\leftarrow$ (2) 8 4 5 3 5 9 3 5 7 7 

(2) 5 5 5 5 6 5 2 3 5 7 7 

(3) 2 4 2 3 5 3 5 9 7 7 7 $\leftarrow$ (5) 18 2 4 3 3 3 6 6 5 3 

(3) 13 ..4 7 3 3 6 6 5 3 $\leftarrow$ (4) 16 2 3 5 5 3 6 6 5 3 

(5) 4 3 5 5 6 5 2 3 5 7 7 $\leftarrow$ (8) 5 5 5 6 5 2 3 5 7 7 

(11) ..4 5 5 5 3 6 6 5 3 $\leftarrow$ (14) 2 4 5 5 5 3 6 6 5 3 

(17) 6 ....3 3 6 6 5 3 $\leftarrow$ (24) ....3 3 6 6 5 3 

(23) 6 .....4 5 3 3 3 $\leftarrow$ (30) .....4 5 3 3 3 

(27) ......4 5 3 3 3 $\leftarrow$ (28) 4 ....4 5 3 3 3 

(38) 1 1 * * 1 $\leftarrow$ (40) 1 * * 1 

\end{tcolorbox}
\begin{tcolorbox}[title={$(57,13)$}]
(1) ..4 2 3 5 3 5 9 7 7 7 $\leftarrow$ (2) 4 4 2 3 5 3 5 9 7 7 7 

(2) 1 2 4 2 3 5 3 5 9 7 7 7 $\leftarrow$ (4) 2 4 2 3 5 3 5 9 7 7 7 

(3) 2 4 3 5 5 6 5 2 3 5 7 7 $\leftarrow$ (6) 4 3 5 5 6 5 2 3 5 7 7 

(3) 11 6 ..4 3 3 3 6 6 5 3 $\leftarrow$ (4) 13 ..4 7 3 3 6 6 5 3 

(9) ...4 5 5 5 3 6 6 5 3 $\leftarrow$ (12) ..4 5 5 5 3 6 6 5 3 

(15) 5 6 .....3 5 7 3 3 $\leftarrow$ (21) 6 .....3 5 7 3 3 

(17) 3 6 .....3 5 7 3 3 $\leftarrow$ (18) 6 ....3 3 6 6 5 3 

(21) 6 ......4 5 3 3 3 $\leftarrow$ (28) ......4 5 3 3 3 

(23) 4 ......4 5 3 3 3 $\leftarrow$ (24) 6 .....4 5 3 3 3 

\end{tcolorbox}
\begin{tcolorbox}[title={$(57,14)$}]
(1) 1 1 2 4 2 3 5 3 5 9 7 7 7 $\leftarrow$ (2) ..4 2 3 5 3 5 9 7 7 7 

(1) ..4 3 5 5 6 5 2 3 5 7 7 $\leftarrow$ (4) 2 4 3 5 5 6 5 2 3 5 7 7 

(3) 9 4 1 1 2 4 7 3 3 6 6 5 3 $\leftarrow$ (4) 11 6 ..4 3 3 3 6 6 5 3 

(7) ....4 5 5 5 3 6 6 5 3 $\leftarrow$ (10) ...4 5 5 5 3 6 6 5 3 

(14) .....4 3 3 3 6 6 5 3 $\leftarrow$ (19) 5 6 .....4 5 3 3 3 

(15) 3 5 6 .....4 5 3 3 3 $\leftarrow$ (16) 5 6 .....3 5 7 3 3 

(17) .......3 3 6 6 5 3 $\leftarrow$ (18) 3 6 .....3 5 7 3 3 

(20) ........3 5 7 3 3 $\leftarrow$ (25) 8 1 1 * 2 4 3 3 3 

(21) 4 .......4 5 3 3 3 $\leftarrow$ (22) 6 ......4 5 3 3 3 

(23) ........4 5 3 3 3 $\leftarrow$ (24) 4 ......4 5 3 3 3 

(27) 5 1 2 3 4 4 1 1 * 1 $\leftarrow$ (28) 6 2 3 4 4 1 1 * 1 

\end{tcolorbox}
\begin{tcolorbox}[title={$(57,15)$}]
(2) ....4 3 5 6 5 2 3 5 7 7 

(3) 8 ..1 ..4 7 3 3 6 6 5 3 $\leftarrow$ (4) 9 4 1 1 2 4 7 3 3 6 6 5 3 

(5) .....4 5 5 5 3 6 6 5 3 $\leftarrow$ (8) ....4 5 5 5 3 6 6 5 3 

(11) 1.........4 5 3 3 3 $\leftarrow$ (18) .......3 3 6 6 5 3 

(15) ........3 3 6 6 5 3 $\leftarrow$ (16) 3 5 6 .....4 5 3 3 3 

(18) .........3 5 7 3 3 $\leftarrow$ (24) ........4 5 3 3 3 

(21) .........4 5 3 3 3 $\leftarrow$ (22) 4 .......4 5 3 3 3 

(27) 3 4 4 1 1 * * 1 $\leftarrow$ (28) 5 1 2 3 4 4 1 1 * 1 

(31) 1 * * * 1 $\leftarrow$ (33) * * * 1 

\end{tcolorbox}
\begin{tcolorbox}[title={$(57,16)$}]
(3) ......4 5 5 5 3 6 6 5 3 $\leftarrow$ (6) .....4 5 5 5 3 6 6 5 3 

(3) 5 2 * 2 4 7 3 3 6 6 5 3 $\leftarrow$ (4) 8 ..1 ..4 7 3 3 6 6 5 3 

(9) 6 ........3 3 6 6 5 3 $\leftarrow$ (16) ........3 3 6 6 5 3 

(11) 10 .........4 5 3 3 3 $\leftarrow$ (12) 1.........4 5 3 3 3 

(15) 6 .........4 5 3 3 3 $\leftarrow$ (22) .........4 5 3 3 3 

(30) 1 1 * * * 1 $\leftarrow$ (32) 1 * * * 1 

\end{tcolorbox}
\begin{tcolorbox}[title={$(57,17)$}]
(1) .......4 5 5 5 3 6 6 5 3 $\leftarrow$ (4) ......4 5 5 5 3 6 6 5 3 

(3) 3 6 ......4 3 3 3 6 6 5 3 $\leftarrow$ (4) 5 2 * 2 4 7 3 3 6 6 5 3 

(7) 5 6 .........3 5 7 3 3 $\leftarrow$ (13) 6 .........3 5 7 3 3 

(9) 3 6 .........3 5 7 3 3 $\leftarrow$ (10) 6 ........3 3 6 6 5 3 

(11) 7 13 1 * * * 1 $\leftarrow$ (12) 10 .........4 5 3 3 3 

(13) 6 ..........4 5 3 3 3 $\leftarrow$ (19) 13 1 * * * 1 

(15) 4 ..........4 5 3 3 3 $\leftarrow$ (16) 6 .........4 5 3 3 3 

\end{tcolorbox}
\begin{tcolorbox}[title={$(58,2)$}]
(55) 3 $\leftarrow$ (59) 

\end{tcolorbox}
\begin{tcolorbox}[title={$(58,3)$}]
(27) 30 1 $\leftarrow$ (28) 31 

(43) 14 1 $\leftarrow$ (44) 15 

(51) 6 1 $\leftarrow$ (52) 7 

(52) 3 3 $\leftarrow$ (58) 1 

(55) 2 1 $\leftarrow$ (56) 3 

\end{tcolorbox}
\begin{tcolorbox}[title={$(58,4)$}]
(19) 21 11 7 $\leftarrow$ (21) 23 15 

(21) 23 7 7 $\leftarrow$ (25) 27 7 

(27) 13 11 7 $\leftarrow$ (29) 15 15 

(27) 29 1 1 $\leftarrow$ (28) 30 1 

(37) 7 7 7 $\leftarrow$ (41) 11 7 

(39) 5 7 7 $\leftarrow$ (45) 7 7 

(43) 13 1 1 $\leftarrow$ (44) 14 1 

(50) 2 3 3 $\leftarrow$ (57) 1 1 

(51) 5 1 1 $\leftarrow$ (52) 6 1 

(55) 1 1 1 $\leftarrow$ (56) 2 1 

\end{tcolorbox}
\begin{tcolorbox}[title={$(58,5)$}]
(13) 14 13 11 7 $\leftarrow$ (14) 15 15 15 

(17) 10 13 11 7 $\leftarrow$ (18) 11 15 15 

(18) 19 7 7 7 $\leftarrow$ (20) 21 11 7 

(19) 20 5 7 7 $\leftarrow$ (22) 23 7 7 

(25) 14 5 7 7 $\leftarrow$ (28) 13 11 7 

(27) 28 1 1 1 $\leftarrow$ (28) 29 1 1 

(35) 4 5 7 7 $\leftarrow$ (38) 7 7 7 

(36) 3 5 7 7 $\leftarrow$ (40) 5 7 7 

(43) 12 1 1 1 $\leftarrow$ (44) 13 1 1 

(49) 1 2 3 3 $\leftarrow$ (56) 1 1 1 

(51) 4 1 1 1 $\leftarrow$ (52) 5 1 1 

\end{tcolorbox}
\begin{tcolorbox}[title={$(58,6)$}]
(5) 5 7 11 15 15 $\leftarrow$ (9) 9 11 15 15 

(7) 7 13 13 11 7 $\leftarrow$ (14) 14 13 11 7 

(11) 7 11 15 7 7 $\leftarrow$ (19) 11 15 7 7 

(17) 9 11 7 7 7 $\leftarrow$ (18) 10 13 11 7 

(18) 18 3 5 7 7 $\leftarrow$ (20) 20 5 7 7 

(19) 7 13 5 7 7 $\leftarrow$ (21) 17 7 7 7 

(23) 5 9 7 7 7 $\leftarrow$ (27) 13 5 7 7 

(25) 11 3 5 7 7 $\leftarrow$ (26) 14 5 7 7 

(27) 24 4 1 1 1 $\leftarrow$ (28) 28 1 1 1 

(28) 12 9 3 3 3 

(34) 2 3 5 7 7 $\leftarrow$ (36) 4 5 7 7 

(37) 3 5 7 3 3 $\leftarrow$ (41) 5 7 3 3 

(41) 3 6 2 3 3 $\leftarrow$ (42) 8 3 3 3 

(43) 2 4 3 3 3 $\leftarrow$ (45) 6 2 3 3 

(43) 8 4 1 1 1 $\leftarrow$ (44) 12 1 1 1 

(47) 4 4 1 1 1 $\leftarrow$ (52) 4 1 1 1 

\end{tcolorbox}
\begin{tcolorbox}[title={$(58,7)$}]
(6) 5 7 11 15 7 7 $\leftarrow$ (11) 10 17 7 7 7 

(7) 6 9 15 7 7 7 $\leftarrow$ (8) 7 13 13 11 7 

(9) 8 9 11 7 7 7 $\leftarrow$ (10) 9 11 15 7 7 

(11) 6 9 11 7 7 7 $\leftarrow$ (12) 7 11 15 7 7 

(13) 12 11 3 5 7 7 $\leftarrow$ (14) 13 13 5 7 7 

(18) 5 5 9 7 7 7 $\leftarrow$ (20) 7 13 5 7 7 

(20) 3 5 9 7 7 7 $\leftarrow$ (24) 5 9 7 7 7 

(21) 13 2 3 5 7 7 $\leftarrow$ (22) 14 4 5 7 7 

(22) 5 9 3 5 7 7 $\leftarrow$ (28) 9 3 5 7 7 

(26) 9 3 6 6 5 3 

(27) 22 * 1 $\leftarrow$ (28) 24 4 1 1 1 

(35) 2 3 5 7 3 3 $\leftarrow$ (36) 3 6 6 5 3 

(41) ..4 3 3 3 $\leftarrow$ (42) 3 6 2 3 3 

(42) 1 2 4 3 3 3 $\leftarrow$ (44) 2 4 3 3 3 

(43) 6 * 1 $\leftarrow$ (44) 8 4 1 1 1 

(44) 3 4 4 1 1 1 $\leftarrow$ (50) * 1 

(47) 2 * 1 $\leftarrow$ (48) 4 4 1 1 1 

\end{tcolorbox}
\begin{tcolorbox}[title={$(58,8)$}]
(1) 27 12 4 5 3 3 3 

(3) 1 2 4 7 11 15 15 $\leftarrow$ (5) 2 4 7 11 15 15 

(5) 3 5 9 15 7 7 7 $\leftarrow$ (9) 5 9 15 7 7 7 

(7) 4 6 9 11 7 7 7 $\leftarrow$ (8) 6 9 15 7 7 7 

(11) 9 3 5 9 7 7 7 $\leftarrow$ (13) 11 5 9 7 7 7 

(13) 7 8 12 9 3 3 3 $\leftarrow$ (14) 12 11 3 5 7 7 

(19) 3 5 9 3 5 7 7 $\leftarrow$ (21) 8 12 9 3 3 3 

(21) 3 5 7 3 5 7 7 $\leftarrow$ (25) 5 7 3 5 7 7 

(21) 7 12 4 5 3 3 3 $\leftarrow$ (22) 13 2 3 5 7 7 

(23) 6 5 2 3 5 7 7 $\leftarrow$ (29) 12 4 5 3 3 3 

(25) 5 5 3 6 6 5 3 

(27) 21 1 * 1 $\leftarrow$ (28) 22 * 1 

(41) 1 1 2 4 3 3 3 $\leftarrow$ (42) ..4 3 3 3 

(42) 2 3 4 4 1 1 1 $\leftarrow$ (49) 1 * 1 

(43) 5 1 * 1 $\leftarrow$ (44) 6 * 1 

(47) 1 1 * 1 $\leftarrow$ (48) 2 * 1 

\end{tcolorbox}
\begin{tcolorbox}[title={$(58,9)$}]
(1) 6 4 6 9 11 7 7 7 $\leftarrow$ (4) 1 2 4 7 11 15 15 

(1) 18 3 5 9 3 5 7 7 $\leftarrow$ (2) 27 12 4 5 3 3 3 

(6) 4 8 5 5 9 7 7 7 $\leftarrow$ (12) 9 3 5 9 7 7 7 

(7) 3 5 10 11 3 5 7 7 $\leftarrow$ (11) 5 10 11 3 5 7 7 

(11) 4 5 3 5 9 7 7 7 $\leftarrow$ (13) 8 3 5 9 7 7 7 

(11) 8 3 5 9 3 5 7 7 $\leftarrow$ (14) 7 8 12 9 3 3 3 

(21) 4 7 3 3 6 6 5 3 $\leftarrow$ (27) 6 3 3 6 6 5 3 

(21) 5 6 3 3 6 6 5 3 $\leftarrow$ (22) 7 12 4 5 3 3 3 

(23) 3 6 3 3 6 6 5 3 $\leftarrow$ (24) 6 5 2 3 5 7 7 

(27) 20 1 1 * 1 $\leftarrow$ (28) 21 1 * 1 

(29) 3 6 2 4 5 3 3 3 $\leftarrow$ (30) 6 2 3 5 7 3 3 

(35) 12 1 1 * 1 $\leftarrow$ (36) 13 1 * 1 

(41) 1 2 3 4 4 1 1 1 $\leftarrow$ (48) 1 1 * 1 

(43) 4 1 1 * 1 $\leftarrow$ (44) 5 1 * 1 

\end{tcolorbox}
\begin{tcolorbox}[title={$(58,10)$}]
(1) 4 2 4 6 9 11 7 7 7 $\leftarrow$ (2) 6 4 6 9 11 7 7 7 

(1) 11 9 3 5 7 3 5 7 7 $\leftarrow$ (2) 18 3 5 9 3 5 7 7 

(2) 4 4 8 5 5 9 7 7 7 $\leftarrow$ (8) 3 5 10 11 3 5 7 7 

(5) 2 3 5 10 11 3 5 7 7 

(7) 5 9 3 5 7 3 5 7 7 $\leftarrow$ (13) 9 3 5 7 3 5 7 7 

(10) 2 3 5 3 5 9 7 7 7 $\leftarrow$ (12) 4 5 3 5 9 7 7 7 

(10) 4 5 3 5 9 3 5 7 7 $\leftarrow$ (12) 8 3 5 9 3 5 7 7 

(19) 2 4 7 3 3 6 6 5 3 $\leftarrow$ (22) 4 7 3 3 6 6 5 3 

(20) 2 3 5 5 3 6 6 5 3 $\leftarrow$ (25) 5 6 2 3 5 7 3 3 

(21) 3 5 6 2 3 5 7 3 3 $\leftarrow$ (22) 5 6 3 3 6 6 5 3 

(23) 2 4 3 3 3 6 6 5 3 $\leftarrow$ (24) 3 6 3 3 6 6 5 3 

(27) 3 6 ..4 5 3 3 3 $\leftarrow$ (28) 5 6 2 4 5 3 3 3 

(27) 16 4 1 1 * 1 $\leftarrow$ (28) 20 1 1 * 1 

(29) ....3 5 7 3 3 $\leftarrow$ (30) 3 6 2 4 5 3 3 3 

(33) 3 6 2 3 4 4 1 1 1 $\leftarrow$ (34) 8 1 1 2 4 3 3 3 

(35) * 2 4 3 3 3 $\leftarrow$ (37) 6 2 3 4 4 1 1 1 

(35) 8 4 1 1 * 1 $\leftarrow$ (36) 12 1 1 * 1 

(39) 4 4 1 1 * 1 $\leftarrow$ (44) 4 1 1 * 1 

\end{tcolorbox}
\begin{tcolorbox}[title={$(58,11)$}]
(1) ...4 6 9 11 7 7 7 $\leftarrow$ (2) 4 2 4 6 9 11 7 7 7 

(2) 5 5 9 3 5 7 3 5 7 7 $\leftarrow$ (8) 5 9 3 5 7 3 5 7 7 

(4) 1 2 3 5 10 11 3 5 7 7 

(11) 4 3 5 6 5 2 3 5 7 7 $\leftarrow$ (14) 5 5 6 5 2 3 5 7 7 

(17) ..4 7 3 3 6 6 5 3 $\leftarrow$ (24) 2 4 3 3 3 6 6 5 3 

(18) 1 2 4 7 3 3 6 6 5 3 $\leftarrow$ (20) 2 4 7 3 3 6 6 5 3 

(21) ..4 3 3 3 6 6 5 3 $\leftarrow$ (22) 3 5 6 2 3 5 7 3 3 

(27) .....3 5 7 3 3 $\leftarrow$ (28) 3 6 ..4 5 3 3 3 

(27) 14 * * 1 $\leftarrow$ (28) 16 4 1 1 * 1 

(33) 2 * 2 4 3 3 3 $\leftarrow$ (34) 3 6 2 3 4 4 1 1 1 

(34) 1 * 2 4 3 3 3 $\leftarrow$ (36) * 2 4 3 3 3 

(35) 6 * * 1 $\leftarrow$ (36) 8 4 1 1 * 1 

(36) 3 4 4 1 1 * 1 $\leftarrow$ (42) * * 1 

(39) 2 * * 1 $\leftarrow$ (40) 4 4 1 1 * 1 

\end{tcolorbox}
\begin{tcolorbox}[title={$(58,12)$}]
(1) 4 5 4 5 3 5 9 3 5 7 7 

(3) 5 5 5 5 6 5 2 3 5 7 7 $\leftarrow$ (9) 5 5 5 6 5 2 3 5 7 7 

(5) 16 ..4 3 3 3 6 6 5 3 $\leftarrow$ (6) 18 2 4 3 3 3 6 6 5 3 

(9) 2 4 3 5 6 5 2 3 5 7 7 $\leftarrow$ (12) 4 3 5 6 5 2 3 5 7 7 

(15) 6 ..4 3 3 3 6 6 5 3 $\leftarrow$ (22) ..4 3 3 3 6 6 5 3 

(17) 1 1 2 4 7 3 3 6 6 5 3 $\leftarrow$ (18) ..4 7 3 3 6 6 5 3 

(27) 13 1 * * 1 $\leftarrow$ (28) 14 * * 1 

(33) 1 1 * 2 4 3 3 3 $\leftarrow$ (34) 2 * 2 4 3 3 3 

(34) 2 3 4 4 1 1 * 1 $\leftarrow$ (41) 1 * * 1 

(35) 5 1 * * 1 $\leftarrow$ (36) 6 * * 1 

(39) 1 1 * * 1 $\leftarrow$ (40) 2 * * 1 

\end{tcolorbox}
\begin{tcolorbox}[title={$(58,13)$}]
(1) 4 3 5 5 5 6 5 2 3 5 7 7 $\leftarrow$ (4) 5 5 5 5 6 5 2 3 5 7 7 

(3) 1 2 4 2 3 5 3 5 9 7 7 7 $\leftarrow$ (5) 2 4 2 3 5 3 5 9 7 7 7 

(5) 13 6 ....3 3 6 6 5 3 $\leftarrow$ (6) 16 ..4 3 3 3 6 6 5 3 

(7) ..4 3 5 6 5 2 3 5 7 7 $\leftarrow$ (10) 2 4 3 5 6 5 2 3 5 7 7 

(13) 4 1 1 2 4 7 3 3 6 6 5 3 $\leftarrow$ (19) 6 ....3 3 6 6 5 3 

(15) 3 6 ....3 3 6 6 5 3 $\leftarrow$ (16) 6 ..4 3 3 3 6 6 5 3 

(21) 3 6 .....4 5 3 3 3 $\leftarrow$ (22) 6 .....3 5 7 3 3 

(27) 12 1 1 * * 1 $\leftarrow$ (28) 13 1 * * 1 

(33) 1 2 3 4 4 1 1 * 1 $\leftarrow$ (40) 1 1 * * 1 

(35) 4 1 1 * * 1 $\leftarrow$ (36) 5 1 * * 1 

\end{tcolorbox}
\begin{tcolorbox}[title={$(58,14)$}]
(2) 1 1 2 4 2 3 5 3 5 9 7 7 7 $\leftarrow$ (4) 1 2 4 2 3 5 3 5 9 7 7 7 

(2) ..4 3 5 5 6 5 2 3 5 7 7 

(5) ...4 3 5 6 5 2 3 5 7 7 $\leftarrow$ (8) ..4 3 5 6 5 2 3 5 7 7 

(5) 11 5 6 .....3 5 7 3 3 $\leftarrow$ (6) 13 6 ....3 3 6 6 5 3 

(11) * 2 4 7 3 3 6 6 5 3 $\leftarrow$ (14) 4 1 1 2 4 7 3 3 6 6 5 3 

(11) 24 4 1 1 * * 1 $\leftarrow$ (17) 5 6 .....3 5 7 3 3 

(15) .....4 3 3 3 6 6 5 3 $\leftarrow$ (16) 3 6 ....3 3 6 6 5 3 

(19) 3 6 ......4 5 3 3 3 $\leftarrow$ (20) 5 6 .....4 5 3 3 3 

(21) ........3 5 7 3 3 $\leftarrow$ (22) 3 6 .....4 5 3 3 3 

(25) 3 6 2 3 4 4 1 1 * 1 $\leftarrow$ (26) 8 1 1 * 2 4 3 3 3 

(27) * * 2 4 3 3 3 $\leftarrow$ (29) 6 2 3 4 4 1 1 * 1 

(27) 8 4 1 1 * * 1 $\leftarrow$ (28) 12 1 1 * * 1 

(31) 4 4 1 1 * * 1 $\leftarrow$ (36) 4 1 1 * * 1 

\end{tcolorbox}
\begin{tcolorbox}[title={$(58,15)$}]
(3) ....4 3 5 6 5 2 3 5 7 7 $\leftarrow$ (6) ...4 3 5 6 5 2 3 5 7 7 

(5) 10 .....4 3 3 3 6 6 5 3 $\leftarrow$ (6) 11 5 6 .....3 5 7 3 3 

(9) 2 * 2 4 7 3 3 6 6 5 3 $\leftarrow$ (16) .....4 3 3 3 6 6 5 3 

(10) 1 * 2 4 7 3 3 6 6 5 3 $\leftarrow$ (12) * 2 4 7 3 3 6 6 5 3 

(11) 22 * * * 1 $\leftarrow$ (12) 24 4 1 1 * * 1 

(19) .........3 5 7 3 3 $\leftarrow$ (20) 3 6 ......4 5 3 3 3 

(25) 2 * * 2 4 3 3 3 $\leftarrow$ (26) 3 6 2 3 4 4 1 1 * 1 

(26) 1 * * 2 4 3 3 3 $\leftarrow$ (28) * * 2 4 3 3 3 

(27) 6 * * * 1 $\leftarrow$ (28) 8 4 1 1 * * 1 

(28) 3 4 4 1 1 * * 1 $\leftarrow$ (34) * * * 1 

(31) 2 * * * 1 $\leftarrow$ (32) 4 4 1 1 * * 1 

\end{tcolorbox}
\begin{tcolorbox}[title={$(58,16)$}]
(1) .....4 3 5 6 5 2 3 5 7 7 $\leftarrow$ (4) ....4 3 5 6 5 2 3 5 7 7 

(5) 7 1.........4 5 3 3 3 $\leftarrow$ (6) 10 .....4 3 3 3 6 6 5 3 

(7) 6 ......4 3 3 3 6 6 5 3 $\leftarrow$ (13) 1.........4 5 3 3 3 

(9) 1 1 * 2 4 7 3 3 6 6 5 3 $\leftarrow$ (10) 2 * 2 4 7 3 3 6 6 5 3 

(11) 21 1 * * * 1 $\leftarrow$ (12) 22 * * * 1 

(25) 1 1 * * 2 4 3 3 3 $\leftarrow$ (26) 2 * * 2 4 3 3 3 

(26) 2 3 4 4 1 1 * * 1 $\leftarrow$ (33) 1 * * * 1 

(27) 5 1 * * * 1 $\leftarrow$ (28) 6 * * * 1 

(31) 1 1 * * * 1 $\leftarrow$ (32) 2 * * * 1 

\end{tcolorbox}
\begin{tcolorbox}[title={$(59,3)$}]
(27) 29 3 $\leftarrow$ (29) 31 

(43) 13 3 $\leftarrow$ (45) 15 

(51) 5 3 $\leftarrow$ (53) 7 

(53) 3 3 $\leftarrow$ (57) 3 

\end{tcolorbox}
\begin{tcolorbox}[title={$(59,4)$}]
(21) 22 13 3 $\leftarrow$ (22) 23 15 

(25) 26 5 3 $\leftarrow$ (26) 27 7 

(26) 27 3 3 $\leftarrow$ (28) 29 3 

(29) 14 13 3 $\leftarrow$ (30) 15 15 

(41) 10 5 3 $\leftarrow$ (42) 11 7 

(42) 11 3 3 $\leftarrow$ (44) 13 3 

(45) 6 5 3 $\leftarrow$ (46) 7 7 

(50) 3 3 3 $\leftarrow$ (52) 5 3 

(51) 2 3 3 $\leftarrow$ (54) 3 3 

\end{tcolorbox}
\begin{tcolorbox}[title={$(59,5)$}]
(11) 7 11 15 15 $\leftarrow$ (19) 11 15 15 

(15) 13 13 11 7 $\leftarrow$ (29) 13 11 7 

(19) 19 7 7 7 $\leftarrow$ (21) 21 11 7 

(21) 21 11 3 3 $\leftarrow$ (22) 22 13 3 

(25) 25 3 3 3 $\leftarrow$ (26) 26 5 3 

(29) 13 11 3 3 $\leftarrow$ (30) 14 13 3 

(37) 3 5 7 7 $\leftarrow$ (41) 5 7 7 

(39) 6 6 5 3 $\leftarrow$ (46) 6 5 3 

(41) 9 3 3 3 $\leftarrow$ (42) 10 5 3 

(50) 1 2 3 3 $\leftarrow$ (52) 2 3 3 

\end{tcolorbox}
\begin{tcolorbox}[title={$(59,6)$}]
(6) 5 7 11 15 15 $\leftarrow$ (10) 9 11 15 15 

(7) 4 7 11 15 15 $\leftarrow$ (12) 7 11 15 15 

(14) 9 15 7 7 7 $\leftarrow$ (27) 14 5 7 7 

(15) 11 14 5 7 7 $\leftarrow$ (16) 13 13 11 7 

(18) 9 11 7 7 7 $\leftarrow$ (22) 17 7 7 7 

(19) 7 14 5 7 7 $\leftarrow$ (20) 11 15 7 7 

(19) 18 3 5 7 7 $\leftarrow$ (21) 20 5 7 7 

(21) 20 9 3 3 3 $\leftarrow$ (22) 21 11 3 3 

(26) 11 3 5 7 7 $\leftarrow$ (28) 13 5 7 7 

(29) 12 9 3 3 3 $\leftarrow$ (30) 13 11 3 3 

(35) 2 3 5 7 7 $\leftarrow$ (37) 4 5 7 7 

(38) 3 5 7 3 3 $\leftarrow$ (43) 8 3 3 3 

(39) 4 7 3 3 3 $\leftarrow$ (40) 6 6 5 3 

(41) 4 5 3 3 3 $\leftarrow$ (42) 5 7 3 3 

(45) 5 1 2 3 3 $\leftarrow$ (46) 6 2 3 3 

\end{tcolorbox}
\begin{tcolorbox}[title={$(59,7)$}]
(1) 28 12 9 3 3 3 $\leftarrow$ (8) 4 7 11 15 15 

(7) 5 7 11 15 7 7 $\leftarrow$ (9) 7 13 13 11 7 

(10) 8 9 11 7 7 7 $\leftarrow$ (23) 14 4 5 7 7 

(11) 9 7 13 5 7 7 $\leftarrow$ (12) 10 17 7 7 7 

(12) 6 9 11 7 7 7 $\leftarrow$ (21) 7 13 5 7 7 

(15) 7 14 4 5 7 7 $\leftarrow$ (16) 11 14 5 7 7 

(17) 12 12 9 3 3 3 $\leftarrow$ (20) 18 3 5 7 7 

(19) 5 5 9 7 7 7 $\leftarrow$ (20) 7 14 5 7 7 

(19) 17 3 6 6 5 3 $\leftarrow$ (22) 20 9 3 3 3 

(21) 3 5 9 7 7 7 $\leftarrow$ (25) 5 9 7 7 7 

(23) 5 9 3 5 7 7 $\leftarrow$ (29) 9 3 5 7 7 

(27) 9 3 6 6 5 3 $\leftarrow$ (30) 12 9 3 3 3 

(33) 3 3 6 6 5 3 $\leftarrow$ (36) 2 3 5 7 7 

(36) 2 3 5 7 3 3 $\leftarrow$ (42) 4 5 3 3 3 

(39) 2 4 5 3 3 3 $\leftarrow$ (40) 4 7 3 3 3 

(43) 1 2 4 3 3 3 $\leftarrow$ (45) 2 4 3 3 3 

(45) 3 4 4 1 1 1 $\leftarrow$ (46) 5 1 2 3 3 

\end{tcolorbox}
\begin{tcolorbox}[title={$(59,8)$}]
(1) 26 9 3 6 6 5 3 $\leftarrow$ (2) 28 12 9 3 3 3 

(6) 3 5 9 15 7 7 7 $\leftarrow$ (10) 5 9 15 7 7 7 

(8) 4 6 9 11 7 7 7 $\leftarrow$ (16) 7 14 4 5 7 7 

(11) 8 5 5 9 7 7 7 $\leftarrow$ (12) 9 7 13 5 7 7 

(17) 10 9 3 6 6 5 3 $\leftarrow$ (18) 12 12 9 3 3 3 

(18) 11 12 4 5 3 3 3 $\leftarrow$ (20) 17 3 6 6 5 3 

(20) 3 5 9 3 5 7 7 $\leftarrow$ (26) 5 7 3 5 7 7 

(21) 6 9 3 6 6 5 3 $\leftarrow$ (22) 8 12 9 3 3 3 

(22) 3 5 7 3 5 7 7 $\leftarrow$ (24) 5 9 3 5 7 7 

(26) 5 5 3 6 6 5 3 $\leftarrow$ (28) 9 3 6 6 5 3 

(29) 10 2 4 5 3 3 3 $\leftarrow$ (30) 12 4 5 3 3 3 

(33) 6 2 4 5 3 3 3 $\leftarrow$ (40) 2 4 5 3 3 3 

(42) 1 1 2 4 3 3 3 $\leftarrow$ (44) 1 2 4 3 3 3 

(43) 2 3 4 4 1 1 1 $\leftarrow$ (46) 3 4 4 1 1 1 

\end{tcolorbox}
\begin{tcolorbox}[title={$(59,9)$}]
(1) 25 5 5 3 6 6 5 3 $\leftarrow$ (2) 26 9 3 6 6 5 3 

(5) 13 11 12 4 5 3 3 3 $\leftarrow$ (14) 8 3 5 9 7 7 7 

(7) 4 8 5 5 9 7 7 7 $\leftarrow$ (12) 8 5 5 9 7 7 7 

(11) 4 6 3 5 9 7 7 7 $\leftarrow$ (12) 5 10 11 3 5 7 7 

(17) 9 5 5 3 6 6 5 3 $\leftarrow$ (18) 10 9 3 6 6 5 3 

(19) 5 6 5 2 3 5 7 7 $\leftarrow$ (25) 6 5 2 3 5 7 7 

(21) 3 6 5 2 3 5 7 7 $\leftarrow$ (22) 6 9 3 6 6 5 3 

(27) 3 6 2 3 5 7 3 3 $\leftarrow$ (28) 6 3 3 6 6 5 3 

(29) 7 13 1 * 1 $\leftarrow$ (30) 10 2 4 5 3 3 3 

(31) 6 ..4 5 3 3 3 $\leftarrow$ (37) 13 1 * 1 

(33) 4 ..4 5 3 3 3 $\leftarrow$ (34) 6 2 4 5 3 3 3 

(42) 1 2 3 4 4 1 1 1 $\leftarrow$ (44) 2 3 4 4 1 1 1 

\end{tcolorbox}
\begin{tcolorbox}[title={$(59,10)$}]
(2) 11 9 3 5 7 3 5 7 7 $\leftarrow$ (9) 3 5 10 11 3 5 7 7 

(3) 4 4 8 5 5 9 7 7 7 $\leftarrow$ (8) 4 8 5 5 9 7 7 7 

(5) 12 9 5 5 3 6 6 5 3 $\leftarrow$ (6) 13 11 12 4 5 3 3 3 

(6) 2 3 5 10 11 3 5 7 7 

(11) 2 3 5 3 5 9 7 7 7 $\leftarrow$ (12) 4 6 3 5 9 7 7 7 

(11) 4 5 3 5 9 3 5 7 7 $\leftarrow$ (13) 8 3 5 9 3 5 7 7 

(17) 4 5 5 5 3 6 6 5 3 $\leftarrow$ (20) 5 6 5 2 3 5 7 7 

(21) 2 3 5 5 3 6 6 5 3 $\leftarrow$ (22) 3 6 5 2 3 5 7 7 

(25) 3 5 6 2 4 5 3 3 3 $\leftarrow$ (26) 5 6 2 3 5 7 3 3 

(27) ...3 3 6 6 5 3 $\leftarrow$ (28) 3 6 2 3 5 7 3 3 

(29) 5 8 1 1 2 4 3 3 3 $\leftarrow$ (30) 7 13 1 * 1 

(30) ....3 5 7 3 3 $\leftarrow$ (35) 8 1 1 2 4 3 3 3 

(31) 4 ...4 5 3 3 3 $\leftarrow$ (32) 6 ..4 5 3 3 3 

(33) ....4 5 3 3 3 $\leftarrow$ (34) 4 ..4 5 3 3 3 

(37) 5 1 2 3 4 4 1 1 1 $\leftarrow$ (38) 6 2 3 4 4 1 1 1 

\end{tcolorbox}
\begin{tcolorbox}[title={$(59,11)$}]
(2) ...4 6 9 11 7 7 7 $\leftarrow$ (4) 4 4 8 5 5 9 7 7 7 

(3) 5 5 9 3 5 7 3 5 7 7 $\leftarrow$ (9) 5 9 3 5 7 3 5 7 7 

(3) 8 4 5 3 5 9 3 5 7 7 

(5) 1 2 3 5 10 11 3 5 7 7 $\leftarrow$ (6) 12 9 5 5 3 6 6 5 3 

(5) 16 2 3 5 5 3 6 6 5 3 $\leftarrow$ (12) 4 5 3 5 9 3 5 7 7 

(15) 2 4 5 5 5 3 6 6 5 3 $\leftarrow$ (18) 4 5 5 5 3 6 6 5 3 

(19) 1 2 4 7 3 3 6 6 5 3 $\leftarrow$ (21) 2 4 7 3 3 6 6 5 3 

(25) ....3 3 6 6 5 3 $\leftarrow$ (26) 3 5 6 2 4 5 3 3 3 

(28) .....3 5 7 3 3 $\leftarrow$ (34) ....4 5 3 3 3 

(29) 4 ....4 5 3 3 3 $\leftarrow$ (30) 5 8 1 1 2 4 3 3 3 

(31) .....4 5 3 3 3 $\leftarrow$ (32) 4 ...4 5 3 3 3 

(35) 1 * 2 4 3 3 3 $\leftarrow$ (37) * 2 4 3 3 3 

(37) 3 4 4 1 1 * 1 $\leftarrow$ (38) 5 1 2 3 4 4 1 1 1 

\end{tcolorbox}
\begin{tcolorbox}[title={$(59,12)$}]
(2) 4 5 4 5 3 5 9 3 5 7 7 

(3) 4 4 2 3 5 3 5 9 7 7 7 $\leftarrow$ (4) 8 4 5 3 5 9 3 5 7 7 

(5) 13 ..4 7 3 3 6 6 5 3 $\leftarrow$ (6) 16 2 3 5 5 3 6 6 5 3 

(7) 4 3 5 5 6 5 2 3 5 7 7 $\leftarrow$ (10) 5 5 5 6 5 2 3 5 7 7 

(13) ..4 5 5 5 3 6 6 5 3 $\leftarrow$ (16) 2 4 5 5 5 3 6 6 5 3 

(18) 1 1 2 4 7 3 3 6 6 5 3 $\leftarrow$ (20) 1 2 4 7 3 3 6 6 5 3 

(25) 6 .....4 5 3 3 3 $\leftarrow$ (32) .....4 5 3 3 3 

(29) ......4 5 3 3 3 $\leftarrow$ (30) 4 ....4 5 3 3 3 

(34) 1 1 * 2 4 3 3 3 $\leftarrow$ (36) 1 * 2 4 3 3 3 

(35) 2 3 4 4 1 1 * 1 $\leftarrow$ (38) 3 4 4 1 1 * 1 

\end{tcolorbox}
\begin{tcolorbox}[title={$(59,13)$}]
(2) 4 3 5 5 5 6 5 2 3 5 7 7 

(3) ..4 2 3 5 3 5 9 7 7 7 $\leftarrow$ (4) 4 4 2 3 5 3 5 9 7 7 7 

(5) 2 4 3 5 5 6 5 2 3 5 7 7 $\leftarrow$ (8) 4 3 5 5 6 5 2 3 5 7 7 

(5) 11 6 ..4 3 3 3 6 6 5 3 $\leftarrow$ (6) 13 ..4 7 3 3 6 6 5 3 

(11) ...4 5 5 5 3 6 6 5 3 $\leftarrow$ (14) ..4 5 5 5 3 6 6 5 3 

(19) 3 6 .....3 5 7 3 3 $\leftarrow$ (20) 6 ....3 3 6 6 5 3 

(23) 6 ......4 5 3 3 3 $\leftarrow$ (30) ......4 5 3 3 3 

(25) 4 ......4 5 3 3 3 $\leftarrow$ (26) 6 .....4 5 3 3 3 

(34) 1 2 3 4 4 1 1 * 1 $\leftarrow$ (36) 2 3 4 4 1 1 * 1 

\end{tcolorbox}
\begin{tcolorbox}[title={$(59,14)$}]
(3) 1 1 2 4 2 3 5 3 5 9 7 7 7 $\leftarrow$ (4) ..4 2 3 5 3 5 9 7 7 7 

(3) ..4 3 5 5 6 5 2 3 5 7 7 $\leftarrow$ (6) 2 4 3 5 5 6 5 2 3 5 7 7 

(5) 9 4 1 1 2 4 7 3 3 6 6 5 3 $\leftarrow$ (6) 11 6 ..4 3 3 3 6 6 5 3 

(9) ....4 5 5 5 3 6 6 5 3 $\leftarrow$ (12) ...4 5 5 5 3 6 6 5 3 

(17) 3 5 6 .....4 5 3 3 3 $\leftarrow$ (18) 5 6 .....3 5 7 3 3 

(19) .......3 3 6 6 5 3 $\leftarrow$ (20) 3 6 .....3 5 7 3 3 

(22) ........3 5 7 3 3 $\leftarrow$ (27) 8 1 1 * 2 4 3 3 3 

(23) 4 .......4 5 3 3 3 $\leftarrow$ (24) 6 ......4 5 3 3 3 

(25) ........4 5 3 3 3 $\leftarrow$ (26) 4 ......4 5 3 3 3 

(29) 5 1 2 3 4 4 1 1 * 1 $\leftarrow$ (30) 6 2 3 4 4 1 1 * 1 

\end{tcolorbox}
\begin{tcolorbox}[title={$(59,15)$}]
(1) ...4 3 5 5 6 5 2 3 5 7 7 $\leftarrow$ (4) ..4 3 5 5 6 5 2 3 5 7 7 

(5) 8 ..1 ..4 7 3 3 6 6 5 3 $\leftarrow$ (6) 9 4 1 1 2 4 7 3 3 6 6 5 3 

(7) .....4 5 5 5 3 6 6 5 3 $\leftarrow$ (10) ....4 5 5 5 3 6 6 5 3 

(11) 1 * 2 4 7 3 3 6 6 5 3 $\leftarrow$ (13) * 2 4 7 3 3 6 6 5 3 

(17) ........3 3 6 6 5 3 $\leftarrow$ (18) 3 5 6 .....4 5 3 3 3 

(20) .........3 5 7 3 3 $\leftarrow$ (26) ........4 5 3 3 3 

(23) .........4 5 3 3 3 $\leftarrow$ (24) 4 .......4 5 3 3 3 

(27) 1 * * 2 4 3 3 3 $\leftarrow$ (29) * * 2 4 3 3 3 

(29) 3 4 4 1 1 * * 1 $\leftarrow$ (30) 5 1 2 3 4 4 1 1 * 1 

\end{tcolorbox}
\begin{tcolorbox}[title={$(60,2)$}]
(59) 1 $\leftarrow$ (61) 

\end{tcolorbox}
\begin{tcolorbox}[title={$(60,3)$}]
(29) 30 1 $\leftarrow$ (30) 31 

(45) 14 1 $\leftarrow$ (46) 15 

(53) 6 1 $\leftarrow$ (54) 7 

(57) 2 1 $\leftarrow$ (58) 3 

(58) 1 1 $\leftarrow$ (60) 1 

\end{tcolorbox}
\begin{tcolorbox}[title={$(60,4)$}]
(15) 15 15 15 $\leftarrow$ (23) 23 15 

(23) 23 7 7 $\leftarrow$ (27) 27 7 

(27) 27 3 3 $\leftarrow$ (29) 29 3 

(29) 29 1 1 $\leftarrow$ (30) 30 1 

(39) 7 7 7 $\leftarrow$ (43) 11 7 

(43) 11 3 3 $\leftarrow$ (45) 13 3 

(45) 13 1 1 $\leftarrow$ (46) 14 1 

(51) 3 3 3 $\leftarrow$ (53) 5 3 

(53) 5 1 1 $\leftarrow$ (54) 6 1 

(57) 1 1 1 $\leftarrow$ (58) 2 1 

\end{tcolorbox}
\begin{tcolorbox}[title={$(60,5)$}]
(15) 14 13 11 7 $\leftarrow$ (16) 15 15 15 

(19) 10 13 11 7 $\leftarrow$ (20) 11 15 15 

(20) 19 7 7 7 $\leftarrow$ (24) 23 7 7 

(26) 25 3 3 3 $\leftarrow$ (28) 27 3 3 

(29) 28 1 1 1 $\leftarrow$ (30) 29 1 1 

(38) 3 5 7 7 $\leftarrow$ (42) 5 7 7 

(42) 9 3 3 3 $\leftarrow$ (44) 11 3 3 

(45) 12 1 1 1 $\leftarrow$ (46) 13 1 1 

(51) 1 2 3 3 $\leftarrow$ (53) 2 3 3 

(53) 4 1 1 1 $\leftarrow$ (54) 5 1 1 

\end{tcolorbox}
\begin{tcolorbox}[title={$(60,6)$}]
(7) 5 7 11 15 15 $\leftarrow$ (11) 9 11 15 15 

(11) 9 11 15 7 7 $\leftarrow$ (23) 17 7 7 7 

(13) 7 11 15 7 7 $\leftarrow$ (21) 11 15 7 7 

(15) 9 15 7 7 7 $\leftarrow$ (17) 13 13 11 7 

(15) 13 13 5 7 7 $\leftarrow$ (22) 20 5 7 7 

(19) 9 11 7 7 7 $\leftarrow$ (20) 10 13 11 7 

(27) 11 3 5 7 7 $\leftarrow$ (28) 14 5 7 7 

(29) 24 4 1 1 1 $\leftarrow$ (30) 28 1 1 1 

(37) 3 6 6 5 3 $\leftarrow$ (43) 5 7 3 3 

(39) 3 5 7 3 3 $\leftarrow$ (41) 6 6 5 3 

(43) 3 6 2 3 3 $\leftarrow$ (44) 8 3 3 3 

(45) 8 4 1 1 1 $\leftarrow$ (46) 12 1 1 1 

(49) 4 4 1 1 1 $\leftarrow$ (52) 1 2 3 3 

(51) * 1 $\leftarrow$ (54) 4 1 1 1 

\end{tcolorbox}
\begin{tcolorbox}[title={$(60,7)$}]
(6) 2 4 7 11 15 15 

(8) 5 7 11 15 7 7 $\leftarrow$ (16) 9 15 7 7 7 

(9) 6 9 15 7 7 7 $\leftarrow$ (10) 7 13 13 11 7 

(11) 8 9 11 7 7 7 $\leftarrow$ (12) 9 11 15 7 7 

(13) 6 9 11 7 7 7 $\leftarrow$ (14) 7 11 15 7 7 

(14) 11 5 9 7 7 7 $\leftarrow$ (21) 18 3 5 7 7 

(15) 12 11 3 5 7 7 $\leftarrow$ (16) 13 13 5 7 7 

(20) 5 5 9 7 7 7 

(22) 3 5 9 7 7 7 $\leftarrow$ (26) 5 9 7 7 7 

(23) 13 2 3 5 7 7 $\leftarrow$ (24) 14 4 5 7 7 

(29) 22 * 1 $\leftarrow$ (30) 24 4 1 1 1 

(34) 3 3 6 6 5 3 $\leftarrow$ (40) 3 5 7 3 3 

(37) 2 3 5 7 3 3 $\leftarrow$ (38) 3 6 6 5 3 

(43) ..4 3 3 3 $\leftarrow$ (44) 3 6 2 3 3 

(45) 6 * 1 $\leftarrow$ (46) 8 4 1 1 1 

(49) 2 * 1 $\leftarrow$ (50) 4 4 1 1 1 

(50) 1 * 1 $\leftarrow$ (52) * 1 

\end{tcolorbox}
\begin{tcolorbox}[title={$(60,8)$}]
(3) 27 12 4 5 3 3 3 

(5) 1 2 4 7 11 15 15 

(7) 3 5 9 15 7 7 7 $\leftarrow$ (11) 5 9 15 7 7 7 

(9) 4 6 9 11 7 7 7 $\leftarrow$ (10) 6 9 15 7 7 7 

(13) 9 3 5 9 7 7 7 $\leftarrow$ (19) 12 12 9 3 3 3 

(15) 7 8 12 9 3 3 3 $\leftarrow$ (16) 12 11 3 5 7 7 

(19) 11 12 4 5 3 3 3 $\leftarrow$ (21) 17 3 6 6 5 3 

(21) 3 5 9 3 5 7 7 $\leftarrow$ (27) 5 7 3 5 7 7 

(23) 3 5 7 3 5 7 7 $\leftarrow$ (25) 5 9 3 5 7 7 

(23) 7 12 4 5 3 3 3 $\leftarrow$ (24) 13 2 3 5 7 7 

(27) 5 5 3 6 6 5 3 $\leftarrow$ (29) 9 3 6 6 5 3 

(29) 21 1 * 1 $\leftarrow$ (30) 22 * 1 

(31) 6 2 3 5 7 3 3 $\leftarrow$ (38) 2 3 5 7 3 3 

(43) 1 1 2 4 3 3 3 $\leftarrow$ (44) ..4 3 3 3 

(45) 5 1 * 1 $\leftarrow$ (46) 6 * 1 

(49) 1 1 * 1 $\leftarrow$ (50) 2 * 1 

\end{tcolorbox}
\begin{tcolorbox}[title={$(60,9)$}]
(2) 25 5 5 3 6 6 5 3 

(3) 6 4 6 9 11 7 7 7 $\leftarrow$ (16) 7 8 12 9 3 3 3 

(3) 18 3 5 9 3 5 7 7 $\leftarrow$ (4) 27 12 4 5 3 3 3 

(13) 4 5 3 5 9 7 7 7 

(14) 9 3 5 7 3 5 7 7 $\leftarrow$ (24) 3 5 7 3 5 7 7 

(18) 9 5 5 3 6 6 5 3 $\leftarrow$ (20) 11 12 4 5 3 3 3 

(23) 4 7 3 3 6 6 5 3 $\leftarrow$ (28) 5 5 3 6 6 5 3 

(23) 5 6 3 3 6 6 5 3 $\leftarrow$ (24) 7 12 4 5 3 3 3 

(25) 3 6 3 3 6 6 5 3 $\leftarrow$ (26) 6 5 2 3 5 7 7 

(29) 5 6 2 4 5 3 3 3 $\leftarrow$ (35) 6 2 4 5 3 3 3 

(29) 20 1 1 * 1 $\leftarrow$ (30) 21 1 * 1 

(31) 3 6 2 4 5 3 3 3 $\leftarrow$ (32) 6 2 3 5 7 3 3 

(37) 12 1 1 * 1 $\leftarrow$ (38) 13 1 * 1 

(43) 1 2 3 4 4 1 1 1 $\leftarrow$ (45) 2 3 4 4 1 1 1 

(45) 4 1 1 * 1 $\leftarrow$ (46) 5 1 * 1 

\end{tcolorbox}
\begin{tcolorbox}[title={$(60,10)$}]
(3) 4 2 4 6 9 11 7 7 7 $\leftarrow$ (4) 6 4 6 9 11 7 7 7 

(3) 11 9 3 5 7 3 5 7 7 $\leftarrow$ (4) 18 3 5 9 3 5 7 7 

(7) 2 3 5 10 11 3 5 7 7 $\leftarrow$ (10) 3 5 10 11 3 5 7 7 

(12) 2 3 5 3 5 9 7 7 7 

(15) 5 5 6 5 2 3 5 7 7 $\leftarrow$ (21) 5 6 5 2 3 5 7 7 

(22) 2 3 5 5 3 6 6 5 3 $\leftarrow$ (24) 4 7 3 3 6 6 5 3 

(23) 3 5 6 2 3 5 7 3 3 $\leftarrow$ (24) 5 6 3 3 6 6 5 3 

(25) 2 4 3 3 3 6 6 5 3 $\leftarrow$ (26) 3 6 3 3 6 6 5 3 

(28) ...3 3 6 6 5 3 $\leftarrow$ (33) 6 ..4 5 3 3 3 

(29) 3 6 ..4 5 3 3 3 $\leftarrow$ (30) 5 6 2 4 5 3 3 3 

(29) 16 4 1 1 * 1 $\leftarrow$ (30) 20 1 1 * 1 

(31) ....3 5 7 3 3 $\leftarrow$ (32) 3 6 2 4 5 3 3 3 

(35) 3 6 2 3 4 4 1 1 1 $\leftarrow$ (36) 8 1 1 2 4 3 3 3 

(37) 8 4 1 1 * 1 $\leftarrow$ (38) 12 1 1 * 1 

(41) 4 4 1 1 * 1 $\leftarrow$ (44) 1 2 3 4 4 1 1 1 

(43) * * 1 $\leftarrow$ (46) 4 1 1 * 1 

\end{tcolorbox}
\begin{tcolorbox}[title={$(60,11)$}]
(1) 6 2 3 5 10 11 3 5 7 7 

(3) ...4 6 9 11 7 7 7 $\leftarrow$ (4) 4 2 4 6 9 11 7 7 7 

(4) 5 5 9 3 5 7 3 5 7 7 

(6) 1 2 3 5 10 11 3 5 7 7 $\leftarrow$ (8) 2 3 5 10 11 3 5 7 7 

(7) 18 2 4 3 3 3 6 6 5 3 

(13) 4 3 5 6 5 2 3 5 7 7 $\leftarrow$ (16) 5 5 6 5 2 3 5 7 7 

(19) ..4 7 3 3 6 6 5 3 $\leftarrow$ (22) 2 4 7 3 3 6 6 5 3 

(23) ..4 3 3 3 6 6 5 3 $\leftarrow$ (24) 3 5 6 2 3 5 7 3 3 

(26) ....3 3 6 6 5 3 $\leftarrow$ (32) ....3 5 7 3 3 

(29) .....3 5 7 3 3 $\leftarrow$ (30) 3 6 ..4 5 3 3 3 

(29) 14 * * 1 $\leftarrow$ (30) 16 4 1 1 * 1 

(35) 2 * 2 4 3 3 3 $\leftarrow$ (36) 3 6 2 3 4 4 1 1 1 

(37) 6 * * 1 $\leftarrow$ (38) 8 4 1 1 * 1 

(41) 2 * * 1 $\leftarrow$ (42) 4 4 1 1 * 1 

(42) 1 * * 1 $\leftarrow$ (44) * * 1 

\end{tcolorbox}
\begin{tcolorbox}[title={$(60,12)$}]
(1) 4 7 4 5 3 5 9 3 5 7 7 $\leftarrow$ (2) 6 2 3 5 10 11 3 5 7 7 

(3) 4 5 4 5 3 5 9 3 5 7 7 $\leftarrow$ (5) 8 4 5 3 5 9 3 5 7 7 

(5) 5 5 5 5 6 5 2 3 5 7 7 

(6) 2 4 2 3 5 3 5 9 7 7 7 

(7) 16 ..4 3 3 3 6 6 5 3 $\leftarrow$ (8) 18 2 4 3 3 3 6 6 5 3 

(11) 2 4 3 5 6 5 2 3 5 7 7 $\leftarrow$ (14) 4 3 5 6 5 2 3 5 7 7 

(17) 6 ..4 3 3 3 6 6 5 3 $\leftarrow$ (21) 1 2 4 7 3 3 6 6 5 3 

(19) 1 1 2 4 7 3 3 6 6 5 3 $\leftarrow$ (20) ..4 7 3 3 6 6 5 3 

(23) 6 .....3 5 7 3 3 $\leftarrow$ (30) .....3 5 7 3 3 

(29) 13 1 * * 1 $\leftarrow$ (30) 14 * * 1 

(35) 1 1 * 2 4 3 3 3 $\leftarrow$ (36) 2 * 2 4 3 3 3 

(37) 5 1 * * 1 $\leftarrow$ (38) 6 * * 1 

(41) 1 1 * * 1 $\leftarrow$ (42) 2 * * 1 

\end{tcolorbox}
\begin{tcolorbox}[title={$(60,13)$}]
(1) 2 4 5 4 5 3 5 9 3 5 7 7 $\leftarrow$ (2) 4 7 4 5 3 5 9 3 5 7 7 

(3) 4 3 5 5 5 6 5 2 3 5 7 7 $\leftarrow$ (6) 5 5 5 5 6 5 2 3 5 7 7 

(5) 1 2 4 2 3 5 3 5 9 7 7 7 

(7) 13 6 ....3 3 6 6 5 3 $\leftarrow$ (8) 16 ..4 3 3 3 6 6 5 3 

(9) ..4 3 5 6 5 2 3 5 7 7 $\leftarrow$ (12) 2 4 3 5 6 5 2 3 5 7 7 

(15) 4 1 1 2 4 7 3 3 6 6 5 3 $\leftarrow$ (20) 1 1 2 4 7 3 3 6 6 5 3 

(17) 3 6 ....3 3 6 6 5 3 $\leftarrow$ (18) 6 ..4 3 3 3 6 6 5 3 

(21) 5 6 .....4 5 3 3 3 $\leftarrow$ (27) 6 .....4 5 3 3 3 

(23) 3 6 .....4 5 3 3 3 $\leftarrow$ (24) 6 .....3 5 7 3 3 

(29) 12 1 1 * * 1 $\leftarrow$ (30) 13 1 * * 1 

(35) 1 2 3 4 4 1 1 * 1 $\leftarrow$ (37) 2 3 4 4 1 1 * 1 

(37) 4 1 1 * * 1 $\leftarrow$ (38) 5 1 * * 1 

\end{tcolorbox}
\begin{tcolorbox}[title={$(60,14)$}]
(1) 2 4 3 5 5 5 6 5 2 3 5 7 7 $\leftarrow$ (4) 4 3 5 5 5 6 5 2 3 5 7 7 

(4) 1 1 2 4 2 3 5 3 5 9 7 7 7 

(7) ...4 3 5 6 5 2 3 5 7 7 $\leftarrow$ (10) ..4 3 5 6 5 2 3 5 7 7 

(7) 11 5 6 .....3 5 7 3 3 $\leftarrow$ (8) 13 6 ....3 3 6 6 5 3 

(13) 24 4 1 1 * * 1 $\leftarrow$ (16) 4 1 1 2 4 7 3 3 6 6 5 3 

(17) .....4 3 3 3 6 6 5 3 $\leftarrow$ (18) 3 6 ....3 3 6 6 5 3 

(20) .......3 3 6 6 5 3 $\leftarrow$ (25) 6 ......4 5 3 3 3 

(21) 3 6 ......4 5 3 3 3 $\leftarrow$ (22) 5 6 .....4 5 3 3 3 

(23) ........3 5 7 3 3 $\leftarrow$ (24) 3 6 .....4 5 3 3 3 

(27) 3 6 2 3 4 4 1 1 * 1 $\leftarrow$ (28) 8 1 1 * 2 4 3 3 3 

(29) 8 4 1 1 * * 1 $\leftarrow$ (30) 12 1 1 * * 1 

(33) 4 4 1 1 * * 1 $\leftarrow$ (36) 1 2 3 4 4 1 1 * 1 

(35) * * * 1 $\leftarrow$ (38) 4 1 1 * * 1 

\end{tcolorbox}
\begin{tcolorbox}[title={$(61,3)$}]
(31) 15 15 $\leftarrow$ (47) 15 

(47) 7 7 $\leftarrow$ (55) 7 

(55) 3 3 $\leftarrow$ (59) 3 

(59) 1 1 $\leftarrow$ (61) 1 

\end{tcolorbox}
\begin{tcolorbox}[title={$(61,4)$}]
(22) 21 11 7 

(23) 22 13 3 $\leftarrow$ (24) 23 15 

(27) 26 5 3 $\leftarrow$ (28) 27 7 

(30) 13 11 7 $\leftarrow$ (46) 13 3 

(31) 14 13 3 $\leftarrow$ (32) 15 15 

(40) 7 7 7 $\leftarrow$ (54) 5 3 

(43) 10 5 3 $\leftarrow$ (44) 11 7 

(47) 6 5 3 $\leftarrow$ (48) 7 7 

(52) 3 3 3 $\leftarrow$ (56) 3 3 

(58) 1 1 1 $\leftarrow$ (60) 1 1 

\end{tcolorbox}
\begin{tcolorbox}[title={$(61,5)$}]
(13) 7 11 15 15 $\leftarrow$ (21) 11 15 15 

(16) 14 13 11 7 $\leftarrow$ (45) 11 3 3 

(21) 19 7 7 7 $\leftarrow$ (25) 23 7 7 

(23) 21 11 3 3 $\leftarrow$ (24) 22 13 3 

(27) 25 3 3 3 $\leftarrow$ (28) 26 5 3 

(29) 13 5 7 7 

(31) 13 11 3 3 $\leftarrow$ (32) 14 13 3 

(38) 4 5 7 7 $\leftarrow$ (48) 6 5 3 

(39) 3 5 7 7 $\leftarrow$ (43) 5 7 7 

(43) 9 3 3 3 $\leftarrow$ (44) 10 5 3 

(47) 6 2 3 3 $\leftarrow$ (54) 2 3 3 

\end{tcolorbox}
\begin{tcolorbox}[title={$(61,6)$}]
(8) 5 7 11 15 15 

(9) 4 7 11 15 15 

(13) 10 17 7 7 7 $\leftarrow$ (44) 9 3 3 3 

(17) 11 14 5 7 7 $\leftarrow$ (18) 13 13 11 7 

(20) 9 11 7 7 7 

(21) 7 14 5 7 7 $\leftarrow$ (22) 11 15 7 7 

(22) 7 13 5 7 7 $\leftarrow$ (24) 17 7 7 7 

(23) 20 9 3 3 3 $\leftarrow$ (24) 21 11 3 3 

(28) 11 3 5 7 7 

(30) 9 3 5 7 7 $\leftarrow$ (40) 3 5 7 7 

(31) 12 9 3 3 3 $\leftarrow$ (32) 13 11 3 3 

(37) 2 3 5 7 7 $\leftarrow$ (45) 8 3 3 3 

(41) 4 7 3 3 3 $\leftarrow$ (42) 6 6 5 3 

(43) 4 5 3 3 3 $\leftarrow$ (44) 5 7 3 3 

(46) 2 4 3 3 3 $\leftarrow$ (53) 1 2 3 3 

(47) 5 1 2 3 3 $\leftarrow$ (48) 6 2 3 3 

\end{tcolorbox}
\begin{tcolorbox}[title={$(61,7)$}]
(3) 28 12 9 3 3 3 

(7) 2 4 7 11 15 15 $\leftarrow$ (10) 4 7 11 15 15 

(9) 5 7 11 15 7 7 $\leftarrow$ (17) 9 15 7 7 7 

(12) 8 9 11 7 7 7 $\leftarrow$ (32) 12 9 3 3 3 

(13) 9 7 13 5 7 7 $\leftarrow$ (14) 10 17 7 7 7 

(14) 6 9 11 7 7 7 

(15) 11 5 9 7 7 7 $\leftarrow$ (17) 13 13 5 7 7 

(17) 7 14 4 5 7 7 $\leftarrow$ (18) 11 14 5 7 7 

(21) 5 5 9 7 7 7 $\leftarrow$ (22) 7 14 5 7 7 

(23) 3 5 9 7 7 7 $\leftarrow$ (25) 14 4 5 7 7 

(23) 8 12 9 3 3 3 

(31) 12 4 5 3 3 3 $\leftarrow$ (44) 4 5 3 3 3 

(35) 3 3 6 6 5 3 $\leftarrow$ (41) 3 5 7 3 3 

(41) 2 4 5 3 3 3 $\leftarrow$ (42) 4 7 3 3 3 

(45) 1 2 4 3 3 3 $\leftarrow$ (51) 4 4 1 1 1 

(47) 3 4 4 1 1 1 $\leftarrow$ (48) 5 1 2 3 3 

(51) 1 * 1 $\leftarrow$ (53) * 1 

\end{tcolorbox}
\begin{tcolorbox}[title={$(61,8)$}]
(1) 6 2 4 7 11 15 15 

(3) 26 9 3 6 6 5 3 $\leftarrow$ (4) 28 12 9 3 3 3 

(6) 1 2 4 7 11 15 15 $\leftarrow$ (8) 2 4 7 11 15 15 

(8) 3 5 9 15 7 7 7 

(10) 4 6 9 11 7 7 7 $\leftarrow$ (30) 9 3 6 6 5 3 

(13) 5 10 11 3 5 7 7 $\leftarrow$ (24) 3 5 9 7 7 7 

(13) 8 5 5 9 7 7 7 $\leftarrow$ (14) 9 7 13 5 7 7 

(14) 9 3 5 9 7 7 7 $\leftarrow$ (16) 11 5 9 7 7 7 

(15) 8 3 5 9 7 7 7 $\leftarrow$ (18) 7 14 4 5 7 7 

(19) 10 9 3 6 6 5 3 $\leftarrow$ (20) 12 12 9 3 3 3 

(22) 3 5 9 3 5 7 7 $\leftarrow$ (26) 5 9 3 5 7 7 

(23) 6 9 3 6 6 5 3 $\leftarrow$ (24) 8 12 9 3 3 3 

(29) 6 3 3 6 6 5 3 $\leftarrow$ (42) 2 4 5 3 3 3 

(31) 10 2 4 5 3 3 3 $\leftarrow$ (32) 12 4 5 3 3 3 

(44) 1 1 2 4 3 3 3 $\leftarrow$ (48) 3 4 4 1 1 1 

(50) 1 1 * 1 $\leftarrow$ (52) 1 * 1 

\end{tcolorbox}
\begin{tcolorbox}[title={$(61,9)$}]
(1) 5 1 2 4 7 11 15 15 $\leftarrow$ (2) 6 2 4 7 11 15 15 

(3) 25 5 5 3 6 6 5 3 $\leftarrow$ (4) 26 9 3 6 6 5 3 

(7) 13 11 12 4 5 3 3 3 $\leftarrow$ (27) 6 5 2 3 5 7 7 

(9) 4 8 5 5 9 7 7 7 

(13) 4 6 3 5 9 7 7 7 $\leftarrow$ (14) 5 10 11 3 5 7 7 

(14) 4 5 3 5 9 7 7 7 $\leftarrow$ (16) 8 3 5 9 7 7 7 

(14) 8 3 5 9 3 5 7 7 

(15) 9 3 5 7 3 5 7 7 $\leftarrow$ (25) 3 5 7 3 5 7 7 

(19) 9 5 5 3 6 6 5 3 $\leftarrow$ (20) 10 9 3 6 6 5 3 

(23) 3 6 5 2 3 5 7 7 $\leftarrow$ (24) 6 9 3 6 6 5 3 

(27) 5 6 2 3 5 7 3 3 $\leftarrow$ (39) 13 1 * 1 

(29) 3 6 2 3 5 7 3 3 $\leftarrow$ (30) 6 3 3 6 6 5 3 

(31) 7 13 1 * 1 $\leftarrow$ (32) 10 2 4 5 3 3 3 

(35) 4 ..4 5 3 3 3 $\leftarrow$ (36) 6 2 4 5 3 3 3 

(39) 6 2 3 4 4 1 1 1 $\leftarrow$ (46) 2 3 4 4 1 1 1 

\end{tcolorbox}
\begin{tcolorbox}[title={$(61,10)$}]
(1) 3 6 4 6 9 11 7 7 7 $\leftarrow$ (2) 5 1 2 4 7 11 15 15 

(1) 13 4 5 3 5 9 7 7 7 

(4) 11 9 3 5 7 3 5 7 7 

(5) 4 4 8 5 5 9 7 7 7 $\leftarrow$ (25) 4 7 3 3 6 6 5 3 

(7) 12 9 5 5 3 6 6 5 3 $\leftarrow$ (8) 13 11 12 4 5 3 3 3 

(10) 5 9 3 5 7 3 5 7 7 $\leftarrow$ (16) 9 3 5 7 3 5 7 7 

(13) 2 3 5 3 5 9 7 7 7 $\leftarrow$ (14) 4 6 3 5 9 7 7 7 

(13) 4 5 3 5 9 3 5 7 7 

(19) 4 5 5 5 3 6 6 5 3 $\leftarrow$ (22) 5 6 5 2 3 5 7 7 

(23) 2 3 5 5 3 6 6 5 3 $\leftarrow$ (24) 3 6 5 2 3 5 7 7 

(26) 2 4 3 3 3 6 6 5 3 $\leftarrow$ (37) 8 1 1 2 4 3 3 3 

(27) 3 5 6 2 4 5 3 3 3 $\leftarrow$ (28) 5 6 2 3 5 7 3 3 

(29) ...3 3 6 6 5 3 $\leftarrow$ (30) 3 6 2 3 5 7 3 3 

(31) 5 8 1 1 2 4 3 3 3 $\leftarrow$ (32) 7 13 1 * 1 

(33) 4 ...4 5 3 3 3 $\leftarrow$ (34) 6 ..4 5 3 3 3 

(35) ....4 5 3 3 3 $\leftarrow$ (36) 4 ..4 5 3 3 3 

(38) * 2 4 3 3 3 $\leftarrow$ (45) 1 2 3 4 4 1 1 1 

(39) 5 1 2 3 4 4 1 1 1 $\leftarrow$ (40) 6 2 3 4 4 1 1 1 

\end{tcolorbox}
\begin{tcolorbox}[title={$(61,11)$}]
(1) 1..3 5 3 5 9 7 7 7 $\leftarrow$ (2) 13 4 5 3 5 9 7 7 7 

(4) ...4 6 9 11 7 7 7 $\leftarrow$ (9) 2 3 5 10 11 3 5 7 7 

(5) 5 5 9 3 5 7 3 5 7 7 

(7) 1 2 3 5 10 11 3 5 7 7 $\leftarrow$ (8) 12 9 5 5 3 6 6 5 3 

(7) 16 2 3 5 5 3 6 6 5 3 

(11) 5 5 5 6 5 2 3 5 7 7 $\leftarrow$ (17) 5 5 6 5 2 3 5 7 7 

(17) 2 4 5 5 5 3 6 6 5 3 $\leftarrow$ (20) 4 5 5 5 3 6 6 5 3 

(24) ..4 3 3 3 6 6 5 3 $\leftarrow$ (36) ....4 5 3 3 3 

(27) ....3 3 6 6 5 3 $\leftarrow$ (28) 3 5 6 2 4 5 3 3 3 

(31) 4 ....4 5 3 3 3 $\leftarrow$ (32) 5 8 1 1 2 4 3 3 3 

(33) .....4 5 3 3 3 $\leftarrow$ (34) 4 ...4 5 3 3 3 

(37) 1 * 2 4 3 3 3 $\leftarrow$ (43) 4 4 1 1 * 1 

(39) 3 4 4 1 1 * 1 $\leftarrow$ (40) 5 1 2 3 4 4 1 1 1 

(43) 1 * * 1 $\leftarrow$ (45) * * 1 

\end{tcolorbox}
\begin{tcolorbox}[title={$(61,12)$}]
(1) ..4 5 9 3 5 9 7 7 7 $\leftarrow$ (8) 1 2 3 5 10 11 3 5 7 7 

(1) 8 4 2 3 5 3 5 9 7 7 7 $\leftarrow$ (2) 1..3 5 3 5 9 7 7 7 

(4) 4 5 4 5 3 5 9 3 5 7 7 

(5) 4 4 2 3 5 3 5 9 7 7 7 $\leftarrow$ (6) 8 4 5 3 5 9 3 5 7 7 

(7) 2 4 2 3 5 3 5 9 7 7 7 $\leftarrow$ (9) 18 2 4 3 3 3 6 6 5 3 

(7) 13 ..4 7 3 3 6 6 5 3 $\leftarrow$ (8) 16 2 3 5 5 3 6 6 5 3 

(9) 4 3 5 5 6 5 2 3 5 7 7 $\leftarrow$ (12) 5 5 5 6 5 2 3 5 7 7 

(15) ..4 5 5 5 3 6 6 5 3 $\leftarrow$ (18) 2 4 5 5 5 3 6 6 5 3 

(21) 6 ....3 3 6 6 5 3 $\leftarrow$ (34) .....4 5 3 3 3 

(31) ......4 5 3 3 3 $\leftarrow$ (32) 4 ....4 5 3 3 3 

(36) 1 1 * 2 4 3 3 3 $\leftarrow$ (40) 3 4 4 1 1 * 1 

(42) 1 1 * * 1 $\leftarrow$ (44) 1 * * 1 

\end{tcolorbox}
\begin{tcolorbox}[title={$(61,13)$}]
(1) 5 5 5 5 5 6 5 2 3 5 7 7 

(1) 6 2 4 2 3 5 3 5 9 7 7 7 $\leftarrow$ (2) 8 4 2 3 5 3 5 9 7 7 7 

(2) 2 4 5 4 5 3 5 9 3 5 7 7 

(5) ..4 2 3 5 3 5 9 7 7 7 $\leftarrow$ (6) 4 4 2 3 5 3 5 9 7 7 7 

(6) 1 2 4 2 3 5 3 5 9 7 7 7 $\leftarrow$ (8) 2 4 2 3 5 3 5 9 7 7 7 

(7) 2 4 3 5 5 6 5 2 3 5 7 7 $\leftarrow$ (10) 4 3 5 5 6 5 2 3 5 7 7 

(7) 11 6 ..4 3 3 3 6 6 5 3 $\leftarrow$ (8) 13 ..4 7 3 3 6 6 5 3 

(13) ...4 5 5 5 3 6 6 5 3 $\leftarrow$ (16) ..4 5 5 5 3 6 6 5 3 

(19) 5 6 .....3 5 7 3 3 $\leftarrow$ (32) ......4 5 3 3 3 

(21) 3 6 .....3 5 7 3 3 $\leftarrow$ (22) 6 ....3 3 6 6 5 3 

(27) 4 ......4 5 3 3 3 $\leftarrow$ (28) 6 .....4 5 3 3 3 

(31) 6 2 3 4 4 1 1 * 1 $\leftarrow$ (38) 2 3 4 4 1 1 * 1 

\end{tcolorbox}
\begin{tcolorbox}[title={$(62,2)$}]
(31) 31 

\end{tcolorbox}
\begin{tcolorbox}[title={$(62,3)$}]
(30) 29 3 

(31) 30 1 $\leftarrow$ (32) 31 

(47) 14 1 $\leftarrow$ (48) 15 

(55) 6 1 $\leftarrow$ (56) 7 

(59) 2 1 $\leftarrow$ (60) 3 

\end{tcolorbox}
\begin{tcolorbox}[title={$(62,4)$}]
(17) 15 15 15 $\leftarrow$ (45) 11 7 

(23) 21 11 7 $\leftarrow$ (25) 23 15 

(29) 27 3 3 

(31) 13 11 7 $\leftarrow$ (33) 15 15 

(31) 29 1 1 $\leftarrow$ (32) 30 1 

(41) 7 7 7 $\leftarrow$ (49) 7 7 

(47) 13 1 1 $\leftarrow$ (48) 14 1 

(53) 3 3 3 $\leftarrow$ (57) 3 3 

(55) 5 1 1 $\leftarrow$ (56) 6 1 

(59) 1 1 1 $\leftarrow$ (60) 2 1 

\end{tcolorbox}
\begin{tcolorbox}[title={$(62,5)$}]
(1) 22 21 11 7 

(12) 9 11 15 15 

(14) 7 11 15 15 

(17) 14 13 11 7 $\leftarrow$ (18) 15 15 15 

(21) 10 13 11 7 $\leftarrow$ (22) 11 15 15 

(22) 19 7 7 7 $\leftarrow$ (24) 21 11 7 

(23) 20 5 7 7 $\leftarrow$ (26) 23 7 7 

(28) 25 3 3 3 

(29) 14 5 7 7 $\leftarrow$ (32) 13 11 7 

(30) 13 5 7 7 $\leftarrow$ (44) 5 7 7 

(31) 28 1 1 1 $\leftarrow$ (32) 29 1 1 

(39) 4 5 7 7 $\leftarrow$ (42) 7 7 7 

(47) 12 1 1 1 $\leftarrow$ (48) 13 1 1 

(55) 4 1 1 1 $\leftarrow$ (56) 5 1 1 

\end{tcolorbox}
\begin{tcolorbox}[title={$(62,6)$}]
(9) 5 7 11 15 15 

(11) 7 13 13 11 7 $\leftarrow$ (23) 11 15 7 7 

(13) 9 11 15 7 7 $\leftarrow$ (18) 14 13 11 7 

(15) 7 11 15 7 7 $\leftarrow$ (19) 13 13 11 7 

(21) 9 11 7 7 7 $\leftarrow$ (22) 10 13 11 7 

(22) 18 3 5 7 7 $\leftarrow$ (24) 20 5 7 7 

(23) 7 13 5 7 7 $\leftarrow$ (25) 17 7 7 7 

(24) 20 9 3 3 3 

(27) 5 9 7 7 7 

(29) 11 3 5 7 7 $\leftarrow$ (30) 14 5 7 7 

(31) 9 3 5 7 7 $\leftarrow$ (41) 3 5 7 7 

(31) 24 4 1 1 1 $\leftarrow$ (32) 28 1 1 1 

(38) 2 3 5 7 7 $\leftarrow$ (40) 4 5 7 7 

(39) 3 6 6 5 3 $\leftarrow$ (43) 6 6 5 3 

(45) 3 6 2 3 3 $\leftarrow$ (46) 8 3 3 3 

(47) 2 4 3 3 3 $\leftarrow$ (49) 6 2 3 3 

(47) 8 4 1 1 1 $\leftarrow$ (48) 12 1 1 1 

\end{tcolorbox}
\begin{tcolorbox}[title={$(62,7)$}]
(1) 20 9 11 7 7 7 

(1) 28 11 3 5 7 7 $\leftarrow$ (22) 9 11 7 7 7 

(10) 5 7 11 15 7 7 $\leftarrow$ (18) 9 15 7 7 7 

(11) 6 9 15 7 7 7 $\leftarrow$ (12) 7 13 13 11 7 

(12) 5 9 15 7 7 7 

(13) 8 9 11 7 7 7 $\leftarrow$ (14) 9 11 15 7 7 

(15) 6 9 11 7 7 7 $\leftarrow$ (16) 7 11 15 7 7 

(17) 12 11 3 5 7 7 $\leftarrow$ (18) 13 13 5 7 7 

(22) 5 5 9 7 7 7 $\leftarrow$ (24) 7 13 5 7 7 

(22) 17 3 6 6 5 3 

(25) 13 2 3 5 7 7 $\leftarrow$ (26) 14 4 5 7 7 

(28) 5 7 3 5 7 7 $\leftarrow$ (32) 9 3 5 7 7 

(31) 22 * 1 $\leftarrow$ (32) 24 4 1 1 1 

(36) 3 3 6 6 5 3 $\leftarrow$ (42) 3 5 7 3 3 

(39) 2 3 5 7 3 3 $\leftarrow$ (40) 3 6 6 5 3 

(45) ..4 3 3 3 $\leftarrow$ (46) 3 6 2 3 3 

(46) 1 2 4 3 3 3 $\leftarrow$ (48) 2 4 3 3 3 

(47) 6 * 1 $\leftarrow$ (48) 8 4 1 1 1 

(51) 2 * 1 $\leftarrow$ (52) 4 4 1 1 1 

\end{tcolorbox}
\begin{tcolorbox}[title={$(62,8)$}]
(1) 14 6 9 11 7 7 7 $\leftarrow$ (2) 20 9 11 7 7 7 

(1) 23 8 12 9 3 3 3 $\leftarrow$ (2) 28 11 3 5 7 7 

(5) 27 12 4 5 3 3 3 

(7) 1 2 4 7 11 15 15 $\leftarrow$ (9) 2 4 7 11 15 15 

(9) 3 5 9 15 7 7 7 

(11) 4 6 9 11 7 7 7 $\leftarrow$ (12) 6 9 15 7 7 7 

(14) 8 5 5 9 7 7 7 

(15) 9 3 5 9 7 7 7 $\leftarrow$ (17) 11 5 9 7 7 7 

(17) 7 8 12 9 3 3 3 $\leftarrow$ (18) 12 11 3 5 7 7 

(21) 11 12 4 5 3 3 3 

(23) 3 5 9 3 5 7 7 $\leftarrow$ (25) 8 12 9 3 3 3 

(25) 7 12 4 5 3 3 3 $\leftarrow$ (26) 13 2 3 5 7 7 

(29) 5 5 3 6 6 5 3 $\leftarrow$ (33) 12 4 5 3 3 3 

(31) 21 1 * 1 $\leftarrow$ (32) 22 * 1 

(33) 6 2 3 5 7 3 3 $\leftarrow$ (40) 2 3 5 7 3 3 

(45) 1 1 2 4 3 3 3 $\leftarrow$ (46) ..4 3 3 3 

(47) 5 1 * 1 $\leftarrow$ (48) 6 * 1 

(51) 1 1 * 1 $\leftarrow$ (52) 2 * 1 

\end{tcolorbox}
\begin{tcolorbox}[title={$(62,9)$}]
(1) 2 4 3 4 7 11 15 15 $\leftarrow$ (2) 23 8 12 9 3 3 3 

(4) 25 5 5 3 6 6 5 3 

(5) 6 4 6 9 11 7 7 7 $\leftarrow$ (8) 1 2 4 7 11 15 15 

(5) 18 3 5 9 3 5 7 7 $\leftarrow$ (6) 27 12 4 5 3 3 3 

(10) 4 8 5 5 9 7 7 7 

(11) 3 5 10 11 3 5 7 7 $\leftarrow$ (16) 9 3 5 9 7 7 7 

(15) 4 5 3 5 9 7 7 7 $\leftarrow$ (17) 8 3 5 9 7 7 7 

(15) 8 3 5 9 3 5 7 7 $\leftarrow$ (18) 7 8 12 9 3 3 3 

(20) 9 5 5 3 6 6 5 3 

(25) 5 6 3 3 6 6 5 3 $\leftarrow$ (26) 7 12 4 5 3 3 3 

(27) 3 6 3 3 6 6 5 3 $\leftarrow$ (28) 6 5 2 3 5 7 7 

(31) 5 6 2 4 5 3 3 3 $\leftarrow$ (37) 6 2 4 5 3 3 3 

(31) 20 1 1 * 1 $\leftarrow$ (32) 21 1 * 1 

(33) 3 6 2 4 5 3 3 3 $\leftarrow$ (34) 6 2 3 5 7 3 3 

(39) 12 1 1 * 1 $\leftarrow$ (40) 13 1 * 1 

(47) 4 1 1 * 1 $\leftarrow$ (48) 5 1 * 1 

\end{tcolorbox}
\begin{tcolorbox}[title={$(62,10)$}]
(1) 14 8 3 5 9 3 5 7 7 

(2) 3 6 4 6 9 11 7 7 7 

(5) 4 2 4 6 9 11 7 7 7 $\leftarrow$ (6) 6 4 6 9 11 7 7 7 

(5) 11 9 3 5 7 3 5 7 7 $\leftarrow$ (6) 18 3 5 9 3 5 7 7 

(6) 4 4 8 5 5 9 7 7 7 

(11) 5 9 3 5 7 3 5 7 7 $\leftarrow$ (17) 9 3 5 7 3 5 7 7 

(14) 2 3 5 3 5 9 7 7 7 $\leftarrow$ (16) 4 5 3 5 9 7 7 7 

(14) 4 5 3 5 9 3 5 7 7 $\leftarrow$ (16) 8 3 5 9 3 5 7 7 

(23) 2 4 7 3 3 6 6 5 3 $\leftarrow$ (26) 4 7 3 3 6 6 5 3 

(24) 2 3 5 5 3 6 6 5 3 

(25) 3 5 6 2 3 5 7 3 3 $\leftarrow$ (26) 5 6 3 3 6 6 5 3 

(27) 2 4 3 3 3 6 6 5 3 $\leftarrow$ (28) 3 6 3 3 6 6 5 3 

(30) ...3 3 6 6 5 3 $\leftarrow$ (35) 6 ..4 5 3 3 3 

(31) 3 6 ..4 5 3 3 3 $\leftarrow$ (32) 5 6 2 4 5 3 3 3 

(31) 16 4 1 1 * 1 $\leftarrow$ (32) 20 1 1 * 1 

(33) ....3 5 7 3 3 $\leftarrow$ (34) 3 6 2 4 5 3 3 3 

(37) 3 6 2 3 4 4 1 1 1 $\leftarrow$ (38) 8 1 1 2 4 3 3 3 

(39) * 2 4 3 3 3 $\leftarrow$ (41) 6 2 3 4 4 1 1 1 

(39) 8 4 1 1 * 1 $\leftarrow$ (40) 12 1 1 * 1 

\end{tcolorbox}
\begin{tcolorbox}[title={$(62,11)$}]
(1) 13 4 5 3 5 9 3 5 7 7 $\leftarrow$ (2) 14 8 3 5 9 3 5 7 7 

(3) 6 2 3 5 10 11 3 5 7 7 

(5) ...4 6 9 11 7 7 7 $\leftarrow$ (6) 4 2 4 6 9 11 7 7 7 

(6) 5 5 9 3 5 7 3 5 7 7 $\leftarrow$ (12) 5 9 3 5 7 3 5 7 7 

(15) 4 3 5 6 5 2 3 5 7 7 $\leftarrow$ (18) 5 5 6 5 2 3 5 7 7 

(21) ..4 7 3 3 6 6 5 3 

(22) 1 2 4 7 3 3 6 6 5 3 $\leftarrow$ (24) 2 4 7 3 3 6 6 5 3 

(25) ..4 3 3 3 6 6 5 3 $\leftarrow$ (26) 3 5 6 2 3 5 7 3 3 

(28) ....3 3 6 6 5 3 $\leftarrow$ (34) ....3 5 7 3 3 

(31) .....3 5 7 3 3 $\leftarrow$ (32) 3 6 ..4 5 3 3 3 

(31) 14 * * 1 $\leftarrow$ (32) 16 4 1 1 * 1 

(37) 2 * 2 4 3 3 3 $\leftarrow$ (38) 3 6 2 3 4 4 1 1 1 

(38) 1 * 2 4 3 3 3 $\leftarrow$ (40) * 2 4 3 3 3 

(39) 6 * * 1 $\leftarrow$ (40) 8 4 1 1 * 1 

(43) 2 * * 1 $\leftarrow$ (44) 4 4 1 1 * 1 

\end{tcolorbox}
\begin{tcolorbox}[title={$(62,12)$}]
(1) 5 5 5 9 3 5 7 3 5 7 7 

(1) 8 5 4 5 3 5 9 3 5 7 7 $\leftarrow$ (2) 13 4 5 3 5 9 3 5 7 7 

(2) ..4 5 9 3 5 9 7 7 7 

(3) 4 7 4 5 3 5 9 3 5 7 7 $\leftarrow$ (4) 6 2 3 5 10 11 3 5 7 7 

(5) 4 5 4 5 3 5 9 3 5 7 7 $\leftarrow$ (9) 16 2 3 5 5 3 6 6 5 3 

(7) 5 5 5 5 6 5 2 3 5 7 7 $\leftarrow$ (13) 5 5 5 6 5 2 3 5 7 7 

(9) 16 ..4 3 3 3 6 6 5 3 $\leftarrow$ (10) 18 2 4 3 3 3 6 6 5 3 

(13) 2 4 3 5 6 5 2 3 5 7 7 $\leftarrow$ (16) 4 3 5 6 5 2 3 5 7 7 

(19) 6 ..4 3 3 3 6 6 5 3 

(21) 1 1 2 4 7 3 3 6 6 5 3 $\leftarrow$ (22) ..4 7 3 3 6 6 5 3 

(25) 6 .....3 5 7 3 3 $\leftarrow$ (32) .....3 5 7 3 3 

(31) 13 1 * * 1 $\leftarrow$ (32) 14 * * 1 

(37) 1 1 * 2 4 3 3 3 $\leftarrow$ (38) 2 * 2 4 3 3 3 

(39) 5 1 * * 1 $\leftarrow$ (40) 6 * * 1 

(43) 1 1 * * 1 $\leftarrow$ (44) 2 * * 1 

\end{tcolorbox}
\begin{tcolorbox}[title={$(63,1)$}]
(63) 

\end{tcolorbox}
\begin{tcolorbox}[title={$(63,2)$}]
(62) 1 

\end{tcolorbox}
\begin{tcolorbox}[title={$(63,3)$}]
(1) 31 31 

(29) 27 7 

(31) 29 3 $\leftarrow$ (33) 31 

(47) 13 3 $\leftarrow$ (49) 15 

(55) 5 3 $\leftarrow$ (57) 7 

(61) 1 1 

\end{tcolorbox}
\begin{tcolorbox}[title={$(63,4)$}]
(1) 30 29 3 $\leftarrow$ (2) 31 31 

(25) 22 13 3 $\leftarrow$ (26) 23 15 

(29) 26 5 3 $\leftarrow$ (30) 27 7 

(30) 27 3 3 $\leftarrow$ (32) 29 3 

(33) 14 13 3 $\leftarrow$ (34) 15 15 

(45) 10 5 3 $\leftarrow$ (46) 11 7 

(46) 11 3 3 $\leftarrow$ (48) 13 3 

(49) 6 5 3 $\leftarrow$ (50) 7 7 

(54) 3 3 3 $\leftarrow$ (56) 5 3 

(55) 2 3 3 $\leftarrow$ (58) 3 3 

(60) 1 1 1 

\end{tcolorbox}
\begin{tcolorbox}[title={$(63,5)$}]
(1) 29 27 3 3 $\leftarrow$ (2) 30 29 3 

(2) 22 21 11 7 

(13) 9 11 15 15 $\leftarrow$ (33) 13 11 7 

(15) 7 11 15 15 $\leftarrow$ (23) 11 15 15 

(23) 19 7 7 7 $\leftarrow$ (25) 21 11 7 

(25) 21 11 3 3 $\leftarrow$ (26) 22 13 3 

(29) 25 3 3 3 $\leftarrow$ (30) 26 5 3 

(31) 13 5 7 7 $\leftarrow$ (45) 5 7 7 

(33) 13 11 3 3 $\leftarrow$ (34) 14 13 3 

(45) 5 7 3 3 $\leftarrow$ (50) 6 5 3 

(45) 9 3 3 3 $\leftarrow$ (46) 10 5 3 

(54) 1 2 3 3 $\leftarrow$ (56) 2 3 3 

(56) 4 1 1 1 

\end{tcolorbox}
\begin{tcolorbox}[title={$(63,6)$}]
(1) 12 9 11 15 15 $\leftarrow$ (16) 7 11 15 15 

(1) 14 7 11 15 15 

(1) 28 25 3 3 3 $\leftarrow$ (2) 29 27 3 3 

(10) 5 7 11 15 15 $\leftarrow$ (26) 17 7 7 7 

(11) 4 7 11 15 15 

(15) 10 17 7 7 7 $\leftarrow$ (24) 19 7 7 7 

(19) 11 14 5 7 7 $\leftarrow$ (20) 13 13 11 7 

(23) 7 14 5 7 7 $\leftarrow$ (24) 11 15 7 7 

(23) 18 3 5 7 7 $\leftarrow$ (25) 20 5 7 7 

(25) 20 9 3 3 3 $\leftarrow$ (26) 21 11 3 3 

(28) 5 9 7 7 7 $\leftarrow$ (42) 3 5 7 7 

(30) 11 3 5 7 7 $\leftarrow$ (32) 13 5 7 7 

(33) 12 9 3 3 3 $\leftarrow$ (34) 13 11 3 3 

(39) 2 3 5 7 7 $\leftarrow$ (41) 4 5 7 7 

(43) 4 7 3 3 3 $\leftarrow$ (44) 6 6 5 3 

(45) 4 5 3 3 3 $\leftarrow$ (46) 5 7 3 3 

(49) 5 1 2 3 3 $\leftarrow$ (50) 6 2 3 3 

(54) * 1 

\end{tcolorbox}
\begin{tcolorbox}[title={$(63,7)$}]
(1) 9 5 7 11 15 15 

(1) 24 20 9 3 3 3 $\leftarrow$ (2) 28 25 3 3 3 

(5) 28 12 9 3 3 3 

(11) 5 7 11 15 7 7 $\leftarrow$ (13) 7 13 13 11 7 

(13) 5 9 15 7 7 7 

(14) 8 9 11 7 7 7 $\leftarrow$ (19) 13 13 5 7 7 

(15) 9 7 13 5 7 7 $\leftarrow$ (16) 10 17 7 7 7 

(16) 6 9 11 7 7 7 

(19) 7 14 4 5 7 7 $\leftarrow$ (20) 11 14 5 7 7 

(21) 12 12 9 3 3 3 $\leftarrow$ (24) 18 3 5 7 7 

(23) 5 5 9 7 7 7 $\leftarrow$ (24) 7 14 5 7 7 

(23) 17 3 6 6 5 3 $\leftarrow$ (26) 20 9 3 3 3 

(25) 3 5 9 7 7 7 

(27) 5 9 3 5 7 7 $\leftarrow$ (41) 3 6 6 5 3 

(29) 5 7 3 5 7 7 $\leftarrow$ (33) 9 3 5 7 7 

(31) 9 3 6 6 5 3 $\leftarrow$ (34) 12 9 3 3 3 

(37) 3 3 6 6 5 3 $\leftarrow$ (40) 2 3 5 7 7 

(43) 2 4 5 3 3 3 $\leftarrow$ (44) 4 7 3 3 3 

(47) 1 2 4 3 3 3 $\leftarrow$ (49) 2 4 3 3 3 

(49) 3 4 4 1 1 1 $\leftarrow$ (50) 5 1 2 3 3 

(53) 1 * 1 

\end{tcolorbox}
\begin{tcolorbox}[title={$(63,8)$}]
(1) 22 17 3 6 6 5 3 $\leftarrow$ (2) 24 20 9 3 3 3 

(2) 14 6 9 11 7 7 7 

(3) 6 2 4 7 11 15 15 

(5) 26 9 3 6 6 5 3 $\leftarrow$ (6) 28 12 9 3 3 3 

(10) 3 5 9 15 7 7 7 

(12) 4 6 9 11 7 7 7 $\leftarrow$ (18) 11 5 9 7 7 7 

(15) 5 10 11 3 5 7 7 $\leftarrow$ (20) 7 14 4 5 7 7 

(15) 8 5 5 9 7 7 7 $\leftarrow$ (16) 9 7 13 5 7 7 

(21) 10 9 3 6 6 5 3 $\leftarrow$ (22) 12 12 9 3 3 3 

(22) 11 12 4 5 3 3 3 $\leftarrow$ (24) 17 3 6 6 5 3 

(24) 3 5 9 3 5 7 7 

(25) 6 9 3 6 6 5 3 $\leftarrow$ (26) 8 12 9 3 3 3 

(26) 3 5 7 3 5 7 7 $\leftarrow$ (28) 5 9 3 5 7 7 

(30) 5 5 3 6 6 5 3 $\leftarrow$ (32) 9 3 6 6 5 3 

(31) 6 3 3 6 6 5 3 $\leftarrow$ (38) 3 3 6 6 5 3 

(33) 10 2 4 5 3 3 3 $\leftarrow$ (34) 12 4 5 3 3 3 

(46) 1 1 2 4 3 3 3 $\leftarrow$ (48) 1 2 4 3 3 3 

(47) 2 3 4 4 1 1 1 $\leftarrow$ (50) 3 4 4 1 1 1 

(52) 1 1 * 1 

\end{tcolorbox}
\begin{tcolorbox}[title={$(63,9)$}]
(1) 9 3 5 9 15 7 7 7 

(1) 14 8 5 5 9 7 7 7 

(1) 21 11 12 4 5 3 3 3 $\leftarrow$ (2) 22 17 3 6 6 5 3 

(2) 2 4 3 4 7 11 15 15 

(3) 5 1 2 4 7 11 15 15 $\leftarrow$ (4) 6 2 4 7 11 15 15 

(5) 25 5 5 3 6 6 5 3 $\leftarrow$ (6) 26 9 3 6 6 5 3 

(9) 13 11 12 4 5 3 3 3 $\leftarrow$ (16) 8 5 5 9 7 7 7 

(11) 4 8 5 5 9 7 7 7 

(12) 3 5 10 11 3 5 7 7 $\leftarrow$ (18) 8 3 5 9 7 7 7 

(15) 4 6 3 5 9 7 7 7 $\leftarrow$ (16) 5 10 11 3 5 7 7 

(21) 9 5 5 3 6 6 5 3 $\leftarrow$ (22) 10 9 3 6 6 5 3 

(23) 5 6 5 2 3 5 7 7 $\leftarrow$ (29) 6 5 2 3 5 7 7 

(25) 3 6 5 2 3 5 7 7 $\leftarrow$ (26) 6 9 3 6 6 5 3 

(29) 5 6 2 3 5 7 3 3 $\leftarrow$ (35) 6 2 3 5 7 3 3 

(31) 3 6 2 3 5 7 3 3 $\leftarrow$ (32) 6 3 3 6 6 5 3 

(33) 7 13 1 * 1 $\leftarrow$ (34) 10 2 4 5 3 3 3 

(37) 4 ..4 5 3 3 3 $\leftarrow$ (38) 6 2 4 5 3 3 3 

(46) 1 2 3 4 4 1 1 1 $\leftarrow$ (48) 2 3 4 4 1 1 1 

(48) 4 1 1 * 1 

\end{tcolorbox}
\begin{tcolorbox}[title={$(63,10)$}]
(1) 20 9 5 5 3 6 6 5 3 $\leftarrow$ (2) 21 11 12 4 5 3 3 3 

(3) 3 6 4 6 9 11 7 7 7 $\leftarrow$ (4) 5 1 2 4 7 11 15 15 

(3) 13 4 5 3 5 9 7 7 7 

(6) 11 9 3 5 7 3 5 7 7 

(7) 4 4 8 5 5 9 7 7 7 $\leftarrow$ (12) 4 8 5 5 9 7 7 7 

(9) 12 9 5 5 3 6 6 5 3 $\leftarrow$ (10) 13 11 12 4 5 3 3 3 

(10) 2 3 5 10 11 3 5 7 7 $\leftarrow$ (17) 4 5 3 5 9 7 7 7 

(15) 2 3 5 3 5 9 7 7 7 $\leftarrow$ (16) 4 6 3 5 9 7 7 7 

(15) 4 5 3 5 9 3 5 7 7 $\leftarrow$ (17) 8 3 5 9 3 5 7 7 

(21) 4 5 5 5 3 6 6 5 3 $\leftarrow$ (24) 5 6 5 2 3 5 7 7 

(25) 2 3 5 5 3 6 6 5 3 $\leftarrow$ (26) 3 6 5 2 3 5 7 7 

(28) 2 4 3 3 3 6 6 5 3 $\leftarrow$ (33) 5 6 2 4 5 3 3 3 

(29) 3 5 6 2 4 5 3 3 3 $\leftarrow$ (30) 5 6 2 3 5 7 3 3 

(31) ...3 3 6 6 5 3 $\leftarrow$ (32) 3 6 2 3 5 7 3 3 

(33) 5 8 1 1 2 4 3 3 3 $\leftarrow$ (34) 7 13 1 * 1 

(35) 4 ...4 5 3 3 3 $\leftarrow$ (36) 6 ..4 5 3 3 3 

(37) ....4 5 3 3 3 $\leftarrow$ (38) 4 ..4 5 3 3 3 

(41) 5 1 2 3 4 4 1 1 1 $\leftarrow$ (42) 6 2 3 4 4 1 1 1 

(46) * * 1 

\end{tcolorbox}
\begin{tcolorbox}[title={$(63,11)$}]
(1) 2 3 6 4 6 9 11 7 7 7 

(1) 6 4 4 8 5 5 9 7 7 7 

(1) 24 2 3 5 5 3 6 6 5 3 

(3) 1..3 5 3 5 9 7 7 7 $\leftarrow$ (4) 13 4 5 3 5 9 7 7 7 

(6) ...4 6 9 11 7 7 7 $\leftarrow$ (8) 4 4 8 5 5 9 7 7 7 

(7) 5 5 9 3 5 7 3 5 7 7 $\leftarrow$ (13) 5 9 3 5 7 3 5 7 7 

(7) 8 4 5 3 5 9 3 5 7 7 $\leftarrow$ (16) 2 3 5 3 5 9 7 7 7 

(9) 1 2 3 5 10 11 3 5 7 7 $\leftarrow$ (10) 12 9 5 5 3 6 6 5 3 

(19) 2 4 5 5 5 3 6 6 5 3 $\leftarrow$ (22) 4 5 5 5 3 6 6 5 3 

(23) 1 2 4 7 3 3 6 6 5 3 $\leftarrow$ (25) 2 4 7 3 3 6 6 5 3 

(26) ..4 3 3 3 6 6 5 3 $\leftarrow$ (32) ...3 3 6 6 5 3 

(29) ....3 3 6 6 5 3 $\leftarrow$ (30) 3 5 6 2 4 5 3 3 3 

(33) 4 ....4 5 3 3 3 $\leftarrow$ (34) 5 8 1 1 2 4 3 3 3 

(35) .....4 5 3 3 3 $\leftarrow$ (36) 4 ...4 5 3 3 3 

(39) 1 * 2 4 3 3 3 $\leftarrow$ (41) * 2 4 3 3 3 

(41) 3 4 4 1 1 * 1 $\leftarrow$ (42) 5 1 2 3 4 4 1 1 1 

(45) 1 * * 1 

\end{tcolorbox}
\begin{tcolorbox}[title={$(64,2)$}]
(1) 63 

(61) 3 

(63) 1 $\leftarrow$ (65) 

\end{tcolorbox}
\begin{tcolorbox}[title={$(64,3)$}]
(1) 62 1 $\leftarrow$ (2) 63 

(33) 30 1 $\leftarrow$ (34) 31 

(49) 14 1 $\leftarrow$ (50) 15 

(57) 6 1 $\leftarrow$ (58) 7 

(61) 2 1 $\leftarrow$ (62) 3 

(62) 1 1 $\leftarrow$ (64) 1 

\end{tcolorbox}
\begin{tcolorbox}[title={$(64,4)$}]
(1) 29 27 7 

(1) 61 1 1 $\leftarrow$ (2) 62 1 

(19) 15 15 15 $\leftarrow$ (35) 15 15 

(27) 23 7 7 

(31) 27 3 3 $\leftarrow$ (33) 29 3 

(33) 29 1 1 $\leftarrow$ (34) 30 1 

(43) 7 7 7 $\leftarrow$ (51) 7 7 

(47) 11 3 3 $\leftarrow$ (49) 13 3 

(49) 13 1 1 $\leftarrow$ (50) 14 1 

(55) 3 3 3 $\leftarrow$ (57) 5 3 

(57) 5 1 1 $\leftarrow$ (58) 6 1 

(61) 1 1 1 $\leftarrow$ (62) 2 1 

\end{tcolorbox}
\begin{tcolorbox}[title={$(64,5)$}]
(1) 60 1 1 1 $\leftarrow$ (2) 61 1 1 

(3) 22 21 11 7 

(14) 9 11 15 15 $\leftarrow$ (34) 13 11 7 

(19) 14 13 11 7 $\leftarrow$ (20) 15 15 15 

(23) 10 13 11 7 $\leftarrow$ (24) 11 15 15 

(30) 25 3 3 3 $\leftarrow$ (32) 27 3 3 

(31) 14 5 7 7 $\leftarrow$ (46) 5 7 7 

(33) 28 1 1 1 $\leftarrow$ (34) 29 1 1 

(46) 9 3 3 3 $\leftarrow$ (48) 11 3 3 

(47) 8 3 3 3 $\leftarrow$ (56) 3 3 3 

(49) 12 1 1 1 $\leftarrow$ (50) 13 1 1 

(55) 1 2 3 3 $\leftarrow$ (57) 2 3 3 

(57) 4 1 1 1 $\leftarrow$ (58) 5 1 1 

\end{tcolorbox}
\begin{tcolorbox}[title={$(64,6)$}]
(1) 56 4 1 1 1 $\leftarrow$ (2) 60 1 1 1 

(2) 12 9 11 15 15 $\leftarrow$ (4) 22 21 11 7 

(2) 14 7 11 15 15 

(11) 5 7 11 15 15 $\leftarrow$ (27) 17 7 7 7 

(12) 4 7 11 15 15 

(15) 9 11 15 7 7 $\leftarrow$ (20) 14 13 11 7 

(17) 7 11 15 7 7 

(19) 9 15 7 7 7 $\leftarrow$ (21) 13 13 11 7 

(23) 9 11 7 7 7 $\leftarrow$ (24) 10 13 11 7 

(25) 7 13 5 7 7 

(27) 14 4 5 7 7 $\leftarrow$ (43) 3 5 7 7 

(29) 5 9 7 7 7 $\leftarrow$ (33) 13 5 7 7 

(31) 11 3 5 7 7 $\leftarrow$ (32) 14 5 7 7 

(33) 24 4 1 1 1 $\leftarrow$ (34) 28 1 1 1 

(43) 3 5 7 3 3 $\leftarrow$ (45) 6 6 5 3 

(46) 4 5 3 3 3 $\leftarrow$ (51) 6 2 3 3 

(47) 3 6 2 3 3 $\leftarrow$ (48) 8 3 3 3 

(49) 8 4 1 1 1 $\leftarrow$ (50) 12 1 1 1 

(53) 4 4 1 1 1 $\leftarrow$ (56) 1 2 3 3 

(55) * 1 $\leftarrow$ (58) 4 1 1 1 

\end{tcolorbox}
\begin{tcolorbox}[title={$(64,7)$}]
(1) 11 4 7 11 15 15 

(1) 54 * 1 $\leftarrow$ (2) 56 4 1 1 1 

(2) 9 5 7 11 15 15 

(3) 20 9 11 7 7 7 

(3) 28 11 3 5 7 7 

(10) 2 4 7 11 15 15 

(12) 5 7 11 15 7 7 $\leftarrow$ (17) 10 17 7 7 7 

(13) 6 9 15 7 7 7 $\leftarrow$ (14) 7 13 13 11 7 

(14) 5 9 15 7 7 7 $\leftarrow$ (20) 9 15 7 7 7 

(15) 8 9 11 7 7 7 $\leftarrow$ (16) 9 11 15 7 7 

(17) 6 9 11 7 7 7 $\leftarrow$ (18) 7 11 15 7 7 

(19) 12 11 3 5 7 7 $\leftarrow$ (20) 13 13 5 7 7 

(24) 5 5 9 7 7 7 

(26) 3 5 9 7 7 7 $\leftarrow$ (32) 11 3 5 7 7 

(27) 13 2 3 5 7 7 $\leftarrow$ (28) 14 4 5 7 7 

(30) 5 7 3 5 7 7 $\leftarrow$ (34) 9 3 5 7 7 

(33) 22 * 1 $\leftarrow$ (34) 24 4 1 1 1 

(41) 2 3 5 7 3 3 $\leftarrow$ (42) 3 6 6 5 3 

(44) 2 4 5 3 3 3 $\leftarrow$ (50) 2 4 3 3 3 

(47) ..4 3 3 3 $\leftarrow$ (48) 3 6 2 3 3 

(49) 6 * 1 $\leftarrow$ (50) 8 4 1 1 1 

(53) 2 * 1 $\leftarrow$ (54) 4 4 1 1 1 

(54) 1 * 1 $\leftarrow$ (56) * 1 

\end{tcolorbox}
\begin{tcolorbox}[title={$(64,8)$}]
(1) 13 5 9 15 7 7 7 

(1) 25 3 5 9 7 7 7 

(1) 53 1 * 1 $\leftarrow$ (2) 54 * 1 

(3) 14 6 9 11 7 7 7 $\leftarrow$ (4) 20 9 11 7 7 7 

(3) 23 8 12 9 3 3 3 $\leftarrow$ (4) 28 11 3 5 7 7 

(7) 27 12 4 5 3 3 3 

(9) 1 2 4 7 11 15 15 $\leftarrow$ (16) 8 9 11 7 7 7 

(11) 3 5 9 15 7 7 7 $\leftarrow$ (18) 6 9 11 7 7 7 

(13) 4 6 9 11 7 7 7 $\leftarrow$ (14) 6 9 15 7 7 7 

(17) 9 3 5 9 7 7 7 

(19) 7 8 12 9 3 3 3 $\leftarrow$ (20) 12 11 3 5 7 7 

(23) 11 12 4 5 3 3 3 $\leftarrow$ (25) 17 3 6 6 5 3 

(25) 3 5 9 3 5 7 7 $\leftarrow$ (35) 12 4 5 3 3 3 

(27) 3 5 7 3 5 7 7 $\leftarrow$ (29) 5 9 3 5 7 7 

(27) 7 12 4 5 3 3 3 $\leftarrow$ (28) 13 2 3 5 7 7 

(31) 5 5 3 6 6 5 3 $\leftarrow$ (33) 9 3 6 6 5 3 

(33) 21 1 * 1 $\leftarrow$ (34) 22 * 1 

(41) 13 1 * 1 $\leftarrow$ (49) 1 2 4 3 3 3 

(47) 1 1 2 4 3 3 3 $\leftarrow$ (48) ..4 3 3 3 

(49) 5 1 * 1 $\leftarrow$ (50) 6 * 1 

(53) 1 1 * 1 $\leftarrow$ (54) 2 * 1 

\end{tcolorbox}
\begin{tcolorbox}[title={$(64,9)$}]
(1) 3 6 2 4 7 11 15 15 $\leftarrow$ (2) 25 3 5 9 7 7 7 

(1) 52 1 1 * 1 $\leftarrow$ (2) 53 1 * 1 

(2) 9 3 5 9 15 7 7 7 

(2) 14 8 5 5 9 7 7 7 

(3) 2 4 3 4 7 11 15 15 $\leftarrow$ (4) 23 8 12 9 3 3 3 

(6) 25 5 5 3 6 6 5 3 

(7) 6 4 6 9 11 7 7 7 $\leftarrow$ (14) 4 6 9 11 7 7 7 

(7) 18 3 5 9 3 5 7 7 $\leftarrow$ (8) 27 12 4 5 3 3 3 

(13) 3 5 10 11 3 5 7 7 $\leftarrow$ (17) 5 10 11 3 5 7 7 

(18) 9 3 5 7 3 5 7 7 

(22) 9 5 5 3 6 6 5 3 $\leftarrow$ (24) 11 12 4 5 3 3 3 

(27) 4 7 3 3 6 6 5 3 $\leftarrow$ (32) 5 5 3 6 6 5 3 

(27) 5 6 3 3 6 6 5 3 $\leftarrow$ (28) 7 12 4 5 3 3 3 

(29) 3 6 3 3 6 6 5 3 $\leftarrow$ (30) 6 5 2 3 5 7 7 

(33) 20 1 1 * 1 $\leftarrow$ (34) 21 1 * 1 

(35) 3 6 2 4 5 3 3 3 $\leftarrow$ (36) 6 2 3 5 7 3 3 

(39) 8 1 1 2 4 3 3 3 $\leftarrow$ (48) 1 1 2 4 3 3 3 

(41) 12 1 1 * 1 $\leftarrow$ (42) 13 1 * 1 

(47) 1 2 3 4 4 1 1 1 $\leftarrow$ (49) 2 3 4 4 1 1 1 

(49) 4 1 1 * 1 $\leftarrow$ (50) 5 1 * 1 

\end{tcolorbox}
\begin{tcolorbox}[title={$(64,10)$}]
(1) ..4 3 4 7 11 15 15 $\leftarrow$ (2) 3 6 2 4 7 11 15 15 

(1) 11 4 8 5 5 9 7 7 7 

(1) 48 4 1 1 * 1 $\leftarrow$ (2) 52 1 1 * 1 

(2) 20 9 5 5 3 6 6 5 3 

(3) 14 8 3 5 9 3 5 7 7 

(4) 3 6 4 6 9 11 7 7 7 $\leftarrow$ (11) 13 11 12 4 5 3 3 3 

(7) 4 2 4 6 9 11 7 7 7 $\leftarrow$ (8) 6 4 6 9 11 7 7 7 

(7) 11 9 3 5 7 3 5 7 7 $\leftarrow$ (8) 18 3 5 9 3 5 7 7 

(11) 2 3 5 10 11 3 5 7 7 $\leftarrow$ (14) 3 5 10 11 3 5 7 7 

(16) 4 5 3 5 9 3 5 7 7 

(19) 5 5 6 5 2 3 5 7 7 $\leftarrow$ (25) 5 6 5 2 3 5 7 7 

(26) 2 3 5 5 3 6 6 5 3 $\leftarrow$ (28) 4 7 3 3 6 6 5 3 

(27) 3 5 6 2 3 5 7 3 3 $\leftarrow$ (28) 5 6 3 3 6 6 5 3 

(29) 2 4 3 3 3 6 6 5 3 $\leftarrow$ (30) 3 6 3 3 6 6 5 3 

(33) 3 6 ..4 5 3 3 3 $\leftarrow$ (34) 5 6 2 4 5 3 3 3 

(33) 16 4 1 1 * 1 $\leftarrow$ (34) 20 1 1 * 1 

(35) ....3 5 7 3 3 $\leftarrow$ (36) 3 6 2 4 5 3 3 3 

(38) ....4 5 3 3 3 $\leftarrow$ (43) 6 2 3 4 4 1 1 1 

(39) 3 6 2 3 4 4 1 1 1 $\leftarrow$ (40) 8 1 1 2 4 3 3 3 

(41) 8 4 1 1 * 1 $\leftarrow$ (42) 12 1 1 * 1 

(45) 4 4 1 1 * 1 $\leftarrow$ (48) 1 2 3 4 4 1 1 1 

(47) * * 1 $\leftarrow$ (50) 4 1 1 * 1 

\end{tcolorbox}
\begin{tcolorbox}[title={$(65,3)$}]
(1) 61 3 

(3) 31 31 

(27) 23 15 

(31) 27 7 $\leftarrow$ (35) 31 

(47) 11 7 $\leftarrow$ (51) 15 

(59) 3 3 

(63) 1 1 $\leftarrow$ (65) 1 

\end{tcolorbox}
\begin{tcolorbox}[title={$(65,4)$}]
(2) 29 27 7 

(3) 30 29 3 $\leftarrow$ (4) 31 31 

(26) 21 11 7 

(27) 22 13 3 $\leftarrow$ (28) 23 15 

(28) 23 7 7 $\leftarrow$ (34) 29 3 

(31) 26 5 3 $\leftarrow$ (32) 27 7 

(35) 14 13 3 $\leftarrow$ (36) 15 15 

(44) 7 7 7 $\leftarrow$ (50) 13 3 

(47) 10 5 3 $\leftarrow$ (48) 11 7 

(51) 6 5 3 $\leftarrow$ (52) 7 7 

(62) 1 1 1 $\leftarrow$ (64) 1 1 

\end{tcolorbox}
\begin{tcolorbox}[title={$(65,5)$}]
(1) 27 23 7 7 

(3) 29 27 3 3 $\leftarrow$ (4) 30 29 3 

(15) 9 11 15 15 $\leftarrow$ (21) 15 15 15 

(17) 7 11 15 15 

(25) 11 15 7 7 

(25) 19 7 7 7 

(26) 20 5 7 7 $\leftarrow$ (33) 27 3 3 

(27) 21 11 3 3 $\leftarrow$ (28) 22 13 3 

(31) 25 3 3 3 $\leftarrow$ (32) 26 5 3 

(35) 13 11 3 3 $\leftarrow$ (36) 14 13 3 

(42) 4 5 7 7 $\leftarrow$ (49) 11 3 3 

(47) 5 7 3 3 $\leftarrow$ (52) 6 5 3 

(47) 9 3 3 3 $\leftarrow$ (48) 10 5 3 

\end{tcolorbox}
\begin{tcolorbox}[title={$(65,6)$}]
(1) 10 9 15 15 15 

(3) 12 9 11 15 15 $\leftarrow$ (5) 22 21 11 7 

(3) 14 7 11 15 15 

(3) 28 25 3 3 3 $\leftarrow$ (4) 29 27 3 3 

(12) 5 7 11 15 15 $\leftarrow$ (16) 9 11 15 15 

(13) 4 7 11 15 15 $\leftarrow$ (18) 7 11 15 15 

(21) 11 14 5 7 7 $\leftarrow$ (22) 13 13 11 7 

(24) 9 11 7 7 7 

(25) 7 14 5 7 7 $\leftarrow$ (26) 11 15 7 7 

(25) 18 3 5 7 7 $\leftarrow$ (32) 25 3 3 3 

(26) 7 13 5 7 7 $\leftarrow$ (28) 17 7 7 7 

(27) 20 9 3 3 3 $\leftarrow$ (28) 21 11 3 3 

(30) 5 9 7 7 7 $\leftarrow$ (34) 13 5 7 7 

(35) 12 9 3 3 3 $\leftarrow$ (36) 13 11 3 3 

(41) 2 3 5 7 7 $\leftarrow$ (48) 9 3 3 3 

(44) 3 5 7 3 3 $\leftarrow$ (49) 8 3 3 3 

(45) 4 7 3 3 3 $\leftarrow$ (46) 6 6 5 3 

(47) 4 5 3 3 3 $\leftarrow$ (48) 5 7 3 3 

(51) 5 1 2 3 3 $\leftarrow$ (52) 6 2 3 3 

\end{tcolorbox}
\begin{tcolorbox}[title={$(65,7)$}]
(1) 25 7 13 5 7 7 $\leftarrow$ (4) 12 9 11 15 15 

(2) 11 4 7 11 15 15 $\leftarrow$ (4) 14 7 11 15 15 

(3) 9 5 7 11 15 15 

(3) 24 20 9 3 3 3 $\leftarrow$ (4) 28 25 3 3 3 

(7) 28 12 9 3 3 3 

(11) 2 4 7 11 15 15 $\leftarrow$ (14) 4 7 11 15 15 

(13) 5 7 11 15 7 7 $\leftarrow$ (17) 9 11 15 7 7 

(15) 5 9 15 7 7 7 $\leftarrow$ (21) 9 15 7 7 7 

(17) 9 7 13 5 7 7 $\leftarrow$ (18) 10 17 7 7 7 

(19) 11 5 9 7 7 7 $\leftarrow$ (21) 13 13 5 7 7 

(21) 7 14 4 5 7 7 $\leftarrow$ (22) 11 14 5 7 7 

(23) 12 12 9 3 3 3 $\leftarrow$ (28) 20 9 3 3 3 

(25) 5 5 9 7 7 7 $\leftarrow$ (26) 7 14 5 7 7 

(27) 3 5 9 7 7 7 $\leftarrow$ (29) 14 4 5 7 7 

(27) 8 12 9 3 3 3 $\leftarrow$ (36) 12 9 3 3 3 

(31) 5 7 3 5 7 7 $\leftarrow$ (35) 9 3 5 7 7 

(39) 3 3 6 6 5 3 $\leftarrow$ (43) 3 6 6 5 3 

(42) 2 3 5 7 3 3 $\leftarrow$ (48) 4 5 3 3 3 

(45) 2 4 5 3 3 3 $\leftarrow$ (46) 4 7 3 3 3 

(51) 3 4 4 1 1 1 $\leftarrow$ (52) 5 1 2 3 3 

(55) 1 * 1 $\leftarrow$ (57) * 1 

\end{tcolorbox}
\begin{tcolorbox}[title={$(65,8)$}]
(1) 10 2 4 7 11 15 15 

(1) 24 5 5 9 7 7 7 $\leftarrow$ (2) 25 7 13 5 7 7 

(2) 13 5 9 15 7 7 7 

(3) 22 17 3 6 6 5 3 $\leftarrow$ (4) 24 20 9 3 3 3 

(4) 14 6 9 11 7 7 7 

(5) 6 2 4 7 11 15 15 

(7) 26 9 3 6 6 5 3 $\leftarrow$ (8) 28 12 9 3 3 3 

(10) 1 2 4 7 11 15 15 $\leftarrow$ (12) 2 4 7 11 15 15 

(12) 3 5 9 15 7 7 7 $\leftarrow$ (16) 5 9 15 7 7 7 

(17) 8 5 5 9 7 7 7 $\leftarrow$ (18) 9 7 13 5 7 7 

(18) 9 3 5 9 7 7 7 $\leftarrow$ (20) 11 5 9 7 7 7 

(19) 8 3 5 9 7 7 7 $\leftarrow$ (22) 7 14 4 5 7 7 

(20) 7 8 12 9 3 3 3 $\leftarrow$ (26) 17 3 6 6 5 3 

(23) 10 9 3 6 6 5 3 $\leftarrow$ (24) 12 12 9 3 3 3 

(26) 3 5 9 3 5 7 7 $\leftarrow$ (34) 9 3 6 6 5 3 

(27) 6 9 3 6 6 5 3 $\leftarrow$ (28) 8 12 9 3 3 3 

(28) 3 5 7 3 5 7 7 $\leftarrow$ (32) 5 7 3 5 7 7 

(33) 6 3 3 6 6 5 3 $\leftarrow$ (40) 3 3 6 6 5 3 

(35) 10 2 4 5 3 3 3 $\leftarrow$ (36) 12 4 5 3 3 3 

(39) 6 2 4 5 3 3 3 $\leftarrow$ (46) 2 4 5 3 3 3 

(54) 1 1 * 1 $\leftarrow$ (56) 1 * 1 

\end{tcolorbox}
\begin{tcolorbox}[title={$(65,9)$}]
(1) 5 10 8 9 11 7 7 7 $\leftarrow$ (2) 10 2 4 7 11 15 15 

(1) 10 20 5 7 3 5 7 7 

(1) 17 9 3 5 9 7 7 7 $\leftarrow$ (2) 24 5 5 9 7 7 7 

(3) 9 3 5 9 15 7 7 7 

(3) 14 8 5 5 9 7 7 7 

(3) 21 11 12 4 5 3 3 3 $\leftarrow$ (4) 22 17 3 6 6 5 3 

(4) 2 4 3 4 7 11 15 15 $\leftarrow$ (9) 27 12 4 5 3 3 3 

(5) 5 1 2 4 7 11 15 15 $\leftarrow$ (6) 6 2 4 7 11 15 15 

(7) 25 5 5 3 6 6 5 3 $\leftarrow$ (8) 26 9 3 6 6 5 3 

(13) 4 8 5 5 9 7 7 7 $\leftarrow$ (18) 8 5 5 9 7 7 7 

(17) 4 6 3 5 9 7 7 7 $\leftarrow$ (18) 5 10 11 3 5 7 7 

(18) 4 5 3 5 9 7 7 7 $\leftarrow$ (20) 8 3 5 9 7 7 7 

(18) 8 3 5 9 3 5 7 7 $\leftarrow$ (25) 11 12 4 5 3 3 3 

(19) 9 3 5 7 3 5 7 7 $\leftarrow$ (33) 5 5 3 6 6 5 3 

(23) 9 5 5 3 6 6 5 3 $\leftarrow$ (24) 10 9 3 6 6 5 3 

(27) 3 6 5 2 3 5 7 7 $\leftarrow$ (28) 6 9 3 6 6 5 3 

(31) 5 6 2 3 5 7 3 3 $\leftarrow$ (37) 6 2 3 5 7 3 3 

(33) 3 6 2 3 5 7 3 3 $\leftarrow$ (34) 6 3 3 6 6 5 3 

(35) 7 13 1 * 1 $\leftarrow$ (36) 10 2 4 5 3 3 3 

(37) 6 ..4 5 3 3 3 $\leftarrow$ (43) 13 1 * 1 

(39) 4 ..4 5 3 3 3 $\leftarrow$ (40) 6 2 4 5 3 3 3 

\end{tcolorbox}
\begin{tcolorbox}[title={$(66,2)$}]
(3) 63 

(59) 7 

(63) 3 $\leftarrow$ (67) 

\end{tcolorbox}
\begin{tcolorbox}[title={$(66,3)$}]
(2) 61 3 

(3) 62 1 $\leftarrow$ (4) 63 

(35) 30 1 $\leftarrow$ (36) 31 

(51) 14 1 $\leftarrow$ (52) 15 

(58) 5 3 

(59) 6 1 $\leftarrow$ (60) 7 

(60) 3 3 $\leftarrow$ (66) 1 

(63) 2 1 $\leftarrow$ (64) 3 

\end{tcolorbox}
\begin{tcolorbox}[title={$(66,4)$}]
(1) 27 23 15 

(1) 59 3 3 

(3) 29 27 7 $\leftarrow$ (5) 31 31 

(3) 61 1 1 $\leftarrow$ (4) 62 1 

(25) 11 15 15 

(27) 21 11 7 $\leftarrow$ (29) 23 15 

(29) 23 7 7 $\leftarrow$ (33) 27 7 

(35) 13 11 7 $\leftarrow$ (37) 15 15 

(35) 29 1 1 $\leftarrow$ (36) 30 1 

(45) 7 7 7 $\leftarrow$ (49) 11 7 

(47) 5 7 7 $\leftarrow$ (53) 7 7 

(51) 13 1 1 $\leftarrow$ (52) 14 1 

(57) 3 3 3 

(58) 2 3 3 $\leftarrow$ (65) 1 1 

(59) 5 1 1 $\leftarrow$ (60) 6 1 

(63) 1 1 1 $\leftarrow$ (64) 2 1 

\end{tcolorbox}
\begin{tcolorbox}[title={$(66,5)$}]
(1) 26 21 11 7 $\leftarrow$ (2) 27 23 15 

(2) 27 23 7 7 $\leftarrow$ (4) 29 27 7 

(3) 60 1 1 1 $\leftarrow$ (4) 61 1 1 

(21) 14 13 11 7 $\leftarrow$ (22) 15 15 15 

(25) 10 13 11 7 $\leftarrow$ (26) 11 15 15 

(26) 19 7 7 7 $\leftarrow$ (28) 21 11 7 

(27) 20 5 7 7 $\leftarrow$ (30) 23 7 7 

(33) 14 5 7 7 $\leftarrow$ (36) 13 11 7 

(35) 28 1 1 1 $\leftarrow$ (36) 29 1 1 

(43) 4 5 7 7 $\leftarrow$ (46) 7 7 7 

(44) 3 5 7 7 $\leftarrow$ (48) 5 7 7 

(51) 12 1 1 1 $\leftarrow$ (52) 13 1 1 

(57) 1 2 3 3 $\leftarrow$ (64) 1 1 1 

(59) 4 1 1 1 $\leftarrow$ (60) 5 1 1 

\end{tcolorbox}
\begin{tcolorbox}[title={$(66,6)$}]
(1) 25 19 7 7 7 $\leftarrow$ (2) 26 21 11 7 

(2) 10 9 15 15 15 

(3) 56 4 1 1 1 $\leftarrow$ (4) 60 1 1 1 

(13) 5 7 11 15 15 $\leftarrow$ (17) 9 11 15 15 

(15) 7 13 13 11 7 $\leftarrow$ (22) 14 13 11 7 

(19) 7 11 15 7 7 

(25) 9 11 7 7 7 $\leftarrow$ (26) 10 13 11 7 

(26) 18 3 5 7 7 $\leftarrow$ (28) 20 5 7 7 

(27) 7 13 5 7 7 $\leftarrow$ (29) 17 7 7 7 

(31) 5 9 7 7 7 $\leftarrow$ (35) 13 5 7 7 

(33) 11 3 5 7 7 $\leftarrow$ (34) 14 5 7 7 

(35) 24 4 1 1 1 $\leftarrow$ (36) 28 1 1 1 

(42) 2 3 5 7 7 $\leftarrow$ (44) 4 5 7 7 

(45) 3 5 7 3 3 $\leftarrow$ (49) 5 7 3 3 

(49) 3 6 2 3 3 $\leftarrow$ (50) 8 3 3 3 

(51) 2 4 3 3 3 $\leftarrow$ (53) 6 2 3 3 

(51) 8 4 1 1 1 $\leftarrow$ (52) 12 1 1 1 

(55) 4 4 1 1 1 $\leftarrow$ (60) 4 1 1 1 

\end{tcolorbox}
\begin{tcolorbox}[title={$(66,7)$}]
(1) 6 11 7 11 15 15 

(3) 11 4 7 11 15 15 $\leftarrow$ (5) 14 7 11 15 15 

(3) 54 * 1 $\leftarrow$ (4) 56 4 1 1 1 

(4) 9 5 7 11 15 15 

(5) 20 9 11 7 7 7 

(5) 28 11 3 5 7 7 

(14) 5 7 11 15 7 7 $\leftarrow$ (19) 10 17 7 7 7 

(15) 6 9 15 7 7 7 $\leftarrow$ (16) 7 13 13 11 7 

(17) 8 9 11 7 7 7 $\leftarrow$ (18) 9 11 15 7 7 

(19) 6 9 11 7 7 7 $\leftarrow$ (20) 7 11 15 7 7 

(21) 12 11 3 5 7 7 $\leftarrow$ (22) 13 13 5 7 7 

(26) 5 5 9 7 7 7 $\leftarrow$ (28) 7 13 5 7 7 

(28) 3 5 9 7 7 7 $\leftarrow$ (32) 5 9 7 7 7 

(29) 13 2 3 5 7 7 $\leftarrow$ (30) 14 4 5 7 7 

(30) 5 9 3 5 7 7 $\leftarrow$ (36) 9 3 5 7 7 

(35) 22 * 1 $\leftarrow$ (36) 24 4 1 1 1 

(43) 2 3 5 7 3 3 $\leftarrow$ (44) 3 6 6 5 3 

(49) ..4 3 3 3 $\leftarrow$ (50) 3 6 2 3 3 

(50) 1 2 4 3 3 3 $\leftarrow$ (52) 2 4 3 3 3 

(51) 6 * 1 $\leftarrow$ (52) 8 4 1 1 1 

(52) 3 4 4 1 1 1 $\leftarrow$ (58) * 1 

(55) 2 * 1 $\leftarrow$ (56) 4 4 1 1 1 

\end{tcolorbox}
\begin{tcolorbox}[title={$(66,8)$}]
(1) 5 7 5 7 11 15 15 

(1) 12 8 9 15 7 7 7 $\leftarrow$ (4) 11 4 7 11 15 15 

(3) 13 5 9 15 7 7 7 

(3) 25 3 5 9 7 7 7 $\leftarrow$ (18) 8 9 11 7 7 7 

(3) 53 1 * 1 $\leftarrow$ (4) 54 * 1 

(5) 14 6 9 11 7 7 7 $\leftarrow$ (6) 20 9 11 7 7 7 

(5) 23 8 12 9 3 3 3 $\leftarrow$ (6) 28 11 3 5 7 7 

(11) 1 2 4 7 11 15 15 $\leftarrow$ (13) 2 4 7 11 15 15 

(13) 3 5 9 15 7 7 7 $\leftarrow$ (17) 5 9 15 7 7 7 

(15) 4 6 9 11 7 7 7 $\leftarrow$ (16) 6 9 15 7 7 7 

(19) 9 3 5 9 7 7 7 $\leftarrow$ (21) 11 5 9 7 7 7 

(21) 7 8 12 9 3 3 3 $\leftarrow$ (22) 12 11 3 5 7 7 

(27) 3 5 9 3 5 7 7 $\leftarrow$ (29) 8 12 9 3 3 3 

(29) 3 5 7 3 5 7 7 $\leftarrow$ (33) 5 7 3 5 7 7 

(29) 7 12 4 5 3 3 3 $\leftarrow$ (30) 13 2 3 5 7 7 

(31) 6 5 2 3 5 7 7 $\leftarrow$ (37) 12 4 5 3 3 3 

(35) 21 1 * 1 $\leftarrow$ (36) 22 * 1 

(49) 1 1 2 4 3 3 3 $\leftarrow$ (50) ..4 3 3 3 

(50) 2 3 4 4 1 1 1 $\leftarrow$ (57) 1 * 1 

(51) 5 1 * 1 $\leftarrow$ (52) 6 * 1 

(55) 1 1 * 1 $\leftarrow$ (56) 2 * 1 

\end{tcolorbox}
\begin{tcolorbox}[title={$(67,3)$}]
(1) 59 7 

(3) 61 3 $\leftarrow$ (5) 63 

(35) 29 3 $\leftarrow$ (37) 31 

(51) 13 3 $\leftarrow$ (53) 15 

(59) 5 3 $\leftarrow$ (61) 7 

(61) 3 3 $\leftarrow$ (65) 3 

\end{tcolorbox}
\begin{tcolorbox}[title={$(67,4)$}]
(1) 58 5 3 $\leftarrow$ (2) 59 7 

(2) 59 3 3 $\leftarrow$ (4) 61 3 

(5) 30 29 3 $\leftarrow$ (6) 31 31 

(29) 22 13 3 $\leftarrow$ (30) 23 15 

(33) 26 5 3 $\leftarrow$ (34) 27 7 

(34) 27 3 3 $\leftarrow$ (36) 29 3 

(37) 14 13 3 $\leftarrow$ (38) 15 15 

(49) 10 5 3 $\leftarrow$ (50) 11 7 

(50) 11 3 3 $\leftarrow$ (52) 13 3 

(53) 6 5 3 $\leftarrow$ (54) 7 7 

(58) 3 3 3 $\leftarrow$ (60) 5 3 

(59) 2 3 3 $\leftarrow$ (62) 3 3 

\end{tcolorbox}
\begin{tcolorbox}[title={$(67,5)$}]
(1) 57 3 3 3 $\leftarrow$ (2) 58 5 3 

(3) 27 23 7 7 $\leftarrow$ (5) 29 27 7 

(5) 29 27 3 3 $\leftarrow$ (6) 30 29 3 

(6) 22 21 11 7 

(19) 7 11 15 15 

(23) 13 13 11 7 

(27) 11 15 7 7 $\leftarrow$ (37) 13 11 7 

(27) 19 7 7 7 $\leftarrow$ (29) 21 11 7 

(29) 21 11 3 3 $\leftarrow$ (30) 22 13 3 

(33) 25 3 3 3 $\leftarrow$ (34) 26 5 3 

(37) 13 11 3 3 $\leftarrow$ (38) 14 13 3 

(45) 3 5 7 7 $\leftarrow$ (49) 5 7 7 

(47) 6 6 5 3 $\leftarrow$ (54) 6 5 3 

(49) 9 3 3 3 $\leftarrow$ (50) 10 5 3 

(58) 1 2 3 3 $\leftarrow$ (60) 2 3 3 

\end{tcolorbox}
\begin{tcolorbox}[title={$(67,6)$}]
(2) 25 19 7 7 7 $\leftarrow$ (4) 27 23 7 7 

(3) 10 9 15 15 15 

(5) 12 9 11 15 15 

(5) 28 25 3 3 3 $\leftarrow$ (6) 29 27 3 3 

(14) 5 7 11 15 15 $\leftarrow$ (18) 9 11 15 15 

(15) 4 7 11 15 15 $\leftarrow$ (20) 7 11 15 15 

(22) 9 15 7 7 7 

(23) 11 14 5 7 7 $\leftarrow$ (24) 13 13 11 7 

(26) 9 11 7 7 7 $\leftarrow$ (30) 17 7 7 7 

(27) 7 14 5 7 7 $\leftarrow$ (28) 11 15 7 7 

(27) 18 3 5 7 7 $\leftarrow$ (29) 20 5 7 7 

(29) 20 9 3 3 3 $\leftarrow$ (30) 21 11 3 3 

(34) 11 3 5 7 7 $\leftarrow$ (36) 13 5 7 7 

(37) 12 9 3 3 3 $\leftarrow$ (38) 13 11 3 3 

(43) 2 3 5 7 7 $\leftarrow$ (45) 4 5 7 7 

(46) 3 5 7 3 3 $\leftarrow$ (51) 8 3 3 3 

(47) 4 7 3 3 3 $\leftarrow$ (48) 6 6 5 3 

(49) 4 5 3 3 3 $\leftarrow$ (50) 5 7 3 3 

(53) 5 1 2 3 3 $\leftarrow$ (54) 6 2 3 3 

\end{tcolorbox}
\begin{tcolorbox}[title={$(67,7)$}]
(1) 19 7 11 15 7 7 

(2) 6 11 7 11 15 15 

(3) 25 7 13 5 7 7 

(5) 9 5 7 11 15 15 $\leftarrow$ (19) 9 11 15 7 7 

(5) 24 20 9 3 3 3 $\leftarrow$ (6) 28 25 3 3 3 

(9) 28 12 9 3 3 3 $\leftarrow$ (16) 4 7 11 15 15 

(15) 5 7 11 15 7 7 $\leftarrow$ (17) 7 13 13 11 7 

(19) 9 7 13 5 7 7 $\leftarrow$ (20) 10 17 7 7 7 

(20) 6 9 11 7 7 7 

(23) 7 14 4 5 7 7 $\leftarrow$ (24) 11 14 5 7 7 

(25) 12 12 9 3 3 3 $\leftarrow$ (28) 18 3 5 7 7 

(27) 5 5 9 7 7 7 $\leftarrow$ (28) 7 14 5 7 7 

(27) 17 3 6 6 5 3 $\leftarrow$ (30) 20 9 3 3 3 

(29) 3 5 9 7 7 7 $\leftarrow$ (33) 5 9 7 7 7 

(31) 5 9 3 5 7 7 $\leftarrow$ (37) 9 3 5 7 7 

(35) 9 3 6 6 5 3 $\leftarrow$ (38) 12 9 3 3 3 

(41) 3 3 6 6 5 3 $\leftarrow$ (44) 2 3 5 7 7 

(44) 2 3 5 7 3 3 $\leftarrow$ (50) 4 5 3 3 3 

(47) 2 4 5 3 3 3 $\leftarrow$ (48) 4 7 3 3 3 

(51) 1 2 4 3 3 3 $\leftarrow$ (53) 2 4 3 3 3 

(53) 3 4 4 1 1 1 $\leftarrow$ (54) 5 1 2 3 3 

\end{tcolorbox}
\begin{tcolorbox}[title={$(68,2)$}]
(67) 1 $\leftarrow$ (69) 

\end{tcolorbox}
\begin{tcolorbox}[title={$(68,3)$}]
(5) 62 1 $\leftarrow$ (6) 63 

(37) 30 1 $\leftarrow$ (38) 31 

(53) 14 1 $\leftarrow$ (54) 15 

(61) 6 1 $\leftarrow$ (62) 7 

(65) 2 1 $\leftarrow$ (66) 3 

(66) 1 1 $\leftarrow$ (68) 1 

\end{tcolorbox}
\begin{tcolorbox}[title={$(68,4)$}]
(3) 27 23 15 

(3) 59 3 3 $\leftarrow$ (5) 61 3 

(5) 61 1 1 $\leftarrow$ (6) 62 1 

(23) 15 15 15 $\leftarrow$ (39) 15 15 

(27) 11 15 15 

(31) 23 7 7 $\leftarrow$ (35) 27 7 

(35) 27 3 3 $\leftarrow$ (37) 29 3 

(37) 29 1 1 $\leftarrow$ (38) 30 1 

(47) 7 7 7 $\leftarrow$ (51) 11 7 

(51) 11 3 3 $\leftarrow$ (53) 13 3 

(53) 13 1 1 $\leftarrow$ (54) 14 1 

(59) 3 3 3 $\leftarrow$ (61) 5 3 

(61) 5 1 1 $\leftarrow$ (62) 6 1 

(65) 1 1 1 $\leftarrow$ (66) 2 1 

\end{tcolorbox}
\begin{tcolorbox}[title={$(68,5)$}]
(2) 57 3 3 3 $\leftarrow$ (4) 59 3 3 

(3) 26 21 11 7 $\leftarrow$ (4) 27 23 15 

(5) 60 1 1 1 $\leftarrow$ (6) 61 1 1 

(7) 22 21 11 7 

(23) 14 13 11 7 $\leftarrow$ (24) 15 15 15 

(27) 10 13 11 7 $\leftarrow$ (28) 11 15 15 

(28) 19 7 7 7 $\leftarrow$ (32) 23 7 7 

(34) 25 3 3 3 $\leftarrow$ (36) 27 3 3 

(35) 14 5 7 7 $\leftarrow$ (48) 7 7 7 

(37) 28 1 1 1 $\leftarrow$ (38) 29 1 1 

(46) 3 5 7 7 $\leftarrow$ (50) 5 7 7 

(50) 9 3 3 3 $\leftarrow$ (52) 11 3 3 

(53) 12 1 1 1 $\leftarrow$ (54) 13 1 1 

(59) 1 2 3 3 $\leftarrow$ (61) 2 3 3 

(61) 4 1 1 1 $\leftarrow$ (62) 5 1 1 

\end{tcolorbox}
\begin{tcolorbox}[title={$(68,6)$}]
(1) 19 7 11 15 15 

(1) 23 13 13 11 7 

(3) 25 19 7 7 7 $\leftarrow$ (4) 26 21 11 7 

(4) 10 9 15 15 15 

(5) 56 4 1 1 1 $\leftarrow$ (6) 60 1 1 1 

(6) 12 9 11 15 15 $\leftarrow$ (8) 22 21 11 7 

(6) 14 7 11 15 15 

(15) 5 7 11 15 15 $\leftarrow$ (19) 9 11 15 15 

(21) 7 11 15 7 7 

(23) 9 15 7 7 7 $\leftarrow$ (25) 13 13 11 7 

(23) 13 13 5 7 7 $\leftarrow$ (30) 20 5 7 7 

(27) 9 11 7 7 7 $\leftarrow$ (28) 10 13 11 7 

(29) 7 13 5 7 7 $\leftarrow$ (37) 13 5 7 7 

(31) 14 4 5 7 7 $\leftarrow$ (46) 4 5 7 7 

(35) 11 3 5 7 7 $\leftarrow$ (36) 14 5 7 7 

(37) 24 4 1 1 1 $\leftarrow$ (38) 28 1 1 1 

(45) 3 6 6 5 3 $\leftarrow$ (51) 5 7 3 3 

(47) 3 5 7 3 3 $\leftarrow$ (49) 6 6 5 3 

(51) 3 6 2 3 3 $\leftarrow$ (52) 8 3 3 3 

(53) 8 4 1 1 1 $\leftarrow$ (54) 12 1 1 1 

(57) 4 4 1 1 1 $\leftarrow$ (60) 1 2 3 3 

(59) * 1 $\leftarrow$ (62) 4 1 1 1 

\end{tcolorbox}
\begin{tcolorbox}[title={$(69,3)$}]
(3) 59 7 

(7) 31 31 

(31) 23 15 $\leftarrow$ (39) 31 

(55) 7 7 

(63) 3 3 $\leftarrow$ (67) 3 

(67) 1 1 $\leftarrow$ (69) 1 

\end{tcolorbox}
\begin{tcolorbox}[title={$(69,4)$}]
(3) 58 5 3 $\leftarrow$ (4) 59 7 

(6) 29 27 7 

(7) 30 29 3 $\leftarrow$ (8) 31 31 

(30) 21 11 7 $\leftarrow$ (38) 29 3 

(31) 22 13 3 $\leftarrow$ (32) 23 15 

(35) 26 5 3 $\leftarrow$ (36) 27 7 

(38) 13 11 7 

(39) 14 13 3 $\leftarrow$ (40) 15 15 

(51) 10 5 3 $\leftarrow$ (52) 11 7 

(55) 6 5 3 $\leftarrow$ (56) 7 7 

(60) 3 3 3 $\leftarrow$ (64) 3 3 

(66) 1 1 1 $\leftarrow$ (68) 1 1 

\end{tcolorbox}
\begin{tcolorbox}[title={$(69,5)$}]
(1) 27 11 15 15 

(3) 57 3 3 3 $\leftarrow$ (4) 58 5 3 

(5) 27 23 7 7 

(7) 29 27 3 3 $\leftarrow$ (8) 30 29 3 

(21) 7 11 15 15 

(24) 14 13 11 7 $\leftarrow$ (37) 27 3 3 

(29) 11 15 7 7 

(29) 19 7 7 7 $\leftarrow$ (33) 23 7 7 

(31) 17 7 7 7 $\leftarrow$ (49) 7 7 7 

(31) 21 11 3 3 $\leftarrow$ (32) 22 13 3 

(35) 25 3 3 3 $\leftarrow$ (36) 26 5 3 

(39) 13 11 3 3 $\leftarrow$ (40) 14 13 3 

(47) 3 5 7 7 $\leftarrow$ (51) 5 7 7 

(51) 9 3 3 3 $\leftarrow$ (52) 10 5 3 

(55) 6 2 3 3 $\leftarrow$ (62) 2 3 3 

\end{tcolorbox}
\begin{tcolorbox}[title={$(70,2)$}]
(7) 63 

(55) 15 

(63) 7 $\leftarrow$ (71) 

\end{tcolorbox}
\begin{tcolorbox}[title={$(70,3)$}]
(6) 61 3 

(7) 62 1 $\leftarrow$ (8) 63 

(39) 30 1 $\leftarrow$ (40) 31 

(54) 13 3 

(55) 14 1 $\leftarrow$ (56) 15 

(62) 5 3 $\leftarrow$ (70) 1 

(63) 6 1 $\leftarrow$ (64) 7 

(67) 2 1 $\leftarrow$ (68) 3 

\end{tcolorbox}
\begin{tcolorbox}[title={$(70,4)$}]
(1) 55 7 7 

(5) 27 23 15 

(5) 59 3 3 

(7) 29 27 7 $\leftarrow$ (9) 31 31 

(7) 61 1 1 $\leftarrow$ (8) 62 1 

(25) 15 15 15 $\leftarrow$ (37) 27 7 

(29) 11 15 15 

(31) 21 11 7 $\leftarrow$ (33) 23 15 

(39) 13 11 7 $\leftarrow$ (41) 15 15 

(39) 29 1 1 $\leftarrow$ (40) 30 1 

(53) 11 3 3 

(55) 13 1 1 $\leftarrow$ (56) 14 1 

(56) 6 5 3 $\leftarrow$ (69) 1 1 

(61) 3 3 3 $\leftarrow$ (65) 3 3 

(63) 5 1 1 $\leftarrow$ (64) 6 1 

(67) 1 1 1 $\leftarrow$ (68) 2 1 

\end{tcolorbox}

.
\end{appendices}

\end{document}